%% file: suleimanov-talanov.tex
\documentclass[12pt]{amsart}

\usepackage{fullpage}

\usepackage{mathpazo} 

\usepackage{mathabx}

\usepackage{amssymb}
\usepackage{bbm}
\usepackage{comment}
\usepackage{epstopdf}
\usepackage{enumerate}
\usepackage{float}
\usepackage{graphicx}
\usepackage{overpic}
\usepackage{mathrsfs}
\usepackage{mathtools}
\usepackage{multirow}
\usepackage{stmaryrd}
\usepackage{rotating} 
\usepackage[dvipsnames]{xcolor}
\usepackage{tikz}  
\usetikzlibrary{decorations.markings}
\usetikzlibrary{shapes,decorations,calc,fit}
\usetikzlibrary{arrows.meta}
\tikzset{->-/.style={decoration={ markings, mark=at position #1 with {\arrow[scale=1.4]{>}}},postaction={decorate}}}
\tikzset{-<-/.style={decoration={ markings, mark=at position #1 with {\arrow[scale=1.4]{<}}},postaction={decorate}}}

\usepackage{xcolor}

\usepackage{thmtools}
\usepackage{thm-restate}

\usepackage[pagebackref,hypertexnames=false]{hyperref}

\numberwithin{equation}{section}

\newcommand{\dd}{\mathrm{d}}
\newcommand{\ee}{\mathrm{e}}
\newcommand{\ii}{\mathrm{i}}
\newcommand{\SP}{\lambda}
\newcommand{\phaseint}{\Phi}
\newcommand{\tailint}{\Xi}
\newcommand{\evdensity}{\varrho}
\newcommand{\C}{\mathbb{C}} 
\newcommand{\R}{\mathbb{R}}

\newcommand{\Z}{\mathbb{Z}}
\newcommand{\eps}{\epsilon}

\newcommand{\rr}{r}
\newcommand{\bigo}[1]{ \mathcal{O}\! \left( #1 \right) }

\def\eq{\begin{equation}}
\def\endeq{\end{equation}}

\newcommand{\triu}[2][1]{\begin{pmatrix} #1 & #2 \\ 0 & #1 \end{pmatrix}}
\newcommand{\tril}[2][1]{\begin{pmatrix} #1 & 0 \\ #2 & #1 \end{pmatrix}}
\newcommand{\diag}[2]{\begin{pmatrix} #1 & 0 \\ 0 & #2 \end{pmatrix}}

\DeclareMathOperator{\imag}{Im}
\DeclareMathOperator{\re}{Re}
\DeclareMathOperator{\Ai}{ {Ai} }

\DeclareMathOperator{\sgn}{sgn}

\renewcommand{\Re}{\operatorname{Re}}
\renewcommand{\Im}{\operatorname{Im}}

\declaretheorem[style=plain, name=Theorem, numberwithin=section]{thm}
[style=plain, name=Theorem, numbered = no]
\declaretheorem{cor}[style=plain, name=Corollary, numberwithin=section]
\declaretheorem{lem}[style=plain, name=Lemma, numberwithin=section]
\declaretheorem{prop}[style=plain, name=Proposition, numberwithin=section]
\declaretheorem{myrhp}[style=plain, name=Riemann-Hilbert Problem, numberwithin=section]

\declaretheorem{definition}[style=definition, name=Definition, numberwithin=section]
[style=definition, name=Definition, numbered=no]
[style=definition, name=Hypothesis, numberwithin=section]
[style=definition, name=Claim, numberwithin=section]
\declaretheorem{note}[style=definition, name=Note, numbered = yes]

\declaretheorem{rem}[style=remark, name=Remark, numberwithin=section]

\title{Suleimanov-Talanov Self-Focusing and the Hierarchy of the Focusing Nonlinear Schr\"odinger Equation}
\author{R. J. Buckingham}
\address{Department of Mathematical Sciences, University of Cincinnati}
\email{buckinrt@uc.edu}
\author{R. M. Jenkins}
\address{Department of Mathematics, University of Central Florida}
\email{robert.jenkins@ucf.edu}
\author{P. D. Miller}
\address{Department of Mathematics, University of Michigan}
\email{millerpd@umich.edu}

\begin{document}
\begin{abstract}
We study the self-focusing of wave packets from the point of view of the semiclassical focusing nonlinear Schr\"odinger 
equation.  A type of finite-time collapse/blowup of the solution of the associated dispersionless limit was investigated by 
Talanov in the 1960's, and recently Suleimanov identified a special solution of the dispersive problem that formally 
regularizes the blowup and is related to the hierarchy of the Painlev\'e-III equation.  In this paper we approximate 
the Talanov solutions in the full dispersive equation using a semiclassical soliton ensemble, a sequence of exact 
reflectionless solutions for a corresponding sequence of values of the semiclassical parameter $\epsilon$ tending to zero, 
approximating the Talanov initial data more and more accurately in the limit $\epsilon\to 0$.  In this setting, we 
rigorously establish the validity of the dispersive saturation of the Talanov blowup obtained by Suleimanov.  We extend the 
result to the full hierarchy of higher focusing nonlinear Schr\"odinger equations, exhibiting new generalizations of the 
Talanov initial data that produce such dispersively regularized extreme focusing in both mixed and pure flows.  We also 
argue that generic perturbations of the Talanov initial data lead to a different singularity of the dispersionless limit, 
namely a gradient catastrophe for which the dispersive regularization is instead based on the tritronqu\'ee solution of the 
Painlev\'e-I equation and the Peregrine breather solution which appears near points in space time corresponding to the 
poles of the former transcendental function as shown by Bertola and Tovbis.  
\vspace{-.4in}
\end{abstract}
\maketitle

{\small\tableofcontents}

\section{Introduction}
\label{sec-intro}
\label{sec:ASK}
\input{sec-intro}

\noindent
{\bf Acknowledgements.}
R. J. Buckingham was supported by the National Science Foundation under grant DMS 2108019. R. M. Jenkins was supported by the National Science Foundation under grant DMS-2307142 and by the Simons Foundation under grant 853620.  P. D. Miller was supported by the National Science Foundation under grants DMS-1513054, DMS-1812625, and DMS-2204896.

The authors would also like to thank the Isaac Newton Institute for Mathematical Sciences, Cambridge, for support and hospitality during the programmes``Dispersive hydrodynamics: mathematics, simulation and experiments, with applications in nonlinear waves'' (2022) and ``Emergent phenomena in nonlinear dispersive waves'' (2024), where work on this paper was undertaken. This work was supported by EPSRC grant EP/R014604/1. 

\section{Preliminary material and results}

\subsection{The NLS hierarchy}
\label{sec-hierarchy}
\input{sec-hierarchy}

\subsection{Singularities of solutions of dispersionless focusing NLS}
\label{sec-dispersionless}
\input{sec-dispersionless}

\subsection{Semiclassical soliton ensembles for initial data generalizing Talanov profiles}
\label{sec-setup}
\input{sec-setup}

\subsection{Results}
\label{sec-results}
\input{sec-results}

\subsection{Nongeneric character of Suleimanov-Talanov focusing}
\label{sec-interpolate}
\input{sec-interpolate}

\section{General semicircular Klaus-Shaw potentials}
\label{sec-gen-scKS}
\input{sec-gen-scKS}

\section{Proof of Theorem~\ref{thm-accuracy-t=0} for $x\in (X_-,X_+)\setminus\{x_0\}$}
\label{sec-zakharov-shabat}
\input{sec-zakharov-shabat}

\subsection{Steepest-descent analysis for $x \in J^+$}
\label{sec-steepest-descent}
\input{sec-steepest-descent}

\subsection{Parametrix construction}
\label{sec-parametrix}
\input{sec-parametrix}

\subsection{Accuracy of the global parametrix}
\label{sec-accuracy}
\input{sec-parametrix-accuracy}

\subsection{Alteration of the nonlinear steepest descent calculation for $x< x_0$}
\label{sec-x<x_0}
\input{sec-x-less-than-x_0}

\subsection{Application of small-norm theory}
\label{sec-proof-Thm1}
\input{sec-proof-Thm1}

\section{Proof of Theorem~\ref{thm-accuracy-t=0} for $x<X_-$ and $x>X_+$}
\label{sec-proof-Thm1-outside}
\input{sec-proof-Thm1-outside}

\section{Dispersive regularization of Talanov focusing}
\label{sec-Suleimanov}
\input{sec-Suleimanov}

\appendix

\section{Proofs of asymptotic properties of $Y_\epsilon(\SP)$ and $T_\epsilon(\SP)$.}
\label{sec-proofs}
\input{sec-proofs}

\end{document}

%% file: sec-intro.tex
The focusing nonlinear Schr\"odinger (NLS) equation 
\eq
\label{nls}
\ii\epsilon \psi_{t_2} + \frac{1}{2}\epsilon^2\psi_{xx} + |\psi|^2\psi = 0
\endeq
is a standard model for time-dependent complex-valued fields in one-dimensional physical 
systems exhibiting both dispersion and cubic nonlinearity.  It is this 
competition between dispersion and nonlinearity in the NLS equation that leads 
mathematically to a host of phenomena including solitons, wave breaking and 
dispersive shock waves, and rogue waves.  Classically, one way to attempt to 
tease out the separate effects of dispersion and nonlinearity is to first 
consider the \emph{dispersionless focusing NLS system}.  
In particular, if we rewrite the focusing NLS equation \eqref{nls} with 
$t:=t_2$ in Madelung coordinates by introducing 
$\psi=\rho^{1/2}\ee^{\ii S/\epsilon}$ for a density $\rho\ge 0$ and phase 
$S\in\mathbb{R}$, then with the momentum defined by $\mu:=\rho S_x$ the 
equation \eqref{nls} becomes exactly the coupled system
\eq
\rho_t + \mu_x=0,\quad \mu_t + \left(\frac{\mu^2}{\rho}-\frac{\rho^2}{2}\right)_x= \frac{\epsilon^2}{4}\left(\rho\left(\frac{\rho_x}{\rho}\right)_x\right)_x.
\endeq
Thinking of $\epsilon$ as a measure of the strength of dispersive effects, 
setting $\epsilon=0$ leads to the dispersionless focusing NLS system 
\begin{equation}
\rho_t + \mu_x=0,\quad \mu_t + \left(\frac{\mu^2}{\rho}-\frac{\rho^2}{2}\right)_x= 0.
\label{eq:dispersionless-focusing}
\end{equation}
The reason behind the ``focusing'' nomenclature is easily seen by considering 
the \emph{Akhmanov-Sukhorukov-Khokhlov solution} \cite{AkhmanovSK66} 
of \eqref{eq:dispersionless-focusing} with the natural initial conditions 
\begin{equation}
\rho(x,0)=A_\mathrm{max}^2\mathrm{sech}^2(x), \quad \mu(x,0)=0.
\end{equation}
Here $A_\mathrm{max}>0$ a fixed parameter.  The solution is determined 
implicitly by the equations
\begin{equation}
\mu=-2t\rho^2\tanh\left(x-\frac{\mu}{\rho}t\right),\quad \rho=(A_\mathrm{max}^2+t^2\rho^2)\mathrm{sech}^2\left(x-\frac{\mu}{\rho}t\right).
\label{eq:ASK-implicit}
\end{equation}
The equations can be solved explicitly for $x=0$, giving 
\begin{equation}
\rho(0,t)=\frac{1-\sqrt{1-4A_\mathrm{max}^2t^2}}{2t^2},\quad |t|<\frac{1}{2A_\mathrm{max}}.  
\end{equation}
The solution exhibits a 
finite-amplitude gradient catastrophe at $(x,t)=(0,\pm 1/(2A_\mathrm{max}))$ 
with value $\rho=2A_\mathrm{max}^2$ and cannot be continued in 
any smooth way outside the indicated interval.  See 
Figure~\ref{fig:Talanov-ASK}.  On the other hand, the solution to the full 
NLS equation \eqref{nls} with corresponding initial condition 
$\psi(x,0)=A_\mathrm{max}\mathrm{sech}(x)$ is known as the 
\emph{Satsuma-Yajima solution} \cite{SatsumaY74}, which exists for all time.  
To compare to the dispersionless focusing NLS system, it makes sense to 
consider the so-called \emph{zero-dispersion} or \emph{semiclassical limit} 
$\epsilon\downarrow 0$.  For the sequence 
$\{\epsilon=\frac{A_\mathrm{max}}{N},\ N=1,2,3,...\}$, it has been proven 
\cite{Kamvissis:2003} that the Satsuma-Yajima solution indeed converges to the 
Akhmanov-Sukhorukov-Khokhlov solution in the semiclassical limit for 
$|t|<1/(2A_\mathrm{max})$.  However, starting at $t=1/(2A_\mathrm{max})$, for 
small $\epsilon$ the Satsuma-Yajima solution displays a marked phase transition beyond 
a certain well-defined caustic curve.  Across this curve the solution's 
amplitude suddenly switches from behavior that is (asymptotically) 
independent of $\epsilon$ to rapid oscillations of wavelength and frequency proportional 
to $\epsilon$ \cite{Kamvissis:2003}.  

In this work we are especially interested in the solution near the first 
point of wave breaking.  The dispersive regularization under the full 
focusing NLS dynamics of the type of singularity appearing in the 
Akhmanov-Sukhorukov-Khokhlov solution (known in catastrophe theory as 
an \emph{elliptic umbilic catastrophe}) was studied by Bertola and Tovbis 
\cite{BertolaT13}. (For generalizations to other similar problems, see 
\cite{DubrovinGKM15} and \cite{LuM22}).  They found that in a neighborhood of 
the elliptic umbilic catastrophe point 
the dispersive solution behaves like a multiscale structure consisting of a 
mesoscale (spacetime scales proportional to $\epsilon^{4/5}$) background profile described by the tritronqu\'ee solution of the Painlev\'e-I equation except near certain points corresponding to its poles that form a 
curvilinear spacetime lattice; near each lattice point one has instead an 
approximation on the microscale (spacetime scales proportional to $\epsilon$) by a copy of the Peregrine breather (rogue 
wave) solution \cite{Peregrine83}.  It is generally understood from numerical 
and analytical studies that this Bertola-Tovbis regularization is, in some 
sense, the ``generic'' breaking behavior for the focusing NLS equation.  
However, it is not the only type of breaking behavior, and it is exactly such 
non-generic breaking behavior that interests us here.  Returning to the 
dispersionless focusing NLS system \eqref{eq:dispersionless-focusing}, we turn our attention to the 
\emph{Talanov solutions} \cite{Talanov65}, which were actually 
discovered slightly before the Akhmanov-Sukhorukov-Khokhlov solution.  Let 
$E\in\mathbb{R}$ and $F>0$ be fixed constants and choose $w(t)$ to be a 
solution of 
\eq
\frac{1}{2}w'(t)^2 -\frac{2F}{w(t)}=E.  
\endeq
Then the Talanov solution is 
\eq
\begin{split}
& \rho(x,t)=Fw(t)^{-3}(w(t)^2-x^2)\chi_{[-w(t),w(t)]}(x), \\
& \mu(x,t)=Fw(t)^{-4}w'(t)x(w(t)^2-x^2)\chi_{[-w(t),w(t)]}(x),
\end{split}
\label{eq:Talanov-solution}
\endeq
where $\chi_{[a,b]}$ is the standard indicator function on $[a,b]$.  In 
particular, the density $\rho(x,t)$ has the form of a 
cutoff parabolic profile corresponding to a semicircular amplitude profile of 
(half) width $w(t)>0$ and maximum value $Fw(t)^{-1}$.  See Section~\ref{sec:Talanov} 
for more details.  The behavior of the solution depends on $E$ as follows.
\begin{itemize}
\item If $E=0$, then $w(t)=(9F)^\frac{1}{3}(t^\circ-t)^\frac{2}{3}$ for an 
arbitrary integration constant $t^\circ$.  In particular, if $F>0$ then 
as $t\to -\infty$ the width grows without bound and the amplitude 
decays to zero monotonically.  On the other hand, the solution only exists 
for $t<t^\circ$ and collapses to zero width ($w(t)\downarrow 0$) and infinite 
amplitude ($Fw(t)^{-1}\uparrow \infty$) as $t\uparrow t^\circ$.  
\item If $E>0$ then it is not possible to solve for $w(t)$ explicitly.  
However, the qualitative behavior is similar to the $E=0$ case.
\item If $E<0$ then the width no longer changes monotonically.  In this case, 
there is a unique time at which $w'(t)=0$ and $w(t)$ is maximized;  we choose this time to be $t=0$ 
by choice of the integration constant.  The solution now only exists on the 
time interval $-t^\circ < t < t^\circ$, where
$2t^\circ = \sqrt{2}\pi F(-E)^{-\frac{3}{2}}$.  The solution 
collapses to zero width and infinite amplitude as $t\downarrow-t^\circ$ or 
$t\uparrow t^\circ$.   
\end{itemize}
The case $E=0$ was recently considered by Suleimanov \cite{Suleimanov17}, 
where he proposed that the dispersive terms in the focusing NLS equation 
serve to arrest the collapse in a specific fashion related to solutions of 
the third Painlev\'e equation and its hierarchy.  For more details see Section~\ref{sec:extreme}.  

In this work we consider the dispersive regularization of a class of 
functions we call \emph{semicircular Klaus-Shaw potentials} (see 
Section~\ref{sec:semiclassical}) that include the Talanov solutions with $E<0$ as 
a special case.  To take advantage of the integrable structure of the 
focusing NLS equation, we use the \emph{semiclassical soliton ensemble} 
approach \cite{Kamvissis:2003}.  Rather than studying a semicircular 
Klaus-Shaw potential directly, we first approximate it by a sequence of 
reflectionless (pure soliton) solutions that converges when $t=0$ to the 
desired initial condition as $\epsilon\downarrow 0$ (see 
Theorem~\ref{thm-accuracy-t=0}).  Then, through the 
formulation of a Riemann-Hilbert problem and the use of the Deift-Zhou 
nonlinear steepest-descent method, we prove rigorously that the local behavior 
in the semiclassical limit at the focusing point is that proposed by 
Suleimanov (see Theorem~\ref{thm:multi-time}).  The following table 
summarizes the relationship of our results on the focusing NLS equation 
to the existing literature.  
\begin{center}
\begin{tabular}{|c|c|c|}
\hline
\parbox{0.4\linewidth}{\small\textbf{Motivating Solution of the \\Dispersionless NLS System}} &
\parbox{0.34\linewidth}{\small\textbf{Solution of the Full NLS \\Equation}} &
\parbox{0.19\linewidth}{\small\textbf{Regularization}}\\
\hline
\parbox{0.4\linewidth}{\small Akhmanov-Sukhorukov-Khokhlov \cite{AkhmanovSK66}} &
\parbox{0.34\linewidth}{\small Satsuma-Yajima \cite{SatsumaY74}} & 
\parbox{0.19\linewidth}{\small Bertola-Tovbis \cite{BertolaT13}}\\
\hline
\parbox{0.4\linewidth}{\small Talanov ($E<0$) \cite{Talanov65}} &
\parbox{0.34\linewidth}{\small Semiclassical soliton ensemble approx.\@ of semicircular Klaus-Shaw potentials (this work)} &
\parbox{0.19\linewidth}{\small Suleimanov \\(this work)}\\
\hline
\end{tabular}
\end{center}
In addition, we extend our results to other equations in the focusing NLS 
hierarchy.  See Theorems~\ref{thm:multi-time}, \ref{thm:mixture}, and 
\ref{thm:pure-flow} below.  We also discuss the non-genericity 
of Suleimanov-Talanov focusing in Section~\ref{sec-interpolate}, where we show 
that such behavior can be easily perturbed into the type of dispersive 
regularization studied by Bertola and Tovbis.  

\begin{note}
Another recent study of the Talanov solution and its perturbations is the paper \cite{DemontisORS23}, in which the authors 
review the Talanov theory and then consider the fully dispersive NLS equation numerically with a version of the Talanov 
initial data that is artificially smoothed at the corner points where the field meets the vacuum.  
\end{note}

\begin{note}
An interesting open question is the behavior of Talanov-type initial data for 
pulses with chirp (a linear phase gradient).  We will investigate the image 
in the scattering transform domain of such functions in future work.  
\end{note}

\begin{note}
The semicircular Klaus-Shaw initial conditions we study here are 
also important in a 
study of the \emph{three-wave resonant interaction equations} in the 
semiclassical limit.  This is a coupled system of three equations with a 
$3\times 3$ Lax pair.  It happens that if the three fields are initially 
disjointly supported then the $x$-equation in the Lax pair reduces at each $x$-value to a 
$2\times 2$ Zakharov-Shabat eigenvalue problem (tensored with a scalar).  
This is the well-studied $x$-equation in the Lax pair for the NLS hierarchy.  
Therefore, if the three-wave resonant interaction equations are posed with 
disjointly supported initial data of the form \eqref{cauchy-data} in each channel, then the 
scattering data can be determined by (i) performing a local analysis on each 
Zakharov-Shabat operator to determine the \emph{exceptional points}, and 
(ii) a global analysis to determine the associated \emph{connection 
coefficients} \cite{Buckingham:2016}.  It is then possible to show using existing theory \cite{Kamvissis:2003} that the semiclassical soliton 
ensemble converges to the original initial data as $\epsilon\downarrow 0$ 
at $t=0$, provided that the amplitude of each packet vanishes to sufficiently high order at the support endpoints.  One aim of this paper is to address the technical challenges in the forward-scattering 
and inverse-scattering steps when the initial data vanishes at the support endpoints like a 
square root.  
The results of our work in this direction are described in Theorem~\ref{thm-accuracy-t=0} and Corollary~\ref{cor-L2-convergence} below.
In future work, we will apply Theorem~\ref{thm-accuracy-t=0} to the 
three-wave semiclassical soliton ensembles defined in \cite{Buckingham:2016} 
to prove convergence at $t=0$ for ensembles corresponding to certain 
disjointly supported initial packets.  
\end{note}

%% file: sec-hierarchy.tex
 The $m^\mathrm{th}$ flow in the focusing NLS hierarchy---which we will generically refer to as the $\mathrm{NLS}_m$ equation---can be written in compact form as $\epsilon\psi_{t_m}=N_m[\psi,\psi^*]$ where $N_m$ is a polynomial in its arguments and their scaled $x$-derivatives ($\epsilon\partial_x$), normalized so that the coefficient of $(\epsilon\partial_x)^m\psi$ (the highest derivative) is $(\frac{1}{2}\ii)^{m-1}$. The $\mathrm{NLS}_3$ equation is better known as the complex modified Korteweg-de Vries (mKdV) equation
\begin{equation}
\epsilon\psi_{t_3} + \frac{3}{2}\epsilon|\psi|^2\psi_x +\frac{1}{4}\epsilon^3\psi_{xxx}=0,
\label{eq:mKdV}
\end{equation}
while $\mathrm{NLS}_4$ takes the form
\begin{equation}
\ii\epsilon\psi_{t_4}-\frac{1}{8}\epsilon^4\psi_{xxxx}-\epsilon^2 |\psi|^2\psi_{xx}-\frac{1}{4}\epsilon^2\psi^2\psi^*_{xx}-\frac{1}{2}\epsilon^2\psi\psi_x\psi^*_x-\frac{3}{4}\epsilon^2\psi^*\psi_x^2 -\frac{3}{4}|\psi|^4\psi=0.
\label{eq:LPD}
\end{equation}
With index omitted NLS will always refer to \eqref{nls}, the $\mathrm{NLS}_2$ equation.

It is well known that all of the equations in the focusing NLS hierarchy can be simultaneously solved, that is, there is a well-defined function $\psi(x,t_2,t_3,\dots,t_M)$ with suitable given initial condition $\psi(x,0,0,\dots,0)=\psi_0(x)$ such that $\epsilon\psi_{t_m}=N_m[\psi,\psi^*]$ holds for each $m=2,3,\dots,M$.  By restricting the times to be proportional by given constants to a single independent variable $t$, i.e. $t_m=a_mt$, we see that as a function of $(x,t)$, $\psi$ satisfies a mixture of the flows:  $\epsilon\psi_t = a_2N_2[\psi,\psi^*] + a_3N_3[\psi,\psi^*] + \cdots + a_MN_M[\psi,\psi^*]$. Some of these mixtures have their own names in the literature. For instance, the combination $\epsilon\psi_t=a_2N_2[\psi,\psi^*]+a_3N_3[\psi,\psi^*]$ is frequently called the Hirota equation and is written in the form
\begin{equation}
\ii\epsilon\psi_t +a_2\left[\frac{1}{2}\epsilon^2\psi_{xx}+|\psi|^2\psi\right] + a_3\left[\ii \frac{3}{2}\epsilon|\psi|^2\psi_x +\ii \frac{1}{4} \eps^3\psi_{xxx}\right]=0.
\label{eq:Hirota}
\end{equation}
Similarly, the mixture $\epsilon\psi_t = a_2N_2[\psi,\psi^*]+a_4 N_4[\psi,\psi^*]$ yields the  Lakshmanan-Porsezian-Daniel (LPD) equation
\begin{multline}
\ii\epsilon\psi_t 
+ a_2\left[ \frac{1}{2}\epsilon^2\psi_{xx}+|\psi|^2\psi \right]  \\
+ a_4\left[ -\frac{1}{8}\epsilon^4\psi_{xxxx}-\epsilon^2 |\psi|^2\psi_{xx}-\frac{1}{4}\epsilon^2\psi^2\psi^*_{xx}-\frac{1}{2}\epsilon^2\psi\psi_x\psi^*_x-\frac{3}{4}\epsilon^2\psi^*\psi_x^2 -\frac{3}{4}|\psi|^4\psi \right]=0.
\label{eq:gLPD}
\end{multline}
Like \eqref{eq:mKdV} and \eqref{eq:LPD}, these equations can be viewed as models for ultrashort pulses propagating in optical fibers.

%% file: sec-dispersionless.tex
\subsubsection{The Talanov solutions}
\label{sec:Talanov}
Following Talanov \cite{Talanov65} (see also \cite[Section II]{DemontisORS23} for a recent review), we seek a solution of the dispersionless system \eqref{eq:dispersionless-focusing} for which 
the density (squared amplitude) $\rho(x,t)$ has the form of a cutoff parabolic profile corresponding to a semicircular amplitude profile of width $w(t)>0$ 
and maximum value $f(t)w(t)^2>0$:
\eq
\rho(x,t)=f(t)(w(t)^2-x^2)\chi(x,t),\quad\chi(x,t):=\chi_{[-w(t),w(t)]}(x).
\label{eq:rho-ansatz}
\endeq
Since the conservation law on $\rho$ in the system \eqref{eq:dispersionless-focusing} implies that the integral over $x$ of $\rho$ is conserved, we compute:
\eq
\int_{\mathbb{R}}\rho(x,t)\,\dd x=f(t)\int_{-w(t)}^{w(t)}(w(t)^2-x^2)\,\dd x =\frac{4}{3}f(t)w(t)^3
\endeq
so for conservation we require that $f(t)=Fw(t)^{-3}$ for some constant $F>0$.  Thus we have
\eq
\rho(x,t)=Fw(t)^{-3}(w(t)^2-x^2)\chi(x,t).
\label{eq:rho-semicircle}
\endeq
From this formula, we see that  
\eq
\rho_t(x,t)=-Fw(t)^{-4}w'(t)(w(t)^2-3x^2)\chi(x,t),\quad |x|\neq w(t)
\endeq
(note that $\rho_t$ has jump discontinuities at $|x|=w(t)$).
Then, using again $\rho_t+\mu_x=0$ we can find $\mu(x,t)$ by integration in $x$ of $-\rho_t$:
\eq
\mu(x,t)=Fw(t)^{-4}w'(t)x(w(t)^2-x^2)\chi(x,t).
\label{eq:mu-semicircle}
\endeq
Note that this is the unique antiderivative of $-\rho_t$ that decays to zero as $x\to\pm\infty$ for all $t$.  
On the support interval $-w(t)<x<w(t)$ it makes sense to calculate the phase derivative 
\eq
S_x(x,t)=\frac{\mu(x,t)}{\rho(x,t)} = \frac{w'(t)}{w(t)}x,\quad -w(t)<x<w(t),
\label{eq:phase-derivative}
\endeq
which shows that at time instants where $w'(t)\neq 0$, the phase profile is quadratic as a function of $x$ (called a \emph{phase chirp} in the physical literature).

Now we turn to the equation governing the momentum $\mu$ in \eqref{eq:dispersionless-focusing} to see how it determines $w(t)$.  After substituting for $\rho$ and $\mu$ from \eqref{eq:rho-semicircle} and \eqref{eq:mu-semicircle} respectively, this equation has terms proportional to both $x$ and $x^3$ and no other dependence on $x$ within the support of $\chi(x,t)$.  An apparent coincidence that ultimately lies behind the fact that the ansatz \eqref{eq:rho-ansatz} is consistent with the dispersionless NLS system is that equating the coefficients of $x$ and of $x^3$ separately amounts to exactly the same equation on $w(t)$, namely
\eq
w''(t)+\frac{2F}{w(t)^2}=0.
\endeq
This autonomous second-order nonlinear differential equation for $w(t)$ can be multiplied by $w'(t)$ and integrated once to yield
\eq
\frac{1}{2}w'(t)^2 -\frac{2F}{w(t)}=E
\label{eq:g-first-order}
\endeq
where $E$ is another integration constant.

The case $E=0$  corresponds to the self-similar collapse considered by Suleimanov \cite{Suleimanov17}.  Assuming $w'(t)<0$ (corresponding to a collapse instead of an expansion), we get for $E=0$:
\eq
w(t)^{\frac{1}{2}}w'(t)=-2\sqrt{F}\implies \frac{2}{3}w(t)^{\frac{3}{2}}=2\sqrt{F}(t^\circ-t)\implies w(t)=(9F)^{\frac{1}{3}}(t^\circ-t)^{\frac{2}{3}}
\endeq
where $t^\circ\in\mathbb{R}$ is a third integration constant with the interpretation of the focusing time, and the solution is defined only for $t<t^\circ$.  
The solution collapses to zero width ($w(t)\downarrow 0$) and infinite 
amplitude ($f(t)w(t)^2=Fw(t)^{-1}\uparrow \infty$) as $t\uparrow t^\circ$.  
However, as $t\to -\infty$ the width grows without bound and the amplitude 
decays to zero monotonically.  
If instead we assume that $w'(t)>0$, we get the same solution with $t^\circ-t$ replaced by $t-t^\circ$ and the solution exists for $t>t^\circ$ evolving from a collapsed state to a spreading and decaying state as $t\to +\infty$.  Either way, the monotonicity of $w(t)$ implies that when $E=0$, the phase derivative $x\mapsto S_x(x,t)$ given by \eqref{eq:phase-derivative} is nonconstant for every time $t<t^\circ$.  

If $E\neq 0$, it is no longer possible to solve \eqref{eq:g-first-order} explicitly for $w(t)$, although one can explicitly find the inverse function by integration.  To analyze this more general case, it is convenient to rescale the variables.  Let $t=\sqrt{2}F |E|^{-\frac{3}{2}}T$ and $w=2F|E|^{-1}W$.  Then \eqref{eq:g-first-order} implies an equivalent equation on $W(T)$ provided that $E\neq 0$:
\eq
W'(T)^2-\frac{1}{W(T)}=\mathrm{sgn}(E)=\pm 1.
\label{eq:G-first-order}
\endeq
Assuming that $W(T)>0$, in the case $E>0$ we easily see that $W'(T)^2\ge 1$ so $W(T)$ is strictly monotone and, just as for the $E=0$ case, the map $x\mapsto S_x(x,t)$ is nonconstant for every $t$.  

Now assume that $E<0$.  In this case, it is obvious from a phase portrait (see the left-hand panel of Figure~\ref{fig:phase-portrait}) that $W'(T)$ changes sign at exactly one point, which we may take without loss of generality to be $T=0$.
\begin{figure}[h]
\begin{center}
\includegraphics[height=0.5\linewidth]{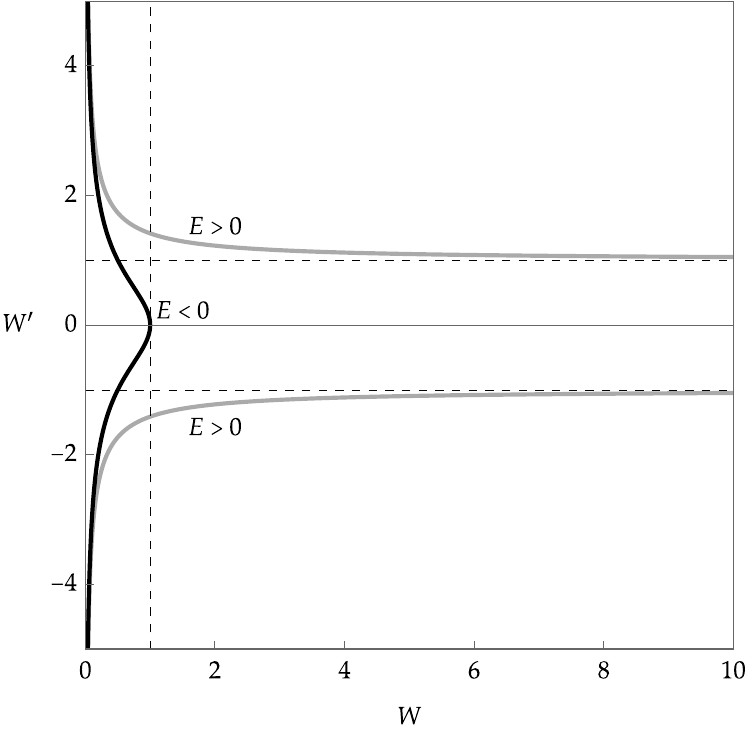}\hspace{0.2\linewidth}%
\includegraphics[height=0.5\linewidth]{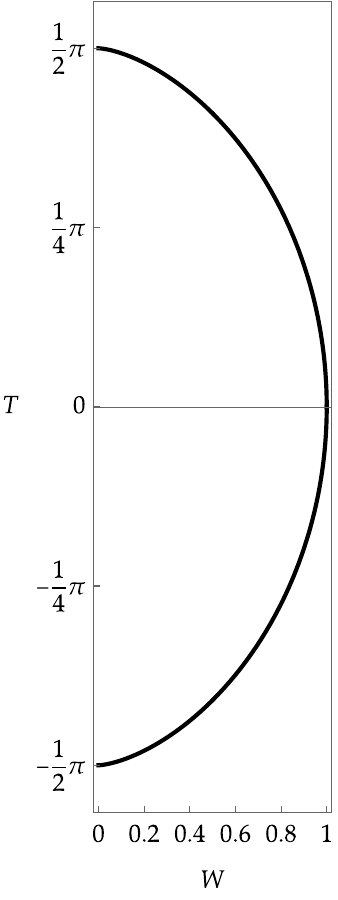}
\end{center}
\caption{Left:  the phase portrait for \eqref{eq:G-first-order} in the part of the $(W,W')$ plane with $W>0$.  Gray:  $E>0$.  Black:  $E<0$.  Compare with \cite[Fig.\@ 3(b)]{Talanov65}.  Right:  the inverse function $\pm T(W)$ in the case $E<0$.}
\label{fig:phase-portrait}
\end{figure}
The orbit of \eqref{eq:G-first-order} for $E<0$ is traversed in the downward direction in finite time $\Delta T$ given by (integrating $|\dd T/\dd W|$ over half of the orbit and doubling the result): 
\eq
\Delta T=2\int_0^1\left(\frac{1}{W}-1\right)^{-1/2}\,\dd W=\pi.
\endeq
The corresponding duration of the solution in the $t$-variable is
\eq
\Delta t = \sqrt{2}F(-E)^{-\frac{3}{2}}\Delta T= \frac{\sqrt{2}\pi F}{(-E)^{\frac{3}{2}}},\quad E<0,\quad F>0.
\label{eq:Talanov-Delta-t}
\endeq
The function $W(T)$ is well-defined for $E<0$ from \eqref{eq:G-first-order} with the initial condition $W(0)=1$, and $W(T)\downarrow 0$ as $|T|\uparrow\tfrac{1}{2}\pi$.  We can give an explicit formula for its inverse function for $T>0$ as a function of $W\in (0,1)$:
\eq
T = \int_W^1\left(\frac{1}{y}-1\right)^{-\frac{1}{2}}\,\dd y = \sqrt{(1-W)W}+\frac{\pi}{4} + \frac{1}{2}\mathrm{Arctan}\left(\frac{1-2W}{2\sqrt{(1-W)W}}\right)
\endeq
which is plotted in the right-hand panel of Figure~\ref{fig:phase-portrait}.
When $t=0$, we have $T=0$ and $W=1$.  Therefore the support of $\chi(x,0)$ and hence also the initial support of $\psi(x,t)$ is $|x|\le w = 2F(-E)^{-1}$.  The maximum amplitude $A_\mathrm{max}:=\max_{x\in\mathbb{R}}|\psi(x,0)|$ is $A_\mathrm{max}=\sqrt{Fw^{-1}}=\sqrt{-E/2}$.  Therefore, the value of $E<0$ is determined directly from $A_\mathrm{max}$  by
\begin{equation}
E:=-2A_\mathrm{max}^2<0
\label{eq:E-determine}
\end{equation}
and then the value of $F>0$ is determined directly from $E$ and the distance between the support endpoints by
\begin{equation}
F:=A_\mathrm{max}^2w(0).
\label{eq:F-determine}
\end{equation}
Hence the duration of the solution becomes 
\eq
\Delta t=\frac{\pi w(0)}{2A_\mathrm{max}}.
\endeq

The Talanov solutions were rediscovered in the setting of shallow water equations by Ovsjannikov \cite{Ovsjannikov79}, where it was shown that a similar approach also applies for the defocusing version of the dispersionless NLS equation.  The latter is obtained from \eqref{eq:dispersionless-focusing} by replacing $-\frac{1}{2}\rho^2$ with $\frac{1}{2}\rho^2$ in the second equation, making the system equivalent to a shallow water model.

\begin{figure}[h]
\begin{center}
\includegraphics[width=0.49\linewidth]{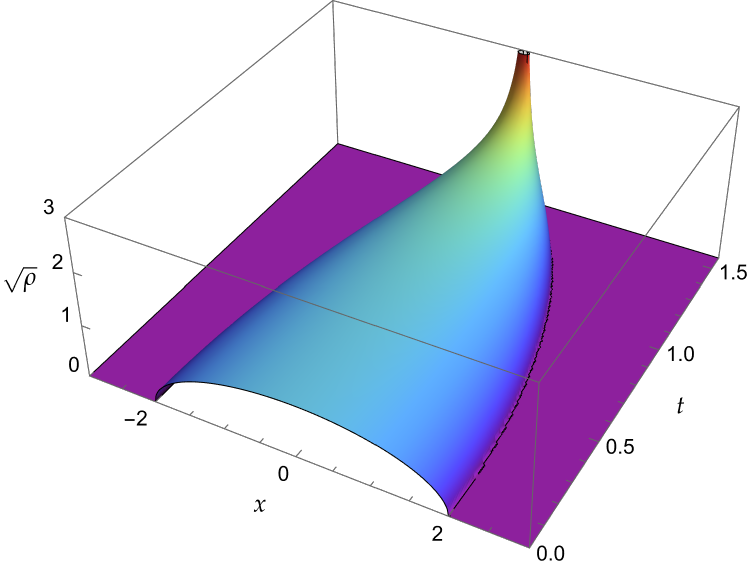}\hfill%
\includegraphics[width=0.49\linewidth]{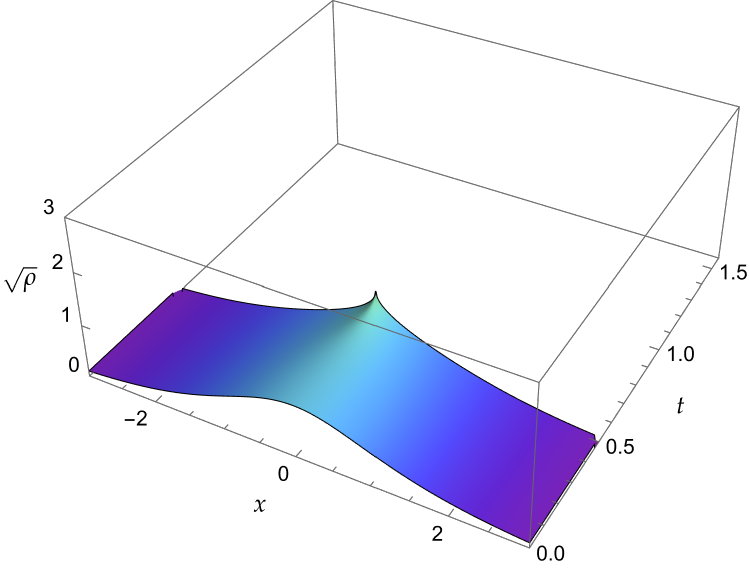}
\end{center}
\caption{Left:  the amplitude $\rho^{1/2}$ for the $E<0$ Talanov solution \eqref{eq:Talanov-solution} of \eqref{eq:dispersionless-focusing} with initial width $w(0)=2$ and amplitude $A_\mathrm{max}=1$, exhibiting the infinite-amplitude collapse and blowup in finite time (at $(x,t)=(0,\pi/2)$).  Right:  the amplitude $\rho^{1/2}$ for the Akhmanov-Sukhorukov-Khokhlov solution \eqref{eq:ASK-implicit} of the same system with $A_\mathrm{max}=1$, exhibiting instead a finite-amplitude gradient catastrophe at $(x,t)=(0,1/2)$ with finite value $\rho^{1/2}=\sqrt{2}$. }
\label{fig:Talanov-ASK}
\end{figure}

%% file: sec-setup.tex
We consider initial data $\psi_0(x)$ of the form
\eq
\label{cauchy-data}
\psi_0(x) = \ee^{\ii\theta}A(x)
\ee^{\ii\kappa x/\epsilon},
\endeq
where $\theta$ and $\kappa$ are real constants.  By gauge invariance $\psi\mapsto \ee^{\ii\theta}\psi$, Galilean invariance, 
and mixing the flows,
we can and will assume without loss of generality that $\theta=\kappa=0$.  

\subsubsection{Semiclassical direct scattering for semicircular Klaus-Shaw potentials}
\label{sec:semiclassical}
We further assume that $A:\mathbb{R}\to\mathbb{R}_{\ge 0}$ is a real-valued function independent of $\eps>0$ that we call a \emph{semicircular Klaus-Shaw potential}.  Such potentials are intended to generalize the initial data for Talanov-type solutions of the dispersionless focusing NLS system \eqref{eq:dispersionless-focusing} in the $E<0$ case.  They are defined as follows.
\begin{definition}[Semicircular Klaus-Shaw potentials]
Let $X_-<X_+$ and $c>0$ be real constants.  A function $A:\mathbb{R}\to\mathbb{R}_{\ge 0}$ is called a \emph{semicircular Klaus-Shaw potential} if 
\begin{itemize}
\item
$A$ has compact support $\mathrm{spt}(A)=[X_-,X_+]$
on which it can be written in the form $A(x)=u(x)\sqrt{(X_+-x)(x-X_-)}$, where $u(x)\ge c$ on $[X_-,X_+]$ and $u$ has an analytic continuation to a complex neighborhood of $[X_-,X_+]\subset\mathbb{C}$; and
\item $A$ has the Klaus-Shaw property:  $A:\mathbb{R}\to\mathbb{R}_{\ge 0}$ is of class $L^1(\mathbb{R}) \cap C^2(\mathbb{R})$, and $A$ has a unique maximizer $x_0 \in \mathbb{R}$.
\item The generic condition $A''(x_0) < 0$ holds at the maximizer. 
\end{itemize}
We denote the maximum value of $A(x)$ by $A_\mathrm{max}:=A(x_0)>0$.
\label{def:semicircularKS}
\end{definition}
Of course if $u(x)\equiv c\ge 0$, then the graph of $A(x)$ is a scaled semicircle that can be written in the form
\begin{equation}
A(x)=\frac{2A_\mathrm{max}\chi_{[X_-,X_+]}(x)}{X_+-X_-}\sqrt{(X_+-x)(x-X_-)},\quad x_0=\frac{1}{2}(X_++X_-),
\label{eq:ExactSemicircle}
\end{equation}
hence the name.  After an irrelevant translation by $x_0$ to recenter $A(x)$ at the origin, this also matches the initial condition of a Talanov-type solution of the dispersionless NLS system with parameters $E=-2A_\mathrm{max}^2<0$ and $F=\frac{1}{2}(X_+-X_-)A_\mathrm{max}^2>0$.  

By inverse-spectral theory \cite{ZakharovS72}, the solution of any equation in the focusing NLS hierarchy with initial data $\psi_0:\mathbb{R}\to\mathbb{C}$ is based on the spectral analysis of the non-selfadjoint Zakharov-Shabat problem:
\eq
\epsilon \frac{\dd{\bf w}}{\dd x} = \begin{pmatrix} -\ii\SP & \psi_0(x) \\ -\psi_0(x)^* & \ii\SP \end{pmatrix}{\bf w}, \quad {\bf w}:\mathbb{R}\to\mathbb{C}^2.
\label{eq:ZS-problem}
\endeq
If $\psi_0(x)=A(x)$ with $A:\mathbb{R}\to\mathbb{R}$ being a semicircular Klaus-Shaw potential, then according to \cite{KlausS02}, the number of eigenvalues is 
finite, and all eigenvalues are simple and lie on the imaginary axis.  We 
choose
\eq
\label{eq:ZS-epsilon-assumption}
\epsilon = \epsilon_N:=\frac{1}{N\pi}\int_\mathbb{R}A(x)\,\dd x
\endeq
for some nonnegative integer $N$.  This ensures there are exactly 
$N$ strictly positive (and simple) imaginary eigenvalues $\SP_n=\ii s_n$, 
$0<s_{N-1}<\cdots<s_1<s_0$.  When $\SP=\SP_n$, the unique solutions $\mathbf{w}=\mathbf{w}_n^\pm(x)$ of \eqref{eq:ZS-problem} given by $\mathbf{w}=\mathbf{w}^-_n(x)=\ee^{-\ii\SP_n x/\epsilon}(1,0)^\top$ for $x\le X_-$ and $\mathbf{w}^+_n(x)=\ee^{\ii\SP_n x/\epsilon}(0,1)^\top$ for $x\ge X_+$ are necessarily proportional; there exists a nonzero connection coefficient $\tau_n\neq 0$ such that $\mathbf{w}_n^-(x)=\tau_n\mathbf{w}_n^+(x)$. Together with the reflection coefficient that may be defined for $\SP\in\mathbb{R}$ (its precise definition is not relevant in this paper), this comprises the scattering data associated with $\psi_0$.

Considering the limit $N\to\infty$ equivalent to $\epsilon\to 0$, we apply the WKB method to \eqref{eq:ZS-problem} in order to approximate the scattering data for $\psi_0(x)=A(x)$ a semicircular Klaus-Shaw potential.  One finds that the reflection coefficient tends to zero as $\epsilon\to 0$ and obtains 
approximate eigenvalues and connection coefficients defined as follows.  
\begin{itemize}
\item Let $\evdensity(s)$ be defined by 
\eq
\begin{split}
\evdensity(s):=&\frac{s}{\pi}\int_{x_-(s)}^{x_+(s)}\frac{\dd x}{\sqrt{A(x)^2-s^2}}\\
  =&-\frac{1}{\pi}\frac{\dd}{\dd s}
\int_{x_-(s)}^{x_+(s)}\sqrt{A(x)^2-s^2}\,\dd x,\quad 0<s<A_\mathrm{max},
\label{eq:ZS-density}
\end{split}
\endeq
where $x_-(s)<x_+(s)$ are the two roots $x$ of $A(x)^2-s^2$, in other words they are the two branches of the inverse function to $s=A(x)$ defined for $0<s<A_\mathrm{max}$.
Then, the \emph{approximate eigenvalues} in the upper half-plane are $\SP=\ii\widetilde{s}_{n}$, $n=0,\dots,N-1$, where $\widetilde{s}_{0},\dots,\widetilde{s}_{N-1}$ are determined uniquely by the Bohr-Sommerfeld quantization rule
\begin{equation}
\phaseint(\ii\widetilde{s}_{n}) = (n+\tfrac{1}{2})\epsilon\pi = \frac{2n+1}{2N}\int_\mathbb{R}A(x)\,\dd x,\quad n=0,\dots,N-1
\label{eq:ZS-approximate-eigenvalues}
\end{equation}
with the \emph{phase integral} $\phaseint$ being defined by
\begin{equation}
\phaseint(\ii s):=\pi\int_{s}^{A_\mathrm{max}}\evdensity(s')\,\dd s'= \int_{x_-(s)}^{x_+(s)}\sqrt{A(x)^2-s^2}\,\dd x,\quad 0<s<A_\mathrm{max}.
\label{eq:ZS-phase-integral}
\end{equation}
Note that $s\mapsto\phaseint(\ii s)$ is real-valued and strictly decreasing on its interval $(0,A_\mathrm{max})$ of definition, and $\phaseint(\ii A_\mathrm{max})=0$, implying that the definition \eqref{eq:ZS-approximate-eigenvalues} is sensible.
\item The connection coefficient $\tau_n$ associated with the eigenvalue best approximated by $\SP=\ii\widetilde{s}_{n}$ is itself approximated by 
\begin{equation}
\tau_n\approx\widetilde{\tau}_{n} :=  (-1)^{n+1} \ee^{\tailint(\ii\widetilde{s}_{n})/\epsilon} 
\label{eq:ZS-proportionality-constant-no-interpolant}
\end{equation}
where the \emph{tail integral} $\tailint$ is defined by
\begin{multline}
\tailint(\ii s):=(x_+(s)+x_-(s))s + \int_{-\infty}^{x_-(s)}\left(\sqrt{s^2-A(x)^2}-s\right)\,\dd x \\
   - \int_{x_+(s)}^{+\infty}\left(\sqrt{s^2-A(x)^2}-s\right)\,\dd x,\quad 0<s<A_\mathrm{max}.
\label{eq:mu-define}
\end{multline}
This is also real-valued on its interval of definition.  Observe that, using the quantization rule \eqref{eq:ZS-approximate-eigenvalues}, the approximate connection coefficients defined in \eqref{eq:ZS-proportionality-constant-no-interpolant} can equivalently be expressed as
\begin{equation}
\widetilde{\tau}_n= \ii(-1)^K\ee^{\ii(2K+1)\phaseint(\ii\widetilde{s}_{n})/\epsilon}\ee^{\tailint(\ii\widetilde{s}_{n})/\epsilon},\quad K\in\mathbb{Z},
\label{eq:ZS-proportionality-constant}
\end{equation}
\end{itemize}
for any arbitrary integer $K$. The freedom to select $K\in\mathbb{Z}$ is essential to our subsequent analysis. 

The functions $\phaseint(\lambda)$ and $\tailint(\lambda)$ have the following properties.
\begin{prop}
Let $A$ be a semicircular Klaus-Shaw potential.  Then the phase integral $\phaseint(\SP)$ defined by \eqref{eq:ZS-phase-integral}
is positive for $s=-\ii\lambda\in (0,A_\mathrm{max})$ with $\phaseint(\ii A_\mathrm{max})=0$.  Also, $\phaseint(\SP)$ is an even analytic function at $\SP=0$, and in particular the following series is convergent:
\begin{equation}
\phaseint(\lambda)=\sum_{k=0}^\infty\phaseint_{k}\lambda^{2k}, 
\label{eq:Psi-series}
\end{equation}
with real coefficients $\phaseint_k$, and the strict inequalities $\phaseint_0>0$ and $\phaseint_1>0$ both hold.
\label{prop:Psi-even-analytic}
\end{prop}
\begin{proof}
That $\phaseint(\ii s)>0$ for $0< s<A_\mathrm{max}$ and that $\phaseint(\ii A_\mathrm{max})=0$ are direct consequences of the definition \eqref{eq:ZS-phase-integral}.
If $A$ is a semicircular Klaus-Shaw potential, then for $0<s<A_\mathrm{max}$ the phase integral \eqref{eq:ZS-phase-integral} can be written as a contour integral:
\begin{equation}
\phaseint(\ii s)=\frac{1}{2}\oint_L R(x;s^2)\,\dd x,\quad 0<s<A_\mathrm{max},
\label{eq:phase-integral-loop}
\end{equation}
where $R(x;s^2)^2=A(x)^2-s^2=u(x)^2(X_+-x)(x-X_-)-s^2$ and $R(x;s^2)$ is analytic in a deleted neighborhood of $[x_-(s),x_+(s)]$ with positive (resp., negative) boundary values on the bottom (resp., top) edge of the cut $[x_-(s),x_+(s)]$, and where $L$ is a positively-oriented loop surrounding the cut.  Since $u$ is analytic in a complex neighborhood of the support interval $[X_-,X_+]$ containing the cut $[x_-(s),x_+(s)]$, we can expand the loop $L$ to enclose the larger interval $[X_-,X_+]$ without changing the value of the integral.  Then as $L$ is independent of $s$, even in a neighborhood of $s=0$, it is clear that the integral is an analytic function of $s^2$ near $s=0$.  Indeed, writing the integrand in the form $R(x;s^2)=R(x;0)(1-s^2/R(x;0)^2)^{\frac{1}{2}}$, where the square root factor converges uniformly to $1$ as $s\to 0$ for $x\in L$, we may expand the integrand in a series of even powers of $s$ and integrate term-by-term:
\begin{equation}
\phaseint(\ii s)=\frac{1}{2}\sum_{k=0}^\infty s^{2k}(-1)^k\binom{1/2}{k}\oint_LR(x;0)^{1-2k}\,\dd x
\end{equation}
for $|s|$ sufficiently small.  Using $\lambda=\ii s$ then proves the claimed series expansion.  The reality of the coefficients $\phaseint_k$  comes from Schwarz symmetry of the integrand and integration contour.  If $k=0,1$, the integrand is integrable at $x=X_\pm$ and by the generalized Cauchy integral theorem we obtain
\begin{equation}
\begin{split}
\phaseint_0&=\binom{1/2}{0}\int_{X_-}^{X_+}u(x)\sqrt{(X_+-x)(x-X_-)}\,\dd x=\int_{X_-}^{X_+}A(x)\,\dd x,\\ \phaseint_1&=\binom{1/2}{1}\int_{X_-}^{X_+}\frac{\dd x}{u(x)\sqrt{(X_+-x)(x-X_-)}},
\end{split}
\label{eq:Psi0-Psi1}
\end{equation}
both of which are obviously positive.
\end{proof}

\begin{prop}
Let $A$ be a semicircular Klaus-Shaw potential.  Then the tail integral $\tailint(\lambda)$ defined by \eqref{eq:mu-define} is an odd analytic function at $\lambda=0$, and in particular the following series is convergent:
\begin{equation}
\tailint(\lambda)=\ii\sum_{k=1}^\infty\tailint_k\lambda^{2k-1},
\end{equation}
with real coefficients $\tailint_k$, and $\tailint_1=-(X_++X_-)$.
\label{prop:mu-odd-analytic}
\end{prop}
\begin{proof}
Since $A$ has compact support in $[X_-,X_+]$, for $0<s<A_\mathrm{max}$ we can rewrite \eqref{eq:mu-define} in the form
\begin{align}
\tailint(\ii s)&=(x_+(s)+x_-(s))s + \int_{X_-}^{x_-(s)}\left(\sqrt{s^2-A(x)^2}-s\right)\,\dd x -\int_{x_+(s)}^{X_+}\left(\sqrt{s^2-A(x)^2}-s\right)\,\dd x\nonumber \\
&=(X_++X_-)s + \int_{X_-}^{x_-(s)}\sqrt{s^2-A(x)^2}\,\dd x -\int_{x_+(s)}^{X_+}\sqrt{s^2-A(x)^2}\,\dd x.
\label{eq:mu-rewrite}
\end{align}
Each integral on the right-hand side makes a similar contribution of order $\bigo{s^3}$, so we just consider the first integral in detail.  Making the change of variables $x=X_-+s^2w$ gives
\begin{equation}
\int_{X_-}^{x_-(s)}\sqrt{s^2-A(x)^2}\,\dd x = s^3\int_0^{(x_-(s)-X_-)/s^2}\sqrt{1-u(X_-+s^2w)^2(X_+-X_--s^2w)w}\,\dd w.
\end{equation}
The upper limit of integration is the positive value $w_-(s)$ of $w$ that makes the integrand vanish and that satisfies $w_-(s)=1/[u(X_-)^2(X_+-X_-)] + \bigo{s^2}$.  We introduce a function $R(w;s^2)$ with a branch cut extending from $w=w_-(s)$ to the left through $w=0$ and satisfying $R(w;s^2)^2=1-u(X_-+s^2w)^2(X_+-X_--s^2w)w$ with $R(w;s^2)>0$ (resp., $R(w;s^2)<0$) on the bottom edge (resp., top edge) of the cut, and then let $L$ denote a teardrop-shaped contour from $w=0$ on the bottom edge to $w=0$ on the top edge and encircling the branch cut between $w=0$ and the branch point $w=w_-(s)$  once in the counterclockwise sense.  We take $L$ to be independent of $s$ for $s>0$ sufficiently small.  Then we have
\begin{equation}
\int_{X_-}^{x_-(s)}\sqrt{s^2-A(x)^2}\,\dd x = \frac{1}{2}s^3\int_L R(w;s^2)\,\dd w.
\end{equation}
Now, $R(w;s^2)$ has a Taylor expansion about $s=0$ in even powers of $s$, and this series is uniformly convergent for $x\in L$, so we may integrate term-by-term.   Thus,
\begin{equation}
\int_{X_-}^{x_-(s)}\sqrt{s^2-A(x)^2}\,\dd x = \frac{1}{2}s^3\sum_{n=0}^\infty\frac{s^{2n}}{n!}\int_L\left.\frac{\dd^nR}{\dd\sigma^n}(w;\sigma)\right|_{\sigma=0}\,\dd w.
\end{equation}
It is easy to check that each Taylor coefficient is an integral that evaluates to a real number by Schwarz symmetry.  Since the second integral in \eqref{eq:mu-rewrite} has a similar expansion, using $\lambda=\ii s$ completes the proof.
\end{proof}

Finally, we note that a semicircular Klaus-Shaw potential $A$ can be recovered from the corresponding phase integral $\phaseint$ and tail integral $\tailint$.
\begin{prop}
Let $A$ be a semicircular Klaus-Shaw potential, and let $\phaseint(\lambda)$ and $\tailint(\lambda)$ be defined by \eqref{eq:ZS-phase-integral} and \eqref{eq:mu-define}, respectively. Then, for $0<s<A_\mathrm{max}$, the inverse functions $x_-(s)<x_+(s)$ of $x\mapsto A(x)$ satisfy
\begin{equation}
x_\pm(s)=\frac{1}{\pi}\int_0^s\frac{\displaystyle\frac{\dd}{\dd m}\tailint(\ii m)\,\dd m}{\sqrt{s^2-m^2}}\mp\frac{1}{\pi}\int_s^{A_\mathrm{max}}\frac{\displaystyle\frac{\dd}{\dd m}\phaseint(\ii m)\,\dd m}{\sqrt{m^2-s^2}}.
\label{eq:xpm-recover}
\end{equation}
\label{prop:Psi-mu-invert}
\end{prop}
\begin{proof}
Directly from the definition of $\phaseint$ in \eqref{eq:ZS-phase-integral}, 
we can compute
\begin{equation}
\frac{\dd}{\dd m}\phaseint(\ii m) = -m\int_{x_-(m)}^{x_+(m)}\frac{\dd x}{\sqrt{A(x)^2-m^2}}.
\end{equation}
Therefore, integrating and changing the integration variable from $m$ to $v=m^2$,
\begin{equation}
\int_s^{A_\mathrm{max}}\frac{\displaystyle\frac{\dd}{\dd m}\phaseint(\ii m)\,\dd m}{\sqrt{m^2-s^2}} = 
-\frac{1}{2}\int_{s^2}^{A_\mathrm{max}^2}\int_{x_-(\sqrt{v})}^{x_+(\sqrt{v})}\frac{\dd x}{\sqrt{A(x)^2-v}}\frac{\dd v}{\sqrt{v-s^2}}.
\end{equation}
Exchanging the order of integration,
\begin{equation}
\int_s^{A_\mathrm{max}}\frac{\displaystyle\frac{\dd}{\dd m}\phaseint(\ii m)\,\dd m}{\sqrt{m^2-s^2}} = -\frac{1}{2}\int_{x_-(s)}^{x_+(s)}\int_{s^2}^{A(x)^2}\frac{\dd v}{\sqrt{A(x)^2-v}\sqrt{v-s^2}}\,\dd x.
\end{equation}
By an affine transformation taking the integration endpoints $v=s^2$ and $v=A(x)^2$ to $\pm 1$, one sees that the inner integral over $v$ is actually independent of $s^2$ and $A(x)^2$, and evaluates to $\pi$. 
Therefore by evaluating the outer integral we obtain
\begin{equation}
\int_s^{A_\mathrm{max}}\frac{\displaystyle\frac{\dd}{\dd m}\phaseint(\ii m)\,\dd m}{\sqrt{m^2-s^2}} = -\frac{\pi}{2}(x_+(s)-x_-(s)).
\label{eq:Psi-integral}
\end{equation}
Similarly, since for $A(x)$ a semicircular Klaus-Shaw potential we can express $\tailint$ using \eqref{eq:mu-rewrite},
we take a derivative to obtain
\begin{align}
\frac{\dd}{\dd m}\tailint(\ii m) &= \frac{\dd}{\dd m}\left[(X_++X_-)m + \int_{X_-}^{x_-(m)}\sqrt{m^2-A(x)^2}\,\dd x -\int_{x_+(m)}^{X_+}\sqrt{m^2-A(x)^2}\,\dd x\right] \nonumber \\
&= X_++X_- + m\int_{X_-}^{x_-(m)}\frac{\dd x}{\sqrt{m^2-A(x)^2}} -m\int_{x_+(m)}^{X_+}\frac{\dd x}{\sqrt{m^2-A(x)^2}}.
\end{align}
So, integrating and changing the integration variable from $m$ to $v=m^2$,
\begin{multline}
\int_0^s\frac{\displaystyle\frac{\dd}{\dd m}\tailint(\ii m)\,\dd m}{\sqrt{s^2-m^2}} = (X_++X_-)\int_0^s\frac{\dd m}{\sqrt{s^2-m^2}} \\
{}+\frac{1}{2}\int_0^{s^2}\int_{X_-}^{x_-(\sqrt{v})}\frac{\dd x}{\sqrt{v-A(x)^2}}\frac{\dd v}{\sqrt{s^2-v}} -\frac{1}{2}\int_0^{s^2}\int_{x_+(\sqrt{v})}^{X_+}\frac{\dd x}{\sqrt{v-A(x)^2}}\frac{\dd v}{\sqrt{s^2-v}}.
\end{multline}
The integral on the first line evaluates to $\pi/2$ independent of $m$, and on the second line we exchange the order of integration in each integral to obtain
\begin{multline}
\int_0^s\frac{\displaystyle\frac{\dd}{\dd m}\tailint(\ii m)\,\dd m}{\sqrt{s^2-m^2}} =\frac{\pi}{2}(X_++X_-)\\ +
\frac{1}{2}\int_{X_-}^{x_-(s)}\int_{A(x)^2}^{s^2}\frac{\dd v}{\sqrt{v-A(x)^2}\sqrt{s^2-v}}\,\dd x -\frac{1}{2}\int_{x_+(s)}^{X_+}\int_{A(x)^2}^{s^2}\frac{\dd v}{\sqrt{v-A(x)^2}\sqrt{s^2-v}}\,\dd x.
\end{multline}
Again the inner integral evaluates to $\pi$ in each case, so we obtain simply
\begin{equation}
\int_0^s\frac{\displaystyle\frac{\dd}{\dd m}\tailint(\ii m)\,\dd m}{\sqrt{s^2-m^2}} =\frac{\pi}{2}(x_+(s)+x_-(s)).
\label{eq:mu-integral}
\end{equation}
Combining \eqref{eq:Psi-integral} and \eqref{eq:mu-integral} yields \eqref{eq:xpm-recover}.
\end{proof}

\begin{cor}
A semicircular Klaus-Shaw potential $A$ is even about $x=\frac{1}{2}(X_++X_-)$ if and only if $\tailint(\lambda)=-\ii (X_++X_-)\lambda$.  
\label{cor:even}
\end{cor}
\begin{proof}
If $A$ is a semicircular Klaus-Shaw potential for which $A(\frac{1}{2}(X_++X_-)+y)$ is an even function of $y$, then $x_-(s)-\frac{1}{2}(X_++X_-)=-(x_+(s)-\frac{1}{2}(X_++X_-))$ holds identically for $0<s<A_\mathrm{max}$, and it then follows from \eqref{eq:mu-rewrite} that $\tailint(\ii s)=(X_++X_-)s$.  On the other hand, if $\tailint(\ii s)=(X_++X_-)s$, then it follows from \eqref{eq:xpm-recover} that $x_-(s)-\frac{1}{2}(X_++X_-)=-(x_+(s)-\frac{1}{2}(X_++X_-))$, which implies that $A$ is even about $x=\frac{1}{2}(X_++X_-)$.
\end{proof}

\begin{rem}
Proposition~\ref{prop:Psi-mu-invert} and Corollary~\ref{cor:even} are also valid for more general Klaus-Shaw potentials.  However Propositions~\ref{prop:Psi-even-analytic} and \ref{prop:mu-odd-analytic} require the more restrictive properties of semicircular Klaus-Shaw potentials.
\label{rem:general-KS}
\end{rem}

\subsubsection{Semiclassical soliton ensembles for semicircular Klaus-Shaw potentials}
We now discard the original initial data $\psi_0(x)=A(x)$ and replace it with the initial 
data corresponding to the pure-soliton solution with discrete spectral data defined in 
\eqref{eq:ZS-density}--\eqref{eq:mu-define}.  This is the semiclassical 
soliton ensemble associated to the initial condition $A(x)$.  Here is the precise definition.
\begin{definition}[Semiclassical soliton ensemble]
Let $A$ be a semicircular Klaus-Shaw potential and let $N>0$ be an integer.  The \emph{semiclassical soliton ensemble} associated with the initial condition $\psi_0(x)=A(x)$ and the index $N$ is the exact solution $\widetilde{\psi}(x,t_2,t_3,\dots,t_M)$ of the first $M-1$ flows in the focusing NLS hierarchy for parameter $\epsilon=\epsilon_N$ given by \eqref{eq:ZS-epsilon-assumption} with initial condition $\widetilde{\psi}(x,0,0,\dots,0)$ that is reflectionless (reflection coefficient vanishing identically) and has eigenvalues $\lambda=\ii\widetilde{s}_n$ defined by \eqref{eq:ZS-approximate-eigenvalues}--\eqref{eq:ZS-phase-integral} with corresponding connection coefficients $\widetilde{\tau}_n$ defined by \eqref{eq:ZS-proportionality-constant-no-interpolant}--\eqref{eq:mu-define}.  Given real constants $a_2,a_3,\dots,a_M$, the semiclassical soliton ensemble for the mixture of the flows corresponding to these constants is the function $\widetilde{\psi}(x,t):=\widetilde{\psi}(x,a_2t,a_3t,\dots,a_Mt)$; it is an exact solution of the prescribed mixture of the flows.
\label{def:SSE}
\end{definition}

We introduce the compact notation $\mathbf{t}=(t_2,t_3,\dots,t_M)\in\mathbb{R}^{M-1}$ to denote the vector of time coordinates, and write $\widetilde{\psi}(x,\mathbf{t})=\widetilde{\psi}(x,t_2,t_3,\dots,t_M)$.
Because $\widetilde{\psi}(x,\mathbf{t})$ is a reflectionless solution of the focusing NLS hierarchy, it can be characterized in terms of the solution of a Riemann-Hilbert problem with purely discrete data.  To formulate this problem, first define the 
set of poles in $\mathbb{C}_+$ by
\eq
P:=\{\ii\widetilde{s}_n\}_{n=0}^{N-1}.
\endeq
\begin{myrhp}[Semiclassical soliton ensemble for the focusing NLS hierarchy]
\label{rhp-meromorphic}
Given $\epsilon>0$ and values of the independent variables $(x,t_2,t_3,\dots,t_M)\in\mathbb{R}^M$, seek a $2\times 2$ matrix function $\widetilde{\bf M}(\SP)=\widetilde{\bf M}(\SP;x,\mathbf{t})$ with the following properties:
\begin{itemize}
\item[]\textit{\textbf{Meromorphicity:}} $\widetilde{\bf M}(\SP)$ is 
analytic for $\SP\in\mathbb{C}\setminus (P\cup P^*)$, with simple poles 
in $P\cup P^*$.
\item[]\textit{\textbf{Residues:}}  
We have the residue conditions
\begin{equation}
\begin{split}
\mathop{\mathrm{Res}}_{\SP=\ii\widetilde{s}_n}\widetilde{\bf M}(\SP)=\lim_{\SP\to \ii\widetilde{s}_n}\widetilde{\bf M}(\SP) \begin{pmatrix} 0 & 0 \\ c_n(x,\mathbf{t}) & 0 \end{pmatrix}, \quad n=0,\dots,N-1,
\label{eq:RHP-pole-C-plus}
\end{split}
\end{equation}
and
\begin{equation}
\begin{split}
\mathop{\mathrm{Res}}_{\SP=-\ii\widetilde{s}_n}\widetilde{\bf M}(\SP)=\lim_{\SP\to-\ii\widetilde{s}_n}\widetilde{\bf M}(\SP) \begin{pmatrix} 0 & -c_n(x,\mathbf{t})^* \\ 0 & 0 \end{pmatrix}, \quad n=0,\dots,N-1,
\end{split}
\label{eq:RHP-pole-C-minus}
\end{equation}
where
\eq
c_n(x,\mathbf{t}):=c_n^0 \ee^{2\ii Q(\ii\widetilde{s}_n;x,\mathbf{t})/\epsilon},
\label{eq:cn}
\endeq
with
\eq
c_n^0 := \widetilde{\tau}_n\frac{\prod_{j=0}^{N-1}\left(\ii\widetilde{s}_n+\ii\widetilde{s}_j\right)}{\prod_{j=0,j\neq n}^{N-1}\left(\ii\widetilde{s}_n-\ii\widetilde{s}_j\right)}, \quad Q(\SP;x,\mathbf{t}):=\SP x+\sum_{m=2}^M\SP^m t_m.
\label{eq:cn0}
\endeq
\item[]\textit{\textbf{Normalization:}} $\widetilde{\bf M}(\SP)\to\mathbb{I}$ as $\SP\to\infty$.
\end{itemize}
\label{rhp:M-sigma}
\end{myrhp}
It can be shown that for each $\epsilon>0$, this problem has a unique solution that is defined and real analytic on the parameter space $(x,t_2,t_3,\dots,t_M)\in\mathbb{R}^M$.  Indeed, one can see this by following the standard method of removing the poles in favor of jumps across small circles surrounding each of them (see for instance \cite{BorgheseJM18}), and doing so in a way that preserves Schwarz symmetry of the conditions as one uses to prove Proposition~\ref{prop:Schwarz} below.  Then one applies Zhou's vanishing lemma \cite{Zhou89}. 

An equivalent Riemann-Hilbert problem useful in some situations arises from the transformation
\begin{equation}
\widetilde{\mathbf{M}}^\updownarrow(\SP;x,\mathbf{t}):=\widetilde{\mathbf{M}}(\SP;x,\mathbf{t})\prod_{j=0}^{N-1}\left(\frac{\SP-\ii\widetilde{s}_j}{\SP+\ii\widetilde{s}_j}\right)^{\sigma_3}.
\label{eq:pole-swap}
\end{equation}
Clearly, this modified matrix is also meromorphic with simple poles only in $P\cup P^*$, and it tends to the identity as $\SP\to\infty$.  
The main effect of \eqref{eq:pole-swap} is to move the simple poles from the first column to the second and vice-versa.  Therefore, \eqref{eq:RHP-pole-C-plus}--\eqref{eq:RHP-pole-C-minus} imply that
$\widetilde{\mathbf{M}}^\updownarrow(\SP;x,\mathbf{t})$ satisfies the conditions of the following Riemann-Hilbert problem.
\begin{myrhp}[Renormalized soliton ensemble for the focusing NLS hierarchy]
\label{rhp-meromorphic_renorm}
Given $\epsilon>0$ and values of the independent variables $(x,t_2,t_3,\dots,t_M)\in\mathbb{R}^M$, seek a $2\times 2$ matrix function $\widetilde{\mathbf{M}}^\updownarrow(\SP) =\widetilde{\mathbf{M}}^\updownarrow(\SP;x,\mathbf{t})$ with the following properties:
\begin{itemize}
\item[]\textit{\textbf{Meromorphicity:}} $\widetilde{\mathbf{M}}^\updownarrow(\SP) $ is 
analytic for $\SP\in\mathbb{C}\setminus (P\cup P^*)$, with simple poles 
in $P\cup P^*$.
\item[]\textit{\textbf{Residues:}}  
We have the residue conditions
\begin{equation}
\begin{split}
\mathop{\mathrm{Res}}_{\SP=\ii\widetilde{s}_n}\widetilde{\mathbf{M}}^\updownarrow(\SP)=\lim_{\SP\to \ii\widetilde{s}_n}\widetilde{\mathbf{M}}^\updownarrow(\SP) \begin{pmatrix} 0 & c^\updownarrow_n(x,\mathbf{t}) \\ 0 & 0 \end{pmatrix}, \quad n=0,\dots,N-1,
\label{eq:RHP-pole-C-plus-renorm}
\end{split}
\end{equation}
and
\begin{equation}
\begin{split}
\mathop{\mathrm{Res}}_{\SP=-\ii\widetilde{s}_n}\widetilde{\mathbf{M}}^\updownarrow(\SP)=\lim_{\SP\to-\ii\widetilde{s}_n}\widetilde{\mathbf{M}}^\updownarrow(\SP) \begin{pmatrix} 0 & 0 \\-c^\updownarrow_n(x,\mathbf{t})^* & 0 \end{pmatrix}, \quad n=0,\dots,N-1,
\end{split}
\label{eq:RHP-pole-C-minus-renorm}
\end{equation}
where
\eq
c^\updownarrow_n(x,\mathbf{t}):=\frac{1}{c_n(x,\mathbf{t})}\frac{\prod_{j=0}^{N-1}(\ii\widetilde{s}_n+\ii\widetilde{s}_j)^2}{\prod_{j=0,j\neq n}^{N-1}(\ii \widetilde{s}_n-\ii\widetilde{s}_j)^2},\quad n=0,\dots,N-1.
\label{eq:cn-updown}
\endeq
\item[]\textit{\textbf{Normalization:}} $\mathbf{M}^\updownarrow(\SP) \to\mathbb{I}$ as $\SP\to\infty$.
\end{itemize}
\label{rhp:M-sigma-renorm}
\end{myrhp}

Once the unique solution of either Riemann-Hilbert Problem \ref{rhp-meromorphic} or \ref{rhp-meromorphic_renorm} is known, the semiclassical soliton ensemble corresponding to the semicircular Klaus-Shaw initial data $\psi_0(x)=A(x)$ is the exact solution to 
the focusing NLS hierarchy given by 
\eq
\widetilde{\psi}(x,\mathbf{t}):=2\ii\lim_{\SP\to\infty}\SP\widetilde{M}_{12}(\SP;x,\mathbf{t}) = 2\ii\lim_{\SP\to\infty}\SP\widetilde{M}^\updownarrow_{12}(\SP;x,\mathbf{t}).
\label{eq:psitilde-reconstruct}
\endeq 
Indeed, it is a consequence of \eqref{eq:pole-swap} that the two formul\ae\ in \eqref{eq:psitilde-reconstruct} are consistent. Then the fact that $\widetilde{\psi}(x,\mathbf{t})$ solves the first $M-1$ equations in the focusing NLS hierarchy follows by a standard dressing calculation that proceeds as follows.  First, one checks that the substitution $\mathbf{L}(\SP)=\widetilde{\mathbf{M}}(\SP)\ee^{-\ii Q(\SP;x,\mathbf{t})\sigma_3}$ yields residue conditions for $\mathbf{L}(\SP)$ that do not depend on the coordinates $(x,t_2,t_3,\dots,t_M)$ and hence $\mathbf{X}(\SP):=\epsilon\mathbf{L}_x(\SP)\mathbf{L}(\SP)^{-1}$ and $\mathbf{T}^{(m)}(\SP):=\epsilon \mathbf{L}_{t_m}(\SP)\mathbf{L}(\SP)^{-1}$ are polynomials in $\SP$ of degree $1$ and $m=2,\dots,M$ respectively.  Expressing the coefficients of these polynomials in terms of the expansion of $\widetilde{\mathbf{M}}(\SP)$ as $\SP\to\infty$ then shows that $\mathbf{L}(\SP)$ is a fundamental simultaneous solution matrix of the Lax equations $\epsilon\mathbf{L}_x=\mathbf{XL}$ and $\epsilon\mathbf{L}_{t_m}=\mathbf{T}^{(m)}\mathbf{L}$ for $m=2,3,\dots,M$.  The compatibility condition $\epsilon\mathbf{X}_{t_m}-\epsilon\mathbf{T}^{(m)}_x + [\mathbf{X},\mathbf{T}^{(m)}]=\mathbf{0}$ is then equivalent to the equation $\epsilon\widetilde{\psi}_{t_m}=N_m[\widetilde{\psi},\widetilde{\psi}^*]$ and hence taking  $m=2,3,\dots,M$ shows that $\widetilde{\psi}(x,\mathbf{t})$ is a simultaneous solution of the first $M-1$ flows of the focusing NLS hierarchy, where $\widetilde{\psi}(x,\mathbf{t})$ is given by \eqref{eq:psitilde-reconstruct}.

\begin{prop}
For each $\epsilon>0$ and $(x,t_2,t_3,\dots,t_M)\in\mathbb{R}^M$, the solution $\widetilde{\mathbf{M}}(\SP;x,\mathbf{t})$ of Riemann-Hilbert Problem~\ref{rhp-meromorphic}, satisfies the Schwarz symmetric property
\begin{equation}
	\widetilde{\mathbf{M}}(\SP;x,\mathbf{t})
	= \widetilde{\mathbf{M}}(\SP^*;x,\mathbf{t})^{-\dagger}, 
	\qquad \lambda \in \C \setminus (P \cup P^*).
\end{equation}
The renormalized function $\widetilde{\mathbf{M}}^{\updownarrow}(\SP;x,\mathbf{t})$ defined by \eqref{eq:pole-swap} inherits the same symmetry. 
\label{prop:Schwarz}
\end{prop}
\begin{proof}
If $\widetilde{\mathbf{M}}(\SP)=\widetilde{\mathbf{M}}(\SP;x,\mathbf{t})$ solves Riemann-Hilbert Problem~\ref{rhp-meromorphic}, first notice that $\det (\widetilde{\mathbf{M}}(\SP))$ is entire, bounded and goes to 1 for $\SP \to \infty$. A Liouville argument then implies that $\det (\widetilde{\mathbf{M}}(\SP))\equiv 1$. 
A straightforward calculation shows that the function $\mathbf{N}(\lambda) := \widetilde{\mathbf{M}}(\SP^*)^{-\dagger}$ also solves Riemann-Hilbert Problem~\ref{rhp-meromorphic}. 
The ratio $\widetilde{\mathbf{M}}(\SP) \mathbf{N}(\lambda)^{-1}$ is then an entire, bounded function which approaches $\mathbb{I}$ as $\SP \to \infty$. 
The result then follows from Liouville's theorem.
\end{proof}

\begin{prop}
Suppose that $c_n(x,\mathbf{t})\in\mathbb{R}$ for $n=0,\dots,N-1$.  Then $\widetilde{\psi}(x,\mathbf{t})\in\ii\mathbb{R}$.  Likewise if $c_n(x,\mathbf{t})\in\ii\mathbb{R}$ for $n=0,\dots,N-1$, then $\widetilde{\psi}(x,\mathbf{t})\in\mathbb{R}$.
\label{prop:real-imag}
\end{prop}
\begin{proof}
In the former (respectively, latter) case, $\mathbf{N}(\SP):=\sigma_1\widetilde{\mathbf{M}}(-\SP)\sigma_1$ (respectively, $\mathbf{N}(\SP):=\sigma_2\widetilde{\mathbf{M}}(-\SP)\sigma_2$) solves Riemann-Hilbert Problem~\ref{rhp-meromorphic} whenever $\widetilde{\mathbf{M}}(\SP)$ does.  Arguing as in the proof of Proposition~\ref{prop:Schwarz},  uniqueness implies that  $\mathbf{N}(\SP)=\widetilde{\mathbf{M}}(\SP)$.  By Proposition~\ref{prop:Schwarz}, we also have $\mathbf{N}(\SP)=\widetilde{\mathbf{M}}(\SP)=\widetilde{\mathbf{M}}(\SP^*)^{-\dagger}=\sigma_2\widetilde{\mathbf{M}}(\SP^*)^{*\top}\sigma_2$.  Equating the Laurent expansions in descending powers of $\SP$ as $\SP\to\infty$ and using \eqref{eq:psitilde-reconstruct} completes the proof.
\end{proof}

Since the connection coefficients $\widetilde{\tau}_n$ defined in \eqref{eq:ZS-proportionality-constant-no-interpolant} are all real numbers, it follows that $c_n^0\in\ii\mathbb{R}$ for all $n=0,\dots,N-1$, and the same holds for $c_n(x,\mathbf{t})$ provided that $t_2=t_4=\dots=0$ (i.e., all even-indexed time coordinates vanish).  Therefore by Proposition~\ref{prop:real-imag}, this condition implies that $\widetilde{\psi}(x,\mathbf{t})$ is real.  In particular, $\widetilde{\psi}(x,\mathbf{0})$ is real for all $x\in\mathbb{R}$.

%% file: sec-results.tex
\subsubsection{Semiclassical soliton ensembles versus semicircular Klaus-Shaw potentials at $t=0$}
Our first result establishes the accuracy of approximating Cauchy data of semicircular Klaus-Shaw type \eqref{cauchy-data} by its semiclassical soliton ensemble approximation.   Recall that the semiclassical soliton ensemble $\widetilde{\psi}(x,\mathbf{t})$ depends on $N\in\mathbb{N}$ (or $\eps=\eps_N$ via \eqref{eq:ZS-epsilon-assumption}), but the semicircular Klaus-Shaw potential $\psi_0(x)=A(x)$ that generates it is fixed.
\begin{thm}[Initial accuracy of semiclassical soliton ensembles]

Let Cauchy data $\psi_0(x) =A(x)$ be given, where $A$ is a semicircular Klaus-Shaw potential  supported on $[X_-, X_+]$ (see Definition~\ref{def:semicircularKS}), and for each $N \in \mathbb{N}$ let $\widetilde{\psi}(x,\mathbf{t})$ be the corresponding semiclassical soliton ensemble (see Definition~\ref{def:SSE}). Then  
\begin{equation}
	\widetilde{\psi}(x,\mathbf{0}) = \psi_0(x) + 
	\begin{cases}  
	    \bigo{\eps^{1/2}} & x \in (X_-,x_0)\cup(x_0, X_+) \\
	    \bigo{\eps^{2}} & x \in [X_-, X_+]^c \\
	\end{cases}
	, \qquad \eps=\eps_N \downarrow 0.
\end{equation} 
These estimates are uniform for $x$ in compact subsets of $\R \setminus \{x_0, X_-, X_+\}$. 
\label{thm-accuracy-t=0}
\end{thm}
The proof will be given in Sections~\ref{sec-zakharov-shabat} and \ref{sec-proof-Thm1-outside} below.  We can apply this result to obtain convergence in the mean-square sense.
\begin{cor}
Under the assumptions of Theorem~\ref{thm-accuracy-t=0}, with $\eps=\eps_N$,
\begin{equation}
\lim_{N\to\infty} \|\widetilde{\psi}(\diamond,\mathbf{0})-\psi_0(\diamond)\|_{L^2(\mathbb{R})}=0.
\label{eq:L2-convergence}
\end{equation}
\label{cor-L2-convergence}
\end{cor}
\begin{proof}
Since $\widetilde{\psi}(x,\mathbf{0})$ is a reflectionless potential, its $L^2$-norm is expressed in terms of the discrete eigenvalues $\lambda=\ii\widetilde{s}_n$, $n=0,\dots,N-1$ by a standard trace formula: 
\begin{equation}
\|\widetilde{\psi}(\diamond,\mathbf{t})\|_{L^2(\mathbb{R})}^2=\|\widetilde{\psi}(\diamond,\mathbf{0})\|_{L^2(\mathbb{R})}^2= 4\sum_{n=0}^{N-1}\epsilon\widetilde{s}_n.
\end{equation}
Now the sum is a Riemann sum approximation of an integral; indeed, combining 
\eqref{eq:ZS-epsilon-assumption} with \eqref{eq:ZS-approximate-eigenvalues} shows that
\begin{equation}
	\sum_{j=0}^{N-1} \eps \widetilde{s}_j = \frac{ \|A \|_1}{\ii \pi N} \sum_{j=0}^{N-1} \phaseint^{-1} \left( (j+\tfrac{1}{2}) \frac{\|A \|_{1}}{ N} \right) 
	= \frac{1}{\ii \pi} \int_0^{\|A\|_{1}} \phaseint^{-1}(\xi) \dd \xi + \bigo{N^{-2}}
	\label{eq:Riemann-Sum}
\end{equation}
where $\|A\|_1 := \int_\R A(x) \dd x$ and the error estimate is standard for the midpoint rule. 
The integral can be evaluated by integration by parts as
\begin{equation}
\frac{1}{\ii \pi} \int_0^{\|A\|_{1}} \phaseint^{-1}(\xi) \dd \xi = \frac{1}{\ii\pi}\xi \phaseint^{-1}(\xi) \Big\vert_{0}^{\|A\|_1} - \frac{1}{\ii\pi}\int_{\phaseint^{-1}(0)}^{\phaseint^{-1}(\|A\|_1)} \phaseint(\SP) \dd \SP = \frac{1}{\pi} \int_0^{A_\text{max}} \phaseint(\ii s) \dd s,
\label{eq:Rewrite-L2-norm}
\end{equation}
where we have used that fact, following from \eqref{eq:ZS-approximate-eigenvalues}, that $\phaseint^{-1}(0) = \ii A_\text{max}$ and $\phaseint^{-1}(\|A\|_1) = 0$. 
Note that upon using \eqref{eq:ZS-phase-integral}, exchanging the integration order and using the substitution $s=A(x)t$ in the (new) inner integral, one gets
\begin{equation}
\frac{1}{\pi}\int_0^{A_\mathrm{max}}\phaseint(\ii s)\,\dd s = \frac{1}{4}\int_{\mathbb{R}}A(x)^2\,\dd x = \frac{1}{4}\|\psi_0\|_{L^2(\mathbb{R})}^2.
\end{equation}
Then we see that
\begin{equation}
\begin{split}
\|\widetilde{\psi}(\diamond,\mathbf{0})-\psi_0(\diamond)\|_{L^2(\mathbb{R})}^2 &= \|\widetilde{\psi}(\diamond,\mathbf{0})\|_{L^2(\mathbb{R})}^2 + \|\psi_0(\diamond)\|_{L^2(\mathbb{R})}^2 -2\mathrm{Re}\left(\langle \widetilde{\psi}(\diamond,\mathbf{0}),\psi_0(\diamond)\rangle\right)\\
&=2\|\psi_0(\diamond)\|_{L^2(\mathbb{R})}^2 + \mathcal{O}(N^{-2}) -2\mathrm{Re}\left(\langle \widetilde{\psi}(\diamond,\mathbf{0}),\psi_0(\diamond)\rangle\right)\\
&=2\mathrm{Re}\left(\langle\psi_0(\diamond)-\widetilde{\psi}(\diamond,\mathbf{0}),\psi_0(\diamond)\rangle\right) + \mathcal{O}(N^{-2}).
\end{split}
\end{equation}
Clearly for any given $\delta>0$ there exists $N_0$ sufficiently large that $N>N_0$ implies
\begin{equation}
\|\widetilde{\psi}(\diamond,\mathbf{0})-\psi_0(\diamond)\|_{L^2(\mathbb{R})}^2\le 2\left|\langle\psi_0(\diamond)-\widetilde{\psi}(\diamond,\mathbf{0}),\psi_0(\diamond)\rangle\right| + \frac{1}{3}\delta.
\end{equation}
Suppose $f\in C_0^\infty(\mathbb{R}\setminus\{X_-,x_0,X_+\})$. Then 
\begin{equation}
\begin{split}
\left|\langle\psi_0(\diamond)-\widetilde{\psi}(\diamond,\mathbf{0}),\psi_0(\diamond)\rangle\right|&\le
\left|\langle\psi_0(\diamond)-\widetilde{\psi}(\diamond,\mathbf{0}),f(\diamond)\rangle\right| + \left|\langle\psi_0(\diamond)-\widetilde{\psi}(\diamond,\mathbf{0}),\psi_0(\diamond)-f(\diamond)\rangle\right|\\ &\le \|\psi_0(\diamond)-\widetilde{\psi}(\diamond,\mathbf{0})\|_{L^\infty(\mathrm{spt}(f))}\|f(\diamond)\|_{L^1(\mathbb{R})} \\
&\qquad{}+ \|\psi_0(\diamond)-\widetilde{\psi}(\diamond,\mathbf{0})\|_{L^2(\mathbb{R})}\|\psi_0(\diamond)-f(\diamond)\|_{L^2(\mathbb{R})}\\
&\le \|\psi_0(\diamond)-\widetilde{\psi}(\diamond,\mathbf{0})\|_{L^\infty(\mathrm{spt}(f))}\|f(\diamond)\|_{L^1(\mathbb{R})} \\
&\qquad{}+\left(\|\psi_0(\diamond)\|_{L^2(\mathbb{R})}+\|\widetilde{\psi}(\diamond,\mathbf{0})\|_{L^2(\mathbb{R})}\right)\|\psi_0(\diamond)-f(\diamond)\|_{L^2(\mathbb{R})}.
\end{split}
\end{equation}
Since $\|\widetilde{\psi}(\diamond,\mathbf{0})\|_{L^2(\mathbb{R})}\to\|\psi_0(\diamond)\|_{L^2(\mathbb{R})}$ as $N\to\infty$, there is some $N_1$ so large that $N>N_1$ guarantees $\|\widetilde{\psi}(\diamond,\mathbf{0})\|_{L^2(\mathbb{R})}\le 2\|\psi_0(\diamond)\|_{L^2(\mathbb{R})}$.  
By a density argument, for each given $\delta>0$ there is a $f=f_\delta\in C_0^\infty(\mathbb{R}\setminus\{X_-,x_0,X_+\})$ such that 
\begin{equation}
6\|\psi_0(\diamond)\|_{L^2(\mathbb{R})}\|\psi_0(\diamond)-f_\delta(\diamond)\|_{L^2(\mathbb{R})}\le\frac{1}{3}\delta.
\end{equation}
Thus, if $N>\max\{N_0,N_1\}$, 
\begin{equation}
\|\widetilde{\psi}(\diamond,\mathbf{0})-\psi_0(\diamond)\|_{L^2(\mathbb{R})}^2\le 2\|\psi_0(\diamond)-\widetilde{\psi}(\diamond,\mathbf{0})\|_{L^\infty(\mathrm{spt}(f_\delta))}\|f_\delta(\diamond)\|_{L^1(\mathbb{R})} + \frac{2}{3}\delta.
\end{equation}
Now by Theorem~\ref{thm-accuracy-t=0} there exists $N_2>0$ such that $N>\max\{N_0,N_1,N_2\}$ implies that
\begin{equation}
\|\widetilde{\psi}(\diamond,\mathbf{0})-\psi_0(\diamond)\|_{L^2(\mathbb{R})}^2\le \delta
\end{equation}
because the support of $f_\delta$ is a compact subset of $\mathbb{R}\setminus\{X_-,x_0,X_+\}$.
\end{proof}

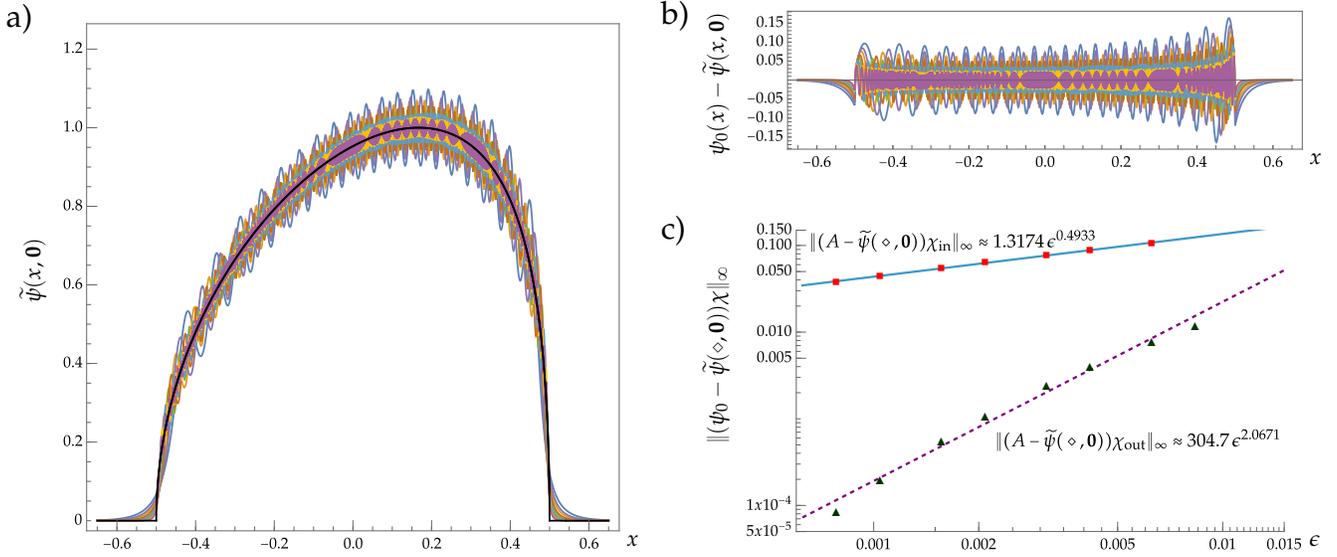
\begin{figure}[htb]
\centering
\begin{minipage}[b]{.45\textwidth}
\centering
\vspace{0pt}
\begin{overpic}[width=\textwidth]{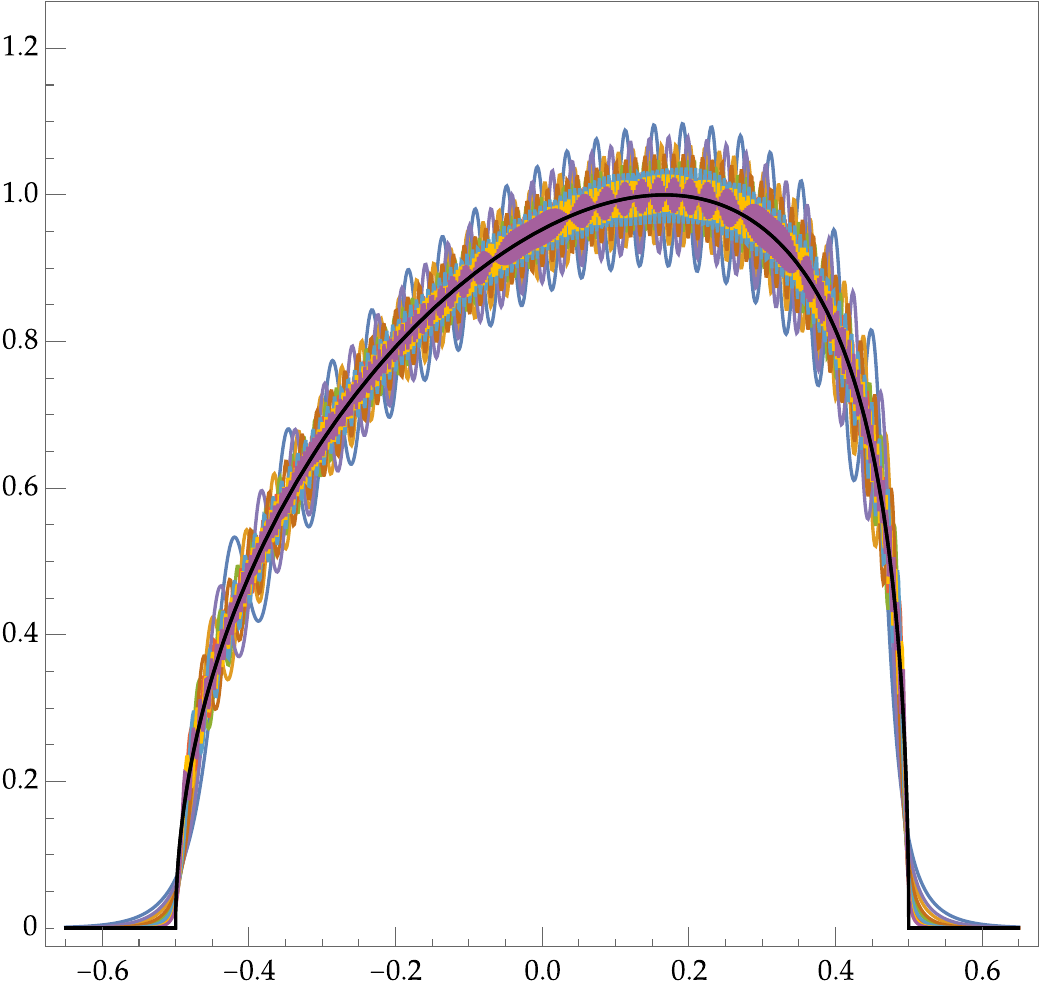}
\put(-10,95){a)}
\put(101,2){\scalebox{0.7}{$x$}}
\put(-10,50){ \scalebox{0.7}{ \rotatebox[origin=c]{90}{$\widetilde \psi(x,\mathbf{0})$}} } 
\end{overpic}
\end{minipage}
\hspace*{\stretch{1}}
\begin{minipage}[b]{.45\textwidth}
\centering
\vspace{0pt}
\begin{overpic}[width=\textwidth]{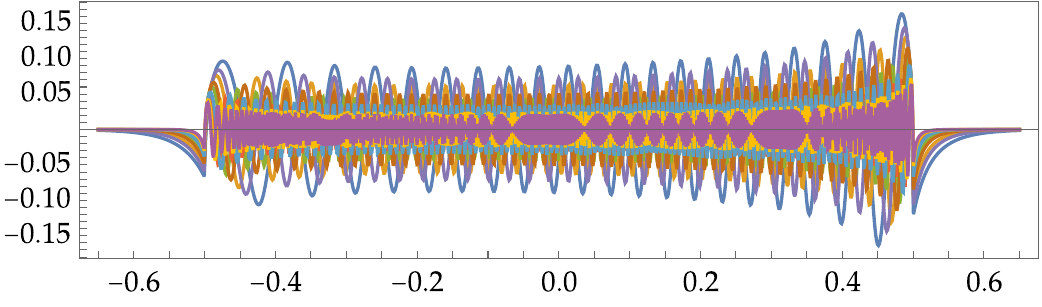}
\put(-15,27){b)}
\put(101,2){\scalebox{0.7}{$x$} }
\put(-10,15){ \scalebox{0.7}{ \rotatebox[origin=c]{90}{$\psi_0(x) - \widetilde \psi(x,\mathbf{0}) $}} } 
\end{overpic}
\vspace*{.1cm}

\begin{overpic}[width=\textwidth]{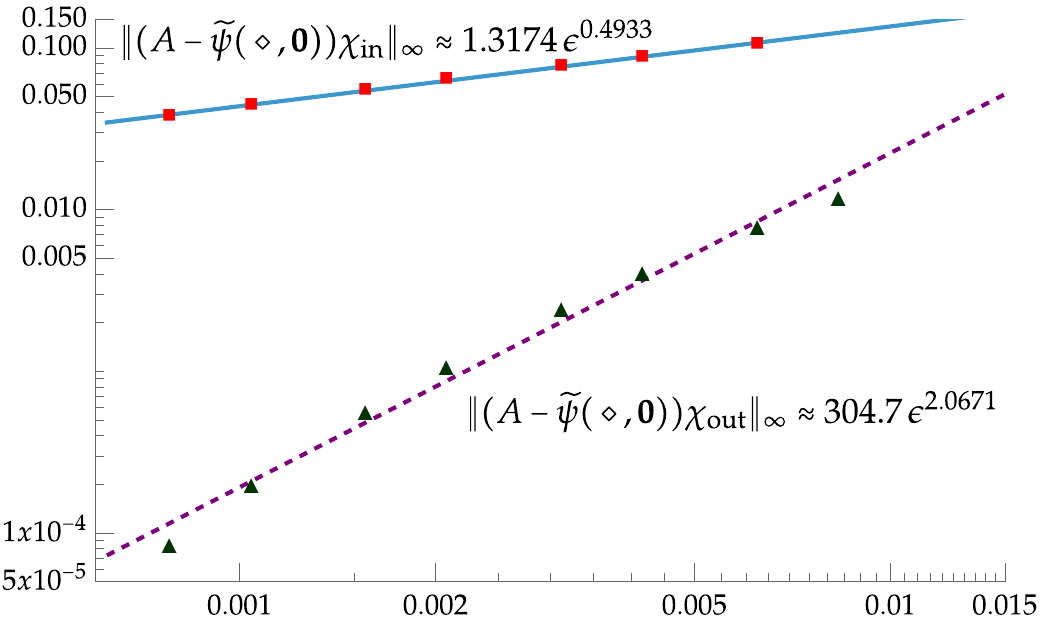}
\put(-15,57){c)}
\put(101,2){\scalebox{0.7}{$\epsilon$} }
\put(-10,35){ \scalebox{0.7}{ \rotatebox[origin=c]{90}{$\| (\psi_0 - \widetilde \psi(\diamond,\mathbf{0})) \chi \|_\infty$}} } 
\end{overpic}
\end{minipage}
\caption{
a) Plots of the numerically computed semiclassical soliton ensemble approximations $\widetilde{\psi}(x,\mathbf{0})$ (shown in colors) of the Klaus-Shaw initial data $\psi_0(x)=A(x)$ (shown in black) given by \eqref{eq:HirotaIC} (with $A_\mathrm{max} = 1$, $X_\pm=\pm \tfrac{1}{2}$, and $\xi = \frac{2}{3}$) for $\eps = \tfrac{1}{80}, \tfrac{1}{120}, \tfrac{1}{160}, \tfrac{1}{240}, \tfrac{1}{320}, \frac{1}{480}, \tfrac{1}{640}, \tfrac{1}{960}, \tfrac{1}{1280}$; 
b) Pointwise plot of the errors $A(x) - \widetilde{\psi}(x,\mathbf{0})$ for each value of $\epsilon$ in part a);
c) Points (red squares/green triangles) show the sup-norm error $\| (A - \widetilde \psi(\diamond,\mathbf{0}))\chi_{\mathrm{in/out}} \|_\infty$ over compact subsets of the interior/exterior of the support ($[-0.45,0.45]$ and $\{ 0.55 \leq |x| \leq 0.65 \}$ respectively) for each value of $\epsilon$ in part a). The lines show the least squares fit of a power law to each data set. The computed powers $0.4933$ and $2.0671$ are in good agreement with the result of Theorem~\ref{thm-accuracy-t=0}.
}
\label{fig:NLS_ICs}
\end{figure}

\begin{rem} Both $\widetilde{\psi}(\diamond,\mathbf{t})$ and the solution $\psi(\diamond,\mathbf{t})$ of the focusing NLS hierarchy for fixed initial data $\psi(\diamond,\mathbf{0})=\psi_0(\diamond)$ satisfy the same $\epsilon$-dependent system of partial differential equations, and both solutions are evidently close when $\mathbf{t}=\mathbf{0}$ and $\eps>0$ is small.  However, despite convergence of the initial data asserted in Theorem~\ref{thm-accuracy-t=0} and Corollary~\ref{cor-L2-convergence}, we cannot guarantee that $\widetilde{\psi}(\diamond,\mathbf{t})$ and $\psi(\diamond,\mathbf{t})$ remain close for nonzero $\mathbf{t}$ in the limit $\eps\downarrow 0$, because there are no known stability results for the Cauchy problem of the focusing NLS hierarchy that are uniform in $\epsilon$.  Indeed, the maximum exponential growth rate for the well-known modulational instability is inversely proportional to $\eps$.  Nonetheless, we will demonstrate below that some predictions about the dynamics for initial data $\psi_0(\diamond)$ carry over also to the semiclassical soliton ensemble initial data $\widetilde{\psi}(\diamond,\mathbf{0})$.
\label{rem-WatchOut}
\end{rem}

\subsubsection{Extreme focusing in the NLS hierarchy}
\label{sec:extreme}
To formulate the next results, we first introduce a certain function $\Psi(X,T_2,T_3,\dots,T_M)$ by means of an auxiliary Riemann-Hilbert problem.
\begin{myrhp}[Rogue wave of infinite order for the NLS hierarchy]
Fix an integer $M\ge 2$.  Given $(X,T_2,T_3,\dots,T_M)\in\mathbb{R}^M$, seek a $2\times 2$ matrix function $\mathbf{R}(\Lambda)=\mathbf{R}(\Lambda;X,T_2,T_3,\dots,T_M)$ with the following properties:
\begin{itemize}
\item[]\textit{\textbf{Analyticity:}} $\mathbf{R}(\Lambda)$ is an analytic function of $\Lambda$ in the domains $|\Lambda|< 1$ and $|\Lambda|>1$. 
\item[]\textit{\textbf{Jump condition:}} $\mathbf{R}(\Lambda)$ takes continuous boundary values on the unit circle from the interior (denoted $\mathbf{R}_-(\Lambda)$) and exterior (denoted $\mathbf{R}_+(\Lambda)$), and these are related by 
\begin{multline}
\mathbf{R}_+(\Lambda)=\mathbf{R}_-(\Lambda)\exp\left(-\ii \left(\Lambda X+\sum_{m=2}^M\Lambda^mT_m + 2\Lambda^{-1}\right)\sigma_3\right)\\
{}\cdot\mathbf{Q}^{-1}\exp\left(\ii \left(\Lambda X+\sum_{m=2}^M\Lambda^mT_m+ 2\Lambda^{-1}\right)\sigma_3\right), \quad|\Lambda|=1,\quad
\mathbf{Q}:=\frac{1}{\sqrt{2}}\begin{pmatrix}1&-1\\1 & 1\end{pmatrix}.
\label{eq:RWIO-jump}
\end{multline}
\item[]\textit{\textbf{Normalization:}} $\mathbf{R}(\Lambda)\to\mathbb{I}$ as $\Lambda\to\infty$.
\end{itemize}
\label{rhp:RWIO}
\end{myrhp}
This problem has a unique solution for each choice of $(X,T_2,T_3,\dots,T_M)\in\mathbb{R}^M$, as follows from Zhou's vanishing lemma \cite{Zhou89}, and the function $\Psi(X,T_2,T_3,\dots,T_M)$ defined by
\begin{equation}
\Psi(X,T_2,T_3,\dots,T_M):=2\ii\lim_{\Lambda\to\infty}\Lambda R_{12}(\Lambda;X,T_2,T_3,\dots,T_M)
\label{eq:RWIO-extract}
\end{equation}
is a complex-valued function whose real and imaginary parts are real-analytic 
function of the $M$ arguments.  In the case $M=2$, this function first 
appeared in the paper of Suleimanov \cite{Suleimanov17}, where it was formally 
proposed as a dispersive regularization of the blowup/collapse predicted by 
Talanov's analysis \cite{Talanov65} of the dispersionless focusing NLS system 
reviewed above in Section~\ref{sec-dispersionless}.  Later, in 
\cite{BilmanLM20}, the same function appeared as a near-field/high-order limit 
of fundamental rogue-wave solutions of the focusing NLS equation with nonzero 
boundary conditions, where it was called the \emph{rogue wave of infinite 
order}.  This function has also been shown to arise in the study of 
boundary layers for the sharp-line Maxwell-Bloch system in characteristic 
coordinates \cite{LiM24}, multiple-pole solutions of the focusing NLS equation 
\cite{BilmanB19}, 
and more general rogue-wave solutions of the focusing NLS equation arising from 
iterated B\"acklund transformations \cite{BilmanM22,BilmanM24b}.
In particular $\Psi(\diamond,0,\dots,0)$ is a real-valued function that is not 
in $L^1(\mathbb{R})$ but is square integrable with 
$\|\Psi(\diamond,0,\dots,0)\|_{L^2(\mathbb{R})}=\sqrt{8}$, and 
$\Psi(0,0,\dots,0)=4$.  In \cite{BilmanLM20} it was proved that this function 
for $M=2$ is an exact solution of the focusing NLS equation in the form 
$\ii \Psi_{T_2} + \frac{1}{2}\Psi_{XX}+|\Psi|^2\Psi=0$, and that it also 
solves equations in the Painlev\'e-III hierarchy of Sakka \cite{Sakka09} as a 
function of $X$ for each fixed $T_2$.  See \cite{BilmanM24} for further 
information about the $M=2$ case.  Similar arguments based on the dressing 
method show that for arbitrary $M=2,3,4,\dots$, the function 
$\Psi(X,T_2,T_3,\dots,T_M)$ is a simultaneous solution of the first $M-1$ 
flows, suitably rescaled by setting $\epsilon=1$ and replacing 
$(x,t_2,t_3,\dots,t_M)$ with $(X,T_2,T_3,\dots,T_M)$, of the focusing NLS 
hierarchy.  See \cite{BilmanBMY25} for further details about 
$\Psi(X,T_2,T_3,\dots,T_M)$ and its generalizations.

A key point is that the solution $\Psi(X,T_2,T_3,\dots,T_M)$ describes the semiclassical asymptotic behavior of solutions of the focusing NLS hierarchy whenever $\phaseint(\SP)$ and $\tailint(\SP)$ are polynomials:
\begin{thm}[Suleimanov-Talanov focusing of the hierarchy and dispersive regularization]
Suppose $\psi_0(x)=A(x)$ is a semicircular Klaus-Shaw potential for which $\phaseint(\SP)$ and $\tailint(\SP)$ are polynomials of exact degree $2\mathcal{P}$ and $2\mathcal{Q}-1$ respectively:
\begin{equation}
\phaseint(\SP)=\sum_{p=0}^\mathcal{P} \phaseint_{p}\lambda^{2p},\quad \tailint(\SP)=\ii\sum_{q=1}^\mathcal{Q} \tailint_q\lambda^{2q-1}, 
\label{eq:polynomial-forms}
\end{equation}
and let $M:=\max\{2\mathcal{P},2\mathcal{Q}-1\}$.  Fixing an arbitrary integer $K$, define a point in $\mathbb{R}^M$ by
\begin{equation}
(x^\circ,t_2^\circ,t_3^\circ,\dots,t_M^\circ):=\begin{cases}
-\frac{1}{2}\left(\tailint_1,(2K+1)\phaseint_1,\tailint_2,(2K+1)\phaseint_2,\dots,(2K+1)\phaseint_\mathcal{P}\right),& M=2\mathcal{P},\\
-\frac{1}{2}\left(\tailint_1,(2K+1)\phaseint_1,\tailint_2,(2K+1)\phaseint_2,\dots,\tailint_\mathcal{Q}\right),&M=2\mathcal{Q}-1,
\end{cases}
\label{eq:focus-K}
\end{equation}
and denote the corresponding times by $\mathbf{t}^\circ:=(t_2^\circ,t_3^\circ,\dots,t_M^\circ)$.
Then the semiclassical soliton ensemble $\widetilde{\psi}(x,\mathbf{t})$ associated with $\psi_0(x)=A(x)$ satisfies 
\begin{multline}
\widetilde{\psi}\left(x^\circ+\frac{\epsilon^2}{\nu}X,\mathbf{t}^\circ + \left(\frac{\epsilon^3}{\nu^{2}}T_2,\frac{\epsilon^4}{\nu^{3}}T_3,\dots,\frac{\epsilon^{M+1}}{\nu^{M}}T_M\right)\right)\\
{}=\ii (-1)^{K+N}  \frac{\nu}{\eps}\Psi(X,T_2,T_3,\dots,T_M) + \bigo{1}
\label{eq:hierarchy-focus-approx}
\end{multline}
as $\eps=\eps_N\downarrow 0$, uniformly for $(X,T_2,T_3,\dots,T_M)\in\mathbb{R}^M$ bounded, where
\begin{equation}
\nu:=\frac{1}{2\pi} \int_0^{A_\mathrm{max}} \phaseint(\ii s) \dd s.
\label{eq:nu-formula}
\end{equation}
\label{thm:multi-time}
\end{thm}
This result shows that $\widetilde{\psi}(x,\mathbf{t})$ exhibits focusing events of large amplitude proportional to $\eps^{-1}$ in the neighborhood of each focal point $(x^\circ,\mathbf{t}^\circ)\in\mathbb{R}^M$, the family of which is parameterized by an arbitrary integer $K$.  In a neighborhood of size proportional to $\eps^2$ in $x$ and proportional to $\eps^{m+1}$ in $t_m$ of each focal point, the wave field takes on a universal form involving the function $\Psi(X,T_2,T_3,\dots,T_M)$.  We refer to this kind of focusing as \emph{Suleimanov-Talanov focusing}.

\subsubsection{Extreme focusing for mixed flows}
According to Proposition~\ref{prop:mu-odd-analytic}, the $x$-coordinate of each focal point is the same and is given explicitly in terms of the support endpoints as $x^\circ = \frac{1}{2}(X_++X_-)$.  On the other hand, the time coordinates of the focal points $\mathbf{t}^\circ$ vary with the index $K\in\mathbb{Z}$, lying equally-spaced along a straight line in the multi-time parameter space $\mathbb{R}^{M-1}$ for $\mathbf{t}$.  Considering now a particular mixture of the flows in the focusing NLS hierarchy defined by relating the coordinates $t_2,t_3,\dots,t_M$ to a single real time variable $t$ by $t_m=a_mt$ for some fixed real constants $a_2,a_3,\dots,a_M$, it becomes clear that the focal points correspond to rare events that do not occur at all for most mixtures.  If they do occur, the fact that the rescaled local time coordinates $T_2,T_3,\dots,T_M$ should be bounded while the unscaled time variables $t_2,t_3,\dots,t_M$ should be in fixed proportion means that the limiting function $\Psi$ should be evaluated at $T_2=T_3=\cdots=T_{M-1}=0$.  However, it is also clear that the type of phenomena that can occur depends on whether the line containing the focal points passes through the origin $\mathbf{t}=\mathbf{0}$.  This happens exactly when $\tailint(\SP)$ is a linear monomial, or equivalently by Corollary~\ref{cor:even}, when $\psi_0(x)=A(x)$ is even about the midpoint $x^\circ$ of its support interval $[X_-,X_+]$.  See Figure~\ref{fig:focusing-times}.  
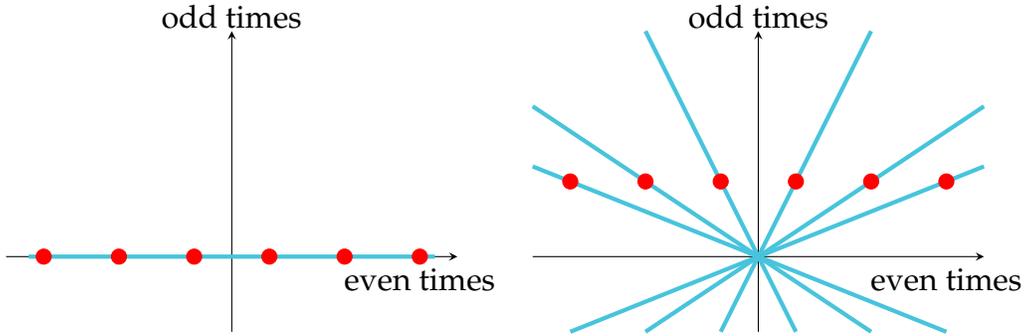
\begin{figure}[h]
\begin{tikzpicture}[>=stealth]
\begin{scope}[xshift=-3.5cm]
\draw[thin,->,black] (0cm,-1cm) -- (0cm,3cm);
\draw[thin,->,black] (-3cm,0cm) -- (3cm,0cm);
\draw[ultra thick,SkyBlue] (-2.7cm,0cm) -- (2.7cm,0cm);
\node at (0cm,3.2cm) {odd times};
\node at (2.5cm,-0.3cm) {even times};
\draw[fill,red] (-0.5cm,0cm) circle [radius=0.1cm];
\draw[fill,red] (-1.5cm,0cm) circle [radius=0.1cm];
\draw[fill,red] (-2.5cm,0cm) circle [radius=0.1cm];
\draw[fill,red] (0.5cm,0cm) circle [radius=0.1cm];
\draw[fill,red] (1.5cm,0cm) circle [radius=0.1cm];
\draw[fill,red] (2.5cm,0cm) circle [radius=0.1cm];
\end{scope}
\begin{scope}[xshift=3.5cm]
\draw[thin,->,black] (0cm,-1cm) -- (0cm,3cm);
\draw[thin,->,black] (-3cm,0cm) -- (3cm,0cm);
\draw[ultra thick,SkyBlue] (0.5cm,-1cm) -- (-1.5cm,3cm);
\draw[ultra thick,SkyBlue] (1.5cm,-1cm) -- (-3cm,2cm);
\draw[ultra thick,SkyBlue] (2.5cm,-1cm) -- (-3cm,1.2cm);
\draw[ultra thick,SkyBlue] (-0.5cm,-1cm) -- (1.5cm,3cm);
\draw[ultra thick,SkyBlue] (-1.5cm,-1cm) -- (3cm,2cm);
\draw[ultra thick,SkyBlue] (-2.5cm,-1cm) -- (3cm,1.2cm);
\node at (0cm,3.2cm) {odd times};
\node at (2.5cm,-0.3cm) {even times};
\draw[fill,red] (-0.5cm,1cm) circle [radius=0.1cm];
\draw[fill,red] (-1.5cm,1cm) circle [radius=0.1cm];
\draw[fill,red] (-2.5cm,1cm) circle [radius=0.1cm];
\draw[fill,red] (0.5cm,1cm) circle [radius=0.1cm];
\draw[fill,red] (1.5cm,1cm) circle [radius=0.1cm];
\draw[fill,red] (2.5cm,1cm) circle [radius=0.1cm];
\end{scope}
\end{tikzpicture}
\caption{Left:  the case that $\tailint(\SP)$ is a linear monomial.  Red points indicate the focus times $\mathbf{t}^\circ$ for different integers $K$, and there is a mixture of even flows (blue line) that experiences each focus.  Right:  the case that $\tailint(\SP)$ has a cubic or higher-order term, where there is a different mixture of flows for each $K\in\mathbb{Z}$ that focuses just once.}
\label{fig:focusing-times}
\end{figure}
Our result for mixtures of flows in the focusing NLS hierarchy is as follows:

\begin{thm}[Suleimanov-Talanov focusing of mixed flows and dispersive regularization]
Under the assumptions of Theorem~\ref{thm:multi-time}, 
\begin{enumerate}
\item If $\tailint(\SP)$ is a linear monomial, i.e., $\psi_0(x)=A(x)$ is even about $x^\circ$, then 
all mixed flows of the focusing NLS hierarchy that undergo Suleimanov-Talanov focusing correspond to  coordinates (a combination of the even flows only in the hierarchy)
\begin{equation}
(a_2,a_3,\dots,a_M)=-\frac{1}{2}\alpha\left(\phaseint_1,0,\phaseint_2,0,\dots,\phaseint_\mathcal{P}\right),\quad M=2\mathcal{P},
\end{equation}
for a fixed real $\alpha\neq 0$, and the flow with commensurate coordinates $t_m=a_mt$ exhibits infinitely many Suleimanov-Talanov focusing events periodically in time $t$ with period $2/|\alpha|$.  Specifically,
if we write $\widetilde{\psi}(x,t):=\widetilde{\psi}(x,(a_2,a_3,\dots,a_M)t)$ and set $t^\circ:=(2K+1)/\alpha$ for an arbitrary integer $K$, then with $\nu$ defined by \eqref{eq:nu-formula},
\begin{equation}
\widetilde{\psi}\left(x^\circ+\frac{\epsilon^2}{\nu}X,t^\circ + \frac{\epsilon^{M+1}}{a_M\nu^M}T_M\right)= \ii (-1)^{K+N} \frac{\nu}{\eps}\Psi(X,0,0,\dots,0,T_M) + \bigo{1}
\label{eq:mixed-flow-focus}
\end{equation}
as $\eps\to 0$, uniformly for $(X,T_M)\in\mathbb{R}^2$ bounded.  
\item Otherwise, for each integer $K$, mixed flows of the focusing NLS hierarchy corresponding to coordinates $t_m=a_mt$ with
\begin{multline}
(a_2,a_3,\dots,a_M)=\\-\frac{1}{2}\alpha\begin{cases}\left((2K+1)\phaseint_1,\tailint_2,(2K+1)\phaseint_2,\tailint_3,\dots,(2K+1)\phaseint_\mathcal{P}\right), &M=2\mathcal{P}\\
\left((2K+1)\phaseint_1,\tailint_2,(2K+1)\phaseint_2,\tailint_3,\dots,\tailint_{2\mathcal{Q}-1}\right),& M=2\mathcal{Q}-1
\end{cases}
\end{multline}
exhibit exactly one Suleimanov-Talanov focusing event near time $t=t^\circ:=1/\alpha$, where $\widetilde{\psi}(x,t):=\widetilde{\psi}(x,(a_2,a_3,\dots,a_M)t)$ is characterized by \eqref{eq:mixed-flow-focus} in the limit $\eps\to 0$ with $(X,T_M)\in\mathbb{R}^2$ bounded.
\end{enumerate}
\label{thm:mixture}
\end{thm}

\subsubsection{Application to the focusing NLS equation}
In particular, this theorem allows us to prove a rigorous version of the result conjectured in the paper of Suleimanov \cite{Suleimanov17}.  To apply Theorem~\ref{thm:mixture} in the case of initial data $\psi_0(x)=A(x)$ given by \eqref{eq:ExactSemicircle} as is consistent with a Talanov-type solution for $E<0$ of the dispersionless focusing NLS system \eqref{eq:dispersionless-focusing}, first note that from the definition  \eqref{eq:ZS-phase-integral} of $\phaseint(\SP)$, by a residue calculation at $x=\infty$ (see \cite[Eqn.\@ (4.27)]{Buckingham:2016}),
\begin{equation}
\begin{split}
\phaseint(\ii s)&=\int_{x_-(s)}^{x_+(s)}\sqrt{\frac{4A_\mathrm{max}^2(x-X_-)(X_+-x)}{(X_+-X_-)^2}-s^2}\,\dd x\\ &=\frac{\pi(X_+-X_-)}{4A_\mathrm{max}}(A_\mathrm{max}^2-s^2),\quad 0<s<A_\mathrm{max}
\end{split}
\label{eq:semicircle-Psi}
\end{equation}
and an easier calculation starting from \eqref{eq:mu-define} gives\footnote{This corrects \cite[Eqn.\@ (4.29)]{Buckingham:2016}, which includes an extraneous factor of $\frac{1}{2}$.}
\begin{equation}
\tailint(\ii s)=(X_++X_-)s,\quad 0<s<A_\mathrm{max}.
\label{eq:semicircle-mu}
\end{equation}
Replacing $s$ with $s=-\ii\SP$ we see that $\phaseint(\SP)$ and $\tailint(\SP)$ are the following polynomials in $\SP$
\begin{equation}
\phaseint(\SP)=\frac{\pi(X_+-X_-)}{4A_\mathrm{max}}(A_\mathrm{max}^2+\SP^2)\quad\text{and}\quad
\tailint(\SP)=-\ii(X_++X_-)\SP.
\label{eq:Psi-mu-polynomials}
\end{equation}
According to Propositions~\ref{prop:Psi-even-analytic} and \ref{prop:mu-odd-analytic}, these are the simplest possible for a semicircular Klaus-Shaw potential (having the minimal number of terms in their Taylor series at $\SP=0$).  They match the form \eqref{eq:polynomial-forms} with $\mathcal{P}=\mathcal{Q}=1$, and hence $M=\max\{2\mathcal{P},2\mathcal{Q}-1\}=2$.  Our rigorous version of Suleimanov's result \cite{Suleimanov17} is then as follows:

\begin{cor}[Suleimanov-Talanov focusing in the NLS equation]
Let $\widetilde{\psi}(x,t_2)$ denote the semiclassical soliton ensemble with $M=2$, for the Talanov-type initial condition $\psi_0(x)=A(x)$ given by \eqref{eq:ExactSemicircle}.  Then $\widetilde{\psi}(x,t_2)$ is an exact solution of the focusing NLS equation \eqref{nls} with initial condition close to $\psi_0(x)$ as described by Theorem~\ref{thm-accuracy-t=0}.  Furthermore, $\widetilde{\psi}(x,t_2)$ exhibits Suleimanov-Talanov focusing near $x=x^\circ:=\frac{1}{2}(X_++X_-)$ periodically in time $t_2$ in the sense that for each integer $K$, defining a focus time by
\begin{equation}
t_K^\circ:=-\frac{\pi(X_+-X_-)}{8A_\mathrm{max}}(2K+1),
\end{equation}
we have
\begin{multline}
\widetilde{\psi}\left(x^\circ+\frac{12\eps^2}{A_\mathrm{max}^2(X_+-X_-)}X, t_K^\circ+ \frac{144\epsilon^3}{A_\mathrm{max}^4(X_+-X_-)^2}T_2\right)\\=(-1)^{K+N} \ii \frac{A_\mathrm{max}^2(X_+-X_-)}{12\eps}\Psi(X,T_2) + \bigo{1}
\label{eq:cor.Talanov.coords}
\end{multline}
as $\eps\to 0$ through the integer sequence $\eps=\eps_N:=\frac{1}{4}A_\mathrm{max}(X_+-X_-)N^{-1}$, $N=1,2,3,\dots$, where the error term is uniform for bounded $(X,T_2)\in\mathbb{R}^2$.
\label{cor:Talanov}
\end{cor}
\begin{proof}
Apply Theorem~\ref{thm:mixture} in the case $M=2$ with $\mathcal{P}=\mathcal{Q}=1$, in which case the first scenario holds.  We take $a_2=1$ to ensure that $\widetilde{\psi}$ solves the focusing NLS equation in the form \eqref{nls}.
\end{proof}
\begin{rem}
This result differs from the claim in \cite{Suleimanov17} in two ways.  Firstly, it concerns initial data $\psi_0(\diamond)=A(\diamond)$ corresponding to an integration constant $E<0$ instead of $E=0$ for the dispersionless system \eqref{eq:dispersionless-focusing}.  Second, it is not a statement about the initial-value problem for \eqref{nls} with Cauchy data $\psi_0(\diamond)=A(\diamond)$ but rather with modified Cauchy data $\widetilde{\psi}(\diamond,\mathbf{0})$.  While $\psi_0(\diamond)$ and $\widetilde{\psi}(\diamond,\mathbf{0})$ are close according to Theorem~\ref{thm-accuracy-t=0} and Corollary~\ref{cor-L2-convergence}, these two initial conditions are not equal.  Therefore, in light of the strong instabilities pointed out in Remark~\ref{rem-WatchOut}, it is remarkable that the dispersionless theory makes such an accurate prediction.
\end{rem}
The focusing events nearest to $t_2=0$ correspond to $K=-1,0$, where the solution grows to size proportional to $\eps^{-1}$ near the points $(x,t_2)=(\frac{1}{2}(X_++X_-),\pm\pi(X_+-X_-)/(8A_\mathrm{max}))$.  The time coordinates here are precisely $\pm\frac{1}{2}\Delta t$ as defined in \eqref{eq:Talanov-Delta-t} using also \eqref{eq:E-determine}--\eqref{eq:F-determine} with the initial width being $w(0)=\frac{1}{2}(X_+-X_-)$.  Thus the solution becomes large exactly near the points predicted by the dispersionless Talanov theory \cite{Talanov65} for $E<0$ as discussed in Section~\ref{sec-dispersionless} and near each of these two points the blowing up and collapsing solution is dispersively regularized according to the prediction of Suleimanov \cite{Suleimanov17}.  Moreover, the NLS solution survives beyond the time interval $(-\frac{1}{2}\Delta t,\frac{1}{2}\Delta t)$ and exhibits periodic ``breathing'' of period $\Delta t$ consisting of alternate periods of dispersive spreading and refocusing.  See Figure~\ref{fig:NLSfigure}.
\begin{figure}[htb]
\includegraphics[width=\textwidth]{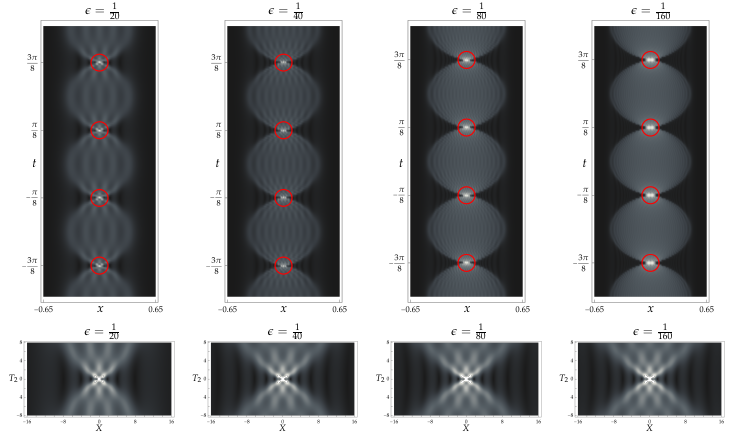}
\caption{Top row: 
Density plots of $|\widetilde \psi(x,t_2)|$ for different values of $\eps$, where $\widetilde{\psi}$ is the semiclassical soliton ensemble solution of NLS generated from the simplest semicircular Klaus-Shaw potential $A(x) = \sqrt{1- (x/2)^2} \chi_{[-\frac{1}{2}, \frac{1}{2}]}(x)$. 
Red circles indicate the location of the Suleimanov-Talanov focusing events described by Corollary~\ref{cor:Talanov}. 
Bottom row: 
Plots of $|\widetilde{\psi}(x,t_2)|$ in the rescaled local coordinates  $(X, T_2)$ defined by \eqref{eq:cor.Talanov.coords} centered at the focus point $(x^\circ, t_{-1}^\circ) = ( 0 ,\tfrac{\pi}{8})$. 
}
\label{fig:NLSfigure}
\end{figure}

\subsubsection{Application to other mixed flows:  Hirota and LPD equations}
There are many other semicircular Klaus-Shaw potentials for which $\phaseint(\SP)$ and $\tailint(\SP)$ are polynomials.  To prove this, we introduce polynomial perturbations of \eqref{eq:Psi-mu-polynomials}:
\begin{equation}
\begin{aligned}
&\phaseint(\SP) = \phaseint^\mathrm{sc}(\SP)\left(1+\sum_{k=1}^{\mathcal{P}-1}B_k\SP^{2k}\right), \quad
&&\phaseint^\mathrm{sc}(\SP):= \phaseint_0 \left( 1+ \frac{\SP^2}{ A_\mathrm{max}^2} \right),\\ 
&\tailint(\SP) =\tailint^\mathrm{sc}(\SP)+\ii\sum_{q=2}^\mathcal{Q}\tailint_q\SP^{2q-1}, \quad
&&\tailint^\mathrm{sc}(\SP):= \ii \tailint_1 \SP,
\end{aligned}
\label{eq:phaseandtail-perturbations}
\end{equation}
for some real coefficients $\phaseint_0$, $A_\mathrm{max}$, $B_1,\dots,B_{\mathcal{P}-1}$, $\tailint_1,\dots,\tailint_{\mathcal{Q}}$.  The perturbation of $\phaseint^{\mathrm{sc}}(\SP)$ is taken to be relative instead of additive in order to fix the maximum amplitude $A_\mathrm{max}$, since Proposition~\ref{prop:Psi-even-analytic} requires that $\phaseint(\ii A_\mathrm{max})=0$. 
If the coefficients $B_1,\dots,B_{\mathcal{P}-1},\Xi_1,\dots,\Xi_\mathcal{Q}$ are sufficiently small, then these expressions will be the phase integral $\phaseint$ and tail integral $\tailint$ of a semicircular Klaus-Shaw potential $A(x)$ with maximum amplitude $A_\mathrm{max}$ as determined by 
\eqref{eq:ZS-phase-integral} and \eqref{eq:mu-define} respectively.  To see this, we apply Proposition~\ref{prop:Psi-mu-invert}.  
First, we calculate
\begin{equation}
\begin{split}
\frac{1}{\pi}\int_0^s\frac{\displaystyle\frac{\dd}{\dd m}\tailint(\ii m)\,\dd m}{\sqrt{s^2-m^2}}  
&= -\frac{\tailint_1}{2}+\frac{1}{\pi}\sum_{q=2}^\mathcal{Q}(-1)^q(2q-1)\tailint_q\int_0^s\frac{m^{2q-2}\,\dd m}{\sqrt{s^2-m^2}} \\
&=-\frac{\tailint_1}{2} +\frac{1}{\pi}\sum_{q=2}^\mathcal{Q}(-1)^q(2q-1)\tailint_q\int_0^1\frac{v^{2q-2}\,\dd v}{\sqrt{1-v^2}}\cdot s^{2q-2}\\
&=-\frac{\tailint_1}{2} +\sum_{q=2}^\mathcal{Q}\frac{(-1)^q\tailint_q}{2}\frac{(2q-1)!!}{(2q-2)!!}\cdot s^{2q-2}.
\end{split}
\end{equation}
Also,
\begin{multline}
\frac{1}{\pi}\int_s^{A_\mathrm{max}}\frac{\displaystyle\frac{\dd}{\dd m}\phaseint(\ii m)\,\dd m}{\sqrt{m^2-s^2}} =
-\frac{2\phaseint_0}{\pi A_\mathrm{max}^2} \sqrt{A_\mathrm{max}^2-s^2} \\
+\frac{2\phaseint_0}{\pi A_\mathrm{max}^2} \sum_{k=1}^{\mathcal{P}-1}(-1)^kB_k\left(A_\mathrm{max}^2k\int_s^{A_\mathrm{max}}\frac{m^{2k-1}\,\dd m}{\sqrt{m^2-s^2}}-(k+1)\int_s^{A_\mathrm{max}}\frac{m^{2k+1}\,\dd m}{\sqrt{m^2-s^2}}\right),
\end{multline}
and since by the substitution $m=\sqrt{s^2+(A_\mathrm{max}^2-s^2)z^2}$,
\begin{equation}
\begin{split}
\int_s^{A_\mathrm{max}}\frac{m^{2k-1}\,\dd m}{\sqrt{m^2-s^2}}
&= \sqrt{A_\mathrm{max}^2-s^2}\int_0^1(s^2+(A_\mathrm{max}^2-s^2)z^2)^{k-1}\,\dd z,\\
\int_s^{A_\mathrm{max}}\frac{m^{2k+1}\,\dd m}{\sqrt{m^2-s^2}}&=\sqrt{A_\mathrm{max}^2-s^2}\int_0^1(s^2+(A_\mathrm{max}^2-s^2)z^2)^k\,\dd z,
\end{split}
\end{equation}
we obtain
\begin{equation}
\frac{1}{\pi}\int_s^{A_\mathrm{max}}\frac{\displaystyle\frac{\dd}{\dd m}\phaseint(\ii m)\,\dd m}{\sqrt{m^2-s^2}}
=-\frac{2\phaseint_0}{\pi A_\mathrm{max}^2} \sqrt{A_\mathrm{max}^2-s^2}
\left(1+\sum_{k=1}^{\mathcal{P}-1}(-1)^kB_k P_{k}(s^2)\right), 
\end{equation}
where $P_{k}(s^2)$ is a polynomial in $s^2$ of degree $k$ given by
\begin{equation}
P_{k}(s^2):=\int_0^1\left((k+1)(s^2+(A_\mathrm{max}^2-s^2)z^2)-A_\mathrm{max}^2k\right)(s^2+(A_\mathrm{max}^2-s^2)z^2)^{k-1}\,\dd z.
\end{equation}
This yields
\begin{multline}
x_\pm(s)=-\frac{\tailint_1}{2} +\sum_{q=2}^\mathcal{Q}\frac{(-1)^q\tailint_q}{2}\frac{(2q-1)!!}{(2q-2)!!}\cdot s^{2q-2}\\
\pm \frac{2\phaseint_0}{\pi A_\mathrm{max}^2} \sqrt{A_\mathrm{max}^2-s^2}\left(1+\sum_{k=1}^{\mathcal{P}-1}(-1)^kB_kP_{k}(s^2)\right).
\label{eq:inverse-functions-polynomials}
\end{multline}
These will be the two inverse functions of a semicircular Klaus-Shaw potential $A(x)$ with maximum amplitude $A_\mathrm{max}$ provided that $x_+(s)>x_-(s)$ and $x_+(s)$ (resp., $x_-(s)$) is decreasing (resp., increasing) on $0<s<A_\mathrm{max}$.  For fixed $\mathcal{Q}$ and $\mathcal{P}$, this is clearly the case as long as the coefficients $B_1,\dots,B_{\mathcal{P}-1}$ and $\tailint_2,\dots,\tailint_\mathcal{Q}$ are sufficiently small.  The maximizer $x_0$ of the semicircular Klaus-Shaw potential $A(x)$ obtained is 
\begin{equation}
x_0:=x_\pm(A_\mathrm{max})=-\frac{\tailint_1}{2} +\sum_{q=2}^{\mathcal{Q}}\frac{(-1)^q\tailint_q}{2}\frac{(2q-1)!!}{(2q-2)!!}A_\mathrm{max}^{2q-2},
\end{equation}
and the support endpoints  $X_\pm$ are given by
\begin{equation}
X_\pm := x_\pm(0)=-\frac{\tailint_1}{2}  \pm \frac{2\phaseint_0}{\pi A_\mathrm{max}}\left(1+\sum_{k=1}^{\mathcal{P}-1}(-1)^kB_kP_{k}(0)\right).
\label{eq:xpm-zero}
\end{equation}
Note that this gives $\tailint_1 = -(X_+ + X_-)$, as is consistent with Proposition~\ref{prop:mu-odd-analytic}. 

For low-degree examples it is possible to invert the above relationships and express things in terms of the Klaus-Shaw potential directly. 
Consider the case that $\mathcal{P}=1$ and $\mathcal{Q}=2$.  
Then the corresponding inverse functions from \eqref{eq:inverse-functions-polynomials} are 
\begin{equation}
x_\pm(s)=\frac{X_++X_-}{2}+\frac{3}{4}\tailint_2s^2 \pm \frac{X_+ - X_-}{2 A_\mathrm{max}} \sqrt{A_\mathrm{max}^2-s^2},\quad 0<s<A_\mathrm{max},
\label{eq:Hirota-inverse-functions}
\end{equation}
where \eqref{eq:xpm-zero} has been used to express $\phaseint_0$ in terms of $X_+ - X_-$ and $A_\mathrm{max}^2$. 
Therefore also
\begin{equation}
x_\pm'(s)=\left(\frac{3}{2}\tailint_2 \mp  \frac{X_+ - X_-}{2 A_\mathrm{max}} \frac{1}{\sqrt{A_\mathrm{max}^2-s^2}}\right)s,\quad 0<s<A_\mathrm{max}.
\end{equation}
From this, we can see easily that $\mp x_\pm'(s)>0$ on $(0,A_\mathrm{max})$ with linear vanishing at $s=0$ as necessarily holds for the inverse functions of every semicircular Klaus-Shaw potential if and only if 
\begin{equation}
3 |\tailint_2|<\frac{X_+-X_-}{A_\mathrm{max}^2}.
\end{equation}

Enforcing this inequality on  the coefficient $\tailint_2$ by setting
\begin{equation}
	\tailint_2 := \frac{X_+ - X_-}{3 A_\mathrm{max}^2} \xi, \qquad \xi \in (-1,1),
 \end{equation}
 we obtain a semicircular Klaus-Shaw potential with support $[X_-,X_+]$ that is additionally parametrized by $\xi$. We can write the potential explicitly by replacing $x_\pm(s)$ on the left-hand side of \eqref{eq:Hirota-inverse-functions} with $x$, yielding a quadratic equation for $s^2=A(x)^2$:
\begin{gather}
\frac{\xi^2}{4} \frac{s^4}{A_\mathrm{max}^4} + \left[ 1 - \xi \frac{2x - X_+ - X_-}{X_+ - X_-} \right] \frac{s^2}{A_\mathrm{max}^2} + \left( \frac{2x - X_+ - X_-}{X_+-X_-} \right)^2 - 1 = 0.
\end{gather}
Noting that the constant (in $s$) term above is negative for $x \in (X_-, X_+)$, the roots $s^2$ have opposite signs for $x\in [X_-,X_+]$, so selecting the positive root and taking a square root gives \eqref{eq:HirotaIC}. 
Applying Theorem~\ref{thm:mixture} to the semicircular Klaus-Shaw potential \eqref{eq:HirotaIC} in the case that $M=\max\{2\mathcal{P},2\mathcal{Q}-1\}=3$ (Hirota equation) yields the following corollary.

\begin{cor}[Suleimanov-Talanov focusing in the Hirota equation]
Consider the Hirota equation \eqref{eq:Hirota} with nonzero coefficients $a_2$ and $a_3$. For parameters $A_\mathrm{max} >0$, $X_+ > X_-$, and $\xi \in (-1,1)$, let $\psi_0(x)= A(x)$ be the semicircular Klaus-Shaw potential  
\begin{equation}\label{eq:HirotaIC}
	A(x) =  \frac{\sqrt{2} A_\mathrm{max}}{|\xi|} \left[  \sqrt{1 - 2 y(x) \xi + \xi^2} - 1 + y(x) \xi   \right]^{\mathrlap{{1/2}}} \chi_{[X_-, X_+]}(x), 
	\quad y(x)= \frac{2x - X_+ - X_-}{X_+ - X_-}, 
\end{equation}
and denote by $\widetilde \psi(x,t)$ the semiclassical soliton ensemble solution of \eqref{eq:Hirota} corresponding to $\psi_0$. Fix an integer $K\in\mathbb{Z}$.  Then whenever the parameters 
$(a_2,a_3; A_\mathrm{max}, \xi)$ 
satisfy
\begin{equation}\label{eq:Hirota.focus.condition}
	4 a_2 \xi - 3\pi (2K+1) a_3 A_\mathrm{max} = 0,  
\end{equation}
$\widetilde \psi(x,t)$ undergoes a single Suleimanov-Talanov focusing event near the point 
\begin{equation}
(x^\circ, t^\circ_K)=\left(\frac{1}{2}(X_++X_-), -\frac{\pi(X_+-X_-)}{8 a_2 A_\mathrm{max}}(2K+1)\right)
\end{equation}
in the sense that
\begin{multline}
\widetilde{\psi}\left(x^\circ + \frac{12\eps^2}{A_\mathrm{max}^2(X_+-X_-)}X,t^\circ_K+ \frac{\eps^4}{a_3}\left(\frac{12}{ A_\mathrm{max}^2(X_+-X_-)} \right)^3 T_3\right)\\
=(-1)^{K+N} \ii  \frac{A_\mathrm{max}^2(X_+-X_-)}{12\eps}\Psi(X,0,T_3) + \bigo{1}
\label{eq:cor.Hirota.coords}
\end{multline}
as $\eps\to 0$ through the integer sequence $\eps=\eps_N$, $N=1,2,3,\dots$, with the error estimate being uniform for bounded $(X,T_3)\in\mathbb{R}^2$.
\label{cor:Hirota}
\end{cor}

\begin{proof}
Using the calculations in the paragraph preceding Corollary~\ref{cor:Hirota}, a Klaus-Shaw potential of the form \eqref{eq:HirotaIC} produces polynomial phase and tail integrals given by
\begin{equation}\label{eq:Hirota_phase_and_tail}
	\phaseint(\SP) = \frac{\pi(X_+ - X_-)}{4 A_\mathrm{max}} ( A_\mathrm{max}^2+ \SP^2),
	\qquad
	\tailint(\SP) = - \ii (X_+ + X_-) \SP + \frac{\ii (X_+ - X_-) \xi}{3 A_\mathrm{max}^2} \SP^3.
\end{equation}
We then apply Theorem~\ref{thm:mixture} in the case $M=3$ with $\mathcal{P}=1$ and $\mathcal{Q}=2$, in which case the second scenario holds.  
\end{proof}

As an illustration of Corollary~\ref{cor:Hirota}, we fixed a suitable semicircular Klaus-Shaw potential consistent with the hypotheses as well as a small value of $\eps$, and then constructed the corresponding semiclassical soliton ensembles for the Hirota equation \eqref{eq:Hirota} with a fixed coefficient $a_2=1$, varying only the coefficient $a_3$.  The plots are shown in Figure~\ref{fig:Hirota}.  Suleimanov-Talanov focusing is only observed for certain quantized values of $a_3$, and for those values it occurs precisely once.  Although these focusing events occur for different equations and at different times, upon rescaling about the predicted focus coordinates $(x^\circ,t_K^\circ)$ to the coordinates $(X,T_3)$ the plots all appear similar, pointing toward the universal nature of the limiting function $\Psi(X,0,T_3)$.
\begin{figure}[htb]
\includegraphics[width=.95\textwidth]{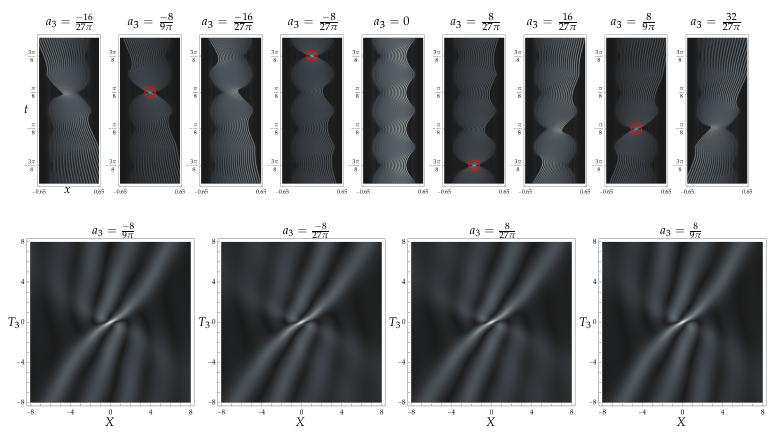}
\caption{
Top row:
Density plots of $|\widetilde \psi(x,t)|$ for the semiclassical soliton ensemble solutions, generated by the Klaus-Shaw potential \eqref{eq:HirotaIC} (with $X_\pm = \pm \tfrac{1}{2}, A_\mathrm{max}=1$, and $\xi=\tfrac{2}{3}$), of the Hirota equation \eqref{eq:Hirota} with $\eps=\tfrac{1}{80}$, $a_2=1$ and varying values of $a_3$.
Red circles indicate the locations of Suleimanov-Talanov focusing events as described in Corollary~\ref{cor:Hirota}. 
Unlike the even NLS flow, these occur only for quantized values of $a_3$ and at most once in the spacetime.  
Bottom row:
Density plots of $|\widetilde \psi(x,t)|$ in the rescaled local coordinates $(X, T_3)$ defined by \eqref{eq:cor.Hirota.coords} centered at the focus point $(x^\circ, t_K^\circ)$.
}
\label{fig:Hirota}
\end{figure}

As a second example, consider the case $\mathcal{P}=2$ and $\mathcal{Q}=1$, so that $M=\max\{2\mathcal{P},2\mathcal{Q}-1\}=4$ and we are thus in the setting of the LPD equation. From \eqref{eq:inverse-functions-polynomials}, the inverse functions become 
\begin{equation}
	x_\pm(s) = -\frac{\tailint_1}{2}  \pm \frac{2\phaseint_0}{\pi A_\mathrm{max}^2} \sqrt{A_\mathrm{max}^2-s^2} \left(1 + \frac{b_1}{3}\left(1 -  \frac{4s^2}{A_\mathrm{max}^2} \right) \right) , \quad 0<s<A_\mathrm{max}, 
\label{eq:xpm-lpd}
\end{equation}
where in \eqref{eq:phaseandtail-perturbations} we have written $B_1 = b_1 A_{\mathrm{max}}^{-2}$ to simplify the resulting formul\ae. The derivatives are  
\begin{equation}
	x'_\pm(s) = \mp \frac{2\phaseint_0}{\pi A_\mathrm{max}^2} \frac{s}{ \sqrt{A_\mathrm{max}^2-s^2}} \left( 1 + b_1\left( 3  - 4 \frac{s^2}{A_\mathrm{max}^2} \right) \right), \quad 0<s<A_\mathrm{max},
\end{equation}
and so the monotonicity condition $\mp x_\mp'(s)>0$ on $(0, A_\mathrm{max})$ guaranteeing that $A(x)$ is a semicircular Klaus-Shaw potential is satisfied if and only if 
\begin{equation}
	-\frac{1}{3} < b_1 <  1.
\end{equation}
Assuming that $b_1$ satisfies this inequality, we get a semicircular Klaus-Shaw potential whose support is the interval $[X_-, X_+]$ where, according to \eqref{eq:xpm-zero},
\begin{equation}
	X_\pm = x_\pm(0) = -\frac{\tailint_1}{2} \pm \frac{2\phaseint_0}{3\pi A_\mathrm{max}} (3+b_1).
	\label{eq:LPD-Xpm}
\end{equation}
The relations \eqref{eq:LPD-Xpm} yield expressions for the polynomial coefficients $\tailint_1$ and $\Phi_0$ in terms of the physical parameters of the initial condition:
\begin{equation}
	\tailint_1 = - ( X_+ + X_-), \qquad \Phi_0 = \frac{3 \pi A_\mathrm{max} (X_+-X_-)}{4(3+b_1)}.
\end{equation}
The potential $A(x)$ is then given implicitly by replacing $x_\pm(s)$ on the left-hand side of \eqref{eq:xpm-lpd}
with $x$, yielding a sextic equation for $s=A(x)$:
\begin{equation}
	\left[ \frac{2x - X_+ -X_- }{X_+ - X_-} \right]^2 = 
	\left( 1 - \frac{ A(x)^2}{A^2_\mathrm{max}} \right)\left( 1 - \gamma \frac{ A(x)^2}{A^2_\mathrm{max}} \right)^2, \quad
	\gamma := \frac{4b_1}{3+b_1},
\label{eq:lpd-IC-implicit}
\end{equation}
which has a unique solution such that $A(x) \in [0,A_\mathrm{max}]$ for each $x \in [X_-, X_+]$.  Note that $-\frac{1}{3}<b_1<1$ corresponds to $-\frac{1}{2}<\gamma<1$.

Using \eqref{eq:ZS-approximate-eigenvalues}-\eqref{eq:ZS-phase-integral} with $\Phi(\lambda)=\Phi_0(1+\lambda^2/A_\mathrm{max}^2)(1+b_1\lambda^2/A_\mathrm{max}^2)$, the semiclassical soliton ensemble $\widetilde \psi$ corresponding to this Klaus-Shaw potential is obtained via \eqref{eq:psitilde-reconstruct} from the solution of Riemann-Hilbert Problem~\ref{rhp-meromorphic} where the poles $\ii \widetilde{s}_n\in P\subset\mathbb{C}_+$ are given by
\begin{equation} 
	\ii \widetilde{s}_n = \ii A_\mathrm{max} \left[ \frac{1+b_1}{2b_1} \left(  1- \sqrt{ 1 -  \frac{4 b_1}{(1+b_1)^2}   \left(1 - \frac{2n+1}{2N} \right)} \right)\right]^{1/2}, \quad n = 0 , \dots, N-1;
\end{equation}
the residue coefficients $c_n^0$ are given by \eqref{eq:cn0} with $\widetilde{\tau}_n = (-1)^{n+1} \ee^{(X_+ + X_-) \widetilde{s}_n/\eps}$ according to \eqref{eq:ZS-proportionality-constant-no-interpolant}; and $\eps = \eps_N$ is given by \eqref{eq:LPD_epsilon} below.

Applying Theorem~\ref{thm:mixture} to the semicircular Klaus-Shaw potential \eqref{eq:lpd-IC-implicit} yields the following corollary.
\begin{cor}[Suleimanov-Talanov focusing in the LPD equation]
\label{cor:glpd}
Consider the LPD equation \eqref{eq:gLPD} with nonzero coefficients $a_2$ and $a_4$. Given parameters $A_\mathrm{max} >0$, $X_+ > X_-$, and $\gamma \in (-\tfrac{1}{2},1)$, let $\psi_0(x)= A(x)$ be the semicircular Klaus-Shaw potential supported on $[X_-, X_+]$ implicitly defined as the unique solution of 
\begin{equation}
	\left( \frac{2x - X_+  - X_-}{X_+ - X_-} \right)^2 = 
	\left( 1 - \frac{ A(x)^2}{A^2_\mathrm{max}} \right)\left( 1 - \gamma\frac{ A(x)^2}{A^2_\mathrm{max}} \right)^2, \qquad
	x \in [X_-, X_+], 
\label{eq:gLPD.IC2}
\end{equation}
for which $0 <  A(x) <  A_\mathrm{max}$ for each $x \in [X_-, X_+]$.
Denote by $\widetilde \psi(x,t)$ the corresponding semiclassical soliton ensemble solution of \eqref{eq:gLPD}. Then whenever the parameters $(a_2, a_4; A_\mathrm{max}, \gamma)$ satisfy
\begin{equation}\label{eq:gLPD.focus.condition}
	(4+2\gamma) a_4 A_\mathrm{max}^2 -   3 \gamma a_2 = 0,  
\end{equation}
$\widetilde \psi(x,t)$ experiences periodic Suleimanov-Talanov focusing events near the points 
\begin{equation}
(x^\circ, t^\circ_K)=\left(\frac{1}{2}(X_+ + X_-),  -(2K+1) \frac{\pi(X_+ - X_-)(2+\gamma)}{16 a_2 A_\mathrm{max}}\right), 
\quad K \in \Z
\end{equation}
in the sense that for each $K\in\mathbb{Z}$,
\begin{equation}
\widetilde{\psi}\left(x^\circ + \frac{\eps^2}{\nu_\textsc{lpd} }X,  t^\circ_K + \frac{\eps^{5}}{a_4 \nu_\textsc{lpd}^4 } T_4 \right) 
=\ii (-1)^{K+N} \frac{\nu_\textsc{lpd}}{\eps}\Psi(X,0,0,T_4) + \bigo{1}
\label{eq:cor.glpd.coords}
\end{equation}
as $\eps\to 0$ through the integer sequence $\eps=\eps_N$, $N=1,2,3,\dots$, with the error estimate being uniform for bounded $(X,T_4)\in\mathbb{R}^2$. Here
\begin{equation}
	\eps_N = \frac{(4-\gamma)(X_+ - X_-) A_\mathrm{max}}{16N}, 
	\qquad
	\nu_\textsc{lpd} = \frac{(5-2\gamma)(X_+ - X_-)A_\mathrm{max}} {60}.
\label{eq:LPD_epsilon}	
\end{equation}
\end{cor}

\begin{proof}
A Klaus-Shaw potential of the form \eqref{eq:gLPD.IC2} produces phase and tail integrals given by
\begin{equation}\label{eq:gLPD_phase_and_tail}
	\begin{aligned}
	\phaseint(\SP) &= \frac{(4-\gamma) \pi (X_+ - X_-) }{16A_\mathrm{max}} (A_\mathrm{\max}^2 + \SP^2) \left( 1 + \frac{3\gamma}{4-\gamma} \frac{\SP^2}{A_\mathrm{max}^2} \right),
	\\
	\tailint(\SP) &= -\ii (X_+ + X_-) \SP.
	\end{aligned}
\end{equation}
We apply Theorem~\ref{thm:mixture} in the case $M=4$ with $\mathcal{P}=2$ and $\mathcal{Q}=1$, in which case the first scenario holds. 
\end{proof}

An illustration of the prediction of Corollary~\ref{cor:glpd} is shown in Figure~\ref{fig:gLPD}, which displays the expected characteristic periodic focusing.
\begin{figure} [htb]
\centering
\includegraphics[width=\textwidth]{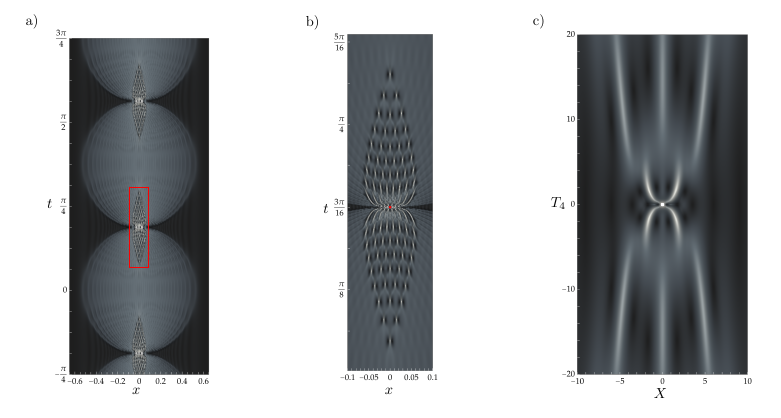}
\caption{a) Density plot of $|\widetilde \psi(x,t)|$ for the semiclassical soliton ensemble solution of the LPD equation \eqref{eq:gLPD} generated by the Klaus-Shaw initial data $A(x)$ given by \eqref{eq:gLPD.IC2} (with $X_\pm =\pm \tfrac{1}{2} = 0$, $A_\mathrm{max}=1$, $\gamma =\tfrac{4}{7}$) with $a_2=1$ and $a_4$ given by \eqref{eq:gLPD.focus.condition} and $\eps = \tfrac{1}{160}$.
b)  A higher resolution computation of the solution in the red rectangular region surrounding the focus point $(x^\circ,t_{-1}^\circ)=(0, \tfrac{3\pi}{16})$. 
c) The solution in the local coordinates $(X, T_4)$ described in Corollary~\ref{cor:glpd} in which the leading order behavior is described by $\Psi(X,0,0,T_4)$.  
}
\label{fig:gLPD}
\end{figure}

\begin{rem}
Under a Madelung-type ansatz $\psi(x,t)=\rho^{1/2}\ee^{\ii S/\eps}$ introducing real variables $\rho$ and $\mu=\rho S_x$, one can obtain a dispersionless system with two unknowns similar to \eqref{eq:dispersionless-focusing} for each mixed flow of the focusing NLS hierarchy.
Although Theorem~\ref{thm:mixture} proves that solutions of various mixed flows exhibit asymptotic behavior in which the amplitude reaches large values proportional to $\eps^{-1}$ on space and time scales that are small as $\eps\to 0$, we do not claim in general that these focusing events represent dispersive regularizations of Talanov-like collapse/blow-up solutions of the corresponding dispersionless systems.  Indeed, the plots shown in Figure~\ref{fig:gLPD} suggest that the earliest catastrophe of the dispersionless LPD system is instead of elliptic-umbilic type, leading to a triangular array of peaks similar to the Bertola-Tovbis regularization for focusing NLS \cite{BertolaT13}.  Moving forward in time from this point, the triangular array apparently develops into a modulated genus-two solution.  As such, the function $\Psi(X,0,0,T_4)$ might be expected to be a dispersive regularization of a Talanov-like collapse/blow-up solution of the genus-two Whitham modulation system for the LPD equation.  The latter is a quasilinear elliptic system with \emph{six} unknowns (see \cite{ForestL86} for the NLS analogue).
\label{rem:higher-genus}
\end{rem}

\subsubsection{Extreme focusing for higher pure flows}
The $n^\mathrm{th}$  \emph{pure flow} of the focusing NLS hierarchy, denoted NLS$_n$, corresponds to the case in which we tie the sequence of time coordinates to a single variable $t\in\mathbb{R}$ by $t_m=a_mt$ where $a_m=\delta_{m,n}$ is the Kronecker delta.  Examples include the mKdV equation \eqref{eq:mKdV} for $n=3$ and \eqref{eq:LPD} for $n=4$.  These equations for $n>2$ do not fall into the category of mixed flows to which Theorem~\ref{thm:mixture} applies, because according to Proposition~\ref{prop:Psi-even-analytic} the coefficient $\phaseint_1$ cannot vanish for semiclassical soliton ensembles generated from any semicircular Klaus-Shaw potential, and as such the mixture must contain a component of the $t_2$ flow (i.e., the NLS equation itself).  However, we can still show that for any given pure flow there exists initial data that leads to Suleimanov-Talanov focusing.

\begin{thm}[Suleimanov-Talanov focusing of pure flows]
Under the assumptions of Theorem~\ref{thm:multi-time}, fix $K\in\mathbb{Z}$, a flow index $n=2,3,\dots,M$, $M:=\max\{\deg(\phaseint),\deg(\tailint)\}$, and a real number $t^\circ$.   Set $\mathbf{t}_n^\circ:=(t_2^\circ,\dots,t_{n-1}^\circ,t_n^\circ-t^\circ,t_{n+1}^\circ,\dots,t_M^\circ)\in\mathbb{R}^{M-1}$ with $t_j^\circ$ defined by \eqref{eq:focus-K} for $j=2,\dots,M$.  Then the function $x\mapsto \psi(x,0):=\widetilde{\psi}(x,\mathbf{t}_n^\circ)$ is an initial datum 
for which the corresponding solution $\psi(x,t)$ of the $n^\mathrm{th}$ pure flow of the focusing NLS hierarchy as a function of $(x,t)\in\mathbb{R}^2$ experiences a Suleimanov-Talanov focusing at $(x,t)=(x^\circ,t^\circ)$ with $x^\circ:=-\frac{1}{2}\tailint_1$.  More precisely,
\begin{equation}
\psi\left(x^\circ+\frac{\epsilon^2}{\nu}X,t^\circ+\frac{\epsilon^{n+1}}{\nu^n}T_n\right)=\ii (-1)^{K+N}\frac{\nu}{\epsilon}\Psi(X,0,\dots,0,T_n,0,\dots,0)+\bigo{1}
\end{equation}
as $\epsilon=\epsilon_N\downarrow 0$, uniformly for $(X,T_n)\in\mathbb{R}^2$ bounded, where $\nu$ is given by \eqref{eq:nu-formula}.
\label{thm:pure-flow}
\end{thm}

\begin{proof}
We use the fact that all of the flows in the focusing NLS hierarchy commute to shift the origin in the space of times $t_2,t_3,\dots$ so that the line in the $t_n$ direction intersects a selected focus point indexed by $K$.
\end{proof}

\begin{rem}
The characterization of initial data in Theorem~\ref{thm:pure-flow} is implicit, in terms of the solution of Riemann-Hilbert Problem~\ref{rhp:M-sigma}.  The shift of origin that is behind the proof would be expected to introduce oscillations of finite amplitude and wavelength proportional to $\eps$, so that the initial data $\psi(x,0)$ would not be an approximation of any Klaus-Shaw potential, semicircular or otherwise.  The presence of such oscillations is clear in examples.
\label{rem:non-Madelung}
\end{rem}

To demonstrate the use of Theorem~\ref{thm:pure-flow}, first consider the NLS$_3$ (mKdV) flow. 
We generate initial data starting with a semicircular Klaus-Shaw potential $A(x)$ of the form \eqref{eq:HirotaIC} which has a quadratic phase integral $\phaseint(\SP)$ and a cubic tail integral $\tailint(\SP)$ given by \eqref{eq:Hirota_phase_and_tail}. 
Using Theorem~\ref{thm:multi-time} we choose a focus point 
\begin{equation}
\begin{split}
	(x^\circ, t_2^\circ, t_3^\circ) &= (-\tfrac{1}{2}\tailint_1,-\tfrac{1}{2}(2K+1)\phaseint_1,-\tfrac{1}{2}\tailint_2)\\
	&=
	\left( \frac{X_+ + X_-}{2}, -(2K+1) \frac{\pi(X_+-X_-)}{\pi A_\mathrm{max}}, -\frac{\xi (X_+ - X_-)}{6} \right)
	\end{split}
\end{equation} 
by fixing an integer $K \in \Z$.
We then construct initial data $\psi(x,0)$ for NLS$_3$ by flowing the semiclassical soliton ensemble $\widetilde{\psi}(x, t_2, t_3)$ defined by $A(x)$ under the $t_2$ (NLS) flow to $t_2^\circ$, i.e., we set $\psi(x,0) := \widetilde{\psi}(x,t_2^\circ, 0)$. 

Note that for each $x\in\mathbb{R}$, $\psi(x,0):=\widetilde{\psi}(x,t_2^\circ,0)$ is purely imaginary.  Indeed, the effect of evaluation at $t_2=t_2^\circ$ is to replace each coefficient $c_n^0$, $n=0,\dots,N-1$, in Riemann-Hilbert Problem~\ref{rhp-meromorphic} with $c_n^0\ee^{2\ii t_2^\circ(\ii \widetilde{s}_n)^2/\eps}$.  Since $\phaseint(\SP)=\phaseint_0+\phaseint_1\lambda^2$, using  \eqref{eq:ZS-approximate-eigenvalues} gives $2\ii t_2^\circ(\ii\widetilde{s}_n)^2/\eps = 2\pi\ii (n+\frac{1}{2})t_2^\circ/\Phi_1-2\ii t_2^\circ\Phi_0/(\eps\Phi_1)$.  Also, combining \eqref{eq:ZS-epsilon-assumption} and \eqref{eq:Psi0-Psi1} gives $\Phi_0/\eps=N\pi$, so $2\ii t_2^\circ(\ii\widetilde{s}_n)^2/\eps = 2\pi\ii (n-N+\frac{1}{2})t_2^\circ/\Phi_1$.  Finally, using $t_2^\circ=-\frac{1}{2}(2K+1)\Phi_1$ gives $2\ii t_2^\circ(\ii\widetilde{s}_n)^2/\eps = -2\pi\ii (n-N+\frac{1}{2})(K+\frac{1}{2})$, and hence $\ee^{2\ii t_2^\circ(\ii\widetilde{s}_n)^2/\eps}=(-1)^{n-N+K+1}\ii$.  Since $c_n^0$ is purely imaginary,  $c_n(x,t_2^\circ,0)=c_n^0\ee^{-2\widetilde{s}_nx/\eps}\ee^{2\ii t_2^\circ(\ii\widetilde{s}_n)^2/\eps}$ is real for every $n=0,\dots,N-1$.  It then follows from Proposition~\ref{prop:real-imag} that $\psi(x,0)=\widetilde{\psi}(x,t_2^\circ,0)$ is purely imaginary for all $x\in\mathbb{R}$.  Since the mKdV equation in the form \eqref{eq:mKdV} preserves this property, we can write the solution for positive $t=t_3$ in the form $\psi(x,t)=\ii u(x,t)$ where $u(x,t)$ is a real-valued solution of the mKdV (NLS$_3$) equation.
According to Theorem~\ref{thm:pure-flow}, the solution $u(x,t)$ undergoes a Suleimanov-Talanov focusing near the point $(x^\circ,t_3^\circ)$.
The results of numerical implementation of this procedure are shown in the first row of Figure~\ref{fig:mkdv_LPD}. 

Next, we illustrate Theorem~\ref{thm:pure-flow} for NLS$_4$, following the same procedure starting with a semicircular Klaus-Shaw potential $A(x)$ given by \eqref{eq:gLPD.IC2} with phase and tail integrals given by \eqref{eq:gLPD_phase_and_tail}; 
the corresponding focus point from Theorem~\eqref{eq:hierarchy-focus-approx} is 
\begin{equation}
	(x^\circ, t_2^\circ, t_3^\circ, t_4^\circ) = 
	\left( \frac{X_+ + X_-}{2},  -(2K+1) \frac{\pi(2-\gamma)(X_+ - X_-)}{16 A_\mathrm{max}},  0, 
	-(2K+1)\frac{ 3\pi \gamma (X_+ - X_-)}{32A_\mathrm{max}^3}  \right).
\end{equation}
For the NLS$_4$ flow, we take as our initial condition $\psi(x,0) = \widetilde{\psi}(x,t_2^\circ, t_3^\circ, 0)$ where $\widetilde{\psi}$ is the semiclassical soliton ensemble corresponding to $A(x)$ given by \eqref{eq:gLPD.IC2}. This initial condition is complex-valued.  The NLS$_4$ numerics are shown in the second row of Figure~\ref{fig:mkdv_LPD}.

\begin{figure}[htb]
\begin{overpic}[width=.3\textwidth]{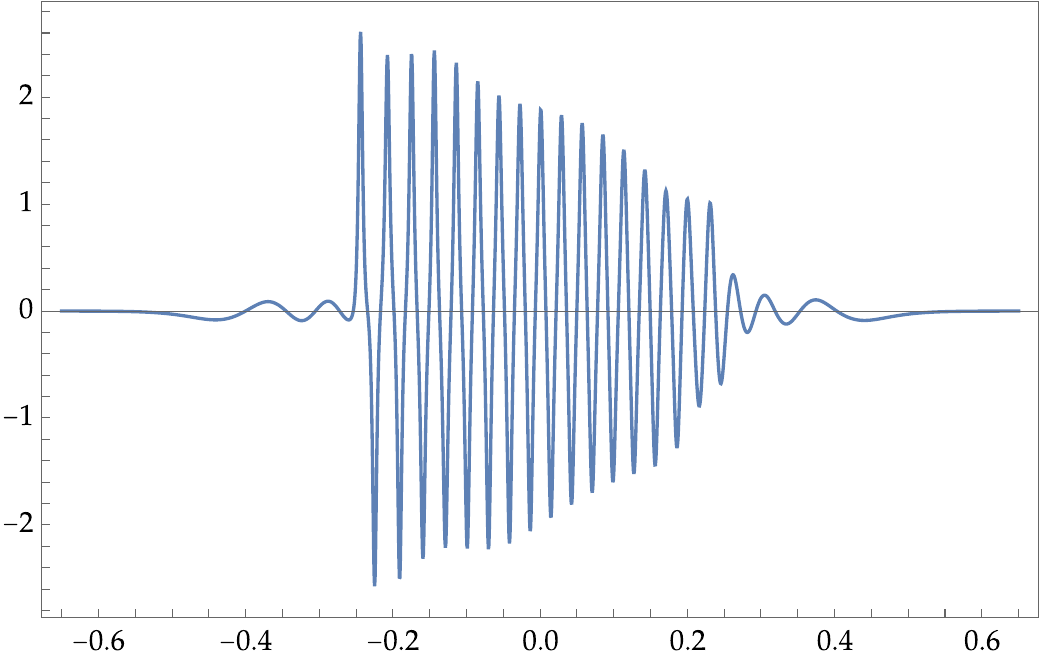} 
\put(-7,64){\scalebox{0.75}{a)}}
\put(51,-3){\scalebox{0.5}{$x$}}
\put(-5,33){\rotatebox[origin=c]{90}{\scalebox{0.5}{$u(x,0)$}}}
\end{overpic}
\hspace*{\stretch{1}}
\begin{overpic}[width=.32\textwidth]{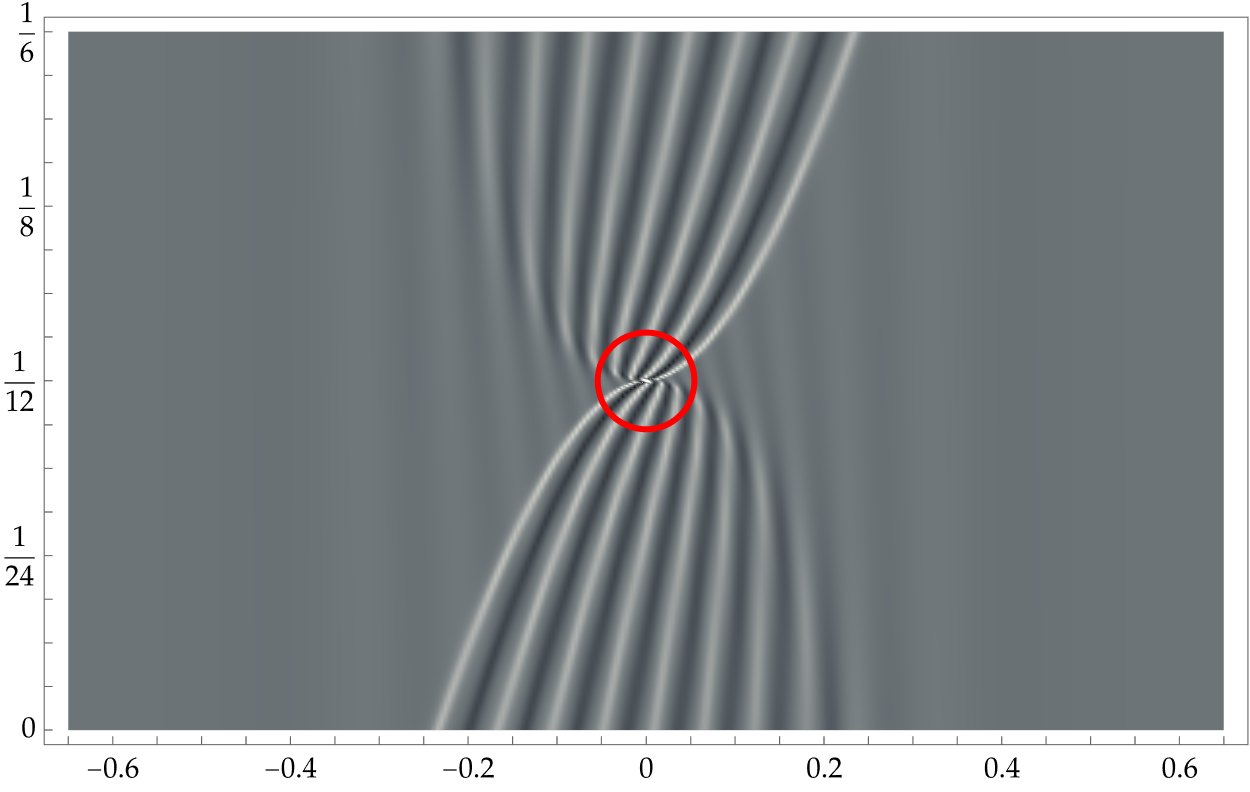}
\put(-7,61){\scalebox{0.75}{b)}}
\put(51,-3){\scalebox{0.5}{$x$}}
\put(-5,33){\rotatebox[origin=c]{90}{\scalebox{0.5}{$t_3$}}}
\end{overpic}
\hspace*{\stretch{1}}
\begin{overpic}[width=.3\textwidth]{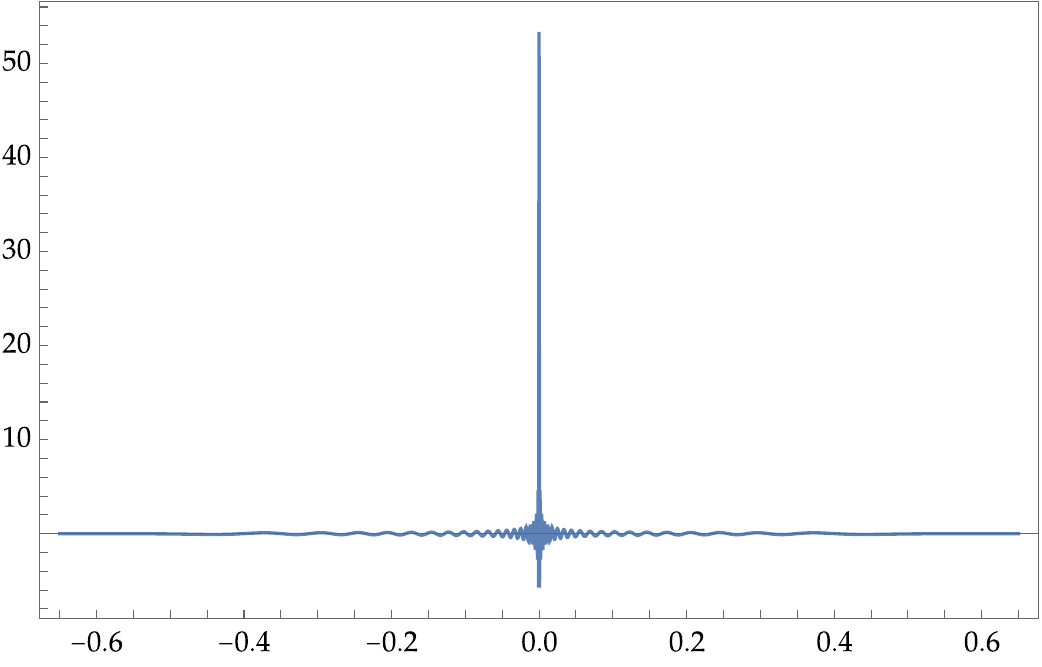}
\put(-7,64){\scalebox{0.75}{c)}}
\put(51,-3){\scalebox{0.5}{$x$}}
\put(-5,33){\rotatebox[origin=c]{90}{\scalebox{0.5}{$u(x,t_3^\circ)$}}}
\put(7,56){\scalebox{0.5}{$t=t^\circ_3=\frac{1}{12}$}}
\end{overpic}

\vspace*{12pt}
\begin{overpic}[width=.3\textwidth]{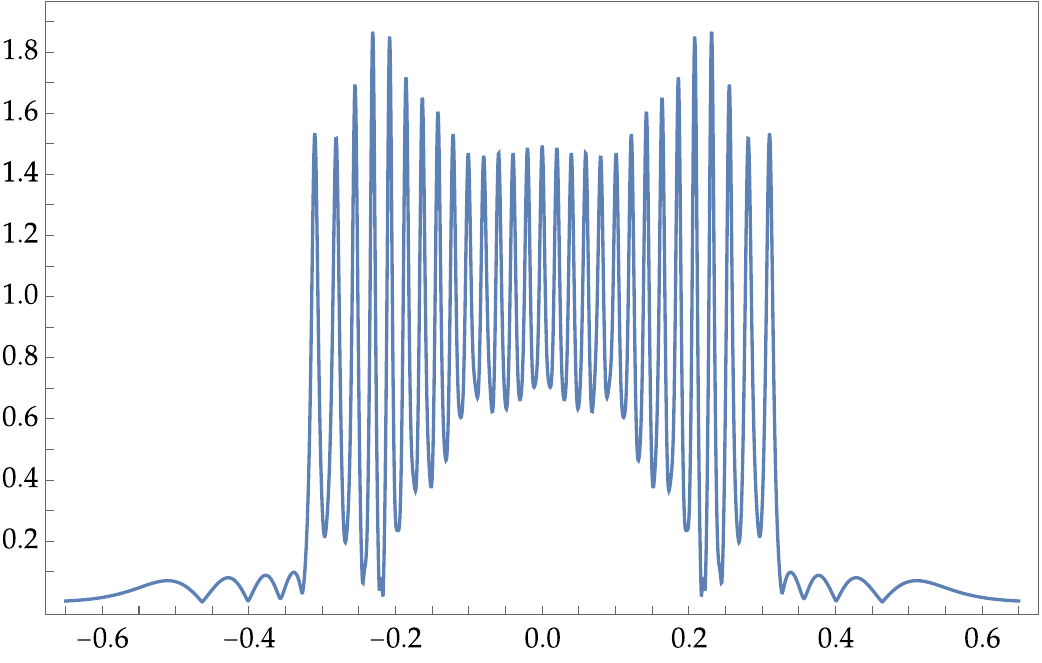}
\put(-7,64){\scalebox{0.75}{d)}}
\put(51,-3){\scalebox{0.5}{$x$}}
\put(-5,33){\rotatebox[origin=c]{90}{\scalebox{0.5}{$|\psi(x,0)|$}}}
\end{overpic}
\hspace*{\stretch{1}}
\begin{overpic}[width=.32\textwidth]{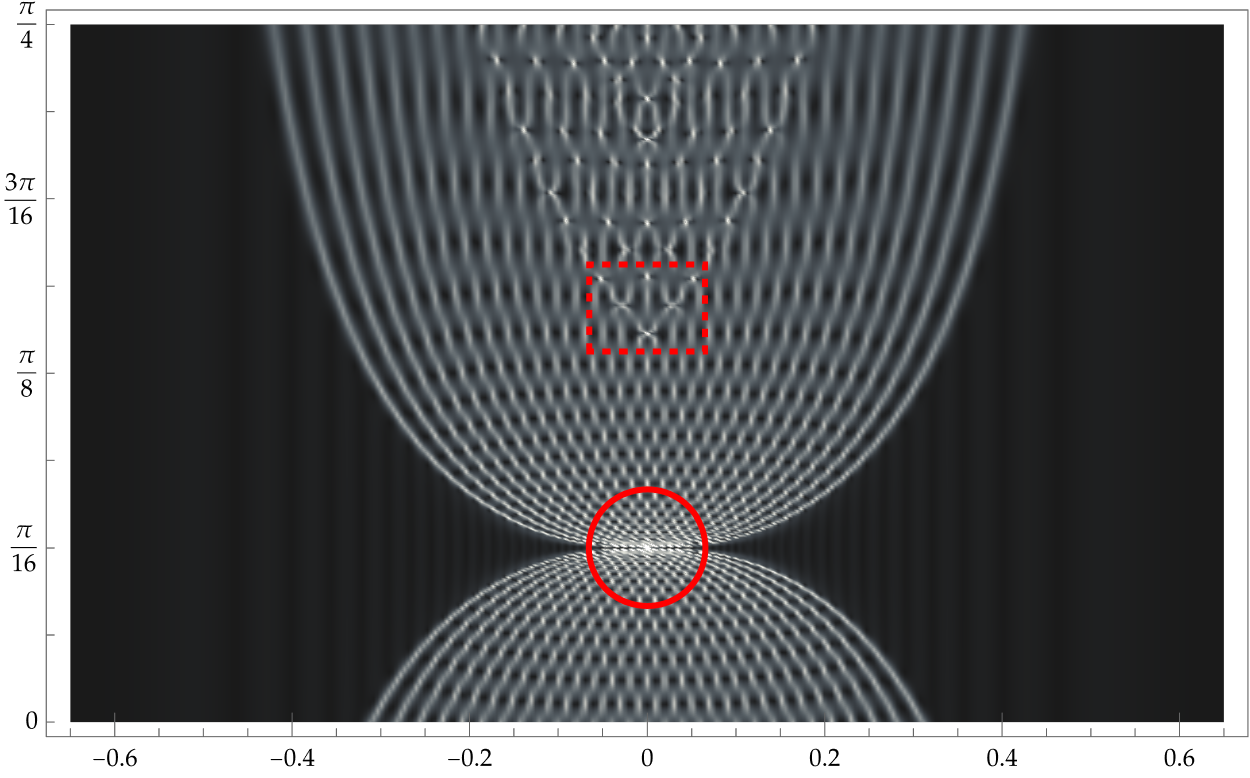}
\put(-7,62){\scalebox{0.75}{e)}}
\put(51,-3){\scalebox{0.5}{$x$}}
\put(-5,33){\rotatebox[origin=c]{90}{\scalebox{0.5}{$t_4$}}}
\end{overpic}
\hspace*{\stretch{1}}
\begin{overpic}[width=.3\textwidth]{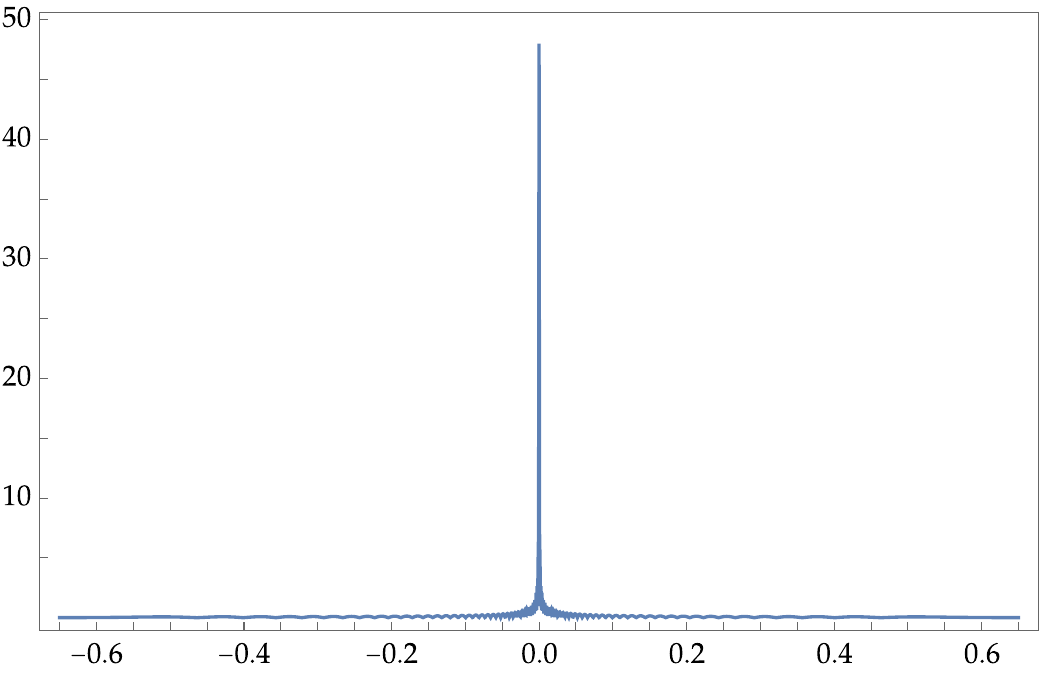}
\put(-7,65){\scalebox{0.75}{f\;\!)}}
\put(51,-3){\scalebox{0.5}{$x$}}
\put(-5,33){\rotatebox[origin=c]{90}{\scalebox{0.5}{$|\psi(x,t_4^\circ)|$}}}
\put(7,56){\scalebox{0.5}{$t = t_4^\circ=\frac{\pi}{16}$}}
\end{overpic}
\caption{Suleimanov-Talanov focusing in higher pure flows of the focusing NLS hierarchy. 
Top row, (left-to-right): Real-valued semiclassical soliton ensemble initial data $u(x,0)$ ($\epsilon=\tfrac{1}{160}$) for the NLS$_3$ (mKdV) equation \eqref{eq:mKdV}; density plot of the solution $u(x,t_3)$ in spacetime; plot of the solution $u(x,t)$ at the focusing time $t_3^\circ = \frac{1}{12}$.
Bottom row, (left-to-right): complex-valued semiclassical soliton ensemble initial data $\psi(x,0)$ ($\epsilon=\tfrac{1}{160}$) for the NLS$_4$ equation \eqref{eq:LPD}; density plot of the solution $\psi(x,t_4)$ in spacetime; plot of the solution $\psi(x,t_4)$ at the focusing time $t_4^\circ = \frac{\pi}{16}$.  
}
\label{fig:mkdv_LPD}
\end{figure}
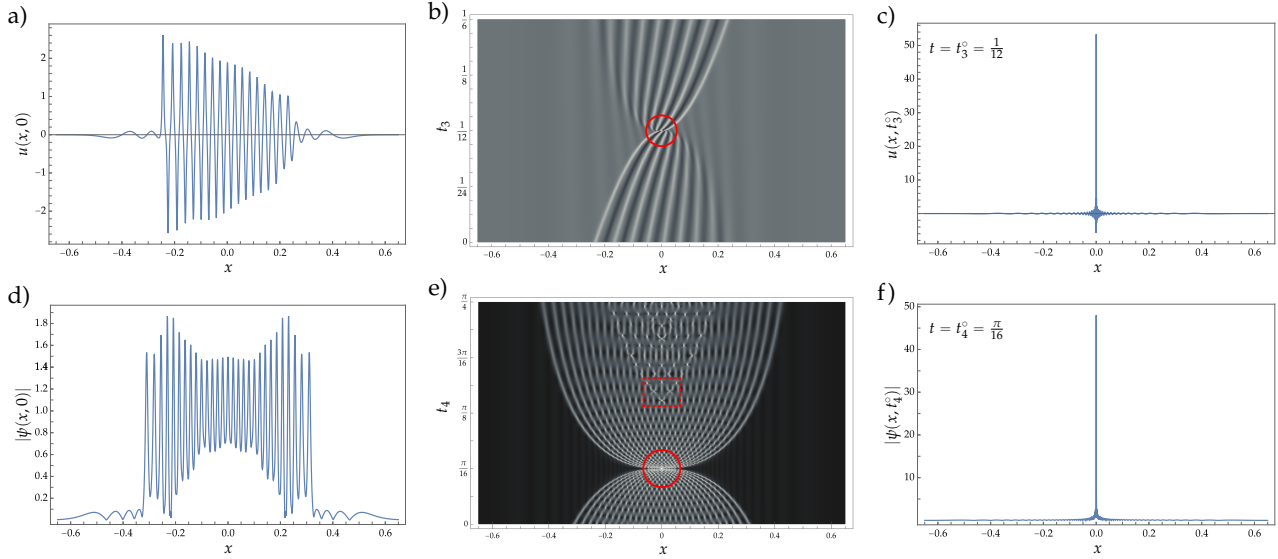

\begin{rem}
Unlike the Talanov solutions of the dispersionless focusing NLS system \eqref{eq:dispersionless-focusing} which may be viewed as the collapse of a modulated plane wave (genus $0$), here we observe that in the mKdV case for special initial data one has extreme-amplitude focusing and collapse instead of a \emph{dispersive shock wave}, or a modulated genus-$1$ structure.  Likewise for the NLS$_4$ equation we have the focusing and collapse of a modulated genus-$2$ structure.  
\end{rem}

%% file: sec-interpolate.tex
Another way to view the unusual nature of Suleimanov-Talanov focusing, going beyond the fact that the semiclassical soliton ensemble for a given semicircular Klaus-Shaw potential $A(x)$ has only been proved to exhibit this kind of asymptotic behavior at a discrete set of points in the multi-time space of the focusing NLS hierarchy (Theorem~\ref{thm:multi-time}) which will only be observed in $1+1$ mixed-flow equations in the hierarchy that have just the right ratios of coefficients (Theorem~\ref{thm:mixture}), is to illustrate how deformations of a Klaus-Shaw potential can easily perturb a Talanov-type infinite-amplitude singularity into an elliptic umblic gradient catastrophe for solutions of the dispersionless approximate system \eqref{eq:dispersionless-focusing}.  

Here, we construct a family of smooth Klaus-Shaw potentials $A(x;\delta)$ interpolating between initial data consistent with a Talanov solution with $E<0$ for $\delta=0$ and and that consistent with an Akhmanov-Sukhorukov-Khokhlov solution for $\delta=1$.  We thus choose for the endpoints
\begin{equation}
A(x;0)=\sqrt{1-x^2}\chi_{[-1,1]}(x)\quad\text{and}\quad A(x;1)=\frac{1}{2}\mathrm{sech}(x),
\end{equation}
both of which have the same $L^1(\mathbb{R})$ norm, namely $\pi/2$.  We will interpolate between these two endpoints by constructing a Klaus-Shaw potential $A(x;\delta)$ such that for each $\delta\in [0,1]$,
\begin{equation}
\|A(\diamond;\delta)\|_{L^1(\mathbb{R})}=\phaseint(0;\delta)=\int_\mathbb{R}A(x;\delta)\,\dd x =\frac{\pi}{2}.
\label{eq:L1-norm-impose}
\end{equation}
The family of potentials $A(x;\delta)$ will be constructed by explicitly interpolating the derivatives $\evdensity(s)=\evdensity(s;\delta)$ of the corresponding phase integrals $\phaseint(\ii s;\delta)$ corresponding to $A(x;\delta)$ by \eqref{eq:ZS-density}.  A direct computation shows that
\begin{equation}
\evdensity(s;0)=-\frac{1}{\pi}\frac{\dd}{\dd s}\int_{x_-(s)}^{x_+(s)}\sqrt{1-x^2-s^2}\,\dd x =\frac{1}{2\pi\ii}\frac{\dd}{\dd s}\oint_LR(x;s)\,\dd x,
\end{equation}
where $R(x;s)^2=x^2+s^2-1$ and $R(x;s)$ is cut in the interval $[x_-(s),x_+(s)]$ with $R(x;s)=x+O(x^{-1})$ as $x\to\infty$, and $L$ is a positively-oriented loop enclosing the cut.  By a residue at $x=\infty$ we then obtain
\begin{equation}
\evdensity(s;0)=s,\quad 0<s<1.
\end{equation}
The function $\evdensity(s;1)$ corresponding to the other endpoint case of $A(x;1)$ can also be calculated directly from \eqref{eq:ZS-density}.  The calculation was done as an example in \cite{Kamvissis:2003}, with the result being that 
\begin{equation}
\evdensity(s;1)=1,\quad 0<s<\frac{1}{2}.
\end{equation}
Thus, we introduce coefficients $a(\delta)$ and $b(\delta)$ to be defined for $0<\delta<1$ and assume a linear combination
\begin{equation}
\evdensity(s;\delta)=a(\delta)+b(\delta)s,\quad 0<s<1-\frac{1}{2}\delta.
\end{equation}
We require that $a(0)=b(1)=0$ and that $b(0)=a(1)=1$ to match the desired endpoints, and we obtain a relation between the functions $a(\delta)$ and $b(\delta)$ by imposing the condition \eqref{eq:L1-norm-impose}.  Using the connection between $\evdensity(s)$ and $\phaseint(\ii s)$ in the definition \eqref{eq:ZS-phase-integral}, we are imposing that
\begin{equation}
\frac{1}{2}=\int_0^{1-\delta/2}(a(\delta)+b(\delta)s)\,\dd s = a(\delta)\left(1-\frac{1}{2}\delta\right)+\frac{1}{2}b(\delta)\left(1-\frac{1}{2}\delta\right)^2.
\label{eq:a-b-relation}
\end{equation}
It is straightforward to check that this relation is consistent with the boundary conditions on $a(\delta)$ and $b(\delta)$ at $\delta=0,1$.  We may therefore supplement it with an arbitrary second independent equation that is also consistent with the boundary conditions; we choose to simply define $b(\delta):=1-\delta$.  Then \eqref{eq:a-b-relation} determines $a(\delta)$ explicitly, and we have determined that
\begin{equation}
a(\delta):=\frac{1}{2}\left(1-\frac{1}{2}\delta\right)^{-1}-\left(1-\frac{1}{2}\delta\right)^2+\frac{1}{2}\left(1-\frac{1}{2}\delta\right),\quad b(\delta):=1-\delta.
\end{equation}
Having determined $\evdensity(s;\delta)$ in this way, we impose the condition that $x\mapsto A(x;\delta)$ is an even function by choosing $\tailint(\ii s;\delta)=0$ for $s\in (0,1-\frac{1}{2}\delta)$ for all $\delta\in [0,1]$ (see Corollary~\ref{cor:even}). 

Now that we have $\evdensity(s;\delta)$ (and hence also $\phaseint(\ii s;\delta)$ by \eqref{eq:ZS-phase-integral}), $\tailint(\ii s;\delta)$, and the maximum amplitude $A_\mathrm{max}(\delta)=1-\frac{1}{2}\delta$, we apply Proposition~\ref{prop:Psi-mu-invert} to obtain the inverse functions $x_\pm(s;\delta)$ of $s=A(x;\delta)$.  We obtain
\begin{equation}
x_+(s;\delta)=-x_-(s;\delta)=\frac{a(\delta)}{2}\ln\left(\frac{1-\frac{1}{2}\delta+\sigma}{1-\frac{1}{2}\delta-\sigma}\right)+b(\delta)\sigma,
\end{equation}
where $\sigma=\sigma(s;\delta)$ given by
\begin{equation}
\sigma(s;\delta):=\sqrt{\left(1-\frac{1}{2}\delta\right)^2-s^2}
\end{equation}
is a monotone decreasing function of $s$ on $(0,1-\frac{1}{2}\delta)$. Note also that
\begin{equation}
\frac{\dd}{\dd\sigma}x_+(s;\delta)=\frac{(1-\frac{1}{2}\delta)a(\delta)}{(1-\frac{1}{2}\delta)^2-\sigma^2} + b(\delta)=\frac{(1-\frac{1}{2}\delta)a(\delta)}{s^2}+b(\delta).
\end{equation}
Since $(1-\frac{1}{2}\delta)a(\delta)>0$ and $b(\delta)>0$ for $0<\delta<1$, $x_+(s;\delta)$ is monotone increasing in $\sigma$ and hence is a monotone decreasing function of $s\in (0,A_\mathrm{max}(\delta))$.  This proves that $A(x;\delta)$ is a Klaus-Shaw potential for all $\delta\in (0,1)$.  

In \cite[Eqn.\@ (6.56)]{Kamvissis:2003}, an implicit formula is given for the solution of the dispersionless focusing NLS system \eqref{eq:dispersionless-focusing} with even initial data $\rho(x,0)=\rho(-x,0)$ having zero initial momentum $\mu(x,0)\equiv 0$, along the symmetry axis $x=0$.  That formula reads
\begin{equation}
t  =\frac{1}{\rho}\int_{A_\mathrm{max}}^{\sqrt{\rho}>A_\mathrm{max}}\frac{s\evdensity(s)\,\dd s}{\sqrt{\rho-s^2}},\quad \rho=\rho(0,t),
\label{eq:x-0-general}
\end{equation}
which uses the analytic continuation of the function $\evdensity(s)$ beyond the right endpoint of the interval $(0,A_\mathrm{max})$ of its definition.  Substituting the $\delta$-dependent quantities $A_\mathrm{max}(\delta)$ and $\evdensity(s;\delta)$ obtained above, the relation \eqref{eq:x-0-general} becomes
\begin{multline}
t=\frac{\pi}{4}b(\delta)+\frac{1}{\rho}\left(a(\delta) +
\frac{b(\delta)}{2}\left(1-\frac{1}{2}\delta\right)\right)\sqrt{\rho-\left(1-\frac{1}{2}\delta\right)^2}\\{}-\frac{b(\delta)}{2}\arctan\left(\frac{1-\frac{1}{2}\delta}{\sqrt{\rho-(1-\frac{1}{2}\delta)^2}}\right).
\label{eq:t-interpolation}
\end{multline}
Hence,
\begin{equation}
\frac{\dd t}{\dd\rho}=\frac{2(1-\frac{1}{2}\delta)^2a(\delta)+(1-\frac{1}{2}\delta)^3b(\delta)-a(\delta)\rho}{2\rho^2\sqrt{\rho-(1-\frac{1}{2}\delta)^2}}.
\end{equation}
If $\delta=0$ we are at the Talanov-type potential endpoint, and $a(\delta)=0$ while $b(\delta)=1$.  Hence also $\dd t/\dd\rho=1/(2\rho^2\sqrt{\rho-1})$ which is positive for $\rho>A_\mathrm{max}(0)=1$ and integrable at $\rho=+\infty$, so $t$ increases to a finite positive limiting value of $\pi/4$ as $\rho\to+\infty$.  This corresponds of course to the finite-time blowup of the Talanov solution with $E<0$:  $\rho(0,t)\uparrow +\infty$ as $t\uparrow \frac{1}{4}\pi$.  However, if $0<\delta\le 1$ we have $a(\delta)> 0$, and hence $\dd t/\dd\rho$ has a simple root at the finite value $\rho=\rho_\mathrm{c}(\delta)$ given by
\begin{equation}
\rho_\mathrm{c}(\delta)=\left(1-\frac{1}{2}\delta\right)^2\left[2+\left(1-\frac{1}{2}\delta\right)\frac{b(\delta)}{a(\delta)}\right]\ge 2\left(1-\frac{1}{2}\delta\right)^2>\left(1-\frac{1}{2}\delta\right)^2,
\end{equation}
at which point $t$ takes the corresponding value $t=t_\mathrm{c}(\delta)$ obtained by evaluating the right-hand side of \eqref{eq:t-interpolation} at $\rho=\rho_\mathrm{c}(\delta)$.  
We have $\dd t/\dd\rho>0$ for $(1-\frac{1}{2}\delta)^2<\rho<\rho_\mathrm{c}(\delta)$, but $t$ has a simple critical point at $\rho=\rho_\mathrm{c}(\delta)$, a nondegenerate local maximum.  Hence for the direct function $\rho=\rho(0,t)$ we have a finite-amplitude gradient catastrophe point at $t=t_\mathrm{c}(\delta)$ with finite value $\rho=\rho_\mathrm{c}(\delta)$.  At the endpoint $\delta=1$ we recover the expected result that $t_\mathrm{c}(1)=1$ and $\rho_\mathrm{c}(1)=\frac{1}{2}$, consistent with the prediction of the Akhmanov-Sukhorukov-Khokhlov solution reviewed in \S~\ref{sec:ASK}.

This calculation shows that an arbitrarily small perturbation measured by $\delta>0$ of a Talanov-type initial condition destroys the infinite-amplitude focusing, replacing it instead with an elliptic umbilic gradient catastrophe point.   This result suggests a form of nongenericity of Talanov-type infinite-amplitude focusing, however it is a challenge to properly formulate this because the space of admissible perturbations of initial data should make sense for the dispersionless focusing NLS system \eqref{eq:dispersionless-focusing}, which is generally of elliptic type, except where $\rho=0$.

Taken together with the prediction of Suleimanov \cite{Suleimanov17} for Talanov pulses with $E=0$ and the Bertola-Tovbis result \cite{BertolaT13} for data leading to an elliptic-umbilic catastrophe, our work on Talanov pulses with $E<0$ (or more properly the corresponding semiclassical soliton ensembles) adds further evidence that 
the scenario that prevails in terms of the type and scale of the dispersive regularization is determined primarily by the nature of the singularity in the dispersionless solution.  If it is an elliptic-umbilic catastrophe (finite amplitude gradient catastrophe), then one has Bertola-Tovbis regularization based on the Painlev\'e-I tritronqu\'ee solution and a field of Peregrine rogue waves at locations corresponding to the poles of the tritronqu\'ee solution.  If it is an infinite-amplitude collapse and blow up, then one has instead Suleimanov-Talanov focusing based on the Painlev\'e-III hierarchy of Sakka.  This rubric seems clearer than that proposed in the recent work \cite{DemontisORS23}.  Note that the quantity denoted $\alpha_0^2+4\gamma_0$ in \cite{DemontisORS23} is proportional by a positive quantity to the integration constant $E$.

%% file: sec-gen-scKS.tex
\subsection{Basic quantities in the semiclassical direct scattering theory for Klaus-Shaw potentials}
\label{sec:basic-quantities}
For a potential $A$ with the Klaus-Shaw property (see Definition~\ref{def:semicircularKS}), we may define the following quantities:
\begin{equation}
L(\SP):=-\ell^-(+\infty;\SP)=-\ell^+(-\infty;\SP)
\label{eq:Lfunc-def}
\end{equation}
for suitable $\SP\in\C$ (see Lemma~\ref{lem:L-alternate} below), wherein with the square root denoting the principal branch,
\begin{equation}
\begin{split}
\ell^-(x;\SP)&:=\int_{-\infty}^x\left[\left(-\SP^2-A(y)^2\right)^{1/2}+\ii\SP\right]\,\dd y,\\
\ell^+(x;\SP)&:=\int_x^{+\infty}\left[\left(-\SP^2-A(y)^2\right)^{1/2}+\ii\SP\right]\,\dd y.
\end{split}
\label{eq:mu-plus-minus-define}
\end{equation}

First, we give an alternate formula for $L(\SP)$.
\begin{lem}
$L(\SP)$ defined by \eqref{eq:Lfunc-def} is an analytic function of $\SP$ in the upper half-plane with the imaginary segment $0<-\ii\SP\le A_\mathrm{max}$ omitted, and it satisfies the asymptotic condition $L(\SP)=\bigo{\SP^{-1}}$ as $\SP\to\infty$ in the upper half-plane.  Also, $L(\SP)$ is equivalently given by
\begin{equation}
L(\SP)=\int_0^{A_\mathrm{max}}\left[\log(-\ii\SP+s)-\log(-\ii\SP-s)\right]\evdensity(s)\dd s,
\label{eq:ZS-nu0hat}
\end{equation}
where $\log$ denotes the principal branch and $\evdensity(s)$ is defined by \eqref{eq:ZS-density}, and its boundary values on the branch cut satisfy the jump condition 
\begin{equation}
\lim_{\delta\downarrow 0}\left[L(\ii s+\delta)-L(\ii s-\delta)\right]=2\ii\phaseint(\ii s),\quad 0<s<A_\mathrm{max},
\label{eq:ZS-nu0-jump}
\end{equation}
where $\phaseint(\ii s)$ is the phase integral defined by \eqref{eq:ZS-phase-integral}.
\label{lem:L-alternate}
\end{lem}
\begin{proof}
The claimed domain of analyticity of $L(\SP)$ and the fact that $L(\SP)=\bigo{\SP^{-1}}$ for large $\SP$ are obvious from the formula \eqref{eq:Lfunc-def} and the definition \eqref{eq:mu-plus-minus-define}.  

Let $\hat{L}(\SP)$ denote the function defined by the right-hand side of \eqref{eq:ZS-nu0hat}.
Replacing $\evdensity(s)$ by the first expression in the definition \eqref{eq:ZS-density}, we first exchange the order of integration to obtain
\begin{equation}
\begin{split}
\hat{L}(\SP)&=\frac{1}{\pi}\int_0^{A_\mathrm{max}}\int_{x_-(s)}^{x_+(s)}\left[\log(-\ii\SP+s)-\log(-\ii\SP-s)\right]\frac{s}{\sqrt{A(x)^2-s^2}}\,\dd x\,\dd s\\
&=
\int_{-\infty}^{+\infty}I(x)\,\dd x,
\end{split}
\end{equation}
where the inner integral over $s$ is
\begin{equation}
I(x):=\frac{1}{\pi}\int_0^{A(x)}\left[\log(-\ii\SP+s)-\log(-\ii\SP-s)\right]
\frac{s}{\sqrt{A(x)^2-s^2}}\,\dd s.
\end{equation}
We may now evaluate $I(x)$ explicitly.  First, observe that the integrand is an even function of $s$, so
\begin{equation}
I(x)=\frac{1}{2\pi}\int_{-A(x)}^{A(x)}\left[\log(-\ii\SP+s)-\log(-\ii\SP-s)\right]\frac{s\,\dd s}{\sqrt{A(x)^2-s^2}}.
\end{equation}
Now let $S(s)$ denote the (odd) function analytic in the domain $\mathbb{C}\setminus [-A(x),A(x)]$ that satisfies the equation $S(s)^2=A(x)^2-s^2$ and the normalization (choice of branch) $S(s)=\ii s + \bigo{s^{-1}}$ as $s\to\infty$ in the complex plane.  Then, by the generalized Cauchy integral theorem,
\begin{equation}
I(x)=\frac{1}{4\pi}\oint_C\left[\log(-\ii\SP+s)-\log(-\ii\SP-s)\right]\frac{s\,\dd s}{S(s)},
\end{equation}
where $C$ is a positively-oriented loop surrounding the branch cut of $S$ for which the non-real horizontal branch cuts of the logarithms are in the unbounded exterior.  Now using $s/S(s)=-S'(s)$, we integrate by parts to obtain
\begin{equation}
I(x)=\frac{1}{4\pi}\oint_C\left[\frac{1}{s-\ii\SP}-\frac{1}{s+\ii\SP}\right]S(s)\,\dd s.
\end{equation}
Expanding the contour $C$ toward $s=\infty$ we may now evaluate $I(x)$ by residues at $s=\pm \ii\SP$ and $s=\infty$:
\begin{equation}
I(x)=-\frac{1}{2}\ii S(\ii\SP)+\frac{1}{2}\ii S(-\ii\SP)-\ii\SP = -\ii S(\ii\SP)-\ii\SP.
\end{equation}
Comparing with \eqref{eq:Lfunc-def}--\eqref{eq:mu-plus-minus-define}, to show $\hat{L}(\SP)=L(\SP)$ and thus complete the proof of the formula \eqref{eq:ZS-nu0hat}, it remains to identify $\ii S(\ii\SP)$ with the principal branch square root $(-\SP^2-A(x)^2)^{1/2}$ as is valid for all $\SP$ in the upper half-plane.

Finally, to confirm the jump condition \eqref{eq:ZS-nu0-jump}, observe from \eqref{eq:ZS-nu0hat} that for $0<s<A_\mathrm{max}$,
\begin{equation}
\lim_{\delta\downarrow 0}\left[L(\ii s+\delta)-L(\ii s-\delta)\right]=
\int_s^{A_\mathrm{max}}2\pi \ii\evdensity(s')\,\dd s',
\end{equation}
which equals $2\ii\phaseint(\ii s)$ according to \eqref{eq:ZS-phase-integral}.
\end{proof}

Now recalling the numbers $\{\widetilde{s}_n\}_{n=0}^{N-1}$ determined from \eqref{eq:ZS-approximate-eigenvalues} and letting
\begin{equation}
\widetilde{a}(\SP):=\prod_{n=0}^{N-1}\frac{\SP-\ii\widetilde{s}_n}{\SP+\ii\widetilde{s}_n},
\label{eq:ZS-tilde-a}
\end{equation}
we introduce a function $Y_\epsilon(\SP)$ given by 
\begin{equation}
Y_\epsilon(\SP):=\frac{\ee^{-L(\SP)/\epsilon}}{\widetilde{a}(\SP)}
\label{eq:ZS-Y-define}
\end{equation}
defining a function analytic for all complex $\SP$ with the vertical segment connecting $-\ii A_\mathrm{max}$ with $\ii A_\mathrm{max}$ omitted.  
Looking at the formula \eqref{eq:ZS-nu0hat}, and taking into account the definition \eqref{eq:ZS-tilde-a} of $\widetilde{a}(\SP)$ as a Blaschke product, one gets the idea that $\epsilon\log(Y_\epsilon(\SP))$ looks like the error in an approximation of the integral $L(\SP)$ by a Riemann sum.
By this reasoning, for some Klaus-Shaw potentials $A(x)$, $\log(Y_\epsilon(\SP))$ has been shown to be as small as $\bigo{\epsilon}$ for suitable $\SP$.  In the case of  semicircular Klaus-Shaw potentials $A(x)$ considered in this paper, we will show below (cf.\ Proposition~\ref{prop:Y-outside}) that for such $\SP$, $\log(Y_\epsilon(\SP))=\bigo{\epsilon^{1/2}}$ (and compute the leading term).

Recalling the domain of analyticity of $L(\SP)$, whose boundary values on the cut $0<-\ii\SP<A_\mathrm{max}$ are related according to \eqref{eq:ZS-nu0-jump} by $L_+(\SP)-L_-(\SP)=2\ii\phaseint(\SP)$ where the subscript $\pm$ indicates the limit from $\pm\re\{\SP\}>0$, we see that the function $\overline{L}(\SP)$ defined by 
\begin{equation}
\overline{L}(\SP):= L(\SP)\mp \ii\phaseint(\SP),\quad 0<\imag\{\SP\}<A_\mathrm{max},\quad
\SP\in\Omega,\quad\text{and}\quad\pm\re\{\SP\}>0,
\label{eq:Lbar-def-3}
\end{equation}
where $\Omega$ is the domain of analyticity of the phase integral $\phaseint(\SP)$,
is analytic where defined and moreover is continuous across the cut, therefore 
actually defining an analytic function on the part of $\Omega$ with $0<\imag\{\SP\}<A_\mathrm{max}$.
For $0<\ii\SP<A_\mathrm{max}$, this function coincides with the average of the boundary values taken by $L(\SP)$
(cf.\@ \eqref{eq:Lbar-def-1} below), 
which explains our use of the ``bar'' notation.  
Then, a function closely related to $Y_\epsilon(\SP)$ is
\begin{equation}
T_\epsilon(\SP):=\frac{2\cos(\phaseint(\SP)/\epsilon)\ee^{-\overline{L}(\SP)/\epsilon}}{\widetilde{a}(\SP)},
\quad 0<\imag\{\SP\}<A_\mathrm{max}\quad\text{and}\quad \SP\in\Omega,
\label{eq:ZS-T-define}
\end{equation}
which, according to \eqref{eq:ZS-approximate-eigenvalues} has removable singularities at the zeros of $\widetilde{a}(\SP)$ along the positive imaginary axis.  The relation between $T_\epsilon(\SP)$ and $Y_\epsilon(\SP)$ is given by the exact identities
\begin{equation}
T_\epsilon(\SP)=Y_\epsilon(\SP)(1+\ee^{\pm 2\ii\phaseint(\SP)/\epsilon}),\quad 0<\imag\{\SP\}<A_\mathrm{max},\quad \SP\in\Omega,\quad\text{and}\quad \pm\re\{\SP\}>0.
\label{eq:T-Y-relation}
\end{equation}
Assuming that $Y_\epsilon(\SP)=1+o(1)$ for $\SP$ bounded away from the branch cut of $Y_\epsilon$, a Cauchy-Riemann argument shows that $T_\epsilon(\SP)=1+o(1)$ holds for $\SP$ to the left or right of the segment $0<-\ii\SP<A_\mathrm{max}$.  It was first noticed in \cite{Miller02}, and further explained in \cite{BaikKMM07} and \cite{LyngM07}, that under some conditions a similar estimate of $T_\epsilon(\SP)$ holds uniformly on compact subsets of its domain of definition, including points $\SP$ with $0<-\ii\SP<A_\mathrm{max}$.  We will prove the version of this result applicable to semicircular Klaus-Shaw potentials $A(x)$ (see Definition~\ref{def:semicircularKS}) below.

\subsection{Semicircular Klaus-Shaw potentials.  Asymptotic properties of $Y_\epsilon(\SP)$ and $T_\epsilon(\SP)$}
\label{sec-prop-YT}

In this section we record several properties of the functions $Y_\epsilon(\SP)$ and $T_\epsilon(\SP)$ that are needed in the steepest-descent analysis that follows. To streamline the presentation, propositions with longer more technical proofs are deferred to Appendix~\ref{sec-proofs}. 

Recall Proposition~\ref{prop:Psi-even-analytic}.  
Now for $0<s<A_\mathrm{max}$, consider 
\begin{equation}
\rr:=\frac{1}{\pi}(\phaseint(0)-\phaseint(\ii s))\quad\text{with inverse}\quad s=s(\rr).
\label{eq:t-and-s-defs}
\end{equation}
The inverse function here is well-defined because $\evdensity(s)>0$, and we have the following result.
\begin{lem}
Let $A$ be a semicircular Klaus-Shaw potential.  Then $s(\rr)$ is analytic on the open interval $0<\rr<\phaseint(0)/\pi$, and furthermore
\begin{equation}
s(\rr)=\sqrt{\rr}v(\rr),\quad 0<\rr<\frac{\phaseint(0)}{\pi},
\label{eq:s-in-terms-of-v}
\end{equation}
where $v(\rr)$ is analytic at $\rr=0$:
\begin{equation}
v(\rr)=\sum_{k=0}^\infty v_k\rr^k,\quad |\rr|<\sigma,\quad v_0>0.
\label{eq:v-def}
\end{equation}
Also, $s$ is also analytic at $\rr=\phaseint(0)/\pi$ and 
$s'(\phaseint(0)/\pi)>0$.
\label{lem:s-properties}
\end{lem}
\begin{rem}
In the case that $A(x)=\psi_0(x)$ is a semicircular Klaus-Shaw potential for which $u(x)\equiv\text{const.}$, it is straightforward to use \eqref{eq:semicircle-Psi}  to obtain
\begin{equation}
\rr = \frac{X_+-X_-}{4A_\mathrm{max}}s^2\implies s =\sqrt{\rr}v(\rr),\quad v(\rr)\equiv v_0=\sqrt{\frac{4A_\mathrm{max}}{X_+-X_-}}>0.
\end{equation}
In particular, $v_k=0$ for all $k>0$ in \eqref{eq:v-def}.
\label{rem:v-for-exact-semicircle}
\end{rem}
\begin{proof}
Analyticity of $s(\rr)$ on the open interval $0<\rr<\phaseint(0)/\pi$ follows from the implicit function theorem because $\phaseint(\ii s)$ is analytic for $0<s<A_\mathrm{max}$ and $\evdensity(s)>0$ holds for each point in this open interval.  
For a semicircular Klaus-Shaw potential $A$, it follows from \eqref{eq:t-and-s-defs} and Proposition~\ref{prop:Psi-even-analytic} that 
$\rr$ is an even analytic function of $s$ at $s=0$, with $\rr=0$ for $s=0$:
\begin{equation}
\rr=-\frac{1}{\pi}\sum_{k=1}^\infty \phaseint_{k}(-1)^ks^{2k},\quad |s|<\delta
\end{equation}
for some radius of convergence $\delta>0$.  Since $\phaseint_2>0$, it follows from the implicit function theorem that 
\begin{equation}
s(\rr)^2 = \sum_{k=1}^\infty \phi_k\rr^k,\quad |\rr|<\sigma
\end{equation}
for some radius of convergence $\sigma>0$, and where $\phi_1>0$.  Taking a positive square root then yields \eqref{eq:s-in-terms-of-v} with \eqref{eq:v-def}.

To study $s(\rr)$ for $\rr$ near $\phaseint(0)/\pi$, we
write $\evdensity(s)$ as a contour integral:
\begin{equation}
\evdensity(s)=\frac{s}{2\pi}\oint_L\frac{\dd x}{R(x;s^2)}
\end{equation}
and note that under the stated condition on $A''$, as $s\uparrow A_\mathrm{max}$ exactly two roots of $R(x;s^2)$ coalesce at $x=x_0$.  This implies that $L$ can be chosen so that $R(x;A_\mathrm{max}^2)^{-1}$ is meromorphic within $L$ having a single pole at $x=x_0$ with residue $-\ii/\sqrt{-A_\mathrm{max}A''(x_0)}$.  Thus $\rr'(A_\mathrm{max})=\evdensity(A_\mathrm{max})>0$, from which it follows that $s'(\phaseint(0)/\pi)>0$ holds for the inverse function.
\end{proof}

\subsubsection{Analysis of $Y_\epsilon(\SP)$ in the limit $\epsilon\downarrow 0$}
The first result concerns the asymptotic behavior of the function $Y_\epsilon(\SP)$ for $\SP$ suitably bounded away from the branch cut of $Y_\epsilon$.  For some small width parameter $\delta>0$, let $\Lambda$ denote the thin ``parabolic'' lens centered on the imaginary segment $0<-\ii\SP<A_\mathrm{max}$ consisting of the points
\begin{equation}
	\label{eq:Lambda-def}
	\Lambda:= \left\{ \SP \in \C \,:\,  |\re\{\SP\}| \le \delta\imag\{\SP\}(A_\mathrm{max}-\imag\{\SP\}),\ 0\le \imag\{\SP\}\le A_\mathrm{max}  \right\}.
\end{equation}

\begin{restatable}[Exterior asymptotic behavior of $Y_\epsilon(\SP)$]{prop}{Yexterior}
Suppose that $A$ is a semicircular Klaus-Shaw potential.
For arbitrary $\sigma>0$, let $D_\sigma$ denote the domain defined by $D_\sigma:=\{\SP\in\mathbb{C}_+\setminus\Lambda: |\SP|>\sigma,\; |\SP-\ii A_\mathrm{max}|>\sigma\}$.  Then 
\begin{equation}
Y_\epsilon(\SP)=
1-\frac{2\ii v_0}{\SP}\left(1-\frac{1}{\sqrt{2}}\right)\zeta(-\tfrac{1}{2})\epsilon^{1/2}+\bigo{\epsilon},
\quad\epsilon\downarrow 0,
\end{equation}
where $\zeta(\cdot)$ denotes the Riemann zeta function and $v_0>0$ denotes the constant defined in Lemma~\ref{lem:s-properties}, holds uniformly for $\SP\in D_\sigma$.
\label{prop:Y-outside}
\end{restatable}

\begin{proof} \hyperlink{proof-Y-outside}
{The proof} can be found in Appendix~\ref{sec-proofs}.
\end{proof}

Note that the explicit term proportional to $\epsilon^{1/2}$ is not present in a result proven for a function similar to $Y_\epsilon(\SP)$ in \cite[Proposition~4.3]{BaikKMM07}.  This discrepancy can be directly traced to the fact that the function $\rho(s)$ here vanishes linearly at $s=0$ while
its analogue in \cite{BaikKMM07} is bounded away from zero at the origin, as the proof will show.

The domain $D_\sigma$ obviously excludes small values of $\SP$.  The next result describes the different way that $Y_\epsilon(\SP)$ behaves for $|\SP|$ small.  Significantly, an approximation accurate in the sense of small relative error can be obtained without excluding $\SP$ from the lens $\Lambda$, i.e., the following result allows $\SP$ to lie among the singularities of $Y_\epsilon(\SP)$.
\begin{restatable}[Asymptotic behavior of $Y_\epsilon(\SP)$ for $\SP\approx 0$]{prop}{Yzetasmall}
Suppose that $A$ is a semicircular Klaus-Shaw potential.
Then 
\begin{equation}
Y_\epsilon(\SP)=\mathcal{Y}_0\left(\frac{\varphi_0(\SP)}{\epsilon^{1/2}}\right)\left(1+\mathcal{E}_0(\SP)\epsilon^{1/2} + \bigo{\epsilon}\right),\quad\epsilon\downarrow 0,\quad \varphi_0(\SP):=\ii\pi^{-1/2}(\phaseint(0)-\phaseint(\SP))^{1/2}
\label{eq:z-zero-define}
\end{equation}
holds uniformly for $\imag\{\SP\}>0$ and $|\SP|$ sufficiently small,  
where the mapping $\SP\mapsto \varphi_0(\SP)$ is conformal near $\SP=0$ with $\varphi_0(0)=0$ and $\varphi_0'(0)=1/v(0)>0$, and where
\begin{equation}
\mathcal{Y}_0(Z):=\ee^{-\ii\pi Z^2\mathrm{sgn}(\re\{Z\})}\prod_{n=0}^\infty\frac{\displaystyle\sqrt{n+\tfrac{1}{2}}-\ii Z}{\displaystyle\sqrt{n+\tfrac{1}{2}}+\ii Z}\ee^{4\ii Z(\sqrt{n+1}-\sqrt{n})},\quad\text{and}
\label{eq:Y0-def}
\end{equation}
\begin{equation}
\mathcal{E}_0(\SP):=-2\ii\frac{\varphi_0(\SP)}{v(-\varphi_0(\SP)^2)}\frac{v(-\varphi_0(\SP)^2)-v(0)}{-\varphi_0(\SP)^2}\left(1-\frac{1}{\sqrt{2}}\right)\zeta(-\tfrac{1}{2}).
\label{eq:E0-define}
\end{equation}
Also, the error terms in \eqref{eq:z-zero-define} proportional to $\epsilon^{1/2}$ and $\epsilon$ both vanish identically in the limit $\SP\to 0$.  In fact, $\mathcal{E}_0(\lambda)$ is analytic at $\lambda=0$ and $\mathcal{E}_0(\lambda)=O(\lambda)$.
\label{prop:Y-zeta-small}
\end{restatable}

\begin{proof} \hyperlink{proof:Yzetasmall}{The proof} can be found in Appendix~\ref{sec-proofs}. \end{proof}

\begin{rem}
Note that for the semicircular Klaus-Shaw potential with constant $u(x)$, given by $A(x)=\psi(x,0)$ defined by \eqref{eq:ExactSemicircle}, Remark~\ref{rem:v-for-exact-semicircle} gives $v(\rr)\equiv v(0)$, so $\mathcal{E}_0(\SP)$ vanishes identically in this case.
\label{rem:exact-semicircle-term-vanishes}
\end{rem}

The model function $\mathcal{Y}_0(Z)$ defined by \eqref{eq:Y0-def} is meromorphic in $Z$ in the right and left $Z$-planes but has a jump discontinuity across the imaginary $Z$-axis, which corresponds to the imaginary $\SP$-axis near $\SP=0$.  The form of this model function is quite different from that which can be obtained (also with a smaller relative error) under the assumption that $\rho(s)$ does not vanish at $s=0$ \cite[Proposition 4.3]{BaikKMM07}.  While we are unable to express $\mathcal{Y}_0(Z)$ in closed form, we can easily see that 
$\mathcal{Y}_0(Z)\ee^{\ii\pi Z^2\mathrm{sgn}(\mathrm{Re}\{Z\})}$ is meromorphic with simple zeros at $Z=-\ii\sqrt{n+\frac{1}{2}}$ and simple poles at $Z=\ii\sqrt{n+\frac{1}{2}}$ for $n\in\mathbb{Z}_{\ge 0}$, a phenomenon that locally captures the features of the factor $\widetilde{a}(\SP)^{-1}$ appearing in the definition \eqref{eq:ZS-Y-define} of $Y_\epsilon(\SP)$.  Although $\mathcal{Y}_0(Z)$ appears to be difficult to analyze directly from its definition in terms of an infinite product, we can easily prove the following.
\begin{restatable}[Behavior of $\mathcal{Y}_0(Z)$ for small and large $Z$]{prop}{Yzeroasymp}
As $Z\to 0$, $\mathcal{Y}_0(Z)=1-2\ii (\sqrt{2}-1)\zeta(\frac{1}{2})Z +\bigo{Z^2}$ where the error term has a jump discontinuity across the imaginary axis in the $Z$-plane.
Also, for each small $\delta>0$, $\mathcal{Y}_0(Z)=1+\bigo{Z^{-1}}$ as $Z\to\infty$ uniformly for $0\le\arg(Z)\le\pi/2-\delta$ and for $\pi/2+\delta\le\arg(Z)\le\pi$.
\label{prop:Y0-asymp}
\end{restatable}

\begin{proof}
\hyperlink{proof:Yzeroasymp}{The proof} can be found in Appendix~\ref{sec-proofs}. 
\end{proof} 

Finally, we have the following result, describing the behavior of $Y_\epsilon(\SP)$ for $\SP$ near $\ii A_\mathrm{max}$ in the limit $\epsilon\downarrow 0$.  As with Proposition~\ref{prop:Y-zeta-small}, 
the relative error terms are controlled even if $\SP$ is near the imaginary axis, i.e., among the poles of $Y_\epsilon(\SP)$.  
\begin{restatable}[Asymptotic behavior of $Y_\epsilon(\SP)$ for $\SP\approx \ii A_{\max}$]{prop}{propYtop}
Suppose that $A$ is a semicircular Klaus-Shaw potential.
Then 
the asymptotic formula 
\eq
Y_\epsilon(\SP) = 
\mathcal{Y}_1\left(\frac{\varphi_1(\SP)}{\epsilon}\right)\left(1+\mathcal{E}_1(\SP)\epsilon^{1/2}+\mathcal{O}\left(\epsilon\right)\right),\quad\epsilon\downarrow 0,\quad \varphi_1(\SP):=-\frac{\phaseint(\SP)}{\pi}
\label{Y-near-iAmax}
\endeq
holds uniformly for $|\SP-\ii A_\mathrm{max}|$ sufficiently small, where the mapping $\SP\mapsto \varphi_1(\SP)$ is conformal near $\SP=\ii A_\mathrm{max}$ with $\varphi_1(\ii A_\mathrm{max})=0$ and $\varphi_1'(\ii A_\mathrm{max})$ negative imaginary, and where
\eq
\mathcal{Y}_1(W):=\frac{1}{\sqrt{2\pi}}W^{-W}\Gamma(W+\tfrac{1}{2})\ee^W
\label{eq:Y1-def}
\endeq
and
\begin{equation}
\mathcal{E}_1(\SP):=-\frac{2\ii v_0}{\SP}\left(1-\frac{1}{\sqrt{2}}\right)\zeta(-\tfrac{1}{2}).
\label{eq:E1-def}
\end{equation}
\label{prop:Y-zeta-near-iAmax}
\end{restatable}

\begin{proof}
\hyperlink{proof:Ytop}{The proof} can be found in Appendix~\ref{sec-proofs}.
\end{proof}

\begin{rem}
The leading term of the relative error is exactly the same as in Proposition~\ref{prop:Y-outside}.
\end{rem}
The model function defined by \eqref{eq:Y1-def} is very similar to that obtained in a similar situation in \cite[Proposition~4.3]{BaikKMM07}, but the error term is larger due to the effect of $\evdensity(s)$ vanishing linearly at $s=0$.  Note that in \eqref{eq:Y1-def}, $W^{-W}$ is the principal branch, and hence $\mathcal{Y}_1(W)$ has a branch cut across the negative real $W$-axis, which corresponds to the imaginary $\SP$-axis below the point $\SP=\ii A_\mathrm{max}$.

\subsubsection{Analysis of $T_\epsilon(\SP)$ in the limit $\epsilon\downarrow 0$}
\label{app:T}

We first give the analogue for $T_\epsilon(\SP)$ of Proposition~\ref{prop:Y-outside}.  The function $T_\epsilon(\SP)$ is analytic for $\SP\in\Lambda$ (unlike $Y_\epsilon(\SP)$, which has poles and a branch cut along the center line of the lens $\Lambda$) and the following result shows that its asymptotic behavior in this domain is simplest for $\SP$ suitably bounded away from the points $\SP=0$ and $\SP=\ii A_\mathrm{max}$.

\begin{restatable}[Basic asymptotic behavior of $T_\epsilon(\SP)$ for $\SP\in\Lambda$]{prop}{Tbasicasymp}
Suppose that $A$ is a semicircular Klaus-Shaw potential.
Let $\sigma>0$ be arbitrary and define $\Lambda_\sigma:=\{\SP\in\Lambda: |\SP|>\sigma,\;|\SP-\ii A_\mathrm{max}|>\sigma\}$.
Then 
\begin{equation}
T_\epsilon(\SP)=
1-\frac{2\ii v_0}{\SP}\left(1-\frac{1}{\sqrt{2}}\right)\zeta(-\tfrac{1}{2})\epsilon^{1/2}+\bigo{\epsilon},
\quad\epsilon\downarrow 0
\end{equation}
holds uniformly for $\SP\in \Lambda_\sigma$.  
\label{prop:T-outside}
\end{restatable}

\begin{proof}
\hyperlink{proof:Tbasicasymp}{The proof} can be found in Appendix~\ref{sec-proofs}.
\end{proof}

Next, we can easily obtain analogues of Propositions~\ref{prop:Y-zeta-small} and \ref{prop:Y-zeta-near-iAmax} with the help of the exact relation \eqref{eq:T-Y-relation} between $T_\epsilon(\SP)$ and $Y_\epsilon(\SP)$.
\begin{restatable}[Asymptotic behavior of $T_\epsilon(\SP)$ for $\SP\approx 0$]{prop}{TatZero}
Suppose that $A$ is a semicircular Klaus-Shaw potential.
Then 
\begin{equation}
T_\epsilon(\SP)=\mathcal{T}_0\left(\frac{\varphi_0(\SP)}{\epsilon^{1/2}}\right)\left(1+\mathcal{E}_0(\SP)\epsilon^{1/2}+\bigo{\epsilon}\right),\quad\epsilon\downarrow 0,\quad \varphi_0(\SP):=\ii\pi^{-1/2}(\phaseint(0)-\phaseint(\SP))^{1/2}
\label{eq:T-zeta-small}
\end{equation}
holds uniformly for $\imag\{\SP\}>0$ and $|\SP|$ sufficiently small, where the mapping $\SP\mapsto \varphi_0(\SP)$ is conformal near $\SP=0$ with $\varphi_0(0)=0$ and $\varphi_0'(0)=1/v(0)>0$, $\mathcal{E}_0(\SP)$ is defined by \eqref{eq:E0-define}, and where
\begin{equation}
\mathcal{T}_0(Z):=2\prod_{n=0}^\infty
\left(1-\frac{\ii Z}{\sqrt{n+\tfrac{1}{2}}}\right)^2\left(1-\frac{Z^2}{n+\tfrac{1}{2}}\right)\ee^{4\ii Z(\sqrt{n+1}-\sqrt{n})}.
\label{eq:T0-def}
\end{equation}
Also, the error terms in \eqref{eq:T-zeta-small} proportional to $\epsilon^{1/2}$ and $\epsilon$ both vanish identically in the limit $\SP\to 0$. 
\label{prop:T-zeta-small}
\end{restatable}

\begin{proof}
\hyperlink{proof:TatZero}{The proof} can be found in Appendix~\ref{sec-proofs}.
\end{proof}

The behavior of  $\mathcal{T}_0(Z)$ for small and large $Z$ is as follows.
\begin{restatable}[Behavior of $\mathcal{T}_0(Z)$ for small and large $Z$]{prop}{TSmallLargeZ}
$\mathcal{T}_0(Z)$ is analytic at $Z=0$ with Taylor expansion $\mathcal{T}_0(Z)=2-4\ii (\sqrt{2}-1)\zeta(\tfrac{1}{2})Z + \bigo{Z^2}$ as $Z\to 0$. 
Also, for each small $\delta>0$, $\mathcal{T}_0(Z)=1-2\ii(1-\sqrt{2})\zeta(-\tfrac{1}{2})Z^{-1} + \bigo{Z^{-2}}$ as $Z\to\infty$ uniformly for $|\arg(-\ii Z)|\le\delta$.
\label{prop:T0-asymp}
\end{restatable}

\begin{proof}
\hyperlink{proof:TSmallLargeZ}{The proof} can be found in Appendix~\ref{sec-proofs}.
\end{proof}

\begin{restatable}[Asymptotic behavior of $T_\epsilon(\SP)$ for $\SP\approx \ii A_{\max}$]{prop}{Ttop}
Suppose that $A$ is a semicircular Klaus-Shaw potential.
Then 
the asymptotic formula
\eq
T_\epsilon(\SP) = 
\mathcal{T}_1\left(\frac{\varphi_1(\SP)}{\epsilon}\right)\left(1+\mathcal{E}_1(\SP)+\mathcal{O}\left(\epsilon\right)\right),\quad\epsilon\downarrow 0,\quad \varphi_1(\SP):=-\frac{\phaseint(\SP)}{\pi}
\label{T-near-iAmax}
\endeq
holds uniformly for $|\SP-\ii A_\mathrm{max}|$ sufficiently small, where the mapping $\SP\mapsto \varphi_1(\SP)$ is  conformal near $\SP=\ii A_\mathrm{max}$ with $\varphi_1(\ii A_\mathrm{max})=0$ and $\varphi_1'(\ii A_\mathrm{max})$ negative imaginary, and where
\eq
\mathcal{T}_1(W):=\frac{\sqrt{2\pi}\ee^W(-W)^{-W}}{\Gamma(\tfrac{1}{2}-W)},
\label{eq:T1-def}
\endeq
and where $\mathcal{E}_1(\SP)$ is defined in \eqref{eq:E1-def}.
\label{prop:T-zeta-near-iAmax}
\end{restatable}

\begin{proof}
\hyperlink{proof:Ttop}{The proof} can be found in Appendix~\ref{sec-proofs}.
\end{proof}

%% file: sec-zakharov-shabat.tex
In this section we apply the Deift-Zhou steepest-descent method to prove Theorem~\ref{thm-accuracy-t=0} for $x \in (X_-, X_+) \setminus \{ x_0 \} $, that is for $x$ inside the support of the initial data $A(x)$ and away from its unique maximizer $x=x_0$. 
The analysis is different for $x$ in one or the other of the open intervals
\begin{equation}
	J^- := (X_-, x_0), \qquad J^+ := (x_0, X_+).
\end{equation}
For technical reasons we do not consider the case when $x = x_0$, though our numerics (see Figure~\ref{fig:NLS_ICs}) suggest nothing interesting happens near the maximizer. 
In Section~\ref{sec:ZS-g} we construct the $g$-functions $g^\pm(\SP;x)$ which we will use for the steepest-descent analysis for $x \in J^\pm$. We then describe the remaining analysis for $x \in J^+$ in Sections~\ref{sec-steepest-descent}-\ref{sec-accuracy}. The modifications necessary to extend the proof to $x \in J^-$ are then described briefly in Section~\ref{sec-x<x_0}. Finally, in Section~\ref{sec-proof-Thm1} we combine the results to complete the proof of Theorem~\ref{thm-accuracy-t=0} for $x\in (X_-,X_+)\setminus\{x_0\}$.

\subsection{Two $g$-functions for inverse scattering}
\label{sec:ZS-g}
Let $A$ be a potential with the Klaus-Shaw property (see Definition~\ref{def:semicircularKS}).
Note that for $L(\SP)$ defined by \eqref{eq:Lfunc-def}, we may consider a point $\SP=\ii s$, $0<s<A_\mathrm{max}$ and compute the average of the boundary values taken on the branch cut at this point:
\begin{multline}
\overline{L}(\ii s):=\frac{1}{2}\left(L_+(\ii s)+L_-(\ii s)\right) \\
= (x_+(s)-x_-(s))s - \int_{-\infty}^{x_-(s)}
\left[\sqrt{s^2-A(x)^2}-s\right]\,\dd x - \int_{x_+(s)}^{+\infty}\left[\sqrt{s^2-A(x)^2}-s\right]\,\dd x.
\label{eq:Lbar-def-1}
\end{multline}
Comparing with the definition \eqref{eq:mu-define} of $\tailint(\ii s)$, we see that
\begin{equation}
\tailint(\ii s)\pm\overline{L}(\ii s)=2x_\pm(s)s-2\int_{x_\pm(s)}^{\pm\infty}\left[\sqrt{s^2-A(x)^2}-s\right]\,\dd x,\quad 0<s<A_\mathrm{max}.
\label{eq:mu-plus-Lbar}
\end{equation}
\begin{lem}
Let $A:\R \to [0,\infty)$ be a semicircular Klaus-Shaw potential with support $[X_-,X_+]$ (see Definition~\ref{def:semicircularKS}).  Then 
\begin{equation}
\tailint(\ii s)\pm\overline{L}(\ii s)=2X_\pm s - 2\int_{x_\pm(s)}^{X_\pm}\sqrt{s^2-A(x)^2}\,\dd x,\quad 0<s<A_\mathrm{max}.
\label{eq:mu-plus-Lbar-cpt}
\end{equation}
Furthermore, $\tailint(\ii s)\pm\overline{L}(\ii s)$ is analytic on $0<s<A_\mathrm{max}$ and extends to an odd analytic function of $s$ in a neighborhood of $s=0$ satisfying $\tailint(\ii s)\pm\overline{L}(\ii s)=2X_\pm s + \bigo{s^3}$.  Also for $0<s<A_\mathrm{max}$, $\mp(\tailint(\ii s)\pm\overline{L}(\ii s))\ge \mp 2X_\pm s$.
\label{lem:mu-Lbar}
\end{lem}
\begin{proof}
Analyticity on $0<s<A_\mathrm{max}$ follows from analyticity of $A(x)$ within its support, which also implies the analyticity of the turning point functions $x_\pm(s)$ on $0<s<A_\mathrm{max}$.  Equation \ref{eq:mu-plus-Lbar-cpt} 
follows since $A(x)$ has compact support $[X_-,X_+]$.
Since for $0<s<A_\mathrm{max}$ it holds that $x_\pm(s)\in (X_-,X_+)$, it follows that the integral may be dropped to obtain the inequality $\mp(\tailint(\ii s)\pm\overline{L}(\ii s))>2X_\pm s$.
Now observe that due to the Klaus-Shaw condition, $A(x)^2$ is monotone on the interval of integration, and hence the inverse function $x(y)$ satisfying $y=A(x)^2$ is well-defined.  Therefore, making the substitution $x=x(y)$ and rescaling by $y=s^2z$ gives
\begin{equation}
\begin{split}
\tailint(\ii s)\pm\overline{L}(\ii s)&=2X_\pm s + 2\int_0^{s^2}\sqrt{s^2-y}\,x'(y)\,\dd y\\
&=2X_\pm s + 2s^3\int_0^1\sqrt{1-z}\,x'(s^2z)\,\dd z,\quad 0<s<A_\mathrm{max}.
\end{split}
\label{eq:mu-pm-Lbar}
\end{equation}
Taking into account the representation $A(x)^2=u(x)^2(X_+-x)(x-X_-)$ on its support, with $u(x)$ being a positive analytic function on a complex neighborhood of $[X_-,X_+]$, it is obvious that $y=A(x)^2$ can be considered to be a univalent function on a neighborhood of each of the support endpoints $X_\pm$.  Therefore, $x'(y)$ is an analytic function of $y$ near $y=0$ and hence the right-hand side of \eqref{eq:mu-pm-Lbar} is an odd analytic function of $s$ near $s=0$.
\end{proof}

Let 
\begin{equation}
	R(\SP;x) = \sqrt{\SP^2 + A(x)^2}
\end{equation}	
be the function analytic for $\SP$ in the complex plane with a vertical cut between $\pm \ii A(x)$ omitted  that satisfies $R(\SP;x)=\SP+\bigo{\SP^{-1}}$ as $\SP\to\infty$.  Consider the function $g^\pm(\SP;x)$ defined in the same domain as $R(\SP;x)$ by the formula
\begin{align}
g^\pm(\SP;x):=&\,\frac{R(\SP;x)}{2\pi \ii}\left[\int_0^{A(x)}
\frac{2xs-\tailint(\ii s)\mp \overline{L}(\ii s)}{\sqrt{A(x)^2-s^2}(s+\ii \SP)}\,\dd s 
+ 
\int_0^{A(x)}\frac{2xs-\tailint(\ii s)\mp\overline{L}(\ii s)}{\sqrt{A(x)^2-s^2}(s-\ii \SP)}\,\dd s\right]\nonumber\\
{}=&\,\frac{R(\SP;x)}{2\pi \ii}\int_0^{A(x)}\frac{\phi^\pm(s;x)\,\dd s}{\sqrt{A(x)^2-s^2}(s^2+\SP^2)},
\label{eq:g-functions-zeta}
\end{align}
where 
\begin{equation}
\phi^\pm(s;x):=2s\cdot(2sx-\tailint(\ii s)\mp\overline{L}(\ii s)),\quad 0<s<A_\mathrm{max}.
\label{eq:phi-pm-define}
\end{equation}  
According to the identities \eqref{eq:mu-plus-Lbar}, $g^\pm(\SP;x)$ is a kind of integral transform of the turning point function $x_\pm(\cdot)$.  If $A(\cdot)$ is a semicircular Klaus-Shaw potential (cf. Definition~\ref{def:semicircularKS}) then by Lemma~\ref{lem:mu-Lbar}, $\phi^\pm(s;x)$ may be uniquely extended to a complex neighborhood of $[-A(x),A(x)]$ as an even analytic function of $s$ (which we also denote by $\phi^\pm(s;x)$), in which case we can write $g^\pm(\SP;x)$ in the form
\begin{equation}
\begin{split}
g^\pm(\SP;x)&=\frac{R(\SP;x)}{4\pi \ii}\int_{-A(x)}^{A(x)}\frac{\phi^\pm(s;x)\,\dd s}{\sqrt{A(x)^2-s^2}(s^2+\SP^2)} \\ &= \frac{R(\SP;x)}{8\pi \ii}\oint_L\frac{\phi^\pm(s;x)\,\dd s}{R(\ii s;x)(s^2+\SP^2)},\quad \text{$\ii\SP$ and $-\ii\SP$ exterior to $L$},
\end{split}
\label{eq:g-functions-zeta-outside}
\end{equation}
where $L$ is a simple closed curve in the domain of analyticity of $\phi^\pm(s;x)$ that encircles the interval $[-A(x),A(x)]$ 
once in the positive sense.  If we want to allow $\ii\SP$ and $-\ii\SP$ to approach the interval $[-A(x),A(x)]$, we can pay the price of two residues
and obtain (using oddness of $R(\diamond;x)$ and evenness of $\phi^\pm(\cdot;x)$)
\begin{equation}
g^\pm(\SP;x)=\frac{\phi^\pm(-\ii\SP;x)}{4\ii\SP}+\frac{R(\SP;x)}{8\pi \ii}
\oint_{L}\frac{\phi^\pm(s;x)\,\dd s}{R(\ii s;x)(s^2+\SP^2)},\quad \text{$\ii\SP$ and $-\ii\SP$ interior to $L$}.
\label{eq:g-functions-zeta-inside}
\end{equation}

\begin{restatable}{prop}{gprop}
Let $A(\cdot)$ be a semicircular Klaus-Shaw potential with support $[X_-,X_+]$ and maximizer $x_0$.  The function $g^\pm(\SP;x)$, for $x \in [X_-,X_+]$, has the following properties.  
\begin{itemize}
\item[G1:] $g^\pm(\SP;x)$ is analytic and uniformly bounded in its domain of definition.  
\item[G2:] $g^\pm(\SP;x)$ is an odd function of $\SP$.  
\item[G3:] $g^\pm(\SP^*;x)=-g^\pm(\SP;x)^*$, and in particular $g^\pm(\SP;x)$ is imaginary for real $\SP\neq 0$.
\item[G4:] The sum of boundary values taken by $g^\pm$ on its branch cut satisfies
\begin{equation}
g_+^\pm(\ii s;x)+g_-^\pm(\ii s;x)=-2sx + \tailint(\ii s)\pm\overline{L}(\ii s),\quad 0<s<A(x).
\label{eq:g-zeta-plus-minus}
\end{equation}
\item[G5:] $g^\pm(\SP;x)=\bigo{\SP^{-1}}$ as $\SP \to \infty$.
\item[G6:] 
If $\pm(x-x_0)>0$, then there exist analytic functions $\SP\mapsto G^\pm_1(\SP;x)$ and $\SP\mapsto G^\pm_2(\SP;x)$ defined in a neighborhood $D(x)$ of $\ii A(x)$ such that 
\begin{equation}
g^\pm(\SP;x)=G_1^\pm(\SP;x)+(-\ii\SP-A(x))^{3/2}G_2^\pm(\SP;x)
\end{equation}
holds for $\SP\in D(x)\setminus B(x)$, where $B(x)$ denotes the branch cut of $g^\pm(\SP;x)$.  
\item[G7:] $g^\pm(\SP;x)\to 0$ as $x\downarrow X_-$ or $x\uparrow X_+$.  
\item[G8:] The partial derivative of $g^\pm(\SP;x)$ with respect to $x$ is given explicitly by $g^\pm_x(\SP;x)=\ii(\SP-R(\SP;x))$.  
\item[G9:] We have the identities
\begin{equation} 
g^+(\SP;x)=-\ii\int_x^{X_+}(\SP-R(\SP;y))\,\dd y,\quad
g^-(\SP;x)=\ii\int_{X_-}^x(\SP-R(\SP;y))\,\dd y,\quad x\in [X_-,X_+].
\label{eq:gplus-gminus-x-integral}
\end{equation}
Recalling the function $L(\SP)$ defined by \eqref{eq:Lfunc-def}--\eqref{eq:mu-plus-minus-define}, we also have the identity
\begin{equation}
L(\SP)-g^+(\SP;x)=-g^-(\SP;x),\quad x\in [X_-,X_+].
\label{eq:gplus-gminus-L-identity}
\end{equation}
\end{itemize}
\label{prop:g-properties}
\end{restatable}

\begin{proof}
\hyperlink{proof-g-prop}{The proof} can be found in Appendix~\ref{sec-g-prop}.
\end{proof}
The formula \eqref{eq:g-functions-zeta-inside} motivates the introduction of a function $h^\pm(\SP;x)$ related to $g^\pm(\SP;x)$ as follows:
\begin{equation}
\begin{split}
h^\pm(\SP;x):=&\pm\left(g^\pm(\SP;x)-\frac{1}{2}(\tailint(\SP)\pm\overline{L}(\SP)+2\ii\SP x)\right)\\
{}=&\pm\frac{R(\SP;x)}{8\pi \ii}\oint_L\frac{\phi^\pm(s;x)\,\dd s}{R(\ii s;x)(s^2+\SP^2)},\quad
\text{$\ii\SP$ and $-\ii\SP$ interior to $L$}.
\end{split}
\label{eq:h-def-zeta}
\end{equation}
\begin{restatable}{prop}{hprop}
Let $A(\cdot)$ be a semicircular Klaus-Shaw potential with support $[X_-,X_+]$ and maximizer $x_0$. For $x \in J^\pm_c$ an arbitrary compact subset of $J^\pm$, the corresponding function $h^\pm(\SP;x)$ has the following properties:
\begin{itemize}
\item[H1:] There is a conformal mapping $\SP\mapsto W(\SP)$ defined in a neighborhood $D(x)$ of $\SP=\ii A(x)$ such that $4h^\pm(\SP;x)^2=W(\SP)^3$ for $\SP\in D(x)$ and $W(\SP)>0$ for $\SP\in D(x)$ with $A(x)<-\ii\SP$.
\item[H2:] Given $\delta>0$ sufficiently small there exists a positive constant $\eta=\eta(J_c^\pm,\delta)$ such that
$h^\pm(\SP;x)>\eta$ for $A(x)+\delta<-\ii\SP<A_\mathrm{max}$ and
$\re\{h^\pm(\SP;x)\}<-\eta$ for $\delta<|\re\{\SP\}|<2\delta$ and $\delta<\imag\{\SP\}<A(x)-\delta$.  
\item[H3:] Given $\delta>0$ sufficiently small there exists a positive constant $\eta=\eta(J^\pm_c,\delta)$ such that $\re\{h^\pm(\SP;x)- \ii\phaseint(\SP)\}>\eta$ holds on the parabolic arc 
$\re\{\SP\}=\delta\imag\{\SP\}(A_\mathrm{max}-\imag\{\SP\})$ with $\delta<\imag\{\SP\}<A_\mathrm{max}$.
Similarly, $\re\{h^\pm(\SP;x)+\ii\phaseint(\SP)\}>\eta$ holds on the parabolic arc 
$\re\{\SP\}=-\delta\imag\{\SP\}(A_\mathrm{max}-\imag\{\SP\})$ with $\delta<\imag\{\SP\}<A_\mathrm{max}$.
\item[H4:] The boundary values $h^\pm_+(\SP;x)$ and $h^\pm_-(\SP;x)$ taken by $h^\pm(\SP;x)$ on the branch cut $-A(x)\le -\ii\SP\le A(x)$ from the right and left half-planes respectively are both analytic at $\SP=0$ with convergent power series consisting of even powers of $\SP$.  Also, $h^\pm_-(\SP;x)=-h^\pm_+(\SP;x)$, and $h^\pm_+(\SP;x)=\ii\alpha^\pm(x) + \ii\beta^\pm(x)\SP^2 + \bigo{\SP^4}$ as $\SP\to 0$ where $\alpha^\pm(x)$ and $\beta^\pm(x)$ are real, and where $\beta^\pm(x) \geq c$, for $c >0$ a constant depending on $J_c^\pm$.
\item[H5:] Recalling that $\phaseint(\SP)$ is an even analytic function of $\SP$ near $\SP=0$, the even analytic function $h^\pm_+(\SP;x)-\ii\phaseint(\SP)$ satisfies $h^\pm_+(\SP;x)-\ii\phaseint(\SP)= \ii(\alpha^\pm(x)-\phaseint_0) +\ii(\beta^\pm(x)-\phaseint_1)\SP^2 + \bigo{\SP^4}$ as $\SP\to 0$, where the real coefficients $\phaseint_0$ and $\phaseint_1$ are given by \eqref{eq:Psi0-Psi1}, and where $\beta^\pm(x)-\phaseint_1\leq - c$, for $c > 0$ a constant depending on $J_c^\pm$.  Also $h^\pm_-(\SP;x)+\ii\phaseint(\SP)=-(h^\pm_+(\SP;x)-\ii\phaseint(\SP))$.
\item[H6:] 
The boundary values taken by $h^\pm(\SP;x)$ on the branch cut $0<-\ii\SP<A(x)$ can be expressed in terms of the difference in boundary values taken by $g^\pm(\SP;x)$ on the same cut:
\begin{equation}
g^\pm_+(\SP;x)-g^\pm_-(\SP;x)=\pm 2h_+^\pm(\SP;x) = \mp 2h_-^\pm(\SP;x),\quad 0<-\ii\SP<A(x).
\label{eq:gdiff-h}
\end{equation}
\end{itemize}
\label{prop:h-properties}
\end{restatable}

\begin{proof}
\hyperlink{proof-h-prop}{The proof} can be found in Appendix~\ref{sec-g-prop}.
\end{proof}

%% file: sec-steepest-descent.tex
\noindent

\subsubsection{Removal of the poles}
For $x \in J^+$, we begin by interpolating the residues at the poles to replace the meromorphic function $\widetilde{\mathbf{M}}(\SP;x,\mathbf{0})$ satisfying Riemann-Hilbert Problem~\ref{rhp-meromorphic} with a sectionally analytic function.  Let $\Sigma_0 = [0, \ii A_\text{max}]$ be oriented from $\ii A_\text{max}$ to $0$; 
let $\Sigma_\pm$ be oriented contours from $\SP = 0$ to $\SP = \ii A_\text{max}$ lying in $\pm\Re\{\SP \}> 0 $ away from its endpoints such that the parabolic lens region $\Lambda$ defined by \eqref{eq:Lambda-def} is enclosed by $\Sigma_+ \cup \Sigma_-$.  
Denote by $\Omega_\pm$ the region enclosed between $\Sigma_0$ and $\Sigma_\pm$. 
See Figure~\ref{fig-Q-jumps}.

\begin{figure}[h]
\begin{center}
\scalebox{.7}{
\begin{tikzpicture}[>=stealth]
\coordinate (O) at (0,0);
\coordinate (iA) at (0,3);
\coordinate (iAc) at (0,-3);
\coordinate (top) at (0,4.5);
\coordinate (bottom) at (0,-4.5);
\coordinate (leftmax) at (-5,0);
\coordinate (rightmax) at (5,0);

\draw[dashed,thick]  (leftmax) node[left] {$\phantom{\Im(\SP)=0}$} -- (rightmax) node[right, scale=1.5] {$\imag(\SP)=0$};

\fill[color=blue!50, fill opacity=0.4] (O) ..controls (2,2) and (2,3.5).. (top) ..controls (-2,3.5) and (-2,2).. (O);
\fill[color=blue!50, fill opacity=0.4] (O) ..controls (2,-2) and (2,-3.5).. (bottom) ..controls (-2,-3.5) and (-2,-2).. (O);
\draw[->-=0.4, ultra thick, draw=black] (O) ..controls (2,2) and (2,3.5).. (top);
\draw[->-=0.4, ultra thick, draw=black] (O) ..controls (-2,2) and (-2,3.5).. (top);
\draw[->-=0.4, ultra thick, draw=black] (O) ..controls (2,-2) and (2,-3.5).. (bottom);
\draw[->-=0.4, ultra thick, draw=black] (O) ..controls (-2,-2) and (-2,-3.5).. (bottom);
\draw[-<-=0.4, ultra thick] (O) -- (top) 
  node[pos=0.7, pin={[pin distance=10mm, scale=1.5, pin edge={<-,black}]20: $\Sigma_0$ }] {};  
\draw[-<-=0.4, ultra thick] (O) -- (bottom)
  node[pos=0.7, pin={[pin distance=10mm, scale=1.5, pin edge={<-,black}]-20: $\Sigma_0^*$ }] {};

\node[scale=1.5] at (.8,2.5) {$\Omega_+$};
\node[scale=1.5] at (-.6,2.5) {$\Omega_-$};
\node[scale=1.5] at (.8,-2.5) {$\Omega_+^*$};
\node[scale=1.5] at (-.6,-2.5) {$\Omega_-^*$};
\node[scale=1.5] at (1.8,1.5) {$\Sigma_+$};
\node[scale=1.5] at (-1.6,1.5) {$\Sigma_-$};
\node[scale=1.5] at (1.8,-1.4) {$\Sigma_+^*$};
\node[scale=1.5] at (-1.6,-1.4) {$\Sigma_-^*$};

\end{tikzpicture} 
}
\end{center}
\caption{\label{fig-Q-jumps} 
The regions $\Omega_\pm$ and the jump contours 
$\Sigma_0$ and $\Sigma_\pm$ (with their orientations) used to define the matrix transformation $\widetilde{\bf{M}} \mapsto \bf{Q}$ in \eqref{eq:Qdef}.
}
\end{figure}
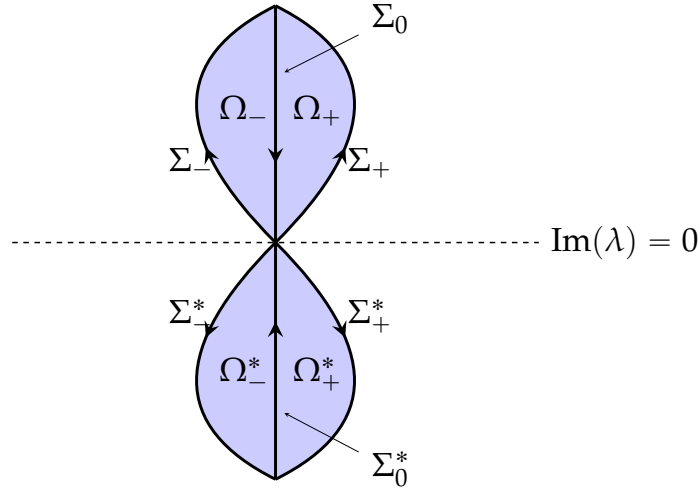

For any fixed $K \in \Z$ and $(x,t_2,t_3,\dots,t_M)\in\mathbb{R}^M$, define the exponent function
\begin{equation}
	2f_K(\SP;x,\mathbf{t}) :=  \ii (2K+1) \phaseint(\SP) + \tailint(\SP) + 2\ii Q(\SP;x,\mathbf{t}),
	\label{eq:fK-define}
\end{equation}
where $Q$ is defined in \eqref{eq:cn0},
which is analytic in $\Omega_+ \cup \Omega_-$. Let $K_+:=0$ and $K_-:=-1$. Define
\eq
{\bf Q}(\SP;x):=\begin{cases} \widetilde{\bf M}(\SP;x,\mathbf{0}) 
\begin{pmatrix} 1 & 0 \\ -\ii(-1)^{K_\pm}\widetilde{a}(\lambda)^{-1} \ee^{2f_{K_\pm}(\SP;x,\mathbf{0})/\epsilon} & 1 \end{pmatrix}, 
& \SP\in\Omega_\pm, \\ 
{\bf Q}(\SP^*;x)^{-\dagger}, & \SP\in\Omega_\pm^*, \\ \widetilde{\bf M}(\SP;x,\mathbf{0}), & \text{otherwise}. \end{cases}
\label{eq:Qdef}
\endeq
The resulting Riemann-Hilbert problem satisfied by $\mathbf{Q}(\SP;x)$ is as follows.

\begin{myrhp}[Sectionally Analytic Problem]
Given $\epsilon>0$ and $x \in J^+$, seek a $2\times 2$ matrix function $\mathbf{Q}(\SP)=\mathbf{Q}(\SP;x)$ with the following properties.
\begin{itemize}
\item[]\textit{\textbf{Analyticity:}}  $\mathbf{Q}(\SP)$ is analytic for $\SP\in\mathbb{C}\setminus\Sigma$ and satisfies the Schwarz symmetry condition $\mathbf{Q}(\SP^*)=\mathbf{Q}(\SP)^{-\dagger}$.
\item[]\textit{\textbf{Jump conditions:}} $\mathbf{Q}(\SP)$ takes continuous boundary values on $\Sigma$ from each maximal connected component of $\mathbb{C}\setminus\Sigma$.  Given a point $\SP\in\mathbb{C}_+$ on one of the oriented arcs of $\Sigma$, let the boundary value taken at $\SP$ by $\mathbf{Q}$ from the left (respectively, right) be denoted $\mathbf{Q}_+(\SP)$ (respectively, $\mathbf{Q}_-(\SP)$).  Then 
\begin{equation}
\mathbf{Q}_+(\SP)=\mathbf{Q}_-(\SP)\mathbf{V^Q}(\SP),
\end{equation}
where the jump matrix $\mathbf{V^Q}(\SP)=\mathbf{V^Q}(\SP;x)$ is defined on the various arcs of $\Sigma\cap\mathbb{C}_+$ by 
\begin{equation}
\mathbf{V^Q}(\SP;x):=\begin{pmatrix} 1 & 0 \\ -\ii T_\epsilon(\SP) \ee^{\varphi^+(\SP;x)/\epsilon} & 1 \end{pmatrix}
,\quad \SP\in\Sigma_0,
\label{eq:jump+++1}
\end{equation}
\eq
\mathbf{V^Q}(\SP;x):= \begin{pmatrix} 1 & 0 \\ -\ii Y_\epsilon(\SP)\ee^{[\varphi^+(\SP;x)\pm 2\ii\phaseint(\SP)]/\epsilon} & 1 \end{pmatrix}, \quad \SP\in\Sigma_\pm,
\label{eq:jump+++2}
\endeq
where 
\eq
\label{eq:varphi+}
\varphi^+(\lambda;x):=\tailint(\lambda)+2\ii Q(\lambda;x,\mathbf{0})+\overline{L}(\lambda).
\endeq
Corresponding jump conditions on the arcs of $\Sigma$ in the lower half-plane are induced by the Schwarz symmetry condition.
\item[]\textit{\textbf{Normalization:}} $\mathbf{Q}(\SP)\to\mathbb{I}$ as $\SP\to\infty$.
\end{itemize}
\label{rhp:+++}
\end{myrhp}
Here in writing down the jumps we have used \eqref{eq:ZS-nu0-jump}, \eqref{eq:ZS-Y-define}, 
\eqref{eq:Lbar-def-3}, and \eqref{eq:ZS-T-define}.

\subsubsection{Installing the $g$-function and lens deformation}
We now introduce the $g$-function.  Let $g^+(\SP;x)$ be defined as in 
\eqref{eq:g-functions-zeta}.
Make the change of variables 
\begin{equation}
\mathbf{R}(\SP;x):= 
\mathbf{Q}(\SP;x) \begin{pmatrix} \ee^{-g^+(\SP;x)/\epsilon} & 0 \\ 0 & \ee^{g^+(\SP;x)/\epsilon} \end{pmatrix},
\label{eq:Q-to-R-case1}
\end{equation}
Now define $B\equiv B(x)$ to be the subset of $\Sigma_0$ in which $g^+$ has 
its jump discontinuity (with orientation inherited from $\Sigma_0$).  We call 
$\overline{B\cup B^*}$ the \emph{band}.  The function ${\bf R}(\SP;x)$ 
satisfies the jumps 
\eq
\mathbf{R}_+(\SP;x)=\mathbf{R}_-(\SP;x)\begin{pmatrix}\ee^{-(g^{+}_+(\SP;x)-g^{+}_-(\SP;x))/\epsilon} & 0 \\
-\ii T_\epsilon(\SP) & \ee^{(g^{+}_+(\SP;x)-g^{+}_-(\SP;x))/\epsilon} \end{pmatrix}, \quad \SP\in B(x),
\label{eq:Q-jump-cut}
\endeq
and
\begin{equation}
\mathbf{R}_+(\SP;x)=\mathbf{R}_-(\SP;x) \begin{pmatrix} 1 & 0 \\ -\ii T_\epsilon(\SP)\ee^{-2h^{+}(\SP;x)/\epsilon} & 1 \end{pmatrix},\quad
 \SP\in\Sigma_0\setminus B(x),
 \label{eq:jump-above-1}
\end{equation}
where, as in \eqref{eq:h-def-zeta},
\begin{equation}
h^{+}(\SP;x):=g^{+}(\SP;x)-\frac{1}{2}\varphi^{+}(\SP;x),\qquad \SP\in \Sigma_0\setminus B(x). 
\label{eq:h+-def}
\end{equation}
To compute the jumps we have used Property G4 of Proposition \ref{prop:g-properties}.

We now prepare to open lenses.  Properties H1--H3 from \S\ref{sec:ZS-g} can 
be used to characterize the analytic continuation of $h^+(\SP;x)$
from $\Sigma_0\setminus B(x)$ to the domains $\Omega_\pm$.
In particular, a matrix factorization and the use of the identity \eqref{eq:gdiff-h} implies that 
\eqref{eq:Q-jump-cut} can be written in the equivalent form:
\eq
\mathbf{R}_+(\SP;x)\mathbf{L}_+(\SP;x)=
\mathbf{R}_-(\SP;x)\mathbf{L}_-(\SP;x)^{-1}
\begin{pmatrix}0 & -\ii T_\epsilon(\SP)^{-1}
\\
-\ii T_\epsilon(\SP)
& 0 \end{pmatrix}, \quad \SP\in B(x),
\label{eq:factored-jump-B-case1}
\endeq
where $\mathbf{L}(\SP;x)$ is the matrix analytic for $\SP\in\Omega_+\cup\Omega_-$ defined by
\eq
\mathbf{L}(\SP;x):=
\begin{pmatrix} 1 & -\ii T_\epsilon(\SP)^{-1}\ee^{2h^+(\SP;x)/\epsilon} \\ 0 & 1 \end{pmatrix}
, \quad  \SP\in \Omega_+\cup\Omega_-,\quad x\in J_+.
\label{eq:L-define-case1}
\endeq
Note that $\mathbf{L}(\SP;x)$ inherits from $h^+(\SP;x)$ a jump discontinuity across $B(x)$, which explains the subscripts in \eqref{eq:factored-jump-B-case1} indicating boundary values taken.  Lastly, we note that the jump condition on $\Sigma_\pm$ takes the form
\begin{equation}
\mathbf{R}_+(\SP;x)=\mathbf{R}_-(\SP;x) \begin{pmatrix} 1 & 0 \\ 
-\ii Y_\epsilon(\SP)\ee^{2[\pm \ii\phaseint(\SP)-h^+(\SP;x)]/\epsilon} & 1 \end{pmatrix}
,\quad\SP\in\Sigma_\pm.
\label{eq:jump-outer-lens-1}
\end{equation}
Define two lens domains $\Lambda_+$ and $\Lambda_-$ bounded by $B(x)$ and the 
two parabolic arcs defined in property H3 of $h^+$ in 
Proposition~\ref{prop:h-properties} in \S\ref{sec:ZS-g}.  We denote these 
parabolic arcs by $\Sigma_{L\pm}$.  Property H3 also 
tells us that the matrix ${\bf L}(\SP;x)$ will decay to the identity as 
$\epsilon\downarrow 0$ on $\Sigma_{L\pm}$.  With this in mind, define 
\begin{equation}
\mathbf{S}(\SP;x):=\begin{cases}
\mathbf{R}(\SP;x)\mathbf{L}(\SP;x),&\quad \SP\in\Lambda_+,\\
\mathbf{R}(\SP;x)\mathbf{L}(\SP;x)^{-1},&\quad\SP\in\Lambda_-,\\
\mathbf{S}(\SP^*;x)^{-\dagger},&\quad\SP\in\Lambda_+^*\cup\Lambda_-^*,\\
\mathbf{R}(\SP;x),&\quad \text{otherwise.}
\end{cases}
\label{eq:S-define-3}
\end{equation}

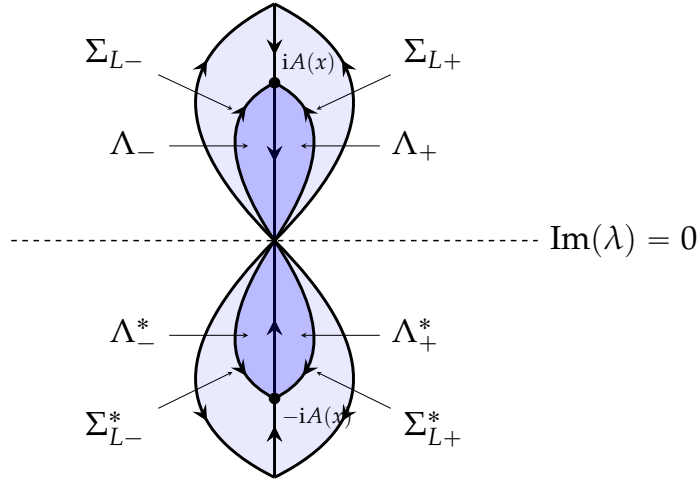
\begin{figure}
\begin{center}
\scalebox{.7}{
\begin{tikzpicture}[>=stealth]

\coordinate (O) at (0,0);
\coordinate (iA) at (0,3);
\coordinate (iAc) at (0,-3);
\coordinate (top) at (0,4.5);
\coordinate (bottom) at (0,-4.5);
\coordinate (leftmax) at (-5,0);
\coordinate (rightmax) at (5,0);

\draw[dashed,thick] (leftmax) node[left] {$\phantom{\Im(\SP)=0}$}-- (rightmax) node[right, scale=1.5] {$\Im(\SP)=0$};

\fill[color=blue!20, fill opacity=0.4] (O) ..controls (2,2) and (2,3.5).. (top) ..controls (-2,3.5) and (-2,2).. (O);
\fill[color=blue!20, fill opacity=0.4] (O) ..controls (2,-2) and (2,-3.5).. (bottom) ..controls (-2,-3.5) and (-2,-2).. (O);
\draw[->-=0.7, ultra thick, draw=black] (O) ..controls (2,2) and (2,3.5).. (top);
\draw[->-=0.7, ultra thick, draw=black] (O) ..controls (-2,2) and (-2,3.5).. (top);
\draw[->-=0.7, ultra thick, draw=black] (O) ..controls (2,-2) and (2,-3.5).. (bottom);
\draw[->-=0.7, ultra thick, draw=black] (O) ..controls (-2,-2) and (-2,-3.5).. (bottom);
\draw[-<-=0.4,-<-=0.85, ultra thick] (O) -- (top);  
\draw[-<-=0.4,-<-=0.85, ultra thick] (O) -- (bottom);

\fill[color=blue, fill opacity=0.2] (O) ..controls (1,1.5) and (1,2.5)..(iA) ..controls (-1,2.5) and (-1,1.5).. (O);
\fill[color=blue, fill opacity=0.2] (O) ..controls (1,-1.5) and (1,-2.5)..(iAc) ..controls (-1,-2.5) and (-1,-1.5).. (O);
\draw[->-=0.8, ultra thick, draw=black] (O) ..controls (1,1.5) and (1,2.5)..(iA)
 node[pos=0.7, pin={[pin distance=10mm, scale=1.5, pin edge={<-,black}]20: $\Sigma_{L+}$ }] {};
\draw[->-=0.8, ultra thick, draw=black] (O) ..controls (-1,1.5) and (-1,2.5)..(iA)
 node[pos=0.7, pin={[pin distance=10mm, scale=1.5, pin edge={<-,black}]160: $\Sigma_{L-}$ }] {};;
\draw[->-=0.8, ultra thick, draw=black] (O) ..controls (1,-1.5) and (1,-2.5)..(iAc)
 node[pos=0.7, pin={[pin distance=10mm, scale=1.5, pin edge={<-,black}]-20: $\Sigma_{L+}^*$ }] {};;
\draw[->-=0.8, ultra thick, draw=black] (O) ..controls (-1,-1.5) and (-1,-2.5)..(iAc) 
 node[pos=0.7, pin={[pin distance=10mm, scale=1.5, pin edge={<-,black}]-160: $\Sigma_{L-}^*$ }] {};;
\draw[fill] (iA) circle [radius=0.1] node[above right] {$\ii A(x)$};
\draw[fill] (iAc) circle [radius=0.1] node[below right] {$-\ii A(x)$};

\draw[->,black] (2,1.8) node[right,scale=1.5]{$\Lambda_+$} -- (0.5,1.8);
\draw[->,black] (-2,1.8) node[left,scale=1.5]{$\Lambda_-$} -- (-0.5,1.8);
\draw[->,black] (2,-1.8) node[right,scale=1.5]{$\Lambda_+^*$} -- (0.5,-1.8);
\draw[->,black] (-2,-1.8) node[left,scale=1.5]{$\Lambda_-^*$} -- (-0.5,-1.8);

\end{tikzpicture} 
}
\caption{
\label{fig-S-jumps} 
The regions $\Lambda_\pm$ and the jump contours $\Sigma_{L\pm}$ (along with their orientations) defining the transformation $\bf{R} \mapsto \bf{S}$ given by \eqref{eq:S-define-3}.
}
\end{center}
\end{figure}

The substitution \eqref{eq:S-define-3} separates the factors in the jump 
conditions \eqref{eq:factored-jump-B-case1} so that the jump matrix for 
$\mathbf{S}(\SP;x)$ on the band $B(x)$ is the  explicit off-diagonal factor on 
the right-hand side of \eqref{eq:factored-jump-B-case1}.  
On the contours $\Sigma_{L\pm}$ we have 
$\mathbf{S}_+(\SP;x)=\mathbf{S}_-(\SP;x)\mathbf{L}(\SP;x)$. From 
Proposition~\ref{prop:T-outside} in \S\ref{app:T}, the jump matrix for 
$\mathbf{S}(\SP;x)$ on $B(x)$ is uniformly an $\bigo{\epsilon^{1/2}}$ 
perturbation of $-\ii\sigma_1$ as long as $\SP$ is bounded away from the 
real axis.  

Let $\Sigma' = \Sigma \cup \Sigma_{L\pm} \cup \Sigma^*_{L\pm}$ denote the contour formed by adjoining the new lens contours to $\Sigma$.  The matrix $\mathbf{S}(\lambda;x)$ is then seen to satisfy the following Riemann-Hilbert problem, which has been ``stabilized'' via the introduction of $g^+(\lambda;x)$. 
\begin{myrhp}[Stabilized Problem]
Given $\epsilon>0$ and $x \in J^+$, seek a $2\times 2$ matrix function $\mathbf{S}(\SP)=\mathbf{S}(\SP;x)$ with the following properties.
\begin{itemize}
\item[]\textit{\textbf{Analyticity:}}  $\mathbf{S}(\SP)$ is analytic for $\SP\in\mathbb{C}\setminus\Sigma'$ and satisfies the Schwarz symmetry condition $\mathbf{S}(\SP^*)=\mathbf{S}(\SP)^{-\dagger}$.
\item[]\textit{\textbf{Jump conditions:}} $\mathbf{S}(\SP)$ takes continuous boundary values on $\Sigma'$ from each maximal connected component of $\mathbb{C}\setminus\Sigma'$.  Given a point $\SP\in\mathbb{C}_+$ on one of the oriented arcs of $\Sigma'$, let the boundary value taken at $\SP$ by $\mathbf{S}$ from the left (respectively, right) be denoted $\mathbf{S}_+(\SP)$ (respectively, $\mathbf{S}_-(\SP)$).  Then 
\begin{equation}
\mathbf{S}_+(\SP)=\mathbf{S}_-(\SP)\mathbf{V^S}(\SP),
\end{equation}
where the jump matrix $\mathbf{V^S}(\SP)=\mathbf{V^S}(\SP;x)$ is defined on the various arcs of $\Sigma'\cap\mathbb{C}_+$ by 
\begin{equation}
\mathbf{V^S}(\SP;x):= \begin{dcases}
	\begin{pmatrix} 1 & 0 \\ -\ii T_\epsilon(\SP)\ee^{-2 h^+(\SP;x)/\epsilon} & 1 \end{pmatrix},
	 & \SP\in\Sigma_0 \setminus B(x), \\
	\begin{pmatrix}0 & -\ii T_\epsilon(\SP)^{-1} \\ -\ii T_\epsilon(\SP) & 0 \end{pmatrix}, 
 	 & \SP\in B(x), \\
	\begin{pmatrix} 1 & 0 \\ -\ii Y_\epsilon(\SP)\ee^{2[\pm\ii \phaseint(\SP)-h^+(\SP;x)]/\epsilon} & 1 \end{pmatrix}, 
	 &\SP\in\Sigma_\pm, \\
	\begin{pmatrix} 1 & -\ii T_\epsilon(\SP)^{-1}\ee^{2h^+(\SP;x)/\epsilon} \\ 0 & 1 \end{pmatrix},
	 &\SP \in \Sigma_{L,\pm}, \\
\end{dcases}
\label{eq:jump+++S}
\end{equation}
where $h^+(\SP;x)$ is given by \eqref{eq:h+-def}. 
Corresponding jump conditions on the arcs of $\Sigma'$ in the lower half-plane are induced by the Schwarz symmetry 
condition.
\item[]\textit{\textbf{Normalization:}} $\mathbf{S}(\SP)\to\mathbb{I}$ as $\SP\to\infty$.
\end{itemize}
\label{rhp:+++S}
\end{myrhp}

%% file: sec-parametrix.tex
\subsubsection{Outer parametrix}

We begin with the construction of an outer parametrix designed to approximately solve the jump condition for $\mathbf{S}(\SP ; x)$ on the vertical band $B(x) \cup B(x)^*$. 
It follows from Proposition~\ref{prop:T-outside} that $\mathbf{V^S}(\lambda;x) =  -\ii \sigma_1 + \bigo{\eps^{1/2}}$ for those $\SP \in B(x)$ bounded away from $\SP =0$ and $\SP = \ii A_\text{max}$. Define
\begin{equation}
\label{eq:S.out.matrix}
	\breve{\mathbf{S}}^{\mathrm{out}}(\SP;x) : = \boldsymbol{\mathcal{E}} \left( \frac{\SP + \ii A(x)}{\SP- \ii A(x)} \right)^{\sigma_3/4} \boldsymbol{\mathcal{E}}^{-1}, \quad \SP \in \C \setminus ( B(x) \cup B(x)^* ),
\end{equation}
where the powers $\pm1/4$ refer to the principal branch, and 
\begin{equation}
	\boldsymbol{\mathcal{E}}:= \frac{1}{\sqrt{2}} \begin{pmatrix} \ee^{-\ii \pi/4} & \ee^{\ii \pi/4} \\ -\ee^{-\ii \pi/4} & \ee^{\ii \pi/4}\end{pmatrix}, 
	\qquad \det(\boldsymbol{\mathcal{E}}) = 1, \qquad \boldsymbol{\mathcal{E}}^{-1} = \boldsymbol{\mathcal{E}}^\dagger
	\label{eq:E-matrix-def}
\end{equation}
is a matrix of eigenvectors for $-\ii \sigma_1$. 

Then $\breve{\mathbf{S}}^{\mathrm{out}}(\SP;x)$ is analytic in its domain of definition, bounded for $\SP$ away from $\pm \ii A(x)$, with $\det (\breve{\mathbf{S}}^{\mathrm{out}}(\SP;x)) = 1$, and 
\begin{equation}\label{eq:S.out.expand}
	\breve{\mathbf{S}}^{\mathrm{out}}(\SP;x) = \mathbb{I} + \frac{ A(x)}{2\ii \SP} \sigma_1+\bigo{\SP^{-2}}, \qquad \SP \to \infty.
\end{equation}
Moreover, the boundary values $\breve{\mathbf{S}}^{\mathrm{out}}_+(\SP;x)$ and $\breve{\mathbf{S}}^{\mathrm{out}}_-(\SP;x)$ satisfy the jump relation 
\begin{equation}
	\breve{\mathbf{S}}^{\mathrm{out}}_+(\SP;x) = \breve{\mathbf{S}}^{\mathrm{out}}_-(\SP;x) (-\ii \sigma_1),  \qquad \lambda \in B(x) \cup B(x)^*,
\end{equation}
where the jump contour orientation is taken downwards from $\ii A(x)$ to $-\ii A(x)$.

\subsubsection{Airy Parametrix at $\SP = \ii A(x)$}
We need local models in neighborhoods of the band endpoints $\pm\ii A(x)$. 
Provided that $|x -x_0|> 0$, where $x_0$ is the maximizer of $A$, the appropriate model is constructed from Airy functions.  See \cite[Appendix B]{BothnerM19} for a complete derivation of this model with a slightly different normalization.
The function $y = \Ai(u)$ is the unique solution of the differential equation $y''(u) = u y(u)$ such that 
\begin{equation}
	\Ai(u) = \frac{ \ee^{-2u^{3/2}/3}}{2u^{1/4}\sqrt{\pi}} \left[ 1- \frac{5}{48} u^{-3/2} + \bigo{u^{-3}} \right],  \qquad u \to \infty, \quad | \arg(u) | < \pi.
	\label{eq:Airy-asymp}
\end{equation}
It is an entire function, and it satisfies the identity
\begin{equation}
	\Ai(u) + \omega \Ai(\omega u) + \omega^2 \Ai (\omega^2 u) = 0, \qquad \omega := \ee^{2\ii \pi/3}.
	\label{eq:Airy-connection}
\end{equation}
Its derivative, $\Ai'(u)$, satisfies
\begin{equation}
	\Ai'(u) = - \frac{ u^{1/4} \ee^{-2u^{3/2}/3}}{2\sqrt{\pi}} \left[ 1 + \frac{7}{48} u^{-3/2} + \bigo{u^{-3}} \right], \qquad u \to \infty, \quad | \arg(u) | < \pi.
	\label{eq:Airy-prime-asymp}
\end{equation}
Setting $u:= \left(\frac{3}{2}\right)^{2/3} z$, we define a matrix function $\mathbf{A}(z)$ as follows:
\begin{equation}
	\mathbf{A}(z) := \begin{dcases}
		\sqrt{2\pi} \left( \tfrac{4}{3} \right)^{\frac{\sigma_3}{6}} 
		\begin{pmatrix} -\Ai'(u) & \omega \Ai'(\omega^2 u) \\ -\ii \Ai(u) &\ii \omega^2 \Ai ( \omega^2 u) \end{pmatrix} 
		\ee^{\frac{\ii}{4} \pi \sigma_3} \ee^{\frac{2}{3}u^{3/2} \sigma_3},
		& \arg(z) \in (0, \tfrac{2\pi}{3}), \\
		\sqrt{2\pi} \left( \tfrac{4}{3} \right)^{\frac{\sigma_3}{6}} 
		\begin{pmatrix} \omega^2 \Ai'(\omega u) & \omega \Ai'(\omega^2 u) \\ \ii \omega \Ai(\omega u) & \ii \omega^2 \Ai ( \omega^2 u) \end{pmatrix} 
		\ee^{\frac{\ii}{4} \pi \sigma_3} \ee^{\frac{2}{3}u^{3/2} \sigma_3},
		& \arg(z) \in (\tfrac{2\pi}{3}, \pi), \\
		\sqrt{2\pi} \left( \tfrac{4}{3} \right)^{\frac{\sigma_3}{6}} 
		\begin{pmatrix} \omega \Ai'(\omega^2 u) & -\omega^2 \Ai'(\omega u) \\ \ii \omega^2 \Ai(\omega^2 u) & -\ii \omega \Ai ( \omega u) \end{pmatrix} 
		\ee^{\frac{\ii}{4} \pi \sigma_3} \ee^{\frac{2}{3}u^{3/2} \sigma_3},
		& \arg(z) \in (-\pi,-\tfrac{2\pi}{3}), \\
		\sqrt{2\pi} \left( \tfrac{4}{3} \right)^{\frac{\sigma_3}{6}} 
		\begin{pmatrix} - \Ai'( u) & -\omega^2 \Ai'(\omega u) \\ -\ii  \Ai( u) & -\ii \omega \Ai ( \omega u) \end{pmatrix} 
		\ee^{\frac{\ii}{4} \pi \sigma_3} \ee^{\frac{2}{3}u^{3/2} \sigma_3},
		& \arg(z) \in (-\tfrac{2\pi}{3} , 0).
	\end{dcases}
\end{equation}
From \eqref{eq:Airy-asymp} and \eqref{eq:Airy-prime-asymp} we get that
\begin{equation}
\mathbf{A}(z) \boldsymbol{\mathcal{E}} z^{-\sigma_3/4}=\mathbb{I}+\begin{pmatrix}\bigo{z^{-3}} & \bigo{z^{-1}}\\\bigo{z^{-2}} & \bigo{z^{-3}}\end{pmatrix},\quad z\to\infty
\label{eq:Airy-parametrix-norm}
\end{equation}
uniformly in all directions of the complex plane, where $\boldsymbol{\mathcal{E}}$ is defined by \eqref{eq:E-matrix-def}.
The matrix function $\mathbf{A}(z)$ is analytic in $z$ with the exception of the four rays $\arg(\pm z)=0$ and $\arg(z)=\pm 2\pi/3$, along which  $\mathbf{A}(z)$ takes continuous boundary values from either side, but  across which it experiences jump discontinuities.  Taking $\arg(-z)=0$ to be oriented away from the origin and all other rays to be oriented toward the origin, the use of \eqref{eq:Airy-connection} yields the following jump conditions relating the boundary values of $\mathbf{A}(z)$:
\begin{equation}
\begin{aligned}
	&\mathbf{A}_+(z)=\mathbf{A}_-(z)\begin{pmatrix}1 & \ii \ee^{-z^{3/2}}\\0 & 1\end{pmatrix}, &&\arg(z)=0, \\
	&\mathbf{A}_+(z)=\mathbf{A}_-(z)\, \ii \sigma_1, &&\arg(-z)=0, \\
	&\mathbf{A}_+(z)=\mathbf{A}_-(z)\begin{pmatrix}1 & 0\\ \ii \ee^{z^{3/2}} & 1\end{pmatrix}, &&\arg(z)=\pm\frac{2\pi}{3}.
\label{eq:Airy-jumps}
\end{aligned}
\end{equation}
These jump conditions along with \eqref{eq:Airy-parametrix-norm} and a Liouville argument imply that $\det(\mathbf{A}(z))=1$.  They capture the behavior of the exponential factors in the jump conditions satisfied by $\mathbf{S}(\SP;x)$ in a neighborhood of $\SP= \ii A(x)$. 

To exactly match the jump conditions \eqref{eq:jump+++S} in a neighborhood $\mathcal{U}_{\ii A(x)}$ of $\SP =\ii A(x)$, we first introduce a local coordinate $W(\SP)$. According to property H1 of Proposition~\ref{prop:h-properties}, the map $\SP \mapsto W(\SP;x):= (2h(\SP;x))^{2/3}$ is conformal on any sufficiently small neighborhood $\mathcal{U}_{\ii A(x)}$ of $\SP = \ii A(x)$ with conformal image a neighborhood of $W =0$ such that $W(\Sigma_0 \cap \mathcal{U}_{\ii A(x)};x)  \subset \R$ with $W(B(x) \cap \mathcal{U}_{\ii A(x)};x) = [-w,0]$ for some $w>0$. Locally, we deform $\Sigma_{L,\pm}$ such that $\arg (W(\Sigma_{L,\pm};x)) = \pm 2\pi/3$, and then we define the local model in the form
\begin{equation}
	\breve{\mathbf{S}}^{\mathrm{Airy}}(\SP;x) := \mathbf{H}^\eps(\SP;x) \mathbf{A}\left( \frac{W(\lambda;x)}{\epsilon^{2/3}}
	\right) \ii\sigma_2 T_\eps(\SP)^{\sigma_3/2}, 
	\qquad \SP \in \mathcal{U}_{\ii A(x)},
\end{equation}
where $\mathbf{H}^\eps(\SP;x)$ is a unit-determinant holomorphic function for $\SP \in \mathcal{U}_{\ii A(x)}$ used to match the local model to the outer model for $\SP \in \partial \mathcal{U}_{\ii A(x)}$.

To define $\mathbf{H}^\eps(\SP;x)$, observe that 
\begin{equation}
	\mathbf{C}_{\ii A(x)} := \left( \frac{W(\SP;x)}{\eps^{2/3}}
	\right)^{\sigma_3/4} \boldsymbol{\mathcal{E}} \ii\sigma_2
\end{equation}
is analytic with unit determinant for $\SP \in \mathcal{U}_{\ii A(x)} \setminus B(x)$, and satisfies the same jump condition as $\breve{\mathbf{S}}^{\mathrm{out}}(\SP;x)$ for $\SP \in B(x) \cap \mathcal{U}_{\ii A(x)}$. 
Let
\begin{equation}
	\mathbf{H}^\eps(\SP;x) := \breve{\mathbf{S}}^{\mathrm{out}}(\SP;x) \mathbf{C}_{\ii A(x)}(\SP)^{-1}, 
	\qquad \SP \in \mathcal{U}_{\ii A(x)},
\end{equation}
then $\mathbf{H}^\eps(\SP;x)$ has no jump on $ B(x) \cap \mathcal{U}_{\ii A(x)}$ and admits a holomorphic extension to $\mathcal{U}_{\ii A(x)}$.

\begin{prop}
The matrix $\breve{\mathbf{S}}_{\mathrm{Airy}}(\SP;x)$ is analytic in $\mathcal{U}_{\ii A(x)} \setminus \Sigma'$ with continuous boundary values on $\mathcal{U}_{\ii A(x)} \cap \Sigma'$ which  match the jump conditions of $\mathbf{S}(\SP;x)$ exactly. Finally, 
\begin{equation}
	\breve{\mathbf{S}}^{\mathrm{Airy}}(\SP;x) \breve{\mathbf{S}}^{\mathrm{out}}(\SP;x)^{-1} = \mathbb{I } + \bigo{\eps^{1/2}}, 
\end{equation}
which holds uniformly for $\SP \in \partial \mathcal{U}_{\ii A(x)}$ and $x \in J_c^+$.
\label{prop:Airy.match}
\end{prop}

\subsubsection{Parametrix at the origin}
The outer model is not uniformly accurate near the origin, so a local model is needed. Let $\mathcal{U}_0$ denote a disk of sufficiently small radius, independent of $\eps$, centered at $\SP =0$. 
Motivated by the local approximations of $Y_\epsilon$ and $T_\epsilon$ for $\SP$ near 0 in Propositions~\ref{prop:Y-zeta-small} and \ref{prop:T-zeta-small}, introduce a rescaled local coordinate on $\mathcal{U}_0$ by first defining
\begin{equation}
	Z= Z(\SP) :=  \frac{\ii}{\sqrt{\eps \pi}} \left( \phaseint(0) - \phaseint(\SP) \right)^{1/2},\quad\lambda=\ii s,\quad 0<s<A_\mathrm{max},
	\label{eq:Z0-coord}
\end{equation}
where the square root is positive. It then follows from Proposition~\ref{eq:Psi-series} that this formula admits continuation to $\mathcal{U}_0$, defining $Z$ as an odd analytic function of $\SP$ with the property that $Z(0)=0$ and $Z'(0) > 0$. It follows that $\SP \mapsto Z$ is a conformal map on the disk $\mathcal{U}_0$ of sufficiently small radius independent of $\eps>0$. 

Next, noting that $h^+$ is analytic in $\mathcal{U}_0 \setminus (B(x) \cup B(x)^*)$ and that $h^+_+(\SP;x) + h^+_-(\SP;x) = 0$ for $\lambda \in B(x) \cup B(x)^*$, the function 
\begin{equation}
	\hat h^+(\SP;x) = \sgn (\Re\{ \SP \}) \cdot h^+(\SP;x)
	\label{eq:h-hat}
\end{equation}
has no jump on the band, and thus extends to an analytic function in $\mathcal{U}_0$. Define for $\lambda\in\mathcal{U}_0$ the functions
\begin{equation}
\begin{aligned}
	m^+(\SP; x) &:= \frac{2}{\eps}\left[ \hat h^+(\SP;x) - \hat h^+(0;x) \right], \\
	n^+(\SP; x) &:= \frac{2}{\eps}\left[ \ii \phaseint(\SP) - \ii \phaseint(0) - \hat h^+(\SP;x) + \hat h^+(0;x) \right]. \\
\end{aligned}
\end{equation}
According to properties H4 and H5 of Proposition~\ref{prop:h-properties} these functions are even, analytic functions of $\SP$, vanishing at $\SP=0$, such that the coefficient of $\SP^2$ in their Taylor expansion at $\SP=0$ is strictly positive imaginary. Since the coordinate map $Z(\SP)$ is an odd, univalent function of $\SP$, we can view $m^+(\SP; x)$ and $n^+(\SP; x)$ as even functions of $Z$ analytic at $Z=0$  with positive imaginary coefficients of $Z^2$ in their respective Taylor expansions:
\begin{equation}
\begin{aligned}
	m^+(\SP; x) &= \ii \mu^+(x) Z(\SP)^2 + \bigo{\eps Z(\SP)^4}, & \mu^+(x) \geq c, \\
	n^+(\SP; x) &= \ii \nu^+(x) Z(\SP)^2 + \bigo{\eps Z(\SP)^4}, & \nu^+(x) \geq c,
\end{aligned}
\label{eq:m-n-expansions}
\end{equation}
where $c >0$ is a constant depending on $J^+_c$.

Next, we define a function which is a local solution of the jump condition satisfied by the outer parametrix for $\SP$ near $0$. Let
\begin{equation}
	\mathbf{C}_0(\SP) := \begin{cases} \ii \sigma_1, & \Re\{\SP\} < 0 \\ \hfill \mathbb{I}, & \Re\{\SP\} > 0. \end{cases}
	\label{eq:C0}
\end{equation}
so that the function 
\begin{equation}
	\mathbf{H}_0(\SP;x) = \dot{\mathbf{S}}^{\mathrm{out}}(\SP;x) \mathbf{C}_0(\SP)^{-1}
	\label{eq:H0}
\end{equation}
has unit determinant, is independent of $\eps$, and is analytic for $\SP \in \mathcal{U}_0$. 

Then if we locally define a function $\mathbf{W}(\SP;x)$ by
\begin{equation}
	\mathbf{W}(\SP;x) = \mathbf{S}(\SP;x) \mathbf{C}_0(\SP)^{-1} \ee^{\hat{h}^+(0;x) \sigma_3/\epsilon},  \quad \SP \in \mathcal{U}_0,
\end{equation}
then $\mathbf{W}(\SP;x)$ is analytic in $\mathcal{U}_0 \setminus (\Sigma' \cap \mathcal{U}_0)$ where it satisfies the jump relation
\begin{equation}
\begin{gathered}
	\mathbf{W}_+(\SP;x) = \mathbf{W}_-(\SP;x) \mathbf{V^W}(\SP;x), \qquad \SP \in \Sigma' \cap \mathcal{U}_0, \\
	 \mathbf{V^W}(\SP;x) := \begin{dcases}
		\tril{-\ii Y_\eps(\SP) \ee^{n^+(\SP;x)}}, & \SP \in \Sigma^+ \cap \mathcal{U}_0, \\
		\triu{-\ii T_\eps(\SP) \ee^{m^+(\SP,x)}}, & \SP \in \Sigma^+_L \cap \mathcal{U}_0, \\	
		\diag{T_\eps(\SP)}{T_\eps(\SP)^{-1}}, & \SP \in B(x) \cap \mathcal{U}_0, \\	
		\tril{ -\ii T_\eps(\SP) \ee^{-m^+(\SP,x)}}, & \SP \in \Sigma^-_L \cap \mathcal{U}_0, \\	
		\triu{-\ii Y_\eps(\SP) \ee^{-n^+(\SP;x)}}, & \SP \in \Sigma^- \cap \mathcal{U}_0, \\
		\mathbf{V^W}(\SP^*;x)^\dagger, & \SP \in \Sigma' \cap \mathcal{U}_0 \cap \C^-.
	 \end{dcases}
\end{gathered}
\label{eq:Vw-def}
\end{equation}

To construct a local model $\breve{\mathbf{S}}^0$ inside $\mathcal{U}_0$, we make two approximations: first, we replace all instances of $Y_\eps(\SP)$ and $T_\eps(\SP)$ in the jump condition $\mathbf{V^S}$ defined by \eqref{eq:jump+++S} with their local approximations $\mathcal{Y}_0(Z(\lambda))$ and $\mathcal{T}_0(Z(\lambda))$ defined by \eqref{eq:Y0-def} and \eqref{eq:T0-def} respectively; second we rewrite the exponential phases in $\mathbf{V^S}$ in terms of the functions $m^+(\SP;x)$ and $n^+(\SP;x)$ and then replace these by the leading-order terms in their Taylor expansions at $\SP =0$ given in \eqref{eq:m-n-expansions}. 
For concreteness, we locally deform the contours $\Sigma' \cap \mathcal{U}_0$ as necessary so that the images are straight line segments with 
\begin{equation}
	\arg(Z( B(x) \cap \mathcal{U}_0)) = \tfrac{\pi}{2}, \quad
	\arg(Z( \Sigma^\pm \cap \mathcal{U}_0))= \tfrac{\pi}{2} \mp \tfrac{\pi}{3}, \quad
	\arg(Z( \Sigma_L^\pm \cap \mathcal{U}_0)) = \tfrac{\pi}{2} \mp \tfrac{\pi}{6}, 
	\label{eq:Z'-image}
\end{equation}	
and preserve the symmetry $Z(\Sigma')^* = Z(\Sigma')$ to define the images in $\C^-$. 
Let $\Sigma^\circledast$ denote the contour in the $Z$-plane consisting of six infinite rays with angles $\pm \arg(Z) \in \{\pi/6, \pi/3, \pi/2\}$ taking all rays oriented outward from the origin.

We then define 
\begin{equation}
	\breve{\mathbf{S}}^{0}(\SP;x) := \mathbf{H}_0(\SP ; x) \ee^{\hat{h}^+(0;x) \sigma_3/\eps} \breve{\mathbf{W}}( Z(\SP) ; \mu^+(x),\nu^+(x)) \ee^{-\hat{h}^+(0;x) \sigma_3/\eps} \mathbf{C}_0(\SP), 
	\quad \SP \in \mathcal{U}_0,
	\label{eq:S0toW0}
\end{equation}
where $\breve{\mathbf{W}}( Z ; \mu, \nu)$ satisfies the following problem:

\begin{myrhp}
Given parameters $\mu>0$ and $\nu>0$, seek a $2\times 2$ matrix function $\breve{\mathbf{W}}(Z)=\breve{\mathbf{W}}(Z;\mu,\nu)$ with the following properties.
\begin{itemize}
\item[]\textit{\textbf{Analyticity:}}  $\breve{\mathbf{W}}(Z)$ is analytic for $Z\in \mathbb{C}\setminus\Sigma^\circledast$ and satisfies the Schwarz symmetry condition $\breve{\mathbf{W}}(Z^*)=\breve{\mathbf{W}}(Z)^{-\dagger}$.
\item[]\textit{\textbf{Jump conditions:}} $\breve{\mathbf{W}}(Z)$ takes continuous boundary values on $\Sigma^\circledast$ from each maximal connected component of $\mathbb{C}\setminus\Sigma^\circledast$.  Given a point $Z$ on one of the oriented arcs of $\Sigma^\circledast$, let the boundary value taken at $Z$ by $\breve{\mathbf{W}}$ from the left (respectively, right) be denoted $\breve{\mathbf{W}}_+(Z)$ (respectively, $\breve{\mathbf{W}}_-(Z)$).  Then 
\begin{equation}
\breve{\mathbf{W}}_+(Z)=\breve{\mathbf{W}}_-(Z)\mathbf{V}^{\breve{\mathbf{W}}}(Z),
\end{equation}
where the jump matrix $\mathbf{V}^{\breve{\mathbf{W}}}(Z)=\mathbf{V}^{\breve{\mathbf{W}}}(Z; \mu, \nu)$ is defined on the various arcs of $\Sigma^\circledast$ by 
\begin{equation}
\mathbf{V}^{\breve{\mathbf{W}}}(Z;\mu,\nu):= \begin{dcases}
	\begin{pmatrix} 1 & 0 \\ -\ii \mathcal{Y}_0(Z) \ee^{\ii \nu Z^2} & 1 \end{pmatrix}, 
	 & \arg(Z) = \frac{\pi}{6}, \\
	\begin{pmatrix} 1 &  -\ii \mathcal{T}_0(Z)^{-1} \ee^{\ii\mu Z^2} \\ 0 & 1 \end{pmatrix},
	 & \arg(Z) = \frac{\pi}{3}, \\
	\begin{pmatrix} \mathcal{T}_0(Z)^{-1} & 0 \\ 0& \mathcal{T}_0(Z)  \end{pmatrix}, 
 	 & \arg(Z) =\frac{\pi}{2}, \\
	\begin{pmatrix} 1 & 0 \\ -\ii \mathcal{T}_0(Z)^{-1} \ee^{-\ii\mu Z^2}  & 1 \end{pmatrix},
	 & \arg(Z) = \frac{2\pi}{3}, \\ 
	\begin{pmatrix} 1 & -\ii \mathcal{Y}_0(Z) \ee^{-\ii \nu Z^2}\\ 0&  1 \end{pmatrix}, 
	 & \arg(Z) = \frac{5\pi}{6}, \\ 
	 \quad \mathbf{V}^{\breve{\mathbf{W}}}(Z^*;\mu,\nu)^\dagger & \arg(Z) \in -\frac{k\pi}{6},\  k=1,\dots,5.
\end{dcases}
\label{eq:jump+++W0}
\end{equation} 
\item[]\textit{\textbf{Normalization:}} $\breve{\mathbf{W}}(Z)\to\mathbb{I}$ as $Z\to\infty$.
\end{itemize}
\label{rhp:+++W0}
\end{myrhp}

\begin{lem}
Uniformly for $\SP \in \Sigma' \cap \mathcal{U}_0$,
\[
	\mathbf{V^W}(\SP;x)\mathbf{V}^{\breve{\mathbf{W}}} \left( Z(\SP);\mu^+(x), \nu^+(x)\right)^{-1} = \mathbb{I} + \bigo{\eps^{1/2}},
\]
where on the arcs $B(x)\cup B(x)^*$ we first replace $\mathbf{V^W}(\lambda;x)$ by its inverse to compare with outward orientation of all arcs of $\Sigma^\circledast$.
\label{lem:VW-match}
\end{lem}
\begin{proof}
First consider $\SP \in B(x)$; then the statement $\mathbf{V^W}(\SP;x)\mathbf{V}^{\breve{\mathbf{W}}}(Z(\SP); \mu^+(x), \nu^+(x))^{-1} = \mathbb{I} + \bigo{\eps^{1/2}}$ follows for such $\SP$  immediately from Proposition~\ref{prop:T-zeta-small}. The analysis on each of the remaining components of $\Sigma' \cap \mathcal{U}_0$ in the upper half-plane is similar, so we give full details for one case, say $\SP \in \Sigma^+$. Recalling \eqref{eq:Z'-image}, $\arg( Z(\Sigma^+ \cap \mathcal{U}_0)) = \pi/6$ and so comparing \eqref{eq:Vw-def} to \eqref{eq:jump+++W0} we have for any matrix norm $\| \cdot \|$
\begin{gather}
	\left\| \mathbf{V^W}(\SP;x) \mathbf{V}^{\breve{\mathbf{W}}}(Z(\SP);\mu^+(x),\nu^+(x))^{-1} - \mathbb{I} \right\| \leq K \left| e(\SP;x) \right|, \\
	e(\SP;x) :=    Y_\eps(\SP) \ee^{n^+(\SP;x)} - \mathcal{Y}_0(Z(\SP)) \ee^{\ii \nu^+(x) Z(\SP)^2} , 
\shortintertext{for some constant $K >0$ depending on the matrix norm. By Proposition~\ref{prop:Y-zeta-small}, }
		e(\SP;x) =   \mathcal{Y}_0(Z(\SP))  \left[ \left(1+\bigo{\eps^{1/2}} \right) \ee^{n^+(\SP;x)}- \ee^{\ii \nu^+(x) Z(\SP)^2}  \right] ,
		\label{eq:e-lambda-x}
\end{gather}
where we note that $ \mathcal{Y}_0(Z(\SP)) $ and $\ee^{n^+(\SP;x)}$ are uniformly bounded on $\Sigma^+ \cap \mathcal{U}_0$, as follows from Proposition~\ref{prop:Y0-asymp} and property H3 of Proposition~\ref{prop:h-properties}. 
Now consider separately those $\SP \in \mathcal{U}_0$ for which $|\SP| \leq \eps^{3/8}$ and those for which $|\SP| \geq \eps^{3/8}$. 
For $|\SP| \leq \eps^{3/8}$, if follows from \eqref{eq:m-n-expansions} that the difference of exponentials in \eqref{eq:e-lambda-x} is $\bigo{\eps^{1/2}}$.
Conversely, for $\SP \in \Sigma^+$ with  $|\SP| \geq \eps^{3/8}$, then both exponential factors are separately small beyond all orders in $\eps$ as $\eps \downarrow 0$.  We conclude that $e(\SP) = \bigo{\eps^{1/2}}$ holds uniformly for $\SP \in \Sigma^+ \cap \mathcal{U}_0$. Similar arguments apply on the other components of $\Sigma' \setminus B(x)$ in the upper-half-plane where the estimate in Proposition~\ref{prop:Y-zeta-small} is replaced by  Proposition~\ref{prop:T-zeta-small} whenever the jump condition depends on $T_\eps(\SP)$ in place of $Y_\eps(\SP)$; similarly property H3 of Proposition~\ref{prop:h-properties} is replaced by property H2 whenever $m^+(\SP;x)$ appears in place of $n^+(\SP;x)$. Once the result is established in the upper half-plane, the symmetries $\mathbf{V^W}(\SP) = \mathbf{V^W}(\SP^*)^\dagger$ and $\mathbf{V}^{\breve{\mathbf{W}}}(\SP) = \mathbf{V}^{\breve{\mathbf{W}}}(\SP^*)^\dagger$ extend the result to the components of $\Sigma' \cap \mathcal{U}_0$ in the lower half-plane. 
\end{proof}

\begin{lem}
	$\mathbf{V}^{\breve{\mathbf{W}}}(Z;\mu,\nu)$ is analytic in all of its arguments along each ray of $\Sigma^\circledast$, and the limiting values as $Z \to 0$ along each ray satisfy the cyclic consistency condition, i.e., their counterclockwise-ordered product is the identity.
\label{lem:W0-consistency}	
\end{lem}

\begin{proof}
The analyticity of the jump matrix $\mathbf{V}^{\breve{\mathbf{W}}}(Z;\mu,\nu)$ on each ray of $\Sigma^\circledast$ in the upper half-plane is obvious from formul\ae\
\eqref{eq:Y0-def}, \eqref{eq:T0-def}, and from the fact that all the exponential factors have the form $\ee^{\pm \ii c Z^2}$ for $c \in \{\mu,\nu\}$. 
To prove the consistency of the jump conditions, we first observe that this condition holds automatically for $\mathbf{V^W}(\SP;x)$ (replaced on the re-oriented arcs $B(x)\cup B(x)^*$ by its inverse) because this jump matrix arose from a sequence of explicit sectionally analytic substitutions continuous up to the boundary of each component of $\C \setminus \Sigma'$, each with unit determinant, applied to the matrix $\widetilde{\mathbf{M}}(\SP;x,\mathbf{0})$ (cf. Riemann-Hilbert problem~\ref{rhp-meromorphic}) which by definition is analytic at $\SP =0$. To this we add the fact that the Taylor approximations \eqref{eq:m-n-expansions} are exact in the limit $Z \to 0$, and, according to Propositions \ref{prop:Y-zeta-small} and \ref{prop:T-zeta-small}, the functions $Y_\eps(\SP)$ and $T_\eps(\SP)$ agree in the limit $\SP \to 0$ along the arcs of $\Sigma'$ with the approximations $\mathcal{Y}_0(Z)$ and $\mathcal{T}_0(Z)$ in the corresponding limit $Z \to 0$ along $\Sigma^\circledast =Z (\Sigma' \cap \mathcal{U}_0)$. Therefore the limiting value of $\mathbf{V}^{\breve{\mathbf{W}}}(Z;\mu, \nu)$ as $Z \to 0$ along a given ray of $\Sigma^\circledast$ agrees exactly with that of $\mathbf{V^W}(\SP;x)$ as $\SP \to 0$ along the corresponding arc of $\Sigma'$ (the limits $\lim_{\SP \to 0} \mathbf{V^W}(\SP;x)$ and $\lim_{Z \to 0} \mathbf{V}^{\breve{\mathbf{W}}}(Z;\mu,\nu)$ along any component of $\Sigma'$ and $\Sigma^\circledast$ are independent of $x$ and $(\mu,\nu)$ respectively). 
\end{proof}

The model $\breve{\mathbf{W}}(Z;\mu,\nu)$ is independent of the dispersion parameter $\eps$ and its jump matrix $\mathbf{V}^{\breve{\mathbf{W}}}(Z;\mu,\nu)$ decays exponentially to the identity as $Z \to \infty$ along each each ray of $\Sigma^\circledast \setminus \ii\R$. The jump along the imaginary axis also decays to identity, but only algebraically.  Before proving an existence result for the solution of our model problem, it is useful to first remove the slowly decaying jump along the imaginary axis.  
Using the asymptotic expansion of $\mathcal{T}_0(Z)$ from Proposition~\ref{prop:T0-asymp} write
\begin{equation}
	\widetilde{\mathcal{T}}_0(Z) := \mathcal{T}_0(Z) \ee^{2\ii (1-\sqrt{2}) \zeta(-\frac{1}{2}) Z^{-1}}, 
	\label{eq:T0-tilde-def}
\end{equation}
so that $\widetilde{\mathcal{T}}_0(Z) = 1+ \bigo{Z^{-2}}$ as $Z \to \infty$ with $-\ii Z >0$. 
Let $\mathcal{B}:[0,\infty) \to [0,1]$ be a $C^\infty$ bump function with the property that 
\begin{equation}
	\mathcal{B}(\xi) = 1 \text{\quad for } \xi \leq \tfrac{1}{2}
	\qquad \text{and} \qquad 
	\mathcal{B}(\xi) = 0 \text{\quad for } \xi\geq 1
\end{equation}
and define the function 
\begin{align}
	D(Z) := \begin{dcases}
	\begin{multlined}[.7\textwidth] \exp \left( \vphantom{\int_0^\infty\frac{2\xi}{\xi^2+Z^2}}\ii (1-\sqrt{2}) \zeta(-\tfrac{1}{2})  \sgn(\Re\{Z\}) Z^{-1} + \right. \\ \left. 
	    \frac{1}{2\pi \ii} \int_{1/2}^\infty   (1-\mathcal{B}(\xi ) ) \log( \widetilde{\mathcal{T}}_0(\ii \xi))  \frac{ 2 \xi}{\xi^2+Z^2}\dd \xi \right),
	\end{multlined} 
	& |Z| > 1, \\		
	 \exp \left( \frac{1}{2\pi \ii} \int_{1/2}^\infty  (1-\mathcal{B}(\xi ) ) \log( \widetilde{\mathcal{T}}_0(\ii \xi))\frac{2 \xi }{\xi^2+Z^2}\dd \xi \right), 
	&|Z| < 1.
	\label{eq:D-def}
\end{dcases}
\end{align}

Let $\partial \mathbb{D}$ denote the unit circle in the $Z$-plane, and let $\partial \mathbb{D}^\pm$ denote the upper and lower unit semi-circles oriented from $Z=-1$ to $Z=1$. Then
\begin{prop}
The function $D(Z)$ is analytic, bounded, nonzero, and satisfies the symmetry $D(Z^*)^* =  D(Z)^{-1}$ for $Z \in \C \setminus (\ii \R \cup \partial \mathbb{D})$. 
For $Z \in \ii \R \cup \partial \mathbb{D}$ it takes continuous boundary values which satisfy the jump relation
\begin{gather}
	D_+(Z) = D_-(Z) v^D(Z) \\
	v^D(Z) = \begin{cases}
		\mathcal{T}_0(Z), & -\ii Z \in (1,\infty), \\
		\widetilde{\mathcal{T}}_0(Z)^{1-\mathcal{B}(|Z|)},  & -\ii Z \in (\tfrac{1}{2},1), \\
		1,  & -\ii Z \in (0, \tfrac{1}{2}), \\
		\ee^{ \ii (1-\sqrt{2}) \zeta(-\frac{1}{2})  \sgn(\Re\{Z\}) Z^{-1}}, & \phantom{-\ii}Z \in \partial \mathbb{D}^+, \\
		v^D(Z^*)^*, & -\ii Z <0 \text{ or }  Z \in \partial \mathbb{D}^-
	\end{cases}
	\label{eq:D-jump}
\end{gather}
and, uniformly for large $Z$, 
\[
	D(Z) = 1+ \bigo{Z^{-1}},  \quad Z \to \infty. 
\]
\label{prop:D}
\end{prop}

\begin{proof}
Using the partial fraction expansion 
$2\xi(\xi^2+Z^2)^{-1}=-\ii (Z-\ii\xi)^{-1} +\ii (Z+\ii\xi)^{-1}$
 and a change of variables one can rewrite the integral term in \eqref{eq:D-def} in the form
\begin{multline}
\frac{1}{2\pi\ii}\int_{1/2}^\infty(1-\mathcal{B}(\xi))\log(\widetilde{\mathcal{T}}_0(\ii\xi))\frac{2\xi}{\xi^2+Z^2}\,\dd\xi \\
=
\frac{1}{2\pi\ii}\int_0^{\ii\infty}(1-\mathcal{B}(|\eta|))\log(\widetilde{\mathcal{T}}_0(\eta))\frac{\dd\eta}{\eta-Z} \\
+\frac{1}{2\pi\ii}\int_0^{-\ii\infty}(1-\mathcal{B}(|\eta|))\log(\widetilde{\mathcal{T}}_0(\eta^*)^*)\frac{\dd\eta}{\eta-Z},
\end{multline}
where in the last term we used the fact that $\widetilde{\mathcal{T}}_0(Z)$ is real-valued on the positive imaginary axis.
Analyticity of $D(Z)$ for $Z \in \C \setminus (\ii \R \cup \partial \mathbb{D})$ and the jump conditions then follow immediately from the Cauchy-Plemelj formula. Boundedness on any compact set follows from the fact that the function $\widetilde{\mathcal{T}}_0(Z)$ is analytic and nonzero on the positive imaginary axis and $\widetilde{\mathcal{T}}_0(Z)= 1 + \bigo{Z^{-2}}$ (cf.\@ Proposition~\ref{prop:T0-asymp})  which implies both the continuity and boundedness of the boundary values for $Z \in \ii \R$. In particular $D(Z)$ is actually analytic for $-\ii Z \in (-1/2,1/2)$. 
The symmetry property $D(Z^*)^* = D(Z)^{-1}$ follows from direct inspection. 
Finally, the decay estimate $D(Z) = 1+ \bigo{Z^{-1}}$ as $Z \to \infty$ follows immediately from the fact that $\log( \widetilde{\mathcal{T}}_0(\ii \xi) ) \in L^1([0,\infty))$ since it decays like $\xi^{-2}$ as $\xi \to \infty$.
\end{proof}

Using $D(Z)$ we transform the model problem. Let 
\begin{equation}
	\mathbf{X}(Z;\mu,\nu) : = \breve{\mathbf{W}}(Z;\mu,\nu) D(Z)^{\sigma_3},
	\label{eq:X0-def}
\end{equation}
then using Proposition~\ref{prop:D} it follows that $\mathbf{X}(Z;\mu,\nu)$ solves the following problem.

\begin{myrhp}
Given $\mu, \nu >0$, seek a $2\times 2$ matrix function $\mathbf{X}(Z)=\mathbf{X}(Z;\mu,\nu)$ with the following properties.
\begin{itemize}
\item[]\textit{\textbf{Analyticity:}}  $\mathbf{X}(Z)$ is analytic for $Z\in \mathbb{C}\setminus (\Sigma^\circledast \cup \partial{\mathbb{D}})$ and satisfies the Schwarz symmetry condition $\mathbf{X}(Z^*)=\mathbf{X}(Z)^{-\dagger}$.
\item[]\textit{\textbf{Jump conditions:}} $\mathbf{X}(Z)$ takes continuous boundary values on $\Sigma^\circledast \cup \partial{\mathbb{D}}$ from each maximal connected component of $\mathbb{C}\setminus (\Sigma^\circledast \cup \partial\mathbb{D})$.  For $Z \in \Sigma^\circledast \cup \partial\mathbb{D}$, denote the boundary value taken from the left (respectively, right) by $\mathbf{X}_+(Z)$ (respectively, $\mathbf{X}_-(Z)$).  Then 
\begin{equation}
	\mathbf{X}_+(Z)=\mathbf{X}_-(Z)\mathbf{V^X}(Z),
\end{equation}
where the jump matrix $\mathbf{V^X}(Z)=\mathbf{V^X}(Z;\mu,\nu)$ is defined on the various arcs of $\Sigma^\circledast \cup \partial \mathbb{D}$ by 
\begin{equation}
\mathbf{V^X}(Z;\mu,\nu):= \begin{dcases}
	\begin{pmatrix} 1 & 0 \\ -\ii \mathcal{Y}_0(Z) D(Z)^{2} \ee^{\ii \nu Z^2} & 1 \end{pmatrix}, 
	 & \arg(Z) = \frac{\pi}{6}, \\
	\begin{pmatrix} 1 &  -\ii \mathcal{T}_0(Z)^{-1} D(Z)^{-2} \ee^{\ii\mu Z^2} \\ 0 & 1 \end{pmatrix},
	 & \arg(Z) = \frac{\pi}{3}, \\
	\left( \frac{\widetilde{\mathcal{T}}_0(Z)^{1-\mathcal{B}(|Z|)}}{\mathcal{T}_0(Z)} \right)^{\sigma_3},
 	 &  -\ii Z \in (0,1) , \\
	\begin{pmatrix} 1 & 0 \\ -\ii \mathcal{T}_0(Z)^{-1} D(Z)^{2} \ee^{-\ii\mu Z^2}  & 1 \end{pmatrix},
	 & \arg(Z) = \frac{2\pi}{3}, \\ 
	\begin{pmatrix} 1 & -\ii \mathcal{Y}_0(Z) D(Z)^{-2} \ee^{-\ii \nu Z^2}\\ 0&  1 \end{pmatrix}, 
	 & \arg(Z) = \frac{5\pi}{6}, \\ 
	 \ee^{\ii (1-\sqrt{2}) \zeta(-\frac{1}{2}) \sgn( \Re \{ Z \} ) Z^{-1} \sigma_3}, & Z \in \partial \mathbb{D}^+, \\
	 \quad \mathbf{V^X}(Z^*;\mu,\nu)^\dagger, & \arg(Z) \in -\frac{k\pi}{6},\  k=1,\dots,5.
\end{dcases}
\label{eq:jump+++X0}
\end{equation}
\item[]\textit{\textbf{Normalization:}} $\mathbf{X}(Z)\to\mathbb{I} + \bigo{Z^{-1}}$ as $Z\to\infty$.
\end{itemize}
\label{rhp:+++X0}
\end{myrhp}

\begin{lem}
	$\mathbf{V^X}(Z;\mu,\nu)$ is analytic in $\mu$ and $\nu$ along each arc of $\Sigma^\circledast \cup \partial \mathbb{D}$, and at each point of self-intersection in the contour $\Sigma^\circledast \cup \partial \mathbb{D}$, the limiting values of the jump along each component contour satisfy the cyclic consistency condition. 
\label{lem:X0-consistency}	
\end{lem}

\begin{proof}	
	Analyticity of $\mathbf{V^X}(Z;\mu,\nu)$ in $\mu$ and $\nu$ is obvious from \eqref{eq:jump+++X0}. Cyclic consistency of the jumps at self-intersection points of the jump contour follows from Lemma~\ref{lem:W0-consistency} and the observation that the function $D(Z)$ defining the transformation \eqref{eq:X0-def} from $\dot{\mathbf{W}} \mapsto \mathbf{X}$ has continuous boundary values as $Z \to \Sigma^\circledast \cup \partial \mathbb{D}$ --- including all points of self intersection --- from each connected component of $\C \setminus (\Sigma^\circledast \cup \partial \mathbb{D})$ which preserves the consistency of the jumps.	
\end{proof}	

\begin{lem}
There exists a unique solution $\mathbf{X}(\SP;\mu,\nu)$ of Riemann-Hilbert Problem \ref{rhp:+++X0} which is uniformly bounded for $(\mu, \nu)$ in any compact subset of $(0,\infty)^2$. 
\end{lem}

\begin{proof} 
Let $\mathcal{S}$ denote the open sector $\mathcal{S} := \{ z\in \C: |\arg(z)| < \pi/3 \}$; the mapping $( \mu,\nu) \mapsto \mathbf{V^X}(\diamond;\mu,\nu)- \mathbb{I}$ is analytic from $\mathcal{S}^2 \to  L^2(\Sigma^\circledast \cup \partial \mathbb{D})\cap L^\infty(\Sigma^\circledast\cup\partial\mathbb{D})$. Cyclic consistency of the jump $\mathbf{V^X}(Z;\mu,\nu)$ at points of self-intersection established in Lemma~\ref{lem:X0-consistency} together with the symmetry $\mathbf{V^X}(Z^*;\mu,\nu) = \mathbf{V^X}(Z;\mu,\nu)^\dagger$ are sufficient to apply Zhou's vanishing lemma \cite{Zhou89} to deduce that the only matrix $\mathbf{X}^0(Z)$ analytic in $\C \setminus (\Sigma^\circledast \cup \partial \mathbb{D})$ solving the homogenous version of the Riemann-Hilbert problem, namely $\mathbf{X}^0_{+}(Z) =\mathbf{X}^0_{-}(Z) \mathbf{V^X}(Z)$ for $Z \in \Sigma^\circledast \cup \partial \mathbb{D}$ and $\mathbf{X}^0(Z) \to \mathbf{0}$ as $Z \to \infty$, is the zero matrix $\mathbf{X}^0(Z) \equiv \mathbf{0}$.
From analytic Fredholm theory it then follows that there exists a unique solution of Riemann-Hilbert Problem~\ref{rhp:+++X0} whose boundary values $\mathbf{X}_\pm(Z;\mu,\nu)$ satisfy the jump condition, and also the normalization condition at $Z = \infty$ in the sense that $\mathbf{X}_\pm(\diamond) - \mathbb{I} \in L^2(\Sigma^\circledast \cup \partial \mathbb{D})$. Moreover these boundary values depend analytically on the parameters $(\mu,\nu) \in \mathcal{S}^2$. 
The solution $\mathbf{X}(Z;\mu,\nu)$ of Riemann-Hilbert problem \ref{rhp:+++X0} is given in terms of its boundary values by 
\begin{equation}
	\mathbf{X}(Z;\mu,\nu) = \mathbb{I} +  \frac{1}{2\pi \ii} \int_{\Sigma^\circledast \cup \partial \mathbb{D}} \frac{ \mathbf{X}_-(\eta;\mu,\nu) (\mathbf{V^X}(\eta;\mu,\nu) - \mathbb{I})}{\eta-Z} \dd \eta.
\end{equation}
Observing that $\mathbf{V^X}(\eta;\mu,\nu) - \mathbb{I}$ decays exponentially to zero along each unbounded component of $\Sigma^\circledast$ for each $(\mu,\nu) \in \mathcal{S}^2$, the factor $(\eta-Z)^{-1}$ can be expanded geometrically for $Z \to \infty$ yielding an asymptotic expansion for $\mathbf{X}(Z;\mu,\nu)$ in powers of $Z^{-1}$.  In particular, $\mathbf{X}(Z;\mu,\nu) = \mathbb{I}+\bigo{Z^{-1}}$ as $Z \to \infty$. Cyclic consistency of the jump matrix $\mathbf{V^X}(Z)$ at points of self-intersection and $C^\infty$ smoothness on each arc of $\Sigma^\circledast \cup \partial \mathbb{D}$ imply that $\mathbf{X}(Z)$ has bounded and continuous boundary values $\mathbf{X}_\pm(Z;\mu,\nu)$. Finally, analyticity of the solution $\mathbf{X}(Z;\mu,\nu)$ for $(\mu,\nu) \in \mathcal{S}^2$ implies uniform boundedness of the solution for parameters $(\mu,\nu)$ in any compact subset of $\mathcal{S}^2$, and hence by restriction to positive real values, in any compact subset of $(0,\infty)^2$.
\end{proof}	

With $\mathbf{X}(Z;\mu,\nu)$ uniquely determined, we then invert the explicit transformations \eqref{eq:S0toW0}, \eqref{eq:X0-def} connecting $\breve{\mathbf{S}}^0(\SP;x)$ to $\mathbf{X}(Z;\mu,\nu)$ to write for $\SP\in\mathcal{U}_0$,
\begin{equation}
	\breve{\mathbf{S}}^0(\SP;x) =  
	\mathbf{H}_0(\SP ; x) \ee^{\hat{h}^+(0;x) \sigma_3/\eps} \mathbf{X}( Z(\SP) ; \mu^+(x), \nu^+(x)) D(Z(\SP))^{-\sigma_3} \ee^{-\hat{h}^+(0;x) \sigma_3/\eps} \mathbf{C}_0(\SP),
	\label{eq:S0fromX0}
\end{equation}
where we recall the definitions in \eqref{eq:Z0-coord}--\eqref{eq:h-hat}, \eqref{eq:C0}--\eqref{eq:H0}, and \eqref{eq:D-def}.

\begin{lem} For $x \in J^+_c$, 
	$\breve{\mathbf{S}}^0(\SP;x)$ is uniformly bounded on $\mathcal{U}_0$ with $\det( \breve{\mathbf{S}}(\SP;x)) = 1$. Also, 
	$\breve{\mathbf{S}}^0(\SP;x) \breve{\mathbf{S}}^{\mathrm{out}}(\SP;x)^{-1} = \mathbb{I} + \bigo{\eps^{1/2}}$ holds uniformly for $\lambda\in\partial \mathcal{U}_0$. Finally,  
	$\mathbf{V}^\mathbf{S}(\SP;x) \mathbf{V}^{\breve{\mathbf{S}}}(\SP;x)^{-1} = \mathbb{I} + \bigo{\eps^{1/2}}$ holds uniformly for $\SP \in \Sigma' \cap \mathcal{U}_0$, where $\mathbf{V^S}(\SP;x)$ and $\mathbf{V}^{\breve{\mathbf{S}}}(\SP;x)$ denote the jump matrices for $\mathbf{S}(\lambda;x)$ and $\breve{\mathbf{S}}^0(\lambda;x)$ respectively. 
\label{lem:approximate-jumps-active}
\end{lem}

\begin{proof}
The fact that $\breve{\mathbf{S}}^0(\SP;x)$ is bounded with unit determinant follows from the corresponding properties of $\mathbf{X}(Z;\mu,\nu)$ and the explicit formula \eqref{eq:S0fromX0}. Then, using \eqref{eq:H0}, we have
\begin{multline}
	\breve{\mathbf{S}}^0(\SP;x) \breve{\mathbf{S}}^{\mathrm{out}}(\SP;x)^{-1} =  \\
	\mathbf{H}_0(\SP ; x) \ee^{\hat{h}^+(0;x) \sigma_3/\eps} \mathbf{X}( Z(\SP) ; \mu^+(x), \nu^+(x)) D(Z(\SP))^{-\sigma_3} \ee^{-\hat{h}^+(0;x) \sigma_3/\eps} \mathbf{H}_0(\SP ; x)^{-1}.
\end{multline}
Combining the large-$Z$ asymptotic behavior of $D(Z)$ and $\mathbf{X}(Z; \mu^+(x), \nu^+(x))$ described in Proposition~\ref{prop:D} and Riemann-Hilbert Problem~\ref{rhp:+++X0} with the local coordinate rescaling in \eqref{eq:Z0-coord}, it follows that $\mathbf{X}( Z(\SP) ; \mu^+(x), \nu^+(x)) D(Z(\SP))^{-\sigma_3} = \mathbb{I} + \bigo{\eps^{1/2}}$ uniformly for $\SP \in \partial \mathcal{U}_0$. Observing that $\hat{h}^+(0) = h_+^+(0) \in \ii \R$, as follows from property H4 in Proposition~\ref{prop:h-properties}, and that $\mathbf{H}_0(\SP;x)$ is analytic on the closure of $\mathcal{U}_0$ and is independent of $\eps$, it follows that $\breve{\mathbf{S}}^0(\SP;x) \breve{\mathbf{S}}^{\mathrm{out}}(\SP;x)^{-1} = \mathbb{I} + \bigo{\eps^{1/2}}$ holds as $\eps \downarrow 0$ uniformly for $\SP \in \partial \mathcal{U}_0$. 

To estimate $\mathbf{V}^\mathbf{S}(\SP;x) \mathbf{V}^{\breve{\mathbf{S}}}(\SP;x)^{-1} - \mathbb{I}$ for $\SP \in \Sigma' \cap \mathcal{U}_0$, we observe that 
\begin{equation}
	\begin{aligned}
	\mathbf{V}^{\breve{\mathbf{S}}}(\SP;x) &= \mathbf{C}_0(\SP)^{-1} \ee^{\hat{h}^+(0) \sigma_3/\eps} \mathbf{V}^{\breve{\mathbf{W}}}(Z(\SP); \mu^+(x), \nu^+(x)) \ee^{-\hat{h}^+(0) \sigma_3/\eps}\mathbf{C}_0(\SP), \\
	\mathbf{V^S}(\SP;x) &= \mathbf{C}_0(\SP)^{-1} \ee^{\hat{h}^+(0) \sigma_3/\eps} \mathbf{V^W}(\SP;x) \ee^{-\hat{h}^+(0) \sigma_3/\eps}\mathbf{C}_0(\SP). 
\end{aligned}
\end{equation}
Since the conjugating factors are all bounded with bounded inverses in $\mathcal{U}_0$, the result follows from Lemma~\ref{lem:VW-match}.
\end{proof}

\subsubsection{Global Parametrix}
For each $x \in J^+$, we can now define a parametrix for $\mathbf{S}(\SP;x)$ as follows:
\begin{equation}
	\breve{\mathbf{S}}(\SP;x) := \begin{dcases}
	\breve{\mathbf{S}}^0(\SP;x), & \SP \in \mathcal{U}_0, \\
	\breve{\mathbf{S}}^\text{Airy}(\SP;x), & \SP \in \mathcal{U}_{\ii A(x)}, \\
	\breve{\mathbf{S}}^\text{Airy}(\SP^*;x)^{-\dagger}, & \SP \in \mathcal{U}_{\ii A(x)}^*, \\
	\breve{\mathbf{S}}^{\mathrm{out}}(\SP;x), & \text{ elsewhere}.
	\end{dcases}
\label{eq:global-parametrix}
\end{equation}

%% file: sec-parametrix-accuracy.tex
To compare the unknown solution $\mathbf{S}(\SP;x)$ of Riemann-Hilbert Problem~\ref{rhp:+++S}, directly constructed from the solution $\widetilde{\mathbf{M}}(\SP;x,0)$ of the soliton ensemble Riemann-Hilbert Problem~\ref{rhp:M-sigma} with our explicitly constructed parametrix $\dot{\mathbf{S}}(\SP;x)$ we consider the error defined by the ratio
\begin{equation}
	\mathbf{E}(\SP;x) := \mathbf{S}(\SP;x) \breve{\mathbf{S}}(\SP;x)^{-1}
	\label{eq:Edef}
\end{equation}
in each region of the complex $\SP$-plane in which both factors of the right-hand side are defined and analytic. Then  $\mathbf{E}(\SP;x)$ is analytic for $\SP \in \Sigma''$, where $\Sigma'' = \Sigma' \cup \partial \mathcal{U}_0 \cup \partial \mathcal{U}_{\ii A(x)} \cup \partial \mathcal{U}^*_{\ii A(x)} $. The following result characterizes the jump matrix of $\mathbf{E}(\SP;x)$ on $\Sigma''$.

\begin{lem} For $x \in J_c^+$, the estimate 
\begin{equation}
	\mathbf{E}_+(\SP;x) = \mathbf{E}_-(\SP;x) \left[ \mathbb{I} + \bigo{\eps^{1/2}} \right], \qquad \SP \in \Sigma''
	\label{eq:error.def}
\end{equation}
holds uniformly as $\eps \downarrow 0$.
\label{lem:error-accuracy}
\end{lem}

\begin{proof}
First consider $\SP$ on the boundaries of the neighborhoods $\mathcal{U}_0$, $\mathcal{U}_{\ii A(x)}$, and $\mathcal{U}_{\ii A(x)}^*$, which is in each case a simple closed curve. Across these curves, $\mathbf{S}(\SP;x)$ has no jump, so $\mathbf{E}_+(\SP;x) = \mathbf{E}_-(\SP;x) \breve{\mathbf{S}}_-(\SP;x) \breve{\mathbf{S}}_+(\SP;x)^{-1}$. Taking the curves $\partial \mathcal{U}_0$, $\partial \mathcal{U}_{\ii A(x)},\ \partial \mathcal{U}_{\ii A(x)}^*$ to have clockwise orientation, $\breve{\mathbf{S}}_+(\SP;x) = \breve{\mathbf{S}}^{\mathrm{out}}(\SP;x)$ in each case, while $\breve{\mathbf{S}}_-(\SP;x) = \breve {\mathbf{S}}^0(\SP;x)$ for $\SP \in \partial \mathcal{U}_0$, $\breve{\mathbf{S}}_-(\SP;x) = \breve{ \mathbf{S}}^{\mathrm{Airy}}(\SP;x)$ for $\SP \in \partial \mathcal{U}_{\ii A(x)}$, and  $\breve{\mathbf{S}}_-(\SP;x) = \breve{\mathbf{S}}^{\mathrm{Airy}}(\SP^*;x)^{-\dagger}$ for $\SP \in \partial \mathcal{U}_{\ii A(x)}^*$.  The bound 
$\mathbf{E}_+(\SP;x) = \mathbf{E}_-(\SP;x) \left[ \mathbb{I} + \bigo{\eps^{1/2}} \right]$ then follows from Proposition~\ref{prop:Airy.match} for $\lambda \in \partial \mathcal{U}_{\ii A(x)}$ and the same estimate holds on $\partial \mathcal{U}_{\ii A(x)}^*$ by Schwarz symmetry of the matrix factors in \eqref{eq:error.def}. Similarly, by Lemma~\ref{lem:approximate-jumps-active}, we have $\mathbf{E}_+(\SP;x) = \mathbf{E}_-(\SP;x) \left[ \mathbb{I} + \bigo{\eps^{1/2}} \right]$ for $\SP \in \partial \mathcal{U}_0$. 

The matrix $\mathbf{S}(\SP;x)$ has jump discontinuities on all of the other component contours of $\Sigma''$. If we write $\mathbf{S}_+(\SP;x) = \mathbf{S}_-(\SP;x) \mathbf{V}^\mathbf{S}(\SP;x)$ and $\breve{\mathbf{S}}_+(\SP;x) = \breve{\mathbf{S}}_-(\SP;x) \mathbf{V}^{\breve{\mathbf{S}}}(\SP;x)$ for $\SP$ on any of these arcs, then 
\begin{equation}
	\mathbf{E}_+(\SP;x) = \mathbf{E}_-(\SP;x) \breve{\mathbf{S}}_-(\SP;x) \mathbf{V}^\mathbf{S}(\SP;x) \mathbf{V}^{\breve{\mathbf{S}}}(\SP;x)^{-1}\breve{\mathbf{S}}_-(\SP;x)^{-1}.
\end{equation}	
For $\SP \in \Sigma'' \cap \left( \mathcal{U}_{\ii A(x)} \cup \mathcal{U}_{\ii A(x)}^* \right)$, it follows immediately from Proposition~\ref{prop:Airy.match} that $\mathbf{E}_+(\SP;x) = \mathbf{E}_-(\SP;x)$. In the remainder of $\Sigma''$ either:  $\breve{\mathbf{S}}(\SP;x) = \breve{\mathbf{S}}^{\mathrm{out}}(\SP;x)$, which is bounded with bounded inverse for $\SP \not\in \left(\mathcal{U}_{\ii A(x)} \cup \mathcal{U}_{\ii A(x)}^* \right)$ as follows immediately from \eqref{eq:S.out.matrix}; or, when $\SP \in \mathcal{U}_0$, $\breve{\mathbf{S}}(\SP;x) = \breve{\mathbf{S}}^0(\SP;x)$ which is bounded with bounded inverse by Lemma~\ref{lem:approximate-jumps-active}.  Thus, it will be sufficient to estimate $\mathbf{V}^\mathbf{S}(\SP;x) \mathbf{V}^{\breve{\mathbf{S}}}(\SP;x)^{-1} - \mathbb{I}$ for $\SP$ lying in $\Sigma'$ outside $\mathcal{U}_{\ii A(x)} \cup \mathcal{U}_{\ii A(x)}^*$. Moreover, it is sufficient to consider only $\SP$ with $\Im \{ \SP \} >0$ because the corresponding estimates for $\SP$ in the lower half-plane follow from Schwarz symmetry. 

To analyze $\mathbf{V}^\mathbf{S}(\SP;x) \mathbf{V}^{\breve{\mathbf{S}}}(\SP;x)^{-1} - \mathbb{I}$ on these arcs, first consider $\SP \in \Sigma_0$ outside the neighborhoods $\mathcal{U}_0$ and $\mathcal{U}_{\ii A(x)}$. If $\SP \in B(x)$, then  $\mathbf{V}^\mathbf{S}(\SP;x)$ is given by \eqref{eq:jump+++S} while $\mathbf{V}^{\breve{\mathbf{S}}}(\SP;x)$ is exactly the same with $T_\eps (\lambda)$ replaced by 1. The fact that $\mathbf{V}^\mathbf{S}(\SP;x) \mathbf{V}^{\breve{\mathbf{S}}}(\SP;x)^{-1} - \mathbb{I} = \bigo{\eps^{1/2}}$ holds uniformly for $\lambda \in B(x)$ bounded away from $0$ and $\ii A_{\mathrm{max}}$ follows from Proposition~\ref{prop:T-outside}. 
If,  instead, $\SP \in \Sigma_0 \setminus\ B(x)$ outside of $\mathcal{U}_{\ii A(x)}$, then $\mathbf{V}^{\breve{\mathbf{S}}}(\SP;x)  = \mathbb{I}$, and $\mathbf{V}^\mathbf{S}(\SP;x)$ is given by \eqref{eq:jump+++S}. Then $\mathbf{V}^\mathbf{S}(\SP;x) \mathbf{V}^{\breve{\mathbf{S}}}(\SP;x)^{-1} - \mathbb{I} = \mathbf{V}^\mathbf{S}(\SP;x) - \mathbb{I}$ is uniformly exponentially small as $\eps \downarrow 0$  as follows form property H2 of Proposition~\ref{prop:h-properties}  together with Propositions~\ref{prop:T-outside} and \ref{prop:T-zeta-near-iAmax} to uniformly bound $T_\epsilon(\lambda)$ (to apply Proposition~\ref{prop:T-zeta-near-iAmax}, note that the model function $\mathcal{T}_1(W)$ defined by \eqref{eq:T1-def} is uniformly bounded for $W\le 0$).

Next, consider $\SP$ in the lens contours $\Sigma_\pm$ and $\Sigma_{\pm,L}$ outside the neighborhoods $\mathcal{U}_0$ and $\mathcal{U}_{\ii A(x)}$. On each of these contours $\mathbf{V}^{\breve{\mathbf{S}}}(\SP;x) = \mathbb{I}$, while $\mathbf{V}^\mathbf{S}(\SP;x)$ is given by \eqref{eq:jump+++S}. In each case we have that $\mathbf{V}^\mathbf{S}(\SP;x) \mathbf{V}^{\breve{\mathbf{S}}}(\SP;x)^{-1} - \mathbb{I} = \mathbf{V}^\mathbf{S}(\SP;x) - \mathbb{I}$ is uniformly exponentially small as $\eps \downarrow 0$. 
For $\SP \in \Sigma_{\pm} \setminus \mathcal{U}_0$ this follows from property H3 of Proposition~\ref{prop:h-properties} together with Proposition~\ref{prop:Y-outside} to control $Y_\epsilon(\SP)$ for $\SP$ away from $\ii A_\mathrm{max}$ and Proposition~\ref{prop:Y-zeta-near-iAmax} to handle the case when $\SP$ approaches this point along $\Sigma_\pm$ (noting that the model function $\mathcal{Y}_1(W)$ defined by \eqref{eq:Y1-def} remains bounded for such $\SP$).
For $\SP \in \Sigma_{L,\pm} \setminus (\mathcal{U}_0 \cup \mathcal{U}_{\ii A(x)})$ this follows from property H2 of Proposition~\ref{prop:h-properties} along with Proposition~\ref{prop:T-outside}. 

Finally, consider $\SP \in \Sigma'' \cap \mathcal{U}_0$. The uniform estimate  $\mathbf{V}^\mathbf{S}(\SP;x) \mathbf{V}^{\breve{\mathbf{S}}}(\SP;x)^{-1} - \mathbb{I} = \bigo{\eps^{1/2}}$  now follows directly from Lemma~\ref{lem:approximate-jumps-active} in this case.
\end{proof}

%% file: sec-x-less-than-x_0.tex
In the steepest descent analysis carried out above it was assumed that  $x \in J^+ = (x_0, X_+)$, where $x_0$ is the unique location of the maximum of the initial amplitude of our semicircular initial data, i.e.,  $A(x_0) =A_\mathrm{max}$. The analysis for $x \in J^-=(X_-,x_0)$ is very similar with minor modifications we will outline below. For $x \in J^-$, our starting point is the renormalized meromorphic Riemann-Hilbert Problem~\ref{rhp-meromorphic_renorm} with solution $\widetilde{\mathbf{M}}^\updownarrow(\SP;x,\mathbf{0})$. The first step in our analysis is to interpolate the poles of $\widetilde{\mathbf{M}}^\updownarrow(\SP;x,\mathbf{0})$ to replace it with a sectionally analytic function. Let $K_+ = -1$ and $K_- = 0$. Akin to \eqref{eq:Qdef}, define
\begin{equation}
	\mathbf{Q}^\updownarrow(\SP;x) := \begin{dcases}
		\widetilde{\mathbf{M}}^\updownarrow(\SP;x,\mathbf{0}) \begin{pmatrix} 1 & \ii (-1)^{K_\pm} \widetilde{a}(\SP)^{-1} \ee^{-2 f_{K_\pm}(\SP;x,\mathbf{0})/\eps} \\ 0 & 1 \end{pmatrix},
		& \SP \in \Omega_\pm, \\
	\mathbf{Q}^\updownarrow(\SP^*;x)^{-\dagger}, & \SP \in \Omega_\pm^*, \\
	\widetilde{\mathbf{M}}^\updownarrow(\SP;x,\mathbf{0}), & \text{otherwise},	
	\end{dcases}
\end{equation}
where we remind the reader of the domains $\Omega^\pm$ and contour $\Sigma$ defined in Figure~\ref{fig-Q-jumps} .
The resulting function $\mathbf{Q}^\updownarrow(\SP;x)$ is analytic for $\SP \in \C \setminus \Sigma$ and satisfies a problem like Riemann-Hilbert Problem~\ref{rhp:+++} with the jump \eqref{eq:jump+++1}-\eqref{eq:jump+++2} replaced by 
\begin{gather}
	\mathbf{V}^{\mathbf{Q}^\updownarrow}(\SP;x) = \begin{dcases}
		\begin{pmatrix}
		1 & - \ii T_\eps(\SP) \ee^{-\varphi^-(\SP;x)/\eps} \\ 0 & 1 
		\end{pmatrix}, & \SP \in \Sigma_0,  \\
		\begin{pmatrix} 
		1 & - \ii Y_\eps(\SP) \ee^{-\left[ \varphi^-(\SP;x) \mp 2\ii \phaseint(\SP) \right]/\eps} \\ 0 & 1 
		\end{pmatrix}, & \SP \in \Sigma_\pm, \\
		\mathbf{V}^{\mathbf{Q}^\updownarrow}(\SP^*;x)^\dagger, & \SP \in \Sigma \cap \C^-,
	\end{dcases}
	\shortintertext{where}
	\varphi^-(\SP;x) := \tailint(\SP) +2\ii Q(\SP;x,\mathbf{0}) - \overline{L}(\SP).
\end{gather}

The next step in the analysis is again to introduce a $g$-function.  For $x \in J^-$, we use the function $g^-(\SP;x)$ (cf. \eqref{eq:g-functions-zeta}) to define
\begin{equation}
\mathbf{R}^\updownarrow(\SP;x) := \mathbf{Q}^\updownarrow(\SP;x)  \begin{pmatrix} \ee^{-g^-(\SP;x)/\eps } & 0 \\ 0 &  \ee^{g^-(\SP;x)/\eps }  \end{pmatrix}.
\end{equation}
Just as was the case for $x\in J^+$, for $x \in J^-$ the function $g^-(\SP;x)$ is analytic for $\SP$ away from the band $\overline{B \cup B^*}$ (here $B\equiv B(x)$). Using \eqref{eq:g-zeta-plus-minus} and \eqref{eq:gdiff-h} we have
\begin{equation}
	\mathbf{R}_+^\updownarrow(\SP;x) = \mathbf{R}_-^\updownarrow(\SP;x) 
	\begin{pmatrix} 
		\ee^{2h^-_+(\SP;x)/ \eps} &-\ii T_\eps(\SP) \\ 0 & \ee^{2h^-_-(\SP;x)/ \eps}
	\end{pmatrix},
	\qquad \SP \in B(x),
\end{equation}
where $h^-(\SP;x)$, analytic for $\SP\in \C \setminus \overline{B \cup B^*}$, is defined by \eqref{eq:h-def-zeta}. 
This jump relation can be rewritten as
\begin{equation}
	\mathbf{R}_+^\updownarrow(\SP;x) \mathbf{L}_+^\updownarrow(\SP;x) = \mathbf{R}_-^\updownarrow(\SP;x) \mathbf{L}_-^\updownarrow(\SP;x)^{-1} 
	\begin{pmatrix} 0 & -\ii T_\eps(\SP) \\ -\ii T_\eps(\SP)^{-1} & 0 \end{pmatrix}, \quad \SP \in B(x), 
\end{equation}
where $\mathbf{L}_+^\updownarrow(\SP;x)$ is the matrix analytic for $\SP \in \Omega_+ \cup \Omega_-$ defined by 
\begin{equation}
	 \mathbf{L}^\updownarrow(\SP;x) := \begin{pmatrix} 1 & 0 \\ -\ii T_\eps(\SP)^{-1} \ee^{2h^-(\SP;x)/\eps } & 1 \end{pmatrix}.
\end{equation}
This factorization  motivates a further deformation. 
As before, let $\Lambda_\pm$ be the two domains bounded by $B(x)$ and a pair of parabolic arcs defined in property H3 of Proposition~\ref{prop:h-properties}; denote the parabolic arcs by $\Sigma_{L\pm}$ (cf. Figure~\ref{fig-S-jumps}).
Let
\begin{equation}
	\mathbf{S}^\updownarrow(\SP;x) := \begin{dcases}
		\mathbf{R}^\updownarrow(\SP;x)  \mathbf{L}^\updownarrow(\SP;x)^{\pm 1}, & \SP \in \Lambda_\pm, \\
		\mathbf{S}^\updownarrow(\SP^*;x)^{-\dagger},  & \SP \in \Lambda_+^* \cup \Lambda_-^*, \\
		\mathbf{R}^\updownarrow(\SP;x), & \text{otherwise.}
	\end{dcases}
\end{equation}
Then $\mathbf{S}^\updownarrow(\SP;x)$ is a function analytic in $\C \setminus \Sigma'$ satisfying Riemann-Hilbert Problem~\ref{rhp:+++S}, except that the jump $\mathbf{V^S}(\SP;x)$ is replaced by 
\begin{equation}
	\mathbf{V}^{\mathbf{S}^\updownarrow}(\SP;x) := \begin{dcases}
		\begin{pmatrix} 1 & -\ii T_\epsilon(\SP) \ee^{-2h^-(\SP;x)/\epsilon} \\ 0 & 1 \end{pmatrix}, & \SP \in \Sigma_0 \setminus B(x), \\
		\begin{pmatrix} 0 & -\ii T_\eps(\SP) \\ -\ii T_\eps(\SP)^{-1} & 0 \end{pmatrix}, & \SP \in  B(x), \\
		\begin{pmatrix} 1 & -\ii Y_\eps(\SP) \ee^{-2\left[h^-(\SP;x) \mp \ii \phaseint(\SP) \right]/\eps} \\ 0 & 1 \end{pmatrix}, & \SP \in \Sigma_\pm, \\
		\begin{pmatrix} 1 & 0 \\ -\ii T_\eps(\SP)^{-1} \ee^{2h^-(\SP;x)/\eps } & 1 \end{pmatrix}, & \SP \in \Sigma_{L\pm}. \\
	\end{dcases}
\label{eq:jump+++S.J-}
\end{equation}

The jump $\mathbf{V}^{\mathbf{S}^\updownarrow}(\SP;x)$ has a well-defined pointwise limit as $\eps \downarrow 0$ away from the points $\SP = 0, \ii A(x), -\ii A(x)$. 
The construction of a global parametrix $\breve{\mathbf{S}}^\updownarrow(\SP;x)$ for $\mathbf{S}^\updownarrow(\SP;x)$ follows along the same lines as was done for $x\in J^+$ in Section~\ref{sec-parametrix}. One important point we emphasize is that inspecting \eqref{eq:jump+++S} and \eqref{eq:jump+++S.J-} and recalling Proposition~\ref{prop:T-outside}, 
\begin{equation}
	\lim_{\epsilon \downarrow 0} \mathbf{V}^\mathbf{S}(\SP;x) = \lim_{\epsilon \downarrow 0} \mathbf{V}^{\mathbf{S}^\updownarrow}(\SP;x) = -\ii \sigma_1, \qquad \SP \in B(x),
\end{equation}
which shows that the outer model $\breve{\mathbf{S}}^\mathrm{out}(\SP;x)$ valid for $x \in J^+$ previously constructed in \eqref{eq:S.out.matrix} works for $x \in J^-$ as well.  In particular, for $x \in J^-$ the endpoints of $B(x)$, which ultimately determine the leading order asymptotic behavior of the semiclassical soliton ensemble $\widetilde{\psi}(x,\mathbf{0})$, are still given by $\pm \ii A(x)$. 
We leave the remaining details to the interested reader. 

Once the global parametrix $\breve{\mathbf{S}}^\updownarrow(\SP;x)$ is constructed, the error 
\begin{equation}
	\mathbf{E}^\updownarrow(\SP;x) = \mathbf{S}^\updownarrow(\SP;x) \breve{\mathbf{S}}^\updownarrow(\SP;x)^{-1}
	\label{eq:Edef-left}
\end{equation}
can be considered, and, repeating the analysis in Section~\ref{sec-accuracy}, one proves the following lemma by mimicking the proof of Lemma~\ref{lem:error-accuracy}.
\begin{lem} 
For $x \in J^-_c $, the estimate
\begin{equation}
	\mathbf{E}^\updownarrow_+(\SP;x) = \mathbf{E}^\updownarrow_-(\SP;x) \left[ \mathbb{I} + \bigo{\eps^{1/2}} \right], \qquad  \SP \in \Sigma'',
\end{equation}
holds uniformly as $\eps \downarrow 0$.
\label{lem:error-accuracy-2}
\end{lem}

%% file: sec-proof-Thm1.tex
To deal with the cases $x\in J^\pm$ simultaneously, for any matrix function $\mathbf{A}$ defined for $x \in J^+$ and $\mathbf{A}^\updownarrow$ defined for $x\in J^-$, set 
\begin{equation}
	\mathbf{A}^\sharp(\SP;x) := \begin{cases}
		\mathbf{A}(\SP;x), & x \in J^+, \\
		\mathbf{A}^\updownarrow(\SP;x),& x \in J^- .
	\end{cases}
\end{equation}
It follows from  \eqref{eq:Edef} and \eqref{eq:Edef-left} that $\mathbf{E}^\sharp(\SP;x) = \mathbb{I} + \bigo{\SP^{-1}}$ as $\SP \to \infty$ and $\mathbf{E}^\sharp(\SP;x)$ is an analytic function for $\SP \in C \setminus \Sigma''$. 
Moreover, from Lemmas~\ref{lem:error-accuracy} and \ref{lem:error-accuracy-2}, the jump matrix $\mathbf{V}^{\mathbf{E}^\sharp}(\SP;x)$ is uniformly an $\bigo{\eps^{1/2}}$ perturbation of the identity matrix for $x$ in compact subsets of $(X_-,X_+) \setminus \{ x_0 \}$ as $\eps \downarrow 0$. 
Finally we observe that the contour $\Sigma''$ is compact and independent of $\eps$.
Together, these facts classify $\mathbf{E}^\sharp(\SP;x)$ as the solution of a Riemann-Hilbert problem of \emph{small-norm type}.  The small-norm theory can be applied in the context of matrix-valued functions that are H\"older continuous up to the boundary of each connected component of $\C \setminus \Sigma''$. The theory establishes the existence of a unique function $\mathbf{E}^\sharp(\SP;x)$ satisfying the normalization and jump conditions in the Riemann-Hilbert problem and yields estimates for $\mathbf{E}^\sharp(\SP;x)-\mathbb{I}$ which are proportional to the product of the above $\bigo{\eps^{1/2}}$ uniform estimate of $\mathbf{V}^{\mathbf{E}^\sharp}(\SP;x) - \mathbb{I}$ and the operator norm of a Cauchy (singular integral) projection operator for the contour $\Sigma''$.  The end result is that a unique solution $\mathbf{E}^\sharp(\SP;x)$ of the error Riemann-Hilbert problem exists and the estimate $\mathbf{E}^\sharp(\SP;x) = \mathbb{I} + \bigo{\eps^{-1/2}}$ holds as $\eps \downarrow 0$ uniformly for $\SP \in \C \setminus \Sigma''$ and $x$ in compact subsets of $(X_-,X_+) \setminus \{ x_0 \}$. Moreover, in the convergent Laurent expansion of $\mathbf{E}^\sharp(\SP;x)$ about $\SP = \infty$,
\begin{equation}
	\mathbf{E}^\sharp(\SP;x) = \mathbb{I} + \sum_{n=0}^\infty \mathbf{E}^{\sharp[n]}(x) \SP^{-n},  \qquad |\SP| > \sup_{\SP' \in \Sigma''} | \SP'|,
	\label{eq:error.matrix.expand}
\end{equation}
all of the coefficients $\mathbf{E}^{\sharp[n]}(x)$, $n \geq 1$, satisfy $\mathbf{E}^{\sharp[n]}(x) = \bigo{\eps^{1/2}}$ as $\eps \downarrow 0$. 

Now, from the definition of \eqref{eq:Edef} and \eqref{eq:Edef-left}, $\mathbf{S}^\sharp(\SP;x) = \mathbf{E}^\sharp(\SP;x) \breve{\mathbf{S}}^\sharp(\SP;x)$. For $|\SP|$  sufficiently large,  $ \breve{\mathbf{S}}^\sharp(\SP;x) = \breve{\mathbf{S}}^{\mathrm{out}}(\SP;x)$ and $\mathbf{S}^\sharp(\SP;x) = \mathbf{M}^\sharp(\SP;x,\mathbf{0}) \ee^{-g^{s}(\SP;x) \sigma_3 }$, where $s = \sgn(x-x_0)$. Therefore for all sufficiently large $\SP$, the solution at $t=0$ of the semiclassical soliton ensemble Riemann-Hilbert Problem~\ref{rhp-meromorphic} for $x \in J^+$ and Riemann-Hilbert Problem~\ref{rhp-meromorphic_renorm} for $x \in J^-$ is given by $\mathbf{M}^\sharp (\SP;x,\mathbf{0}) = \mathbf{E}^\sharp(\SP;x) \breve{\mathbf{S}}^\mathrm{out}(\SP;x) \ee^{g^{s}(\SP;x) \sigma_3 }$.
It then follows from the reconstruction formula \eqref{eq:psitilde-reconstruct}, using Property G5 in Proposition~\ref{prop:g-properties}, \eqref{eq:S.out.expand}, and \eqref{eq:error.matrix.expand}, that
\begin{equation}
	\widetilde{\psi}(x,\mathbf{0}) = A(x) + 2\ii E^{\sharp[1]}_{12}(x)  = A(x) + \bigo{\eps^{1/2}},
\end{equation}
where we recall that $A(x) = \psi_0(x)$ is a Cauchy initial datum of semicircular Klaus-Shaw type. 
This completes the proof of Theorem~\ref{thm-accuracy-t=0} in the case that either $X_-<x<x_0$ or $x_0<x<X_+$.

%% file: sec-proof-Thm1-outside.tex
Suppose that $x>X_+$.
Let $\mathscr{D}$ be a bounded, simply connected domain in the upper 
half-plane $\mathbb{C}_+$ such that $\partial\mathscr{D}$ is a 
simple closed loop that starts and ends at the origin 
and encloses all of the points $\SP=\ii\widetilde{s}_j$, $j=0,\dots,N-1$ for 
all $N$.  See Figure~\ref{fig:mathscrDregions}
\begin{figure}[h]
\begin{tikzpicture}[>=stealth]
\draw[fill, color=blue,fill opacity=0.2] (0cm,0cm) ..controls (-3cm,4cm) and (3cm,4cm).. (0cm,0cm);
\draw[fill, color=blue,fill opacity=0.2] (0cm,0cm) ..controls (3cm,-4cm) and (-3cm,-4cm).. (0cm,0cm);
\draw[dashed,thick] (-3.5cm,0cm) node[left] {$\phantom{\Im(\SP)=0}$}-- (3.5cm,0cm) node[right] {$\Im(\SP)=0$};
\draw[->-=0.5,thick,black] (0cm,0cm) ..controls (-3cm,4cm) and (3cm,4cm).. (0cm,0cm);
\draw[->-=0.5,thick,black] (0cm,0cm) ..controls (3cm,-4cm) and (-3cm,-4cm).. (0cm,0cm);
\draw[fill,red] (0cm,0.4cm) circle [radius=0.1cm];
\draw[fill,red] (0cm,-0.4cm) circle [radius=0.1cm];
\draw[fill,red] (0cm,0.7cm) circle [radius=0.1cm];
\draw[fill,red] (0cm,-0.7cm) circle [radius=0.1cm];
\draw[fill,red] (0cm,1cm) circle [radius=0.1cm];
\draw[fill,red] (0cm,-1cm) circle [radius=0.1cm];
\draw[fill,red] (0cm,1.3cm) circle [radius=0.1cm];
\draw[fill,red] (0cm,-1.3cm) circle [radius=0.1cm];
\draw[fill,red] (0cm,1.7cm) circle [radius=0.1cm];
\draw[fill,red] (0cm,-1.7cm) circle [radius=0.1cm];
\draw[fill,red] (0cm,2.1cm) circle [radius=0.1cm];
\draw[fill,red] (0cm,-2.1cm) circle [radius=0.1cm];
\node at (0.4cm,2.5cm) {$\mathscr{D}$};
\node at (0.4cm,-2.5cm) {$\mathscr{D}^*$};
\end{tikzpicture}
\caption{The regions $\mathscr{D}$ and $\mathscr{D}^*$ and oriented boundaries $\partial\mathscr{D}$ and $\partial\mathscr{D}^*$.}
\label{fig:mathscrDregions}
\end{figure}
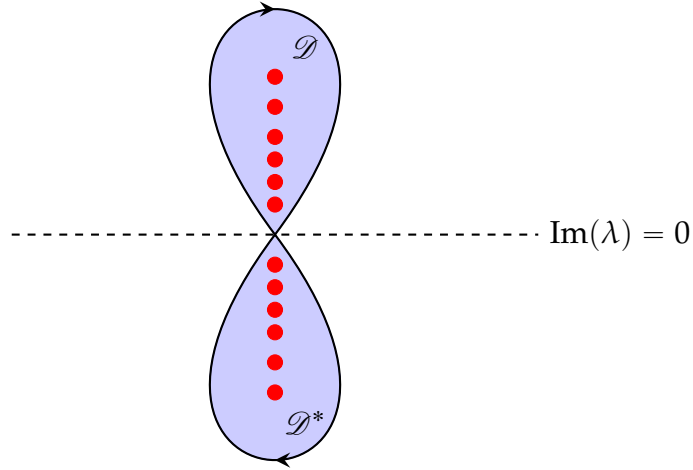

Starting from Riemann-Hilbert Problem~\ref{rhp-meromorphic} with 
solution $\widetilde{\mathbf{M}}(\SP;x,\mathbf{t})$, we define
\begin{equation}
\mathbf{U}(\SP;x):=
\begin{dcases}
\widetilde{\mathbf{M}}(\SP;x,\mathbf{0})\begin{pmatrix} 1 & 0\\
-\ii \widetilde{a}(\SP)^{-1}\ee^{2f_0(\SP;x,\mathbf{0})/\epsilon} & 1\end{pmatrix}, & \SP\in \mathscr{D}, \\
\mathbf{U}(\SP^*;x)^{-\dagger}, & \SP \in \mathscr{D}^*, \\
\widetilde{\mathbf{M}}(\SP;x,\mathbf{0}), & \C \setminus \overline{\mathscr{D} \cup \mathscr{D}^*},
\end{dcases}
\end{equation}
where we recall the definition of $f_K(\SP;x,\mathbf{t})$ in 
\eqref{eq:fK-define}.  The definition $\mathbf{U}(\lambda;x)$ has the effect 
of removing the poles from $\widetilde{\mathbf{M}}(\SP;x,\mathbf{0})$ since, 
for $K=0$,
\begin{equation}
c_n(x,\mathbf{0})=\ii\mathop{\mathrm{Res}}_{\SP=\ii \widetilde{s}_n}\frac{\ee^{2f_0(\SP;x,\mathbf{0})/\epsilon}}{\widetilde{a}(\SP)}.
\end{equation}
The matrix $\mathbf{U}(\lambda;x)$ satisfies the following Riemann-Hilbert problem.

\begin{myrhp}[Problem for $\mathbf{t}=\mathbf{0}$ and $x>X_+$]
Given $\epsilon>0$ and $x \in\mathbb{R}$, seek a $2\times 2$ matrix function 
$\mathbf{U}(\SP)=\mathbf{U}(\SP;x)$ with the following properties.
\begin{itemize}
\item[]\textit{\textbf{Analyticity:}}  $\mathbf{U}(\SP)$ is analytic for 
$\SP\in\mathbb{C}\setminus(\partial\mathscr{D}\cup\partial\mathscr{D}^*)$ 
and satisfies the Schwarz symmetry condition 
$\mathbf{U}(\SP^*)=\mathbf{U}(\SP)^{-\dagger}$.
\item[]\textit{\textbf{Jump conditions:}} $\mathbf{U}(\SP)$ takes continuous 
boundary values on $\partial\mathscr{D}$ and $\partial\mathscr{D}^*$.  
Orienting both loops $\partial\mathscr{D}$ and $\partial\mathscr{D}^*$ 
in the clockwise direction, the jumps are:
\begin{eqnarray}
\mathbf{U}_+(\SP;x) & = & \mathbf{U}_-(\SP;x)\begin{pmatrix}1 & 0\\
\ii \widetilde{a}(\SP)^{-1}\ee^{2f_0(\SP;x,\mathbf{0})/\epsilon} & 1
\end{pmatrix},\quad \SP\in\partial\mathscr{D}, \\
\mathbf{U}_+(\SP;x) & = & \mathbf{U}_-(\SP;x)\begin{pmatrix}1 & \ii \widetilde{a}(\SP^*)^{*-1}\ee^{2f_0(\SP^*; x,\mathbf{0})^*/\epsilon}\\0 & 1\end{pmatrix},\quad\SP\in\partial\mathscr{D}^*.
\end{eqnarray}
\item[]\textit{\textbf{Normalization:}} $\mathbf{U}(\SP)\to\mathbb{I}$ as $\SP\to\infty$.
\end{itemize}
\label{rhp:U}
\end{myrhp}
Using \eqref{eq:ZS-Y-define}, we can rewrite the jump on $\partial\mathscr{D}$ as
\begin{equation}
\mathbf{U}_+(\SP;x) = \mathbf{U}_-(\SP;x)\begin{pmatrix}1 & 0\\
\ii Y_\epsilon(\lambda)\ee^{2F(\SP;x)/\epsilon} & 1
\end{pmatrix},\quad \SP\in\partial\mathscr{D},
\label{U-on-Dboundary}
\end{equation}
where
\begin{equation}
2F(\lambda;x):=2f_0(\lambda;x,\mathbf{0})+L(\lambda).
\label{2F-of-lambda}
\end{equation}
We now show that as $\epsilon\downarrow 0$ the factor $\ee^{2F(\lambda;x)/\eps}$ is exponentially small uniformly for $\lambda\in\partial\mathscr{D}$ bounded away from the origin, provided that the domain $\mathscr{D}$ consists of points sufficiently close to the imaginary interval $0<-\ii\lambda<A_\mathrm{max}$.  From 
\eqref{eq:Lbar-def-3} and \eqref{eq:fK-define},
\begin{equation}
2F(\lambda;x) = \ii\phaseint(\SP) + \tailint(\SP) + 2\ii x\lambda + \overline{L}(\SP) \pm \ii\phaseint(\SP), \quad 0<\imag\{\SP\}<A_\mathrm{max},
\quad\pm\re\{\SP\}>0.
\end{equation}
From Lemma \ref{lem:mu-Lbar}, letting $\lambda$ approach a positive imaginary value $\lambda\to \ii s$ from the domain $\pm\mathrm{Re}\{\lambda\}>0$ gives limiting values
\begin{equation}
2F(\ii s;x) = 2(X_+-x)s - 2\int_{x_+(s)}^{X_+}\sqrt{s^2-A(x)^2}\,\dd x
  +\ii\phaseint(\ii s) \pm \ii\phaseint(\ii s), \quad 0<s<A_\mathrm{max}.
\end{equation}
From \eqref{eq:ZS-phase-integral} we see that $\phaseint(\ii s)$ is purely 
real for $0<s<A_\mathrm{max}$.  Therefore the real part of $2F(\lambda;x)$ is continuous across the imaginary axis, and we deduce that unambiguously
\begin{equation}
\re\{2F(\ii s;x)\} = 2(X_+-x)s - 2\int_{x_+(s)}^{X_+}\sqrt{s^2-A(x)^2}\,\dd x
, \quad 0<s<A_\mathrm{max}.
\end{equation}
Since $x>X_+$ and the integral is manifestly positive, it follows that 
$\re\{2F(\ii s;x)\}<0$ for $0<s<A_\mathrm{max}$, as required.  Moreover, if $s>0$ is bounded away from zero, the upper bound on $\re\{2F(\ii s;x)\}$ can be replaced with a negative constant.  By continuity of $\re\{2F(\lambda;x)\}$, this establishes the desired uniform exponential decay of $\ee^{2F(\lambda;x)/\epsilon}$ as $\epsilon\downarrow 0$.

We now assume further that for some small parameters $\delta>0$ and $\sigma>0$, the part of $\partial\mathscr{D}$ with $|\lambda|>\sigma$ lies within the domain $D_\sigma$ of Proposition~\ref{prop:Y-outside}.  Then it follows that also $Y_\eps(\lambda)$ is uniformly bounded, implying that for $\lambda\in\partial\mathscr{D}$ with $|\lambda|>\sigma$, we have 
$\mathbf{U}_+(\SP;x)=\mathbf{U}_-(\SP;x)(\mathbb{I}+\text{exponentially small})$
holding uniformly for $\SP\in\partial\mathscr{D}$ bounded away from the origin in the limit $\epsilon\downarrow 0$.  
By Schwarz symmetry the same holds for $\SP\in\partial\mathscr{D}^*$ bounded 
away from the origin.  All together, we are assuming that the loop $\partial\mathscr{D}$ is both sufficiently close to the imaginary segment $0<-\ii\lambda<A_\mathrm{max}$ and that there is a fixed minimum distance between this segment and the part of the loop outside a small disk centered at the origin.

To model $\mathbf{U}(\SP;x)$, it then suffices to analyze the jump matrix 
in \eqref{U-on-Dboundary} near the origin.  From \eqref{eq:fK-define} and 
$\tailint_1=-(X_++X_-)$ (see Proposition \eqref{prop:mu-odd-analytic}), the 
function $2f_0(\SP;x,\mathbf{0})$ is analytic at $\SP=0$ with Taylor 
expansion
\begin{equation}
2f_0(\SP;x,\mathbf{0})=\ii\phaseint_0  + 2\ii (x-\overline{X})\SP + \bigo{\SP^2},\quad\overline{X}:=\frac{1}{2}(X_++X_-),\quad \SP\to 0.
\label{2f-near-0}
\end{equation}
Similarly, from \eqref{eq:Lbar-def-3} and \eqref{eq:Lbar-def-1}, $L(\SP)$ 
has the expansions
\begin{equation}
L(\SP) = \pm\ii\phaseint_0-\ii(X_+-X_-)\SP+\bigo{\SP^2}, 
\quad \SP\to 0, \quad \pm\re\{\SP\}>0.
\label{L-near-zero}
\end{equation}
From \eqref{2F-of-lambda}, \eqref{2f-near-0}, and \eqref{L-near-zero}, it 
follows that 
\begin{equation}
2F(\SP;x) = \ii\phaseint_0\pm\ii\phaseint_0+2\ii(x-X_+)\SP + \bigo{\SP^2}, 
\quad \SP\to 0, \quad \pm\re\{\SP\}>0.
\label{eq:2F-Taylor}
\end{equation}
Combining \eqref{eq:ZS-epsilon-assumption} and \eqref{eq:ZS-phase-integral} 
gives $\phaseint_0 = N\epsilon\pi$.  Therefore,
\begin{equation}
\ee^{2F(\lambda;x)/\epsilon} = \ee^{[2\ii(x-X_+)\SP+\bigo{\SP^2}]/\epsilon}
\end{equation}
in a neighborhood of $\lambda=0$ regardless of the sign of $\re\{\SP\}$, which affects the error term but not its order estimate.  
Comparing this with \eqref{U-on-Dboundary} motivates the definition of the 
following global approximation for $\mathbf{U}$:
\begin{equation}
\breve{\mathbf{U}}(\SP;x):=\begin{cases}
\begin{pmatrix}1&0\\-\ii \ee^{2\ii(x-X_+)\SP/\epsilon} & 1\end{pmatrix},&\SP\in \mathscr{D},\\
\begin{pmatrix}1&-\ii \ee^{-2\ii (x-X_+)\SP/\epsilon}\\0 & 1\end{pmatrix},&\SP\in \mathscr{D}^*,\\
\mathbb{I},&\SP\in\mathbb{C}\setminus\overline{\mathscr{D}\cup \mathscr{D}^*}.
\end{cases}
\label{eq:N-parametrix-far-outside}
\end{equation}

Note that $\det(\breve{\mathbf{U}}(\SP;x))=1$ and that $\breve{\mathbf{U}}(\SP;x)=\sigma_2\breve{\mathbf{U}}(\SP^*;x)^*\sigma_2=\breve{\mathbf{U}}(\SP^*;x)^{-\dagger}$.  The error is defined as $\mathbf{E}(\SP;x):=\mathbf{U}(\SP;x)\breve{\mathbf{U}}(\SP;x)^{-1}$.  It is clearly analytic for $\SP\in\mathbb{C}\setminus (\partial \mathscr{D}\cup\partial \mathscr{D}^*)$.  It tends to the identity as $\SP\to\infty$ because $\mathbf{U}(\SP;x)$ does so while $\breve{\mathbf{U}}(\SP;x)^{-1}=\mathbb{I}$ for $|\SP|$ sufficiently large.   Imposing the symmetry $\mathbf{E}(\SP;x)=\sigma_2\mathbf{E}(\SP^*;x)^*\sigma_2=\mathbf{E}(\SP^*;x)^{-\dagger}$, $\mathbf{E}(\SP;x)$ is characterized by its jump condition across $\partial \mathscr{D}$, which according to \eqref{U-on-Dboundary} and \eqref{eq:N-parametrix-far-outside} reads
\begin{multline}
\mathbf{E}_+(\SP;x)=\mathbf{E}_-(\SP;x)\mathbf{V}^\mathbf{E}(\SP;x),\\
\mathbf{V}^\mathbf{E}(\SP;x):=\begin{pmatrix}1&0\\\ii [Y_\epsilon(\SP)\ee^{2F(\SP;x)/\epsilon}-\ee^{2\ii (x-X_+)\SP/\epsilon}] & 1\end{pmatrix},\quad\SP\in\partial \mathscr{D}.
\label{eq:E-jump-partial-D}
\end{multline}
We can write the jump condition across $\partial\mathscr{D}^*$ in the same form $\mathbf{E}_+(\lambda;x)=\mathbf{E}_-(\lambda;x)\mathbf{V}^\mathbf{E}(\lambda;x)$, where $\mathbf{V}^\mathbf{E}(\lambda;x)=\mathbf{V}^\mathbf{E}(\lambda^*;x)^{-\dagger}$ holds for $\lambda\in\partial\mathscr{D}^*$.
With the clockwise orientation of both loops, the jump contour  $\Sigma:=\partial \mathscr{D}\cup\partial \mathscr{D}^*$ is a complete oriented contour that divides the complex plane into complementary regions:  $\mathscr{D}\cup \mathscr{D}^*$ on the right and $\mathbb{C}\setminus\overline{\mathscr{D}\cup \mathscr{D}^*}$ on the left.  

Since $\widetilde{\mathbf{M}}(\SP;x,\mathbf{0})=\mathbf{U}(\SP;x)=\mathbf{E}(\SP;x)$ for $|\SP|$ sufficiently large, using the Plemelj formula we can express $\widetilde{\psi}(x,\mathbf{0})$ exactly in the form
\begin{equation}
\widetilde{\psi}(x,\mathbf{0})=2\ii\lim_{\SP\to\infty}\SP E_{12}(\SP;x) = -\frac{1}{\pi}\int_{\Sigma}(E_{12,+}(\SP;x)-E_{12,-}(\SP;x))\,\dd\SP,
\label{eq:psi-formula-E-1}
\end{equation}
which follows from
\begin{equation}
\mathbf{E}(\SP;x)=\mathbb{I}+\mathcal{C}^{\Sigma}[\mathbf{E}_+(\diamond;x)-\mathbf{E}_-(\diamond;x)](\SP),\quad\lambda\in\mathbb{C}\setminus\Sigma,
\label{eq:E-formula-1}
\end{equation}
in which for an arbitrary oriented contour $\Sigma$, the Cauchy transform of a matrix-valued function $\mathbf{F}(\diamond)$ defined on $\Sigma$ is given by
\begin{equation}
\mathcal{C}^\Sigma[\mathbf{F}(\diamond)](\SP):=\frac{1}{2\pi\ii}\int_\Sigma\frac{\mathbf{F}(\xi)\,\dd\xi}{\xi-\SP},\quad\SP\in\mathbb{C}\setminus \Sigma.
\end{equation}
Letting $\SP$ tend to $\Sigma$ from the right side in \eqref{eq:E-formula-1}, denoting the resulting boundary value by $\boldsymbol{\mu}(\SP;x):=\mathbf{E}_-(\SP;x)$, and using $\mathbf{E}_+(\SP;x)=\mathbf{E}_-(\SP;x)\mathbf{V}^\mathbf{E}(\SP;x)$ yields the singular integral equation
\begin{equation}
\boldsymbol{\mu}(\SP;x)-\mathcal{C}_-^\Sigma[\boldsymbol{\mu}(\diamond;x)(\mathbf{V}^\mathbf{E}(\diamond;x)-\mathbb{I})](\SP)=\mathbb{I},\quad\SP\in\Sigma.
\label{eq:sing-int}
\end{equation}
In terms of the solution $\boldsymbol{\mu}(\SP;x)$ of this integral equation, from \eqref{eq:psi-formula-E-1} we obtain
\begin{equation}
\widetilde{\psi}(x,\mathbf{0})=-\frac{1}{\pi}\int_\Sigma (\boldsymbol{\mu}(\SP;x)(\mathbf{V}^\mathbf{E}(\SP;x)-\mathbb{I}))_{12}\,\dd\SP.
\label{eq:psi-formula-mu}
\end{equation}
We note that since $\Sigma$ may be taken to be a Lipschitz curve independent of $\epsilon$, $\mathcal{C}_-^\Sigma$ is a bounded linear operator on $L^2(\Sigma)$, with fixed operator norm $\|\mathcal{C}_-^\Sigma\|_{L^2(\Sigma)\circlearrowleft}<\infty$.

We claim that $\|\mathbf{V}^\mathbf{E}-\mathbb{I}\|_{L^\infty(\Sigma)}=\bigo{\epsilon^{1/2}}$.  This would be implied by the assertion that
\begin{equation}
\sup_{\SP\in\partial \mathscr{D}}\left|Y_\epsilon(\SP)\ee^{2F(\SP;x)/\epsilon}-\ee^{2\ii (x- X_+)\SP/\epsilon}\right| = \bigo{\epsilon^{1/2}}.
\label{eq:sup-loop}
\end{equation}
Combining Propositions~\ref{prop:Y-outside}, \ref{prop:Y-zeta-small}, and \ref{prop:Y0-asymp} shows that $Y_\epsilon(\SP)$ is uniformly bounded for $\SP\in\partial \mathscr{D}$.
Picking an exponent $p<1$, the Taylor formula \eqref{eq:2F-Taylor} shows that if $x>X_+$, then there is some $C>0$ so that both $\ee^{2F(\SP;x)/\epsilon}$ and $\ee^{2\ii (x-X_+)\SP/\epsilon}$ are uniformly $\bigo{\ee^{-C\epsilon^{p-1}}}$ on $\partial \mathscr{D}$ with $|\SP|>\epsilon^p$.  Therefore,
\begin{equation}
\begin{split}
\mathop{\sup_{\SP\in\partial \mathscr{D}}}_{|\SP|>\epsilon^p}\left|Y_\epsilon(\SP)\ee^{2F(\SP;x)/\epsilon}-\ee^{2\ii (x-X_+)\SP/\epsilon}\right| &= \bigo{\ee^{-C\epsilon^{p-1}}} \\ &= \bigo{\epsilon^{1/2}}.
\end{split}
\label{eq:outer-estimate}
\end{equation}
It therefore remains to prove that if $x>X_+$, then for some $p<1$,
\begin{equation}
\mathop{\sup_{\SP\in \partial \mathscr{D}}}_{|\SP|\le\epsilon^p}\left|\ee^{2\ii (x-X_+)\SP/\epsilon}\right|\cdot\left|Y_\epsilon(\SP)\ee^{[2F(\SP;x)-2\ii (x-X_+)\SP]/\epsilon}-1\right|=\bigo{\epsilon^{1/2}}.
\label{eq:WTS-1}
\end{equation}
For this, we assume now that in a neighborhood of the origin, $\partial\mathscr{D}$ lies within the  sector of opening angle strictly less than $\pi$ symmetric about the positive imaginary axis.  Then, because $x>X_+$, there is some constant $C>0$ such that $|\ee^{2\ii (x-X_+)\SP/\epsilon}|\le \ee^{-C|\SP|/\epsilon}$ holds for all $\SP\in\partial \mathscr{D}$.  If $p>\tfrac{1}{2}$, then $\SP\in \partial \mathscr{D}$ with $|\SP|\le\epsilon^p$ implies, using \eqref{eq:2F-Taylor}, that 
\begin{equation}
\begin{split}
\ee^{[2F(\SP;x)-2\ii (x-X_+)\SP]/\epsilon}&=1+\bigo{\frac{\SP^2}{\epsilon}}\\
&=1+\bigo{\frac{\SP}{\epsilon^{1/2}}}. 
\end{split}
\label{eq:exponential-expand-1}
\end{equation}
Under the same conditions, using Propositions~\ref{prop:Y-zeta-small} and \ref{prop:Y0-asymp} gives
\begin{equation}
\begin{split}
Y_\epsilon(\SP)&=\mathcal{Y}_0\left(\frac{\varphi(\SP)}{\epsilon^{1/2}}\right)\left(1+\bigo{\epsilon^{1/2}\lambda} + \bigo{\epsilon}\right)\\
&=\left(1-2\ii(\sqrt{2}-1)\zeta(\tfrac{1}{2})\frac{\varphi(\SP)}{\epsilon^{1/2}} + \bigo{\frac{\varphi(\SP)^2}{\epsilon}}\right)\left(1+\bigo{\epsilon^{1/2}\lambda}+\bigo{\epsilon}\right)\\
&=\left(1-2\ii(\sqrt{2}-1)\zeta(\tfrac{1}{2})\frac{\varphi'(0)\SP}{\epsilon^{1/2}} + \bigo{\frac{\SP^2}{\epsilon}}\right)\left(1+\bigo{\epsilon^{1/2}\lambda}+\bigo{\epsilon}\right)\\
&=\left(1+k\frac{\lambda}{\epsilon^{1/2}}
+\bigo{\frac{\SP^2}{\epsilon}}\right)\left(1+\bigo{\epsilon^{1/2}\lambda}+\bigo{\epsilon}\right)\\
&=1+k\frac{\lambda}{\epsilon^{1/2}}
+\bigo{\epsilon}+\bigo{\epsilon^{1/2}\lambda}+\bigo{\frac{\lambda^2}{\epsilon}}+\bigo{\frac{\lambda^3}{\epsilon^{1/2}}}\\
&=1+\bigo{\frac{\lambda}{\epsilon^{1/2}}} + \bigo{\epsilon},
\end{split}
\label{eq:Yepsilon-expand-1}
\end{equation}
where $k:=-2\ii(\sqrt{2}-1)\zeta(\frac{1}{2})/v(0)$.
Therefore, if $\SP\in\partial \mathscr{D}$ with $|\SP|\le \epsilon^p$ and $p>\tfrac{1}{2}$,
\begin{equation}
\left|\ee^{2\ii (x-X_+)\SP/\epsilon}\right|\cdot\left|Y_\epsilon(\SP)\ee^{[2F(\SP;x)-2\ii (x-X_+)\SP]/\epsilon}-1\right|
=\bigo{\epsilon^{1/2}\frac{\SP}{\epsilon}\ee^{-C|\SP|/\epsilon}} + \bigo{\epsilon\ee^{-C|\SP|/\epsilon}}.
\label{eq:inner-estimate}
\end{equation}
Because $\ee^{-Cy}$ and $y\ee^{-Cy}$ are both uniformly bounded functions of $y\ge 0$, the result \eqref{eq:WTS-1} follows assuming that $p>\frac{1}{2}$.  Therefore choosing any $p\in (\frac{1}{2},1)$ and combining with \eqref{eq:outer-estimate} yields \eqref{eq:sup-loop} and therefore $\|\mathbf{V}^\mathbf{E}-\mathbb{I}\|_{L^\infty(\Sigma)}=\bigo{\epsilon^{1/2}}$ as desired. 

Since $\|\mathcal{C}_-^\Sigma\|_{L^2(\Sigma)\circlearrowleft}$ is finite and independent of $\epsilon$, it follows that the composition with multiplication on the right by $\mathbf{V}^\mathbf{E}-\mathbb{I}$ yields a bounded linear operator on $L^2(\Sigma)$ with norm $\bigo{\epsilon^{1/2}}$.  Therefore, the singular integral equation \eqref{eq:sing-int} can be solved by Neumann series provided that $\epsilon$ is sufficiently small.  In particular, this implies that (after one explicit iteration)
\begin{equation}
\boldsymbol{\mu}-\mathbb{I}-\mathcal{C}_-^\Sigma[\mathbf{V}^\mathbf{E}-\mathbb{I}] = \bigo{\epsilon}\quad\text{in $L^2(\Sigma)$}.
\end{equation}
Now, the first line of \eqref{eq:outer-estimate} and the estimate \eqref{eq:inner-estimate}, along with the Schwarz symmetry $\mathbf{V}^\mathbf{E}(\lambda;x)=\mathbf{V}^\mathbf{E}(\lambda^*;x)^{-\dagger}$ to obtain similar estimates for $\lambda\in\partial \mathscr{D}^*$, imply that $\mathbf{V}^\mathbf{E}-\mathbb{I}=\bigo{\epsilon}$ in $L^2(\Sigma)$.  Therefore,
by Cauchy-Schwarz \eqref{eq:psi-formula-mu} implies that
\begin{multline}
\widetilde{\psi}(x,\mathbf{0})=-\frac{1}{\pi}\int_\Sigma (\mathbf{V}^\mathbf{E}(\SP;x)-\mathbb{I})_{12}\,\dd\SP \\{}- 
\frac{1}{\pi}\int_\Sigma\left(\mathcal{C}_-^\Sigma[\mathbf{V}^\mathbf{E}(\diamond;x)-\mathbb{I}](\SP)(\mathbf{V}^\mathbf{E}(\SP;x)-\mathbb{I})\right)_{12}\,\dd\SP+
\bigo{\epsilon^2}.
\end{multline}
Since $\mathbf{V}^\mathbf{E}-\mathbb{I}$ is an off-diagonal matrix at each point of $\Sigma$, so is its Cauchy transform $\mathcal{C}_-^\Sigma[\mathbf{V}^\mathbf{E}(\diamond;x)-\mathbb{I}](\SP)$.  Therefore their product is diagonal, so the integral on the second line vanishes and hence
\begin{equation}
\widetilde{\psi}(x,\mathbf{0})=-\frac{1}{\pi}\int_\Sigma (\mathbf{V}^\mathbf{E}(\SP;x)-\mathbb{I})_{12}\,\dd\SP+
\bigo{\epsilon^2}.
\label{eq:psi-leading-term}
\end{equation}
Moreover, $\mathbf{V}^\mathbf{E}(\lambda;x)$ is lower-triangular for $\lambda\in\partial\mathscr{D}$, so there is no contribution to the integral from $\partial \mathscr{D}$:
\begin{equation}
\begin{split}
-\frac{1}{\pi}\int_\Sigma (\mathbf{V}^\mathbf{E}(\SP;x)-\mathbb{I})_{12}\,\dd\SP &=-\frac{1}{\pi}\int_{\partial \mathscr{D}^*}(\mathbf{V}^\mathbf{E}(\SP;x)-\mathbb{I})_{12}\,\dd\SP\\
&=-\frac{\ii}{\pi}\int_{\partial \mathscr{D}^*} \left(Y_\epsilon(\SP^*)^*\ee^{2F(\SP^*;x)^*/\epsilon}-\ee^{-2\ii (x-X_+)\SP/\epsilon}\right)\,\dd\SP\\
&=\left[-\frac{\ii}{\pi}\int_{\partial\mathscr{D}}\left(Y_\epsilon(\SP)\ee^{2F(\SP;x)/\eps}-\ee^{2\ii (x-X_+)\SP/\eps}\right)\,\dd\SP\right]^*,
\end{split}
\label{eq:integral-VminusI-12}
\end{equation}
where in the last line we used Schwarz reflection.
Now, a more refined version of \eqref{eq:inner-estimate} is obtained from using instead the penultimate lines of \eqref{eq:exponential-expand-1} and \eqref{eq:Yepsilon-expand-1}.  For $\SP\in\partial \mathscr{D}$ with $|\SP|\le\epsilon^p$ and $p>\frac{1}{2}$, 
\begin{multline}
Y_\epsilon(\SP)\ee^{2F(\SP;x)/\epsilon}-\ee^{2\ii (x-X_+)\SP/\epsilon}\\
\begin{aligned}
&=\ee^{2\ii (x-X_+)\SP/\epsilon}\left(Y_\epsilon(\SP)\ee^{[2F(\SP;x)-2\ii (x-X_+)\SP]/\epsilon}-1\right) \\{}
&= 
\ee^{2\ii (x-X_+)\SP/\epsilon}\left(k\frac{\lambda}{\epsilon^{1/2}}
+ \bigo{\epsilon}+ \bigo{\epsilon^{3/2}\frac{\SP}{\eps}}+\bigo{\eps\left(\frac{\SP}{\eps}\right)^2}\right.\\
&\qquad{}\left.+\bigo{\eps^{3/2}\left(\frac{\lambda}{\epsilon}\right)^3}+\bigo{\eps^2\left(\frac{\lambda}{\epsilon}\right)^4}+\bigo{\eps^{7/2}\left(\frac{\lambda}{\epsilon}\right)^5}\right).
\end{aligned}
\label{eq:refined-inner-estimate}
\end{multline}
Taking $p\in (\frac{1}{2},1)$ and making use of the first line of \eqref{eq:outer-estimate}, the same left-hand side is exponentially small on the complement of $\partial \mathscr{D}$.  We therefore see that at the cost of an exponentially small error, the integral in \eqref{eq:integral-VminusI-12} of the explicit term on the right-hand side of \eqref{eq:refined-inner-estimate}
can be extended to the closed curve $\partial \mathscr{D}$ where it integrates to zero by Cauchy's theorem.  The remaining six terms on the right-hand side of \eqref{eq:refined-inner-estimate} have unspecified analyticity properties, so to estimate the integral in \eqref{eq:integral-VminusI-12}, the best we can do is $L^1(\partial \mathscr{D})$ estimates, which by scaling are 
all $\bigo{\epsilon^2}$ as $y^p\ee^{-Cy}$ is integrable on $y>0$ for $p=0,\dots,5$.  We conclude from \eqref{eq:psi-leading-term} that $\widetilde{\psi}(x,\mathbf{0})=\bigo{\epsilon^2}$ when $x>X_+$, and clearly the estimate is uniform for $x$ in compact subsets of the indicated interval.

To obtain the corresponding estimate for $x<X_-$, we start instead from the matrix $\mathbf{M}^\updownarrow(\SP;x,\mathbf{0})$ solving Riemann-Hilbert Problem~\ref{rhp-meromorphic_renorm} and proceed similarly, now constructing $F(\SP;x)$ from $f_K(\SP;x,\mathbf{0})$ for a nonzero $K\in\mathbb{Z}$.

This completes the proof of Theorem~\ref{thm-accuracy-t=0}.

%% file: sec-Suleimanov.tex
In this section we prove Theorems~\ref{thm:multi-time}, \ref{thm:mixture}, 
and \ref{thm:pure-flow}.  Our starting point is 
Riemann-Hilbert Problem~\ref{rhp-meromorphic} with solution 
$\widetilde{\mathbf{M}}(\SP;x,\mathbf{t})$.  
Let $D$ denote a half disk in the upper half-plane $\mathbb{C}_+$ with center at the origin and radius $L>0$ sufficiently large to contain the points $\SP=\ii\widetilde{s}_j$, $j=0,\dots,N-1$ for all $N$, and such that the boundary $\partial D$ satisfies $\partial D\cap\mathbb{R}=[-L,L]$.  See Figure~\ref{fig:D-regions}.  
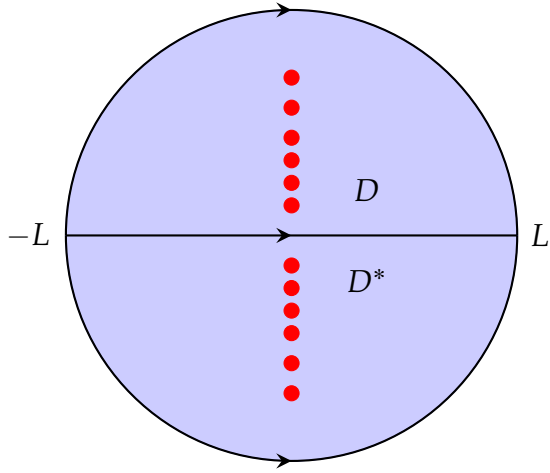
\begin{figure}[h]
\begin{tikzpicture}[>=stealth]
\draw[fill, color=blue,fill opacity=0.2] (0cm,0cm) circle [radius=3cm];
\draw[->-=0.5,thick,black] (-3cm,0cm) -- (3cm,0cm);
\draw[->-=0.5,thick,black] (-3cm,0cm) arc [start angle=180, end angle = 0, radius=3cm];
\draw[->-=0.5,thick,black] (-3cm,0cm) arc [start angle=-180,end angle = 0,radius = 3cm];
\draw[fill,red] (0cm,0.4cm) circle [radius=0.1cm];
\draw[fill,red] (0cm,-0.4cm) circle [radius=0.1cm];
\draw[fill,red] (0cm,0.7cm) circle [radius=0.1cm];
\draw[fill,red] (0cm,-0.7cm) circle [radius=0.1cm];
\draw[fill,red] (0cm,1cm) circle [radius=0.1cm];
\draw[fill,red] (0cm,-1cm) circle [radius=0.1cm];
\draw[fill,red] (0cm,1.3cm) circle [radius=0.1cm];
\draw[fill,red] (0cm,-1.3cm) circle [radius=0.1cm];
\draw[fill,red] (0cm,1.7cm) circle [radius=0.1cm];
\draw[fill,red] (0cm,-1.7cm) circle [radius=0.1cm];
\draw[fill,red] (0cm,2.1cm) circle [radius=0.1cm];
\draw[fill,red] (0cm,-2.1cm) circle [radius=0.1cm];
\node at (1cm,0.6cm) {$D$};
\node at (1cm,-0.6cm) {$D^*$};
\node at (-3.5cm,0cm) {$-L$};
\node at (3.3cm,0cm) {$L$};
\end{tikzpicture}
\caption{The regions $D$ and $D^*$.}
\label{fig:D-regions}
\end{figure}
Fix $K\in\mathbb{Z}$, and let $\widetilde{a}(\SP)$ be defined by the Blaschke product \eqref{eq:ZS-tilde-a}.
Because they are polynomials by assumption, $\phaseint(\SP)$ and $\tailint(\SP)$ are analytic for $\SP\in D$, and 
we define a new unknown related to $\widetilde{\mathbf{M}}(\SP;x,\mathbf{t})$ as follows:
\begin{equation}
\mathbf{N}(\SP;x,\mathbf{t}):=
\begin{dcases}
\widetilde{\mathbf{M}}(\SP;x,\mathbf{t})\begin{pmatrix} 1 & 0\\
-\ii (-1)^K \widetilde{a}(\SP)^{-1}\ee^{2f_K(\SP;x,\mathbf{t})/\epsilon} & 1\end{pmatrix}, & \SP\in D, \\
\mathbf{N}(\SP^*;x,\mathbf{t})^{-\dagger}, & \SP \in D^*, \\
\widetilde{\mathbf{M}}(\SP;x,\mathbf{t}), & \C \setminus \overline{ D \cup D^*},
\end{dcases}
\end{equation}
where we recall the definition of $f_K(\SP;x,\mathbf{t})$ in \eqref{eq:fK-define}.
One then checks from the conditions of Riemann-Hilbert Problem~\ref{rhp-meromorphic} that, since $c_n(x,\mathbf{t})$ can be written in the form
\begin{equation}
c_n(x,\mathbf{t})=\ii (-1)^K\mathop{\mathrm{Res}}_{\SP=\ii \widetilde{s}_n}\frac{\ee^{2f_K(\SP;x,\mathbf{t})/\epsilon}}{\widetilde{a}(\SP)},
\end{equation}
the new unknown $\mathbf{N}(\SP;x,\mathbf{t})$ has removable singularities at all poles of $\widetilde{\mathbf{M}}(\SP;x,\mathbf{t})$ and so is piecewise analytic in $D$, $D^*$, and the unbounded exterior domain, taking continuous boundary values on its jump contour consisting of the three arcs $\partial D\setminus [-L,L]$, $(-L,L)$, and $\partial D^*\setminus[-L,L]$.  The definition also preserves the normalization condition: $\mathbf{N}(\SP;x,\mathbf{t})\to\mathbb{I}$ as $\SP\to\infty$.  Assuming that the three arcs are oriented from $-L$ toward $L$ as shown in Figure~\ref{fig:D-regions}, and that boundary values from the left (right) are indicated with a subscript of ``$+$'' (``$-$'') the jump conditions satisfied by $\mathbf{N}(\SP;x,\mathbf{t})$ are as follows:
\begin{equation}
\mathbf{N}_+(\SP;x,\mathbf{t})=\mathbf{N}_-(\SP;x,\mathbf{t})\begin{pmatrix}1 & 0\\
\ii (-1)^K \widetilde{a}(\SP)^{-1}\ee^{2f_K(\SP;x,\mathbf{t})/\epsilon} & 1
\end{pmatrix},\quad \SP\in\partial D\setminus [-L,L],
\label{eq:jump-N-up}
\end{equation}
\begin{equation}
\mathbf{N}_+(\SP;x,\mathbf{t})=\mathbf{N}_-(\SP;x,\mathbf{t})\begin{pmatrix}1 & -\ii (-1)^K \widetilde{a}(\SP^*)^{*-1}\ee^{2f_K(\SP^*; x,\mathbf{t})^*/\epsilon}\\0 & 1\end{pmatrix},\quad\SP\in\partial D^*\setminus[-L,L],
\end{equation}
and for $\SP\in (-L,L)$,
\begin{multline}
\mathbf{N}_+(\SP;x,\mathbf{t})\\
\begin{aligned}
&=\mathbf{N}_-(\SP;x,\mathbf{t})\begin{pmatrix}1 & \ii(-1)^K\widetilde{a}(\SP)^{*-1}\ee^{2f_K(\SP;x,\mathbf{t})^*/\epsilon}\\0&1\end{pmatrix}\begin{pmatrix}1&0\\-\ii(-1)^K\widetilde{a}(\SP)^{-1}\ee^{2f_K(\SP;x,\mathbf{t})/\epsilon} & 1\end{pmatrix}\\
&=\mathbf{N}_-(\SP;x,\mathbf{t})\begin{pmatrix} 1+\ee^{2[f_K(\SP;x,\mathbf{t})+f_K(\SP;x,\mathbf{t})^*]/\epsilon} & \ii(-1)^K\widetilde{a}(\SP)^{*-1}\ee^{2f_K(\SP;x,\mathbf{t})^*/\epsilon}\\-\ii(-1)^K\widetilde{a}(\SP)^{-1}\ee^{2f_K(\SP;x,\mathbf{t})/\epsilon} & 1\end{pmatrix},
\end{aligned}
\label{eq:N-jump-R}
\end{multline}
where on the last line we used the identity $\widetilde{a}(\SP)\widetilde{a}(\SP^*)^*=1$.  

\begin{rem}
A similar substitution can made in the equivalent Riemann-Hilbert problem for $\widetilde{\mathbf{M}}^\updownarrow(\SP;x,\mathbf{t})$ based on the identity
\begin{equation}
c_n^\updownarrow(x,\mathbf{t})=-\ii (-1)^K\mathop{\mathrm{Res}}_{\SP=\ii\widetilde{s}_n}\frac{\ee^{-2f_K(\SP;x,\mathbf{t})/\epsilon}}{\widetilde{a}(\SP)}.
\end{equation}
One sets 
\begin{equation}
\mathbf{N}^\updownarrow(\SP;x,\mathbf{t}):=\widetilde{\mathbf{M}}^\updownarrow(\SP;x,\mathbf{t})\begin{pmatrix}1 & \ii (-1)^K\widetilde{a}(\SP)^{-1}\ee^{-2f_K(\SP;x,\mathbf{t})/\epsilon}\\0 & 1\end{pmatrix},\quad \SP\in D,
\end{equation}
and then defines $\mathbf{N}^\updownarrow(\SP;x,\mathbf{t}):=\sigma_2\mathbf{N}^\updownarrow(\SP^*;x,\mathbf{t})^*\sigma_2=\mathbf{N}^\updownarrow(\SP;x,\mathbf{t})^{-\dagger}$ for $\SP\in D^*$ to maintain Schwarz symmetry, and sets $\mathbf{N}^\updownarrow(\SP;x,\mathbf{t}):=\widetilde{\mathbf{M}}^\updownarrow(\SP;x,\mathbf{t})$ for $\SP$ outside the closure of $D\cup D^*$.  This defines a matrix function of $\SP$ that is analytic except on the three arcs $\partial D\setminus [-L,L]$, $(-L,L)$, and $\partial D^*\setminus [-L,L]$, and that tends to the identity as $\SP\to\infty$.  The jump conditions across the three arcs read
\begin{equation}
\mathbf{N}^\updownarrow_+(\SP;x,\mathbf{t})=\mathbf{N}^\updownarrow_-(\SP;x,\mathbf{t})\begin{pmatrix}1&-\ii(-1)^K\widetilde{a}(\SP)^{-1}\ee^{-2f_K(\SP;x,\mathbf{t})/\epsilon}\\0 & 1\end{pmatrix},\quad\SP\in\partial D\setminus[-L,L],
\end{equation}
\begin{equation}
\mathbf{N}^\updownarrow_-(\SP;x,\mathbf{t})=\mathbf{N}^\updownarrow_-(\SP;x,\mathbf{t})\begin{pmatrix}1 & 0\\\ii (-1)^K\widetilde{a}(\SP^*)^{*-1}\ee^{-2f_K(\SP^*;x,\mathbf{t})^*/\epsilon} & 1\end{pmatrix},\quad\SP\in\partial D^*\setminus[-L,L],
\end{equation}
and, for $\SP\in (-L,L)$,
\begin{multline}
\mathbf{N}^\updownarrow_+(\SP;x,\mathbf{t})\\
\begin{aligned}
&=\mathbf{N}^\updownarrow_-(\SP;x,\mathbf{t})\begin{pmatrix}1&0\\-\ii (-1)^K\widetilde{a}(\SP)^{*-1}\ee^{-2f_K(\SP;x,\mathbf{t})^*/\epsilon} & 1\end{pmatrix}\begin{pmatrix}1&\ii (-1)^K\widetilde{a}(\SP)^{-1}\ee^{-2f_K(\SP;x,\mathbf{t})/\epsilon} \\0 & 1\end{pmatrix}\\
&=\mathbf{N}^\updownarrow_-(\SP;x,\mathbf{t})\begin{pmatrix}1 & \ii (-1)^K\widetilde{a}(\SP)^{-1}\ee^{-2f_K(\SP;x,\mathbf{t})/\epsilon}\\
-\ii(-1)^K\widetilde{a}(\SP)^{*-1}\ee^{-2f_K(\SP;x,\mathbf{t})^*/\epsilon} & 1+\ee^{-2[f_K(\SP;x,\mathbf{t})+f_K(\SP;x,\mathbf{t})^*]/\epsilon}\end{pmatrix}.
\end{aligned}
\end{multline}
\end{rem}

Note that according to Propositions~\ref{prop:Psi-even-analytic} and \ref{prop:mu-odd-analytic}, the polynomials $\phaseint(\SP)$ and $\tailint(\SP)$ have coefficients that are real and imaginary, respectively, and it therefore follows that
\begin{equation}
f_K(\SP;x,\mathbf{t})+f_K(\SP;x,\mathbf{t})^*=0,\quad\SP\in (-L,L).
\label{eq:fK-imaginary}
\end{equation}
The jump matrix in the jump condition \eqref{eq:N-jump-R} naturally arises as a ``UL'' product; however it also admits a factorization of ``LU'' type.  Indeed, using \eqref{eq:fK-imaginary} and again taking into account that $\widetilde{a}(\SP)\widetilde{a}(\SP^*)^*=1$ it follows that for $\SP\in (-L,L)$,
\begin{multline}
\begin{pmatrix} 1+\ee^{2[f_K(\SP;x,\mathbf{t})+f_K(\SP;x,\mathbf{t})^*]/\epsilon} & \ii(-1)^K\widetilde{a}(\SP)^{*-1}\ee^{2f_K(\SP;x,\mathbf{t})^*/\epsilon}\\-\ii(-1)^K\widetilde{a}(\SP)^{-1}\ee^{2f_K(\SP;x,\mathbf{t})/\epsilon} & 1\end{pmatrix}\\{}=\begin{pmatrix}2^{1/2} & 0\\-2^{-1/2}\ii(-1)^K\widetilde{a}(\SP)^{-1}\ee^{2f_K(\SP;x,\mathbf{t})/\epsilon} & 2^{-1/2}\end{pmatrix}\begin{pmatrix}2^{1/2} & 2^{-1/2}\ii (-1)^K\widetilde{a}(\SP)^{*-1}\ee^{2f_K(\SP;x,\mathbf{t})^*/\epsilon}\\0 & 2^{-1/2}\end{pmatrix}.
\label{eq:N-jump-R-refactor}
\end{multline}
To exploit this alternate factorization, note that $\widetilde{a}(\SP)^{-1}$ is analytic for $\SP\in D^*$, as its poles lie in the domain $D$, and $f(\SP;x,\mathbf{t})$ and $f(\SP^*;x,\mathbf{t})^*$ are polynomials in $\SP$.  Therefore, defining a matrix function $\mathbf{O}(\SP;x,t)$ by 
\begin{equation}
\mathbf{O}(\SP;x,\mathbf{t}):=\mathbf{N}(\SP;x,\mathbf{t})\begin{pmatrix}2^{1/2} & 2^{-1/2}\ii (-1)^K\widetilde{a}(\SP^*)^{*-1}\ee^{2f_K(\SP^*;x,\mathbf{t})^*/\epsilon}\\0 & 2^{-1/2}\end{pmatrix}^{-1},\quad\SP\in D,
\end{equation}
\begin{equation}
\mathbf{O}(\SP;x,\mathbf{t}):=\mathbf{N}(\SP;x,\mathbf{t})\begin{pmatrix}2^{1/2} & 0\\-2^{-1/2}\ii(-1)^K\widetilde{a}(\SP)^{-1}\ee^{2f_K(\SP;x,\mathbf{t})/\epsilon} & 2^{-1/2}\end{pmatrix},\quad\SP\in D^*,
\end{equation}
and by $\mathbf{O}(\SP;x,\mathbf{t}):=\mathbf{N}(\SP;x,\mathbf{t})$ in the exterior domain, we see easily that $\mathbf{O}(\SP;x,\mathbf{t})$ is analytic where defined, and by comparing with \eqref{eq:N-jump-R} and \eqref{eq:N-jump-R-refactor} we see that $\mathbf{O}_+(\SP;x,\mathbf{t})=\mathbf{O}_-(\SP;x,\mathbf{t})$ for $\SP\in (-L,L)$.  Therefore an argument based on Morera's theorem shows that $\mathbf{O}(\SP;x,\mathbf{t})$ may be considered to be analytic in the interior of the closure of $D\cup D^*$, so its only jump occurs across the non-real arcs of $\partial D$ and $\partial D^*$, which form a circle $C$ containing all of the points $\SP=\pm\ii \widetilde{s}_j$, $j=0,\dots,N-1$ in its interior.  Taking the circle $C$ to have clockwise orientation, a computation shows that the jump of $\mathbf{O}(\SP;x,\mathbf{t})$ across $C$ takes the same analytic form regardless of whether $\SP$ is in the lower or upper half-plane, specifically 
\begin{multline}
\mathbf{O}_+(\SP;x,\mathbf{t})\\
\begin{aligned}
&=\mathbf{O}_-(\SP;x,\mathbf{t})\begin{pmatrix}2^{1/2} & 2^{-1/2}\ii (-1)^K\widetilde{a}(\SP^*)^{*-1}\ee^{2f_K(\SP^*;x,\mathbf{t})^*/\epsilon}\\0 & 2^{-1/2}\end{pmatrix}\begin{pmatrix}1 & 0\\
\ii (-1)^K \widetilde{a}(\SP)^{-1}\ee^{2f_K(\SP;x,\mathbf{t})/\epsilon} & 1
\end{pmatrix}\\
&=\mathbf{O}_-(\SP;x,\mathbf{t})\begin{pmatrix}2^{-1/2} & 2^{-1/2}\ii(-1)^K\widetilde{a}(\SP^*)^{*-1}\ee^{2f_K(\SP^*;x,\mathbf{t})^*/\epsilon}\\
2^{-1/2}\ii (-1)^K\widetilde{a}(\SP)^{-1}\ee^{2f_K(\SP;x,\mathbf{t})/\epsilon} & 2^{-1/2}\end{pmatrix}
\end{aligned}
\label{eq:Ojump}
\end{multline}
for $\SP\in C$.

Now fix a positive scaling factor $\nu>0$ to be determined, and consider the asymptotic behavior of the function $\widetilde{a}(\nu\epsilon^{-1}\Lambda)$ as $\epsilon=\epsilon_N\to 0$ with $|\Lambda|=1$ fixed.  We obtain
\begin{equation}
\begin{split}
\widetilde{a}(\nu\epsilon^{-1}\Lambda)&=\prod_{j=0}^{N-1}\frac{1-\ii\nu^{-1}\Lambda^{-1}\epsilon\widetilde{s}_j}{1+\ii\nu^{-1}\Lambda^{-1}\epsilon \widetilde{s}_j}\\
&=\exp\left(\sum_{j=0}^{N-1}\log\left(\frac{1-\ii\nu^{-1}\Lambda^{-1}\epsilon\widetilde{s}_j}{1+\ii\nu^{-1}\Lambda^{-1}\epsilon \widetilde{s}_j}\right)\right)\\
&=\exp\left(-\sum_{j=0}^{N-1}\left[2\ii\nu^{-1}\Lambda^{-1}\epsilon\widetilde{s}_j+\mathcal{O}(\epsilon^3\widetilde{s}_j^3)\right]\right)\\
&=\exp\left(-2\ii\nu^{-1}\Lambda^{-1}\sum_{j=0}^{N-1}\epsilon\widetilde{s}_j\right)(1+\mathcal{O}(\epsilon^2)),
\end{split}
\label{eq:a-expand}
\end{equation}
where we used only the facts that $0<\widetilde{s}_j<A_\mathrm{max}$ for all $j$ and $\epsilon$, and that $N\epsilon=\mathcal{O}(1)$.  Recalling \eqref{eq:Riemann-Sum}--\eqref{eq:Rewrite-L2-norm}, we substitute into \eqref{eq:a-expand} with the choice of $\nu$ given in \eqref{eq:nu-formula}
to obtain that 
\begin{equation}
\widetilde{a}(\nu\epsilon^{-1}\Lambda)=\ee^{-4\ii\Lambda^{-1}}\left(1+ \bigo{\eps^2} \right), 
\quad\epsilon\to 0
\end{equation}
holds uniformly for $|\Lambda|=1$.  
Next, observe that
\begin{equation}
\begin{split}
2f_K(\SP;x,\mathbf{t})&=\ii (2K+1)\phaseint(\SP)+\tailint(\SP) + 2\ii\left(\SP x + \sum_{n=2}^M\SP^nt_n\right)\\
&=\ii\left(\sum_{p=0}^{\mathcal{P}}(2K+1)\phaseint_p\SP^{2p} +\sum_{q=1}^{\mathcal{Q}}\tailint_q\SP^{2q-1} + 2 x\SP +\sum_{m=2}^M2t_m\SP^m\right).
\end{split}
\end{equation}
Consequently, setting $\SP=\nu\eps^{-1}\Lambda$ and replacing $x$ and $\mathbf{t}$ according to the left-hand side of \eqref{eq:hierarchy-focus-approx}, we obtain
\begin{multline}
2f_K\left(\frac{\nu\Lambda}{\eps};x^\circ+\frac{\eps^2}{\nu}X,\mathbf{t}^\circ+\left(\frac{\eps^3}{\nu^2}T_2,\frac{\eps^4}{\nu^3}T_3,\dots,\frac{\eps^{M+1}}{\nu^M}T_M\right)\right)\\{}=
\ii(2K+1)\phaseint_0
+2\ii\eps\left(X\Lambda+\sum_{m=2}^MT_m\Lambda^m\right).
\label{eq:exponent-expand}
\end{multline}

We deal with the constant term $\ii (2K+1)\phaseint_0$ in \eqref{eq:exponent-expand} along with the factors $\ii (-1)^K$ in the jump condition \eqref{eq:Ojump} by introducing one more transformation, a diagonal constant conjugation:
\begin{equation}
\mathbf{P}(\SP;x,\mathbf{t}):=\ee^{-\ii (2K+1)(\pi-2\phaseint_0/\eps)\sigma_3/4}\mathbf{O}(\SP;x,\mathbf{t})\ee^{\ii (2K+1)(\pi-2\phaseint_0/\eps)\sigma_3/4}.
\end{equation}
Note that $\mathbf{P}(\SP;x,\mathbf{t})$ is analytic for $\SP\in\mathbb{C}\setminus C$ and $\mathbf{P}(\SP;x,\mathbf{t})\to \mathbb{I}$ as $\SP\to\infty$.
Next, we specify the radius $L$ of the circle $C$ as $L=\nu/\epsilon$.  In terms of the variable $\Lambda$, this corresponds to $|\Lambda|=1$.
The jump condition across this circle satisfied by $\mathbf{P}(\SP;x,\mathbf{t})$ then takes the form
\begin{equation}
\mathbf{P}_+(\SP;x,\mathbf{t})=\mathbf{P}_-(\SP;x,\mathbf{t})\left[\mathbf{R}_-(\Lambda)^{-1}\mathbf{R}_+(\Lambda) + \bigo{\eps}\right],\quad |\Lambda|=1,
\end{equation}
where the matrix $\mathbf{R}_-(\Lambda)^{-1}\mathbf{R}_+(\Lambda)$ is defined in \eqref{eq:RWIO-jump}, and the error estimate is uniform for $|\Lambda|=1$ and for $(X,T_2,T_3,\dots,T_M)$ bounded in $\mathbb{R}^M$.
The conclusion of small-norm theory is then that 
\begin{equation}
\mathbf{P}\left(\frac{\nu\Lambda}{\eps};x^\circ+\frac{\eps^2}{\nu}X,\mathbf{t}^\circ + \left(\frac{\eps^3}{\nu^2}T_2,\frac{\eps^4}{\nu^3}T_3,\dots,\frac{\eps^{M+1}}{\nu^M}T_M\right)\right) = 
\mathbf{R}(\Lambda;X,T_2,T_3,\dots,T_M) + \bigo{\eps}
\label{eq:P-R-compare}
\end{equation}
as $\eps\to 0$, where $\mathbf{R}(\Lambda;X,T_2,T_3,\dots,T_M)$ is the solution of Riemann-Hilbert Problem~\ref{rhp:RWIO} and the error is uniform for bounded $(X,T_2,T_3,\dots,T_M)$.  Analogous formul\ae\ hold for each of the coefficients of the Laurent expansion of both sides in descending powers of $\Lambda$.
Since the semiclassical soliton ensemble $\widetilde{\psi}(x,\mathbf{t})$ is encoded in $\widetilde{\mathbf{M}}(\SP;x,\mathbf{t})$ by \eqref{eq:psitilde-reconstruct}, since $\mathbf{O}(\SP;x,\mathbf{t})$ agrees with $\widetilde{\mathbf{M}}(\SP;x,\mathbf{t})$ in a neighborhood of $\SP=\infty$, 
and since $O_{12}(\SP;x,\mathbf{t})=(-1)^K\ii\ee^{-\ii\phaseint_0/\eps}P_{12}(\SP;x,\mathbf{t})$, we get that
\begin{equation}
\widetilde{\psi}(x,\mathbf{t})=(-1)^K\ii\ee^{-\ii\phaseint_0/\epsilon}2\ii\lim_{\SP\to\infty}\SP P_{12}(\SP;x,\mathbf{t})=(-1)^K\ii\ee^{-\ii\phaseint_0/\eps}\frac{\nu}{\eps}2\ii\lim_{\Lambda\to\infty}\Lambda P_{12}\left(\frac{\nu\Lambda}{\eps};x,\mathbf{t}\right).
\end{equation}
Therefore, using the Laurent expansion of \eqref{eq:P-R-compare} and combining with \eqref{eq:RWIO-extract} yields 
\eqref{eq:hierarchy-focus-approx} and completes the proof of Theorem~\ref{thm:multi-time}.

The proof of Theorem~\ref{thm:mixture} also follows the same reasoning, except that since the time variables $t_2,t_3,\dots,t_M$ are in proportion by $t_m=a_mt$, taking $t=\epsilon^{M+1}T_M/(a_M\nu^M)$ forces the lower-indexed rescaled times $T_m$ for $m=2,\dots,M-1$ to be small of order $T_m=\bigo{\epsilon^{M-m}}$.  Therefore we may replace $\mathbf{R}(\Lambda;X,T_2,T_3,\dots,T_M)$ by $\mathbf{R}(\Lambda;X,0,0,\dots,0,T_M)$ at the cost of an additional error term proportional to $\epsilon$.  The only remaining part of the proof is to observe the conditions on the flow mixture coefficients $a_2,a_3,\dots,a_M$ such that the line parametrized by $(a_2t,a_3t,\dots,a_Mt)\in\mathbb{R}^{M-1}$ meets all (in the case that $\tailint(\SP)$ is a linear function and the even coefficients are correctly chosen) or one (in the case that $\tailint(\SP)$ has nonlinear terms and all the coefficients are correctly chosen given $K\in\mathbb{Z}$) of the focal points.

%% file: sec-proofs.tex

Here we collect proofs of the results stated in Section~\ref{sec-prop-YT} that describe the local asymptotic behavior of the functions $Y_\epsilon(\SP)$ and $T_\epsilon(\SP)$ in different regions of the complex plane in the limit $\eps \downarrow 0$. For the convenience of the reader the statements of each result proved below is preceded by its (re)statement.

\Yexterior*

\begin{proof}[Proof of Proposition~\ref{prop:Y-outside}]
\label{proof:Y-outside}
\hypertarget{proof-Y-outside}{For} $\SP$ in the indicated domain, in particular we have $\imag\{\SP\}>0$ with $-\ii\SP\not\in [0,A_\mathrm{max}]$, in which case
\begin{equation}
\epsilon\log(Y_\epsilon(\SP)) = \int_0^{A_\mathrm{max}}\log\left(\frac{\SP-\ii s}{\SP+\ii s}\right)\rho(s)\,\dd s-\epsilon\sum_{n=0}^{N-1}
\log\left(\frac{\SP-\ii s_n}{\SP+\ii s_n}\right).
\label{eq:epsilon-log-Y-pre}
\end{equation}
Making a substitution $s\mapsto \rr$ by using \eqref{eq:t-and-s-defs}, this becomes:
\begin{equation}
\epsilon\log(Y_\epsilon(\SP))=\int_0^{\phaseint(0)/\pi}\log\left(\frac{\SP-\ii s(\rr)}{\SP+\ii s(\rr)}\right)\,\dd \rr-\epsilon\sum_{n=0}^{N-1}\log\left(\frac{\SP-\ii s(\rr_n)}{\SP+\ii s(\rr_n)}\right),\quad \rr_n=(n+\tfrac{1}{2})\epsilon.
\label{eq:epsilon-log-Y}
\end{equation}
Note that for convenience, in \eqref{eq:epsilon-log-Y-pre}--\eqref{eq:epsilon-log-Y} we are labeling the sample points $s_n$ and $\rr_n$ in reverse order, increasing with $n$.  The main thrust of the proof is to express the integrand in \eqref{eq:epsilon-log-Y} as a sum of a function with an integrable second derivative, to which a standard theorem concerning Riemann sum approximation applies, and a more singular term that needs special treatment.  The latter term is more singular because unlike in \cite{BaikKMM07} where the analogue of $s(\rr)$ vanishes linearly at $\rr=0$, here $s(\rr)$ vanishes like $\rr^{1/2}$ by \eqref{eq:s-in-terms-of-v}.

To this end, we begin with the convergent series
\begin{equation}
\log\left(\frac{\SP-\ii s(\rr)}{\SP+\ii s(\rr)}\right)=\sum_{k=0}^\infty \kappa_k\left(\frac{s(\rr)}{\SP}\right)^{2k+1},\quad 0<s(\rr)<|\SP|,\quad\kappa_0=-2\ii.
\label{eq:log-series}
\end{equation}
Of course, although this series on the right-hand side is only convergent for $s(\rr)<|\SP|$, the sum of the series appearing on the left-hand side is analytic 
for $0<\rr<\phaseint(0)/\pi$.  We claim that the function $F(\rr;\SP)$ defined by
\begin{equation}
F(\rr;\SP):=\log\left(\frac{\SP-\ii s(\rr)}{\SP+\ii s(\rr)}\right)-\kappa_0v_0\frac{\sqrt{\rr}}{\SP}
\label{eq:F-t-zeta}
\end{equation}
has a second derivative with respect to $\rr$ that is absolutely integrable on $(0,\phaseint(0)/\pi)$ whenever $\imag\{\SP\}>0$ and $-\ii \SP\not\in [0,A_\mathrm{max}]$.   Indeed, comparing with \eqref{eq:log-series}, we can write $F(\rr;\SP)$ in the form
\begin{equation}
F(\rr;\SP)=\rr^{3/2}G(\rr;\SP),
\end{equation}
where $G(\rr;\SP)$ is analytic at $\rr=0$ with $G(0;\SP)\neq 0$, and $G(\diamond;\SP)$ has an analytic continuation to $0\le \rr<\phaseint(0)/\pi$.  Moreover, the condition $A''(x_0)<0$ implies that the analyticity of $G(\rr;\SP)$
extends to the endpoint $\rr=\phaseint(0)/\pi$.  The integrability of $F_{\rr\rr}(\rr;\SP)$ on $(0,\phaseint(0)/\pi)$ then follows from that of $\rr^{-1/2}$ at $\rr=0$.  In more detail, 
we compute explicitly
\begin{equation}
\frac{\dd^2F}{\dd \rr^2}(\rr;\SP)=\frac{4\ii z(\rr)z'(\rr)^2}{(1+z(\rr)^2)^2}+\frac{\kappa_0v_0}{4\SP \rr^{3/2}}-\frac{2\ii z''(\rr)}{1+z(\rr)^2},\quad z(\rr):=\frac{s(\rr)}{\SP}.
\end{equation}
Since from \eqref{eq:s-in-terms-of-v} we have
\begin{gather}
s(\rr)=\rr^{1/2}v(\rr),\quad s'(\rr)=\tfrac{1}{2}\rr^{-1/2}v(\rr)+\rr^{1/2}v'(\rr),\\
s''(\rr)=-\tfrac{1}{4}\rr^{-3/2}v(\rr)+\rr^{-1/2}v'(\rr)+\rr^{1/2}v''(\rr),  \nonumber
\end{gather}
using the fact that $v$ is analytic on $[0,\phaseint(0)/\pi]$ with $v(0)>0$, we get
\begin{equation}
\begin{split}
\frac{\kappa_0v_0}{4\SP \rr^{3/2}}-\frac{2\ii z''(\rr)}{1+z(\rr)^2}&=
\frac{\ii}{2\SP \rr^{1/2}}\frac{\rr^{-1}(v(\rr)-v_0)-4v'(\rr)-4\rr v''(\rr)}{1+z(\rr)^2}-\frac{\ii}{2\SP^3\rr^{1/2}}\frac{v_0v(\rr)^2}{1+z(\rr)^2}\\
&=\bigo{\rr^{-1/2}}.
\end{split}
\label{eq:some-terms}
\end{equation}
Similarly,
\begin{equation}
\begin{split}
\frac{4\ii z(\rr)z'(\rr)^2}{(1+z(\rr)^2)^2}&=\frac{\ii}{\SP^3\rr^{1/2}}\frac{v(\rr)^3 + 4\rr v(\rr)^2v'(\rr) +4\rr^2v(\rr)v'(\rr)^2}{(1+z(\rr)^2)^2}\\
&=\bigo{\rr^{-1/2}}.
\end{split}
\label{eq:some-more-terms}
\end{equation}
The estimates in \eqref{eq:some-terms}--\eqref{eq:some-more-terms} hold uniformly as $\SP$ varies in $D_\sigma$.
It follows that 
\begin{equation}
\int_0^{\phaseint(0)/\pi}|F_{\rr\rr}(\rr;\SP)|\,\dd \rr=\bigo{|\SP|^{-1}},\quad\SP\in D_\sigma.
\label{eq:Ftt-L1-outside}
\end{equation}  

Rewriting \eqref{eq:epsilon-log-Y} using \eqref{eq:F-t-zeta}
\begin{equation}
\epsilon\log(Y_\epsilon(\SP))=\frac{\kappa_0v_0}{\SP}\left[\int_0^{\phaseint(0)/\pi} \rr^{1/2}\,\dd \rr-
\epsilon\sum_{n=0}^{N-1}\rr_n^{1/2}\right] +\left[\int_0^{\phaseint(0)/\pi}F(\rr;\SP)\,\dd \rr - 
\epsilon\sum_{n=0}^{N-1}F(\rr_n;\SP)\right]
\label{eq:epsLogY-ito-F}
\end{equation}
and applying the following basic inequality\footnote{The inequality \eqref{eq:midpoint-rule} also holds if the function $f$ is defined on $(0,\infty)$ with the limits of the integrals adjusted accordingly and with the upper limit of the summation taken to be $k=\infty$.} from the theory of ``midpoint rule'' Riemann sums:
\begin{equation}
\left|\int_0^1 f(x)\,\dd x -\frac{1}{N}\sum_{k=1}^N f\left(\frac{k-\tfrac{1}{2}}{N}\right)\right|\le
\frac{1}{2N^2}\int_0^1|f''(x)|\,\dd x,
\label{eq:midpoint-rule}
\end{equation}
now allows the second term on the right-hand side of \eqref{eq:epsLogY-ito-F} to be easily estimated from \eqref{eq:Ftt-L1-outside}.
Here we use the fact that the sample points $\rr_n$ are equally spaced with spacing $\epsilon$ and centered as midpoints of $N$ equal-length subintervals of $(0,\phaseint(0)/\pi)$; thus
the second term 
on the right-hand side of \eqref{eq:epsLogY-ito-F} is $\bigo{N^{-2}}=\bigo{\epsilon^2}$ as a consequence of \eqref{eq:Ftt-L1-outside}.
It therefore only remains to approximate the first term on the right-hand side of \eqref{eq:epsLogY-ito-F}, for which we calculate directly:
\begin{equation}
\begin{split}
\int_0^{\phaseint(0)/\pi}\rr^{1/2}\,\dd \rr-\epsilon\sum_{n=0}^{N-1}\rr_n^{1/2}&=\int_0^{N\epsilon}\rr^{1/2}\,\dd \rr -\epsilon\sum_{n=0}^{N-1}\rr_n^{1/2}\\
&=\epsilon^{3/2}\left(\tfrac{2}{3}N^{3/2}-\sum_{n=0}^{N-1}(n+\tfrac{1}{2})^{1/2}\right).
\end{split}
\label{eq:first-term}
\end{equation}
Now, for a function $f(x)$ smooth on $[0,N-1]$, the first-order Euler-Maclaurin summation formula reads
\begin{equation}
\sum_{n=0}^{N-1}f(n)=\int_0^{N-1}f(x)\,\dd x +\frac{f(N-1)+f(0)}{2}+\int_0^{N-1}f'(x)(x-\lfloor x\rfloor-\tfrac{1}{2})\,\dd x.
\label{eq:zeta-first}
\end{equation}
Applying this to the function $f(x):=(x+\tfrac{1}{2})^{1/2}$ which is smooth for $x>0$ gives
\begin{equation}
\begin{split}
\sum_{n=0}^{N-1}(n+\tfrac{1}{2})^{1/2}&=\int_0^{N-1}(x+\tfrac{1}{2})^{1/2}\,\dd x +\frac{(N-\tfrac{1}{2})^{1/2}-(\tfrac{1}{2})^{1/2}}{2}+\frac{1}{2}\int_0^{N-1}\frac{x-\lfloor x\rfloor-\tfrac{1}{2}}{(x+\tfrac{1}{2})^{1/2}}\,\dd x\\
&=\tfrac{2}{3}N^{3/2}+\frac{1}{6\sqrt{2}}+\bigo{N^{-1/2}}+\frac{1}{2}\int_0^{N-1}\frac{x-\lfloor x\rfloor-\tfrac{1}{2}}{(x+\tfrac{1}{2})^{1/2}}\,\dd x ,\quad N\to\infty.
\end{split}
\label{eq:sum-of-square-roots}
\end{equation}
The last integral has a limit as $N\to\infty$ because for any positive integer $m\in\mathbb{Z}$,
\begin{equation}
\begin{split}
\int_{m-1}^m\frac{x-\lfloor x\rfloor -\tfrac{1}{2}}{(x+\tfrac{1}{2})^{1/2}}\dd x &= \int_{m-1}^m\frac{x+\tfrac{1}{2}-m}{(x+\tfrac{1}{2})^{1/2}}\,\dd x \\ 
&= \tfrac{2}{3}\left((m+\tfrac{1}{2})^{3/2}-(m-\tfrac{1}{2})^{3/2}\right)-2m\left((m+\tfrac{1}{2})^{1/2}-(m-\tfrac{1}{2})^{1/2}\right)\\
&=\bigo{m^{-3/2}},\quad m\to\infty.
\end{split}
\end{equation}
To identify the limit, we use this result to write
\begin{equation}
\begin{split}
\frac{1}{2}\int_0^{N-1}\frac{x-\lfloor x\rfloor-\tfrac{1}{2}}{(x+\tfrac{1}{2})^{1/2}}\,\dd x&=- \frac{1}{2}\int_{-1/2}^0\frac{x-\lfloor x\rfloor-\tfrac{1}{2}}{(x+\tfrac{1}{2})^{1/2}}\,\dd x+\frac{1}{2}\int_{-1/2}^{\infty}\frac{x-\lfloor x\rfloor-\tfrac{1}{2}}{(x+\tfrac{1}{2})^{1/2}}\,\dd x +\bigo{N^{-1/2}}\\
&=-\frac{1}{2}\int_{-1/2}^0(x+\tfrac{1}{2})^{1/2}\,\dd x +\frac{1}{2}\int_{-1/2}^{\infty}\frac{x-\lfloor x\rfloor-\tfrac{1}{2}}{(x+\tfrac{1}{2})^{1/2}}\,\dd x+\bigo{N^{-1/2}}\\
&=-\frac{1}{6\sqrt{2}}+\frac{1}{2}\int_{-1/2}^{\infty}\frac{x-\lfloor x\rfloor-\tfrac{1}{2}}{(x+\tfrac{1}{2})^{1/2}}\,\dd x+\bigo{N^{-1/2}}
,\quad N\to\infty.
\end{split}
\label{eq:sawtooth-integral}
\end{equation}
According to \cite[25.11.26]{DLMF}, the remaining integral is a Hurwitz zeta function:
\begin{equation}
\frac{1}{2}\int_{-1/2}^{\infty}\frac{x-\lfloor x\rfloor-\tfrac{1}{2}}{(x+\tfrac{1}{2})^{1/2}}\,\dd x=\zeta(-\tfrac{1}{2},\tfrac{1}{2}),
\label{eq:Hurwitz-Zeta}
\end{equation}
which, by \cite[25.11.11]{DLMF} can be written explicitly in terms of a Riemann zeta value:
\begin{equation}
\zeta(-\tfrac{1}{2},\tfrac{1}{2})=\left(\frac{1}{\sqrt{2}}-1\right)\zeta(-\tfrac{1}{2}).
\label{eq:Riemann-Zeta}
\end{equation}
Using \eqref{eq:sum-of-square-roots} with \eqref{eq:sawtooth-integral}--\eqref{eq:Riemann-Zeta} in \eqref{eq:first-term} gives
\begin{equation}
\int_0^{\phaseint(0)/\pi}\rr^{1/2}\,\dd \rr-\epsilon\sum_{n=0}^{N-1}\rr_n^{1/2}=\epsilon^{3/2}\left(1-\frac{1}{\sqrt{2}}\right)\zeta(-\tfrac{1}{2}) +\bigo{\epsilon^2},\quad \epsilon\to 0,
\label{eq:RS-square-root}
\end{equation}
because $N$ is inversely proportional to $\epsilon$.  Note that $\zeta(-\tfrac{1}{2})\approx -0.207886$.

Then, the first term on the right-hand side of \eqref{eq:epsLogY-ito-F} is $\kappa_0v_0(1-1/\sqrt{2})\zeta(-\frac{1}{2})\epsilon^{3/2}/\SP + \bigo{\epsilon^2}$ where the error estimate holds in the limit $\epsilon\to 0$ uniformly for $\SP\in D_\sigma$.  
Now we divide by $\epsilon$ in \eqref{eq:epsLogY-ito-F} and exponentiate to obtain
\[
Y_\epsilon(\SP)=1+\frac{\kappa_0v_0}{\SP}\left(1-\frac{1}{\sqrt{2}}\right)\zeta(-\tfrac{1}{2})\epsilon^{1/2}+\bigo{\epsilon}.
\]
The result then follows by recalling that $\kappa_0=-2\ii$.
\end{proof}

\Yzetasmall*

\begin{proof}
\hypertarget{proof:Yzetasmall}{Let} $\xi$ be related to $\SP$ by the relations $\SP=\ii s(\xi)$ or $\xi=(\phaseint(0)-\phaseint(\SP))/\pi$.  Thus, $\xi$ is an even analytic function of $\SP$ near $\SP=0$ with $\xi=\bigo{\SP^2}$ as $\SP\to 0$.  We may therefore define $\xi^{1/2}$ as an odd conformal mapping on a neighborhood of the origin with the property that it maps small positive imaginary values of $\SP$ to small positive values of $\xi^{1/2}$, and the conformal map defined in \eqref{eq:z-zero-define} is related by a rotation:  $\varphi_0(\SP)=\ii\xi^{1/2}$.
Consider the function
\begin{equation}
G^\pm(\rr;\SP):=\log\left(\frac{\ii s(\xi)\pm \ii s(\rr)}{\xi^{1/2}\pm \rr^{1/2}}\right),\quad 0<\rr<\phaseint(0)/\pi.
\label{Gpm-def}
\end{equation}
If $|\SP|$ is sufficiently small, $G^\pm(\rr;\SP)$ will be an analytic function of $\rr$ in the indicated open interval; indeed, since from \eqref{eq:s-in-terms-of-v} $s(\rr)=\rr^{1/2}v(\rr)$ with $v$ analytic at $\rr=0$ and $v(0)>0$, $\ee^{G^\pm(\rr;\SP)}$
is uniformly bounded and bounded away from zero, and the same is true of $\ee^{G^-(\rr;\SP)}$ even at $\rr=\xi$ should $\xi$ be positive real.  Note that
\begin{equation}
G^\pm(0;\SP)=\log\left(\frac{\ii s(\xi)}{\xi^{1/2}}\right) = \log(\ii v(\xi)),
\end{equation}
which is well-defined with imaginary part close to $\pi/2$ for $\SP$ small.  Then
\begin{equation}
\begin{split}
G^\pm(\rr;\SP)-G^\pm(0;\SP)&=\log\left(1+\left[\frac{\xi^{1/2}v(\xi)\pm \rr^{1/2}v(\rr)}{v(\xi)(\xi^{1/2}\pm \rr^{1/2})}-1\right]\right)\\
&=\log\left(1\mp \frac{\rr^{1/2}}{v(\xi)}\frac{v(\xi)-v(\rr)}{\xi^{1/2}\pm \rr^{1/2}}\right)\\
&=\log\left(1\mp \frac{\xi^{1/2}}{v(\xi)}\rr^{1/2}\frac{v(\xi)-v(\rr)}{\xi-\rr} + \frac{1}{v(\xi)}\rr\frac{v(\xi)-v(\rr)}{\xi-\rr}\right).
\end{split}
\end{equation}
This shows that $G^\pm(\rr;\SP)-G^\pm(0;\SP)$ has a convergent power series expansion about $\rr=0$ in integer powers of $\rr^{1/2}$.  If we isolate the leading term proportional to $\rr^{1/2}$, the remainder will have a second derivative that is absolutely integrable at $\rr=0$.  Therefore, we write
\begin{equation}
G^-(\rr;\SP)-G^+(\rr;\SP)=\frac{2\xi^{1/2}}{v(\xi)}\frac{v(\xi)-v(0)}{\xi}\rr^{1/2} + R(\rr;\SP)
\label{eq:Gminus-plus}
\end{equation}
and 
\begin{equation}
\int_0^{\phaseint(0)/\pi}|R_{\rr\rr}(\rr;\SP)|\,\dd \rr = \bigo{1},\quad \SP\to 0.
\end{equation}
It follows by Riemann sum approximation (cf. \eqref{eq:midpoint-rule} and \eqref{eq:RS-square-root}) that
\begin{multline}
\int_0^{\phaseint(0)/\pi}[G^-(\rr;\SP)-G^+(\rr;\SP)]\,\dd \rr-\epsilon\sum_{n=0}^{N-1}[G^-(\rr_n;\SP)-G^+(\rr_n;\SP)] \\
\begin{aligned}
{}&= \frac{2\xi^{1/2}}{v(\xi)}\frac{v(\xi)-v(0)}{\xi}\left(1-\frac{1}{\sqrt{2}}\right)\zeta(-\tfrac{1}{2})\epsilon^{3/2}+\bigo{\epsilon^2}\\
{}&=\mathcal{E}_0(\SP)\epsilon^{3/2} + \bigo{\epsilon^2}\\
\end{aligned}
\label{eq:RS-approximation}
\end{multline}
holds uniformly for $|\SP|$ sufficiently small.
Now we use these observations to rewrite $\log(Y_\epsilon(\SP))$ as follows (cf. \eqref{eq:epsilon-log-Y}):
\begin{multline}
\log(Y_\epsilon(\SP))=\frac{1}{\epsilon}\int_0^{\phaseint(0)/\pi}\log\left(\frac{\xi^{1/2}-\rr^{1/2}}{\xi^{1/2}+\rr^{1/2}}\right)\,\dd \rr -\sum_{n=0}^{N-1}\log\left(\frac{\xi^{1/2}-\rr_n^{1/2}}{\xi^{1/2}+\rr_n^{1/2}}\right)\\
{} + \mathcal{E}_0(\SP)\epsilon^{1/2} + \bigo{\epsilon},
\label{eq:logY}
\end{multline}
which holds as $\epsilon\downarrow 0$ uniformly for sufficiently small $|\SP|$.
Note that 
\begin{equation}
\log\left(\frac{\xi^{1/2}-\rr^{1/2}}{\xi^{1/2}+\rr^{1/2}}\right)=
\log\left(\frac{\rr^{1/2}-\xi^{1/2}}{\rr^{1/2}+\xi^{1/2}}\right) +
\begin{cases}
-\pi \ii, & 0<\arg(\SP)<\pi/2 \;\text{so}\;-\pi<\arg(\xi)<0,\\
\pi \ii, & \pi/2<\arg(\SP)<\pi\;\text{so}\;0<\arg(\xi)<\pi.
\end{cases}
\end{equation}
Therefore, since by \eqref{eq:ZS-epsilon-assumption} we have $\phaseint(0)=N\pi\epsilon$, \eqref{eq:logY} can be rewritten in both cases of $0<\arg(\SP)<\pi/2$ and $\pi/2<\arg(\SP)<\pi$ the same way:
\begin{multline}
\log(Y_\epsilon(\SP))=\frac{1}{\epsilon}\int_0^{\phaseint(0)/\pi}\log\left(\frac{\rr^{1/2}-\xi^{1/2}}{\rr^{1/2}+\xi^{1/2}}\right)\,\dd \rr - \sum_{n=0}^{N-1}\log\left(\frac{\rr_n^{1/2}-\xi^{1/2}}{\rr_n^{1/2}+\xi^{1/2}}\right)\\
{}+\mathcal{E}_0(\SP)\epsilon^{1/2} + \bigo{\epsilon}.
\label{eq:logY-again}
\end{multline}
Observing that adding $2\xi^{1/2}\rr^{-1/2}$ to the integrand results in absolute convergence both at $\rr=0$ and $\rr=\infty$, 
\begin{multline}
\frac{1}{\epsilon}\int_0^{\phaseint(0)/\pi}\log\left(\frac{\rr^{1/2}-\xi^{1/2}}{\rr^{1/2}+\xi^{1/2}}\right)\,\dd \rr
\\
\begin{aligned}
&= -\frac{2\xi^{1/2}}{\epsilon}\int_0^{\phaseint(0)/\pi}\frac{\dd \rr}{\rr^{1/2}} + 
\frac{1}{\epsilon}\int_0^\infty\left[\log\left(\frac{\rr^{1/2}-\xi^{1/2}}{\rr^{1/2}+\xi^{1/2}}\right) +\frac{2\xi^{1/2}}{\rr^{1/2}}\right]\,\dd \rr \\
&\quad\quad\quad{}-\frac{1}{\epsilon}\int_{\phaseint(0)/\pi}^\infty
\left[\log\left(\frac{\rr^{1/2}-\xi^{1/2}}{\rr^{1/2}+\xi^{1/2}}\right) +\frac{2\xi^{1/2}}{\rr^{1/2}}\right]\,\dd \rr
\\
&=-\frac{4\phaseint(0)^{1/2}\xi^{1/2}}{\pi^{1/2}\epsilon}+\int_0^\infty\left[\log\left(\frac{\tau^{1/2}+\ii Z}{\tau^{1/2}-\ii Z}\right)-\frac{2\ii Z}{\tau^{1/2}}\right]\,\dd\tau\\
&\quad\quad\quad{}-\frac{1}{\epsilon}\int_{\phaseint(0)/\pi}^\infty
\left[\log\left(\frac{\rr^{1/2}-\xi^{1/2}}{\rr^{1/2}+\xi^{1/2}}\right) +\frac{2\xi^{1/2}}{\rr^{1/2}}\right]\,\dd \rr,
\end{aligned}
\end{multline}
where in the integral over $\rr\in (0,\infty)$ we have rescaled by $\rr=\epsilon\tau$ and we set $Z:=\ii\xi^{1/2}/\epsilon^{1/2}=\varphi_0(\SP)/\epsilon^{1/2}$.  Similarly, recalling $\rr_n:=\epsilon(n+\tfrac{1}{2})$,
\begin{multline}
\sum_{n=0}^{N-1}\log\left(\frac{\rr_n^{1/2}-\xi^{1/2}}{\rr_n^{1/2}+\xi^{1/2}}\right)
=-\frac{2\xi^{1/2}}{\epsilon^{1/2}}\sum_{n=0}^{N-1}\frac{1}{\sqrt{n+\tfrac{1}{2}}} 
+ \sum_{n=0}^\infty\left[\log\left(\frac{\sqrt{n+\tfrac{1}{2}}+\ii Z}{\sqrt{n+\tfrac{1}{2}}-\ii Z}\right)-\frac{2\ii Z}{\sqrt{n+\tfrac{1}{2}}}\right]
\\
{}- \sum_{n=N}^\infty\left[\log\left(\frac{\rr_n^{1/2}-\xi^{1/2}}{\rr_n^{1/2}+\xi^{1/2}}\right)+\frac{2\xi^{1/2}}{\rr_n^{1/2}}\right],
\end{multline}
so using
\begin{equation}
\begin{split}
\sum_{n=0}^{N-1}\frac{1}{\sqrt{n+\tfrac{1}{2}}}&=\int_0^N\frac{\dd x}{\sqrt{x}}+
\sum_{n=0}^{N-1}\left[\frac{1}{\sqrt{n+\tfrac{1}{2}}}-\int_n^{n+1}\frac{\dd x}{\sqrt{x}}\right]\\
&=2\sqrt{N}+\sum_{n=0}^{N-1}\left[\frac{1}{\sqrt{n+\tfrac{1}{2}}}-2\sqrt{n+1}+2\sqrt{n}\right]\\
&=2\sqrt{N} + \sum_{n=0}^\infty\left[\frac{1}{\sqrt{n+\tfrac{1}{2}}}-2\sqrt{n+1}+2\sqrt{n}\right] +
\bigo{N^{-3/2}}
\end{split}
\end{equation}
with $\phaseint(0)=N\pi\epsilon$ and $Z=\ii\xi^{1/2}/\epsilon^{1/2}$ we have that
\begin{multline}
\sum_{n=0}^{N-1}\log\left(\frac{\rr_n^{1/2}-\xi^{1/2}}{\rr_n^{1/2}+\xi^{1/2}}\right)=-\frac{4\phaseint(0)^{1/2}\xi^{1/2}}{\pi^{1/2}\epsilon} +
\sum_{n=0}^\infty\left[\log\left(\frac{\sqrt{n+\tfrac{1}{2}}+\ii Z}{\sqrt{n+\tfrac{1}{2}}-\ii Z}\right)-4\ii Z(\sqrt{n+1}-\sqrt{n})\right]\\
{}- \sum_{n=N}^\infty
\left[\log\left(\frac{\rr_n^{1/2}-\xi^{1/2}}{\rr_n^{1/2}+\xi^{1/2}}\right)+\frac{2\xi^{1/2}}{\rr_n^{1/2}}\right]+\bigo{\epsilon}
\end{multline}
holds in the limit $\epsilon\downarrow 0$ uniformly for bounded $\xi$.  Therefore, defining $k(Z):=k_1(Z)+k_2(Z)$, where
\begin{equation}
\begin{split}
k_1(Z)&:=\int_0^\infty\left[\log\left(\frac{\tau^{1/2}+\ii Z}{\tau^{1/2}-\ii Z}\right)-\frac{2\ii Z}{\tau^{1/2}}\right]\,\dd\tau, \\
k_2(Z)&:=- 
\sum_{n=0}^\infty\left[\log\left(\frac{\sqrt{n+\tfrac{1}{2}}+\ii Z}{\sqrt{n+\tfrac{1}{2}}-\ii Z}\right)-4\ii Z(\sqrt{n+1}-\sqrt{n})\right],
\end{split}
\label{eq:k1k2}
\end{equation}
we have 
\begin{multline}
\frac{1}{\epsilon}\int_0^{\phaseint(0)/\pi}\log\left(\frac{\rr^{1/2}-\xi^{1/2}}{\rr^{1/2}+\xi^{1/2}}\right)\,\dd \rr-\sum_{n=0}^{N-1}\log\left(\frac{\rr_n^{1/2}-\xi^{1/2}}{\rr_n^{1/2}+\xi^{1/2}}\right)=k(Z)+\bigo{\epsilon}\\
{}- \frac{1}{\epsilon}\int_{\phaseint(0)/\pi}^\infty\left[\log\left(\frac{\rr^{1/2}-\xi^{1/2}}{\rr^{1/2}+\xi^{1/2}}\right)+\frac{2\xi^{1/2}}{\rr^{1/2}}\right]\,\dd \rr +
\sum_{n=N}^\infty\left[\log\left(\frac{\rr_n^{1/2}-\xi^{1/2}}{\rr_n^{1/2}+\xi^{1/2}}\right)+\frac{2\xi^{1/2}}{\rr_n^{1/2}}\right].
\end{multline}
Assuming only that $\xi$ is sufficiently small independent of $\epsilon$, the terms on the second line are $\bigo{\epsilon}$ by Riemann sum approximation (cf. \eqref{eq:midpoint-rule}) because the integrand has a second derivative that is absolutely integrable on $(\phaseint(0)/\pi,\infty)$.  Combining this result with \eqref{eq:logY-again} and exponentiating yields
\begin{equation}
Y_\epsilon(\SP)=\ee^{k(Z)}\left(1+\mathcal{E}_0(\SP)\epsilon^{1/2}+\bigo{\epsilon}\right),\quad\epsilon\downarrow 0,
\quad Z:=\ii\frac{\xi^{1/2}}{\epsilon^{1/2}}=\frac{\varphi_0(\SP)}{\epsilon^{1/2}},
\end{equation}
holding only under the assumption that $|\SP|$ (and hence also $|\xi|$) is sufficiently small.
Note that $\mathcal{E}_0(\SP)$ is analytic in $\SP$ at $\SP=0$ and vanishes to first order at the origin (in fact, it is an odd function of $\SP$), because $v$ is analytic and nonvanishing.

Finally, note that $k_1(0)=0$ and, differentiating under the integral sign,
\begin{equation}
k_1'(Z)=-2\ii Z^2\int_0^\infty\frac{\dd\tau}{(\tau+Z^2)\sqrt{\tau}} = -2\ii Z^2\int_\mathbb{R}\frac{\dd w}{w^2+Z^2}=
-2\pi \ii Z\,\mathrm{sgn}(\re\{Z\}).
\end{equation}
Therefore $k_1(Z)=-\pi \ii Z^2\,\mathrm{sgn}(\re\{Z\})$.   
The proof is complete upon identifying $\mathcal{Y}_0(Z)$ with $\ee^{k(Z)}$ and observing that $Y_\epsilon(\SP)$ and $\mathcal{Y}_0(Z)$ take the common value of $1$ in the coincident limits $\SP\to 0$ and $Z\to 0$ respectively.
\end{proof}

\Yzeroasymp*

\begin{proof}
\hypertarget{proof:Yzeroasymp}{Using} the representation $\mathcal{Y}_0(Z)=\ee^{k_1(Z)}\ee^{k_2(Z)}=\ee^{-\ii\pi Z^2\mathrm{sgn}(\mathrm{Re}\{Z\})}\ee^{k_2(Z)}$, to obtain the claimed behavior as $Z\to 0$, it only remains to expand the function $k_2(Z)=-k_2(-Z)$ analytic at $Z=0$ in a Taylor series.  In particular,
by differentiation under the sum in \eqref{eq:k1k2},
\begin{equation}
k_2'(0)=-2\ii\sum_{n=0}^{\infty}\left[\frac{1}{\sqrt{n+\tfrac{1}{2}}}-2(\sqrt{n+1}-\sqrt{n})\right]=-2\ii\lim_{N\to\infty}\left[\sum_{n=0}^{N-1}\frac{1}{\sqrt{n+\frac{1}{2}}}-2\sqrt{N}\right],
\end{equation}
and then a computation along the lines of \eqref{eq:zeta-first}--\eqref{eq:Riemann-Zeta} gives $k_2'(0)=-2\ii (\sqrt{2}-1)\zeta(\tfrac{1}{2})$.  

To prove the large-$Z$ asymptotic behavior, let $\gamma>0$ be fixed, and suppose that $\epsilon^{1/2}|Z|\le\gamma$.  Then Proposition~\ref{prop:Y-zeta-small} implies that
\begin{equation}
\mathcal{Y}_0(Z)=Y_\epsilon(\varphi_0^{-1}(\epsilon^{1/2}Z))\left(1-\mathcal{E}_0(\varphi_0^{-1}(\epsilon^{1/2}Z))\epsilon^{1/2} + \bigo{\epsilon}\right),\quad\epsilon\to 0,
\end{equation}
where $\varphi_0^{-1}$ denotes the inverse conformal mapping to $\varphi_0$.  If we also assume that $\frac{1}{2}\gamma\le \epsilon^{1/2}|Z|$ and that $Z$ satisfies the indicated sectorial condition, then Proposition~\ref{prop:Y-outside} applies to $Y_\epsilon(\varphi_0^{-1}(\epsilon^{1/2}Z))$ and we obtain
\begin{multline}
\mathcal{Y}_0(Z)=\left(1-\frac{2\ii v_0}{\varphi_0^{-1}(\epsilon^{1/2}Z)}\left(1-\frac{1}{\sqrt{2}}\right)\zeta(-\tfrac{1}{2})\epsilon^{1/2}+\bigo{\epsilon}\right)\\
{}\cdot\left(1-\mathcal{E}_0(\varphi_0^{-1}(\epsilon^{1/2}Z))\epsilon^{1/2} + \bigo{\epsilon}\right),\quad\epsilon\to 0.
\end{multline}
By the inequality $\epsilon^{1/2}|Z|\le\gamma$ we can clearly write the right-hand side in the form $1+\bigo{Z^{-1}}$, and together with the inequality $\frac{1}{2}\gamma\le\epsilon^{1/2}|Z|$ we obtain that $\epsilon\to 0$ if and only if $|Z|\to\infty$.
\end{proof}

\propYtop*

\begin{proof}
\hypertarget{proof:Ytop}{Starting} from \eqref{eq:epsilon-log-Y}, we note the only difficulty in applying the sort of Riemann-sum arguments employed in the proof of Proposition~\ref{prop:Y-outside} 
arises when $\ii s(\rr)$ is close to $\SP\approx \ii A_\mathrm{max}$, or equivalently, for $\rr$ in a neighborhood of 
$\phaseint(0)/\pi$.  We therefore define a cutoff value $\rr_c\approx \phaseint(0)/(2\pi)$ by
\eq
\rr_c:=N_c\epsilon, \quad \text{where} \quad N_c:=\left[\frac{\phaseint(0)}{2\pi\epsilon}\right]=\left[\frac{N}{2}\right] \quad ([\cdot]=\text{integer part}),
\label{eq:Nc-define}
\endeq
and use it to rewrite \eqref{eq:epsilon-log-Y} in the form
\eq
\epsilon\log(Y_\epsilon(\SP)) = \Delta_{\text{down}}(\SP) + \Delta_{\text{up}}(\SP),
\endeq
where
\eq
\Delta_{\text{down}}(\SP):=\int_0^{\rr_c}\log\left(\frac{\SP-\ii s(\rr)}{\SP+\ii s(\rr)}\right)\,\dd \rr - \epsilon\sum_{n=0}^{N_c-1}\log\left(\frac{\SP-\ii s(\rr_n)}{\SP+\ii s(\rr_n)}\right).
\endeq
Recalling the definition \eqref{eq:F-t-zeta} of $F(\rr;\SP)$, 
we can rewrite this as
\eq
\Delta_{\text{down}}(\SP) = \frac{\kappa_0 v_0}{\SP}\left[\int_0^{\rr_c}\rr^{1/2}\,\dd \rr - \epsilon\sum_{n=0}^{N_c-1}\rr_n^{1/2}\right] + \left[\int_0^{\rr_c}F(\rr;\SP)\,\dd \rr - \epsilon\sum_{n=0}^{N_c-1}F(\rr_n;\SP)\right].
\endeq
Exactly as in  
\eqref{eq:epsLogY-ito-F}--\eqref{eq:RS-square-root}, we obtain after recalling $\kappa_0=-2\ii$ (see \eqref{eq:log-series}) and the definition \eqref{eq:E1-def},
 \eq
\Delta_\text{down}(\SP)=\mathcal{E}_1(\SP)\epsilon^{3/2} +\bigo{\epsilon^2}.
\endeq
From here we have 
\begin{multline}
\log(Y_\epsilon(\SP))=\frac{1}{\epsilon}\int_{\rr_c}^{\phaseint(0)/\pi}\log\left(\frac{\SP-\ii s(\rr)}{\SP+\ii s(\rr)}\right)\,\dd \rr - \sum_{n=N_c}^{N-1}\log\left(\frac{\SP-\ii s(\rr_n)}{\SP+\ii s(\rr_n)}\right)\\
{}+\mathcal{E}_1(\SP)\epsilon^{1/2}+\bigo{\epsilon},
\end{multline}
where the terms on the first line of the right-hand side are exactly $\Delta_\text{up}(\SP)/\epsilon$.
Next, we note
\eq
\log\left(\frac{\SP-\ii s(\rr)}{\SP+\ii s(\rr)}\right) = \log(-\ii(\SP-\ii s(\rr))) - \log(-\ii(\SP+\ii s(\rr))),
\endeq
where all branches are principal.  From \eqref{eq:midpoint-rule}, we 
see 
\eq
-\frac{1}{\epsilon}\int_{\rr_c}^{\phaseint(0)/\pi}\log(-\ii (\SP+\ii s(\rr)))\,\dd \rr + \sum_{n=N_c}^{N-1}\log(-\ii(\SP+\ii s(\rr_n))) = \bigo{\epsilon}
\endeq
holds uniformly for $\SP$ in a neighborhood of $\ii A_{\max}$.  Combining the last three 
equations gives
\begin{multline}
\log(Y_\epsilon(\SP))=\frac{1}{\epsilon}\int_{\rr_c}^{\phaseint(0)/\pi}\log\left(-\ii (\SP-\ii s(\rr))\right)\,\dd \rr - \sum_{n=N_c}^{N-1}\log\left(-\ii (\SP-\ii s(\rr_n))\right)\\
{}+\mathcal{E}_1(\SP)\epsilon^{1/2}+\bigo{\epsilon}.
\end{multline}

As in the proof of Proposition~\ref{prop:Y-zeta-small}, let $\xi$ be related to $\SP$ by the equivalent relations $\SP=\ii s(\xi)$ or $\xi=(\phaseint(0)-\phaseint(\SP))/\pi$, implying that the conformal map $\SP\mapsto\varphi_1(\SP)$ can be represented as $\varphi_1(\SP)=\xi-\phaseint(0)/\pi$.  Therefore,
\begin{multline}
\log(Y_\epsilon(\SP))=\frac{1}{\epsilon}\int_{\rr_c}^{\phaseint(0)/\pi}\log\left(s(\xi)-s(\rr)\right)\,\dd \rr - \sum_{n=N_c}^{N-1}\log\left(s(\xi)-s(\rr_n)\right)\\
{}+\mathcal{E}_1(\SP)\epsilon^{1/2}+\bigo{\epsilon}.
\end{multline}
Analogous to the functions $G^\pm(\rr;\xi)$ defined in \eqref{Gpm-def}, we 
introduce 
\eq
H(\rr;\xi):=\log(\xi-\rr)-\log(s(\xi)-s(\rr)).
\endeq
As $H(\rr;\xi)$ is analytic in $\rr$ for $\xi$ near $\phaseint(0)/\pi$ due to the assumption that $A''(x_0)<0$,  
applying \eqref{eq:midpoint-rule} again gives
\eq
\frac{1}{\epsilon}\int_{\rr_c}^{\phaseint(0)/\pi}H(\rr;\xi)\,\dd \rr - \sum_{n=N_c}^{N-1}H(\rr_n;\xi) = \bigo{\epsilon}.
\endeq  
This allows us to replace $s(\cdot)$ with the identity with no additional error:
\begin{equation}
\log(Y_\epsilon(\SP))=\frac{1}{\epsilon}\int_{\rr_c}^{\phaseint(0)/\pi}\log\left(\xi-\rr\right)\,\dd \rr - \sum_{n=N_c}^{N-1}\log\left(\xi-\rr_n\right)+\mathcal{E}_1(\SP)\epsilon^{1/2}+\bigo{\epsilon}.
\end{equation}
Elementary calculations give
\begin{multline}
\frac{1}{\epsilon}\int_{\rr_c}^{\phaseint(0)/\pi}\log\left(\xi-\rr\right)\,\dd \rr \\
= N_c-N + (W+N-N_c)\log(W+N-N_c) - W\log(W) + (N-N_c)\log(\epsilon),
\end{multline}
where $W=(\xi-\phaseint(0)/\pi)/\epsilon = \varphi_1(\SP)/\epsilon$, and
\eq
-\sum_{n=N_c}^{N-1}\log\left(\xi-\rr_n\right) = -\sum_{m=0}^{N-N_c-1}\log\left(W+\frac{1}{2}+m\right)-(N-N_c)\log(\epsilon).
\endeq
Therefore
\begin{multline}
\log(Y_\epsilon(\SP)) =  N_c-N + (W+N-N_c)\log(W+N-N_c) \\ 
{} - W\log(W) -\sum_{m=0}^{N-N_c-1}\log\left(W+\frac{1}{2}+m\right) +\mathcal{E}_1(\SP)\epsilon^{1/2}+\bigo{\epsilon},
\end{multline}
and so
\eq
\begin{split}
Y_\epsilon(\SP) & = \ee^{N_c-N}(W+N-N_c)^{W-N-N_c}W^{-W}\prod_{m=0}^{N-N_c-1}\frac{1}{W+m+\frac{1}{2}}\left(1+\mathcal{E}_1(\SP)\epsilon^{1/2}+\bigo{\epsilon}\right)\\
 & = \ee^{N_c-N}(W+N-N_c)^{W-N-N_c}W^{-W}\prod_{m=0}^{N-N_c-1}\frac{\Gamma(W+m+\frac{1}{2})}{\Gamma(W+m+\frac{3}{2})}\left(1+\mathcal{E}_1(\SP)\epsilon^{1/2}+\bigo{\epsilon}\right)\\
 & = \ee^{N_c-N}(W+N-N_c)^{W-N-N_c}W^{-W}\frac{\Gamma(W+\frac{1}{2})}{\Gamma(W+N-N_c+\frac{1}{2})}\left(1+\mathcal{E}_1(\SP)\epsilon^{1/2}+\bigo{\epsilon}\right).
\end{split}
\endeq
Stirling's approximation of the gamma function \cite[Eq.~5.11.3]{DLMF} gives
\begin{multline}
\frac{1}{\Gamma(W+N-N_c+\frac{1}{2})} =  \frac{1}{\sqrt{2\pi}}\left(W+N-N_c+\frac{1}{2}\right)^{-(W+N-N_c)}\exp\left(W+N-N_c+\frac{1}{2}\right) \\
 \times \left(1+\bigo{\frac{1}{W+N-N_c+\frac{1}{2}}}\right),
\end{multline}
provided that $W+N-N_c+\tfrac{1}{2}$ lies outside a thin sector centered on the negative imaginary axis.  Recalling the definition \eqref{eq:Nc-define} of $N_c$ as well as the identity $\phaseint(0)=N\pi\epsilon$, we see that the condition that $|\SP-\ii A_\mathrm{max}|$ is sufficiently small guarantees that $|\epsilon W|\le\delta$ holds for some $\delta>0$ small, which in turn implies both the desired sectorial condition as well as the estimate
$(W+N-N_c+\frac{1}{2})^{-1}=\bigo{\epsilon}$.  Therefore,
\begin{multline}
Y_\epsilon(\SP) = \frac{1}{\sqrt{2\pi}}\left(\frac{W+N-N_c}{W+N-N_c+\frac{1}{2}}\right)^{W-N-N_c}W^{-W}\Gamma(W+\tfrac{1}{2})\ee^{W+\frac{1}{2}}\\
{}\cdot\left(1+\mathcal{E}_1(\SP)\epsilon^{1/2}+\bigo{\epsilon}\right).
\label{eq:top-improvement}
\end{multline}
Finally, using 
\eq
\left(\frac{n}{n+\frac{1}{2}}\right)^n = \left(1+\frac{1/2}{n}\right)^{-n} = \ee^{-1/2}\left(1+\bigo{n^{-1}}\right)
\endeq
with $n=W+N-N_c$ gives \eqref{Y-near-iAmax} and completes the proof.
\end{proof}

\Tbasicasymp*

\begin{proof}
\hypertarget{proof:Tbasicasymp}{We begin} by suitably modifying the analysis of $Y_\epsilon(\SP)$ as in the proof of Proposition~\ref{prop:Y-outside}.  We write $Y_\epsilon(\SP)=Y_\epsilon^+(\SP)/Y_\epsilon^-(\SP)$, where
\begin{equation}
Y_\epsilon^\pm(\SP):=
\exp\left(\frac{1}{\epsilon}\int_0^{\phaseint(0)/\pi}\log(-\ii(\SP\mp \ii s(\rr)))\,\dd \rr-\sum_{n=0}^{N-1}\log(-\ii(\SP \mp \ii s(\rr_n)))\right).
\end{equation}
We apply Riemann-sum arguments to $\epsilon\log(Y_\epsilon^-(\SP))$ by defining
\begin{equation}
F^-(\rr;\SP):=\log(-\ii(\SP+\ii s(\rr)))-\frac{\ii v_0\rr^{1/2}}{\SP}.
\end{equation}
By direct calculation using \eqref{eq:s-in-terms-of-v} with $v$ analytic on $[0,\phaseint(0)/\pi]$,
\begin{equation}
\begin{split}
-4\SP^{-1} (\ii s(\rr)+\SP)^2F_{\rr\rr}^-(\rr;\SP)&=
\ii \rr^{-3/2}(v(\rr)-v_0)-2\SP^{-1}v(\rr)\rr^{-1}(v(\rr)-v_0) - 4\ii \rr^{-1/2}v'(\rr)\\
&\quad{}-4\ii \rr^{1/2}v''(\rr)+4\SP^{-1}\rr(v(\rr)v''(\rr)-v'(\rr)^2)+\ii \SP^{-2}\rr^{-1/2}v_0v(\rr)^2\\
&=\bigo{\rr^{-1/2}}
\end{split}
\end{equation}
holds uniformly for $\SP\in\Lambda_\sigma$ (which in particular bounds $\SP$ away from zero) and $0<\rr<\phaseint(0)/\pi$.  Since $\ii s(\rr)+\SP$ is bounded away from zero for $\SP\in\Lambda_\sigma$ and $0<\rr<\phaseint(0)/\pi$, therefore
also $F_{\rr\rr}^-(\rr;\SP)=\bigo{\rr^{-1/2}}$, and so
\begin{equation}
\int_0^{\phaseint(0)/\pi}|F_{\rr\rr}^-(\rr;\SP)|\,\dd \rr = \bigo{1}
\end{equation}
holds uniformly for $\SP\in\Lambda_\sigma$.
Therefore, applying \eqref{eq:midpoint-rule} gives
\begin{equation}
\epsilon \log(Y_\epsilon^-(\SP))=\frac{\ii v_0}{\SP}\left[\int_0^{\phaseint(0)/\pi}\rr^{1/2}\,\dd \rr - 
\epsilon\sum_{n=0}^{N-1}\rr_n^{1/2}\right] + 
\bigo{\epsilon^2}
\end{equation}
as $\epsilon\downarrow 0$ 
uniformly for $\SP\in \Lambda_\sigma$.
Then using \eqref{eq:RS-square-root}, dividing by $\epsilon$ and exponentiating, we get 
\begin{equation}
Y_\epsilon^-(\SP)=1+\frac{\ii v_0}{\SP}\left(1-\frac{1}{\sqrt{2}}\right)\zeta(-\tfrac{1}{2})\epsilon^{1/2}+\bigo{\epsilon}.
\label{eq:Y-minus-estimate}
\end{equation}

To analyze $Y_\epsilon^+(\SP)$, we introduce the function
\begin{equation}
F^+(\rr;\SP):=\log(-\ii (\SP-\ii s(\rr)))-\log(\phaseint(0)-\phaseint(\SP)-\pi \rr)+\frac{\ii v_0\rr^{1/2}}{\SP}.
\end{equation}
Here the purpose of the second term is to compensate for the singularity of the first term that will occur should it be the case that $-\ii\SP\in (0,A_\mathrm{max})$.  By direct calculation, we have
\begin{equation}
F^+_{\rr\rr}(\rr;\SP)=\frac{s(\rr)s''(\rr)-s(\xi)s''(\rr)-s'(\rr)^2}{(s(\rr)-s(\xi))^2}+\frac{1}{(\rr-\xi)^2}-\frac{v_0}{4s(\xi)\rr^{3/2}},
\label{eq:F-plus-tt}
\end{equation}
where for $\SP\in\Lambda$ we determine a unique value $\xi$ by the relations $\SP=\ii s(\xi)$ or $\xi=(\phaseint(0)-\phaseint(\SP))/\pi$.  It is straightforward to check by Taylor expansion of $s(\rr)$ about $\rr=\xi$ that this function has a removable singularity at $\rr=\xi$.  It is also easy to check that for each $\SP\in\Lambda_\sigma$, $\rr^{1/2}F_{\rr\rr}^+(\rr;\SP)$ is an analytic function of $\rr$ for $0<\rr\le \phaseint(0)/\pi$ that has a finite 
limit as $\rr\downarrow 0$.  
Moreover $\rr^{1/2}F_{\rr\rr}^+(\rr;\SP)$ is uniformly bounded for $\SP\in\Lambda_\sigma$ and $0<\rr<\phaseint(0)/\pi$, and consequently
\begin{equation}
\int_0^{\phaseint(0)/\pi}|F_{\rr\rr}^+(\rr;\SP)|\,\dd \rr = \bigo{1}\quad\text{uniformly for $\SP\in \Lambda_\sigma$}.
\end{equation}
Applying \eqref{eq:midpoint-rule} then gives
\begin{multline}
\epsilon\log (Y_\epsilon^+(\SP))=\int_0^{\phaseint(0)/\pi}\log(\phaseint(0)-\phaseint(\SP)-\pi \rr)\,\dd \rr - 
\epsilon\sum_{n=0}^{N-1}\log(\phaseint(0)-\phaseint(\SP)-\pi \rr_n) \\
- \frac{\ii v_0}{\SP}\left[\int_0^{\phaseint(0)/\pi}\rr^{1/2}\,\dd \rr - \epsilon\sum_{n=0}^{N-1}\rr_n^{1/2}\right]
+\bigo{\epsilon^2}.
\end{multline}
Using \eqref{eq:RS-square-root} we may write this in the form
\begin{multline}
\epsilon\log (Y_\epsilon^+(\SP))=\int_0^{\phaseint(0)/\pi}\log(\phaseint(0)-\phaseint(\SP)-\pi \rr)\,\dd \rr - 
\epsilon\sum_{n=0}^{N-1}\log(\phaseint(0)-\phaseint(\SP)-\pi \rr_n) \\
-\frac{\ii v_0}{\SP}\left(1-\frac{1}{\sqrt{2}}\right)\zeta(-\tfrac{1}{2})\epsilon^{3/2} + \bigo{\epsilon^2}.
\end{multline}
Therefore, 
\begin{equation}
Y_\epsilon^+(\SP)=\frac{\displaystyle\exp\left(\frac{1}{\epsilon}\int_0^{\phaseint(0)/\pi}\log(\phaseint(0)-\phaseint(\SP)-\pi \rr)\,\dd \rr\right)}{\displaystyle \prod_{n=0}^{N-1}(\phaseint(0)-\phaseint(\SP)-\pi \rr_n)} 
\left(1-\frac{\ii v_0}{\SP}\left(1-\frac{1}{\sqrt{2}}\right)\zeta(-\tfrac{1}{2})\epsilon^{1/2}+\bigo{\epsilon}\right)
\label{eq:Y-plus-formula}
\end{equation}
holds in the limit $\epsilon\downarrow 0$ uniformly for $\SP\in\Lambda_\sigma$.  

Combining \eqref{eq:Y-minus-estimate} and \eqref{eq:Y-plus-formula} with $Y_\epsilon(\SP)=Y_\epsilon^+(\SP)/Y_\epsilon^-(\SP)$ and \eqref{eq:T-Y-relation} shows that
\begin{equation}
T_\epsilon(\SP)=\widetilde{T}_\epsilon(\phaseint(\SP))
\left(1-\frac{2\ii v_0}{\SP}\left(1-\frac{1}{\sqrt{2}}\right)\zeta(-\tfrac{1}{2})\epsilon^{1/2}+\bigo{\epsilon}\right),
\quad\epsilon\downarrow 0
\label{eq:T-tildeT}
\end{equation}
holds uniformly for $\SP\in\Lambda_\sigma$, where, using also $\rr_n=\epsilon(n+\tfrac{1}{2})$,
\begin{equation}
\widetilde{T}_\epsilon(\phaseint):=2\cos(\phaseint/\epsilon)\ee^{\pm \ii\phaseint/\epsilon}\frac{\displaystyle
\exp\left(\frac{1}{\epsilon}\int_0^{\phaseint(0)/\pi}\log(\phaseint(0)-\phaseint-\pi \rr)\,\dd \rr\right)}{\displaystyle \prod_{n=0}^{N-1}(\phaseint(0)-\phaseint-\pi \epsilon(n+\tfrac{1}{2}))}, \quad \pm\imag\{\phaseint\}>0.
\end{equation}
Evaluating the integral explicitly, using the identities $\Gamma(z+1)=z\Gamma(z)$ and $\Gamma(\tfrac{1}{2}-w)\Gamma(\tfrac{1}{2}+w)\cos(\pi w)=\pi$ (cf. \cite[Eq.~5.5.3]{DLMF}) with $N=\phaseint(0)/(\pi\epsilon)$, and also taking into account that $\log(-\phaseint)=\log(\phaseint)\mp \ii\pi$ for $\pm\imag\{\phaseint\}>0$ gives the explicit formula
\begin{equation}
\widetilde{T}_\epsilon(\phaseint)=\frac{\displaystyle 2\pi\left(\frac{\phaseint}{\pi\epsilon}\right)^{\phaseint/(\pi\epsilon)}\ee^{-\phaseint/(\pi\epsilon)}\left(\frac{\phaseint(0)-\phaseint}{\pi\epsilon}\right)^{(\phaseint(0)-\phaseint)/(\pi\epsilon)}\ee^{-(\phaseint(0)-\phaseint)/(\pi\epsilon)}}{\displaystyle\Gamma(\tfrac{1}{2}+\phaseint/(\pi\epsilon))\Gamma(\tfrac{1}{2}+(\phaseint(0)-\phaseint)/(\pi\epsilon))}.
\end{equation}
It then follows from Stirling's formula (cf. \cite[Eq.~5.11.3]{DLMF}) that for $\SP\in\Lambda_\sigma$,
\begin{equation}
\begin{split}
\widetilde{T}_\epsilon(\phaseint(\SP)) &= 1+\bigo{\epsilon\phaseint(\SP)^{-1}}+\bigo{\epsilon(\phaseint(0)-\phaseint(\SP))^{-1}} \\ &= 1+\bigo{\epsilon},
\end{split}
\end{equation}
where to get the second line we use the fact that 
$\phaseint(\SP)$ and $\phaseint(0)-\phaseint(\SP)$ are both bounded away from zero for $\SP\in\Lambda_\sigma$.
Combining this result with \eqref{eq:T-tildeT} then proves the proposition.
\end{proof}

\TatZero*

\begin{proof}
\hypertarget{proof:TatZero}{Noting} that $\re\{\varphi_0(\SP)\}$ has the same sign as $\re\{\SP\}$,
combining \eqref{eq:T-Y-relation} with Proposition~\ref{prop:Y-zeta-small} and using $\phaseint(0)=N\pi\epsilon$ for $N\in\mathbb{Z}$ 
yields \eqref{eq:T-zeta-small} with
\begin{equation}
\mathcal{T}_0(Z)=2\mathcal{Y}_0(Z)\cos(\pi Z^2)\ee^{\ii\pi Z^2\mathrm{sgn}(\mathrm{Re}\{Z\})} = 2\cos(\pi Z^2)\prod_{n=0}^\infty\frac{\sqrt{n+\tfrac{1}{2}}-\ii Z}{\sqrt{n+\tfrac{1}{2}}+\ii Z}\ee^{4\ii Z(\sqrt{n+1}-\sqrt{n})}.
\label{eq:T0-first}
\end{equation}
Then, to obtain \eqref{eq:T0-def} one simply uses the infinite product expansion of $\cos(\pi Z^2)$, cf. \cite[Eq.~4.22.2]{DLMF}.  It is straightforward to check that $T_\epsilon(\SP)$ and $\mathcal{T}_0(Z)$ take the common value of $2$ in the coincident limits $\SP\to 0$ and $Z\to 0$ respectively.
\end{proof}

\TSmallLargeZ*

\begin{proof}
\hypertarget{proof:TSmallLargeZ}{We} adapt the proof of Proposition~\ref{prop:Y0-asymp}.  To consider $Z$ small, we get from the first equality in \eqref{eq:T0-first} that $\mathcal{T}_0(Z)=2\ee^{k_2(Z)}\cos(\pi Z^2)$ where $k_2(Z)$ is the odd function analytic at the origin defined by \eqref{eq:k1k2}.  In the proof of Proposition~\ref{prop:Y0-asymp} it was shown that $k_2'(0)=-2\ii (\sqrt{2}-1)\zeta(\tfrac{1}{2})$.  This proves the claimed behavior of $\mathcal{T}_0(Z)$ near $Z=0$.  

For the behavior as $Z\to\infty$, let $\gamma>0$ be fixed and assume that $\epsilon^{1/2}|Z|\le\gamma$.  Then by Proposition~\ref{prop:T-zeta-small}
\begin{equation}
\mathcal{T}_0(Z)=T_\epsilon(\varphi_0^{-1}(\epsilon^{1/2}Z))\left(1-\mathcal{E}_0(\varphi_0^{-1}(\epsilon^{1/2}Z))\epsilon^{1/2}+\bigo{\epsilon}\right),\quad\epsilon\to 0,
\end{equation}
where $\varphi_0^{-1}$ is the inverse of the conformal map $\varphi_0$.  Assuming also that $\tfrac{1}{2}\gamma\le\epsilon^{1/2} |Z|$ and applying Proposition~\ref{prop:T-outside} to $T_\epsilon(\varphi^{-1}(\epsilon^{1/2}Z))$ gives 
\begin{multline}
\mathcal{T}_0(Z)=\left(1-\frac{2\ii v_0}{\varphi_0^{-1}(\epsilon^{1/2}Z)}\left(1-\frac{1}{\sqrt{2}}\right)\zeta(-\tfrac{1}{2})\epsilon^{1/2}+\bigo{\epsilon}\right)\\
{}\cdot\left(1-\mathcal{E}_0(\varphi_0^{-1}(\epsilon^{1/2}Z))\epsilon^{1/2}+\bigo{\epsilon}\right),\quad\epsilon\to 0.
\end{multline}
Then using the inequalities $\tfrac{1}{2}\gamma\le\epsilon^{1/2}|Z|\le\gamma$ as in the proof of Proposition~\ref{prop:Y0-asymp} gives $\mathcal{T}_0(Z)=1+\bigo{Z^{-1}}$ as $Z\to\infty$ in the indicated sector.
\end{proof}

\Ttop*

\begin{proof}
\hypertarget{proof:Ttop}{Combining} \eqref{eq:T-Y-relation} with Proposition~\ref{prop:Y-zeta-near-iAmax}, taking into account that $\pm\re\{\SP\}>0$ corresponds to $\mp\imag\{\varphi_1(\SP)\}>0$ yields \eqref{T-near-iAmax} with
\begin{equation}
\mathcal{T}_1(W)=\sqrt{\frac{2}{\pi}}\cos(\pi W)\Gamma(W+\tfrac{1}{2})e^WW^{-W}\ee^{\ii\pi W\mathrm{sgn}(\imag\{W\})}.
\end{equation}
But in terms of principal branches, $W^{-W}\ee^{\ii\pi W\mathrm{sgn}(\imag\{W\})}=(-W)^{-W}$, so using $\Gamma(\tfrac{1}{2}-W)\Gamma(\tfrac{1}{2}+W)\cos(\pi W)=\pi$ (cf. \cite[Eq.~5.5.3]{DLMF}) the formula \eqref{eq:T1-def} follows.
\end{proof}

\section{Proofs of the properties of $g(\SP;x)$ and $h(\SP;x)$}
\label{sec-g-prop}
\gprop*

\begin{proof}
\hypertarget{proof-g-prop}{To prove} G1, the analyticity of $g^\pm(\SP;x)$ in the indicated domain is obvious from the definition \eqref{eq:g-functions-zeta}; the same formula shows that $g^\pm(\SP;x)$ is uniformly bounded at least for $\SP$ bounded away from the branch cut $-A(x)\le -\ii\SP\le A(x)$.  On the other hand, using \eqref{eq:g-functions-zeta-inside} shows that $g^\pm(\SP;x)$ is continuous up to its branch cut, so G1 is established.  

To prove G2, note that oddness of $g^\pm(\SP;x)$ is obvious from the first line of \eqref{eq:g-functions-zeta} given that $R(\SP;x)$ is an odd function of $\SP$.  The Schwarz symmetry property G3 also follows immediately from \eqref{eq:g-functions-zeta} given that $R(\SP^*;x)=R(\SP;x)^*$.

Property G4 is a direct consequence of the formula \eqref{eq:g-functions-zeta-inside}, upon taking into account \eqref{eq:phi-pm-define} and
the fact that $R(\SP;x)$ changes sign across the branch cut.  Combining G1 and G2 proves that $g^\pm(\SP;x)=O(\SP^{-1})$ as $\SP\to\infty$, i.e., property G5.  

To prove property G6, we take $g^\pm(\SP;x)$ in the form \eqref{eq:g-functions-zeta-inside} to allow $\SP$ near $\ii A(x)$ and identify $G_1^\pm(\SP;x)$ with $\phi^\pm(-\ii\SP;x)/(4\ii\SP)$ analytic at $\SP=\ii A(x)$.  Since the loop integral over $L$ in \eqref{eq:g-functions-zeta-inside} is analytic near $\SP=\ii A(x)$ and since $R(\SP;x)$ vanishes to order $1/2$ at $\SP=\ii A(x)$, it remains to show that the integral vanishes for $\SP=\ii A(x)$ under the condition $\pm (x-x_0)>0$ (and $x\in (X_-,X_+)$; otherwise $g^\pm(\SP;x)=0$ and the result holds trivially).  That is, we need to show that $\pm(x-x_0)>0$ and $x\in (X_-,X_+)$ implies $M_0^\pm(x)=0$, where
\begin{equation}
M_0^\pm(x):=\oint_L\frac{\phi^\pm(s;x)\,\dd s}{R(\ii s;x)(s^2-A(x)^2)}=-\oint_L\frac{\phi^\pm(s;x)\,\dd s}{R(\ii s;x)^3}.
\end{equation}
Recall $L$ is positively oriented and surrounds the branch cut of $R$.
Noting the identity 
\begin{equation}
\frac{\dd}{\dd s}\frac{1}{R(\ii s;x)}=\frac{s}{R(\ii s;x)^3},
\end{equation}
we integrate by parts to obtain
\begin{equation}
M_0^\pm(x)=\oint_L\frac{\dd}{\dd s}\left(\frac{\phi^\pm(s;x)}{s}\right)\frac{\dd s}{R(\ii s;x)}.
\end{equation}
The integrand is now integrable at $s=\pm A(x)$, so the loop $L$ can be contracted to the interval $[-A(x),A(x)]$.  Using also that $\phi^\pm(s;x)$ is an even function of $s$ yields
\begin{equation}
M_0^\pm(x)=2\int_{-A(x)}^{A(x)}\frac{\dd}{\dd s}\left(\frac{\phi^\pm(s;x)}{s}\right)\frac{\dd s}{\sqrt{A(x)^2-s^2}} = 4\int_0^{A(x)}\frac{\dd}{\dd s}\left(\frac{\phi^\pm(s;x)}{s}\right)\frac{\dd s}{\sqrt{A(x)^2-s^2}},
\end{equation}
so from \eqref{eq:phi-pm-define} we get
\begin{equation}
M_0^\pm(x)=8\int_0^{A(x)}\frac{2x-\ii\tailint'(\ii s)\mp \ii\overline{L}'(\ii s)}{\sqrt{A(x)^2-s^2}}\,\dd s.
\end{equation}
Then, using \eqref{eq:mu-plus-Lbar-cpt},
\begin{equation}
\begin{split}
M_0^+(x)&=16\int_0^{A(x)}\left[x-X_++\int_{x_+(s)}^{X_+}\frac{s\,\dd y}{\sqrt{s^2-A(y)^2}}\right]\frac{\dd s}{\sqrt{A(x)^2-s^2}}\\
&=16\int_0^{A(x)}\left[-\int_x^{X_+}\,\dd y +\int_{x_+(s)}^{X_+}\frac{s\,\dd y}{\sqrt{s^2-A(y)^2}}\right]\frac{\dd s}{\sqrt{A(x)^2-s^2}}.
\end{split}
\end{equation}
If $x_0<x<X_+$, then exchanging the order of integration yields
\begin{equation}
M_0^+(x)=16\int_x^{X_+}\left[\int_{A(y)}^{A(x)}\frac{s\,\dd s}{\sqrt{s^2-A(y)^2}\sqrt{A(x)^2-s^2}}
-\int_0^{A(x)}\frac{\dd s}{\sqrt{A(x)^2-s^2}}\right]\,\dd y.
\end{equation}
Both of the inner integrals can be computed exactly and they are both equal to $\pi/2$, hence if $x_0<x<X_+$ we deduce that $M_0^+(x)=0$.
Similarly,
\begin{equation}
\begin{split}
M_0^-(x)&=16\int_0^{A(x)}\left[x-X_--\int_{X_-}^{x_-(s)}\frac{s\,\dd y}{\sqrt{s^2-A(y)^2}}\right]\frac{\dd s}{\sqrt{A(x)^2-s^2}}\\
&= 16\int_0^{A(x)}\left[\int_{X_-}^x\,\dd y-\int_{X_-}^{x_-(s)}\frac{s\,\dd y}{\sqrt{s^2-A(y)^2}}\right]\frac{\dd s}{\sqrt{A(x)^2-s^2}}.
\end{split}
\end{equation}
Under the assumption that $X_-<x<x_0$, exchanging the integration order gives
\begin{equation}
M_0^-(x)=16\int_{X_-}^x\left[\int_0^{A(x)}\frac{\dd s}{\sqrt{A(x)^2-s^2}}-\int_{A(y)}^{A(x)}\frac{s\,\dd s}{\sqrt{s^2-A(y)^2}\sqrt{A(x)^2-s^2}}\right]\,\dd y
\end{equation}
and again the inner integrals cancel, yielding $M_0^-(x)=0$ for $X_-<x<x_0$.

To prove G7, we may start with \eqref{eq:g-functions-zeta-outside} and simply observe that as $x$ tends to either support endpoint from within $(X_-,X_+)$, $R(\SP;x)\to \SP$ uniformly for $\SP$ bounded away from the origin.  Hence for each given $\SP\neq 0$ we fix an integration contour $L$ surrounding the interval $[-A(x),A(x)]$ with $\pm \ii\SP$ on the exterior, and observe that the integrand in \eqref{eq:g-functions-zeta-outside} converges uniformly on $L$ to a function analytic on the interior of $L$.  Hence the integral converges to zero by Cauchy's theorem and the prefactor $R(\SP;x)/(8\pi \ii)$ remains bounded in the limit, which proves that $g^\pm(\SP;x)\to 0$.

To establish property G8, we differentiate the formula \eqref{eq:g-zeta-plus-minus} with respect to $x$, noting that $\tailint(\ii s)\pm \overline{L}(\ii s)$ is independent of $x$.  Thus $g^\pm_x(\SP;x)$ is a function of $\SP$ analytic in the same domain as $g(\SP;x)$ itself, that is bounded, that satisfies a natural analogue of the Schwarz symmetry property G3, and that satisfies the differentiated boundary condition
\begin{equation}
g_{x+}^\pm(\ii s;x)+g_{x-}^\pm(\ii s;x)=-2s,\quad 0<s<A(x).
\end{equation}
Therefore, $g_x^\pm(\SP;x)$ necessarily has the form
\begin{equation}
\begin{split}
g_x^\pm(\SP;x)&=\frac{R(\SP;x)}{2\pi \ii}\left[\int_0^{A(x)}\frac{2s\,\dd s}{\sqrt{A(x)^2-s^2}(s+\ii\SP)}
+\int_0^{A(x)}\frac{2s\,\dd s}{\sqrt{A(x)^2-s^2}(s-\ii\SP)}\right]\\
&=\frac{R(\SP;x)}{\pi \ii}\int_{-A(x)}^{A(x)}\frac{s^2\,\dd s}{\sqrt{A(x)^2-s^2}(s^2+\SP^2)}\\
&=\frac{R(\SP;x)}{2\pi \ii}\oint_L\frac{s^2\,\dd s}{R(\ii s;x)(s^2+\SP^2)},\quad\text{$\ii\SP$ and $-\ii\SP$ exterior to $L$.}
\end{split}
\end{equation}
Evaluating this latter integral by residues at $s=\pm \ii\SP$ and $s=\infty$ shows that $g_x^\pm(\SP;x)=\ii(\SP-R(\SP;x))$, giving the claimed result.

Finally, to prove property G9, note that combining properties G7 and G8 gives 
\eqref{eq:gplus-gminus-x-integral}.  
Since for a semicircular Klaus-Shaw potential with support $[X_-,X_+]$ the formul\ae \eqref{eq:Lfunc-def}--\eqref{eq:mu-plus-minus-define} together with the definition of $R(\SP;x)$ yield
\begin{equation}
L(\SP)=-\ii\int_{X_-}^{X_+}(\SP-R(\SP;y))\,\dd y,
\end{equation}
we obtain the identity \eqref{eq:gplus-gminus-L-identity}.
\end{proof}

\hprop*

\begin{proof}
\hypertarget{proof-h-prop}{To prove} property H1, 
note that from the second line of \eqref{eq:h-def-zeta}, $h^\pm(\SP;x)$ can be written in the form 
\begin{equation}
h^\pm(\SP;x)=\frac{R(\SP;x)}{2\pi \ii}H^\pm(\SP;x),
\label{eq:hH}
\end{equation}
where 
$H^\pm(\SP;x)$ is analytic in $\SP$ for $-A_\mathrm{max}<-\ii\SP<A_\mathrm{max}$.  Observe that for fixed $x\neq x_0$, $H^\pm(\ii A(x);x)=0$ when $\pm(x-x_0)>0$ (by property G6 of Proposition~\ref{prop:g-properties}).  We now show that 
\begin{equation}
\ii\frac{\partial H^\pm}{\partial \SP}(\SP;x)>0,\quad \pm(x-x_0)>0,\quad 0<-\ii\SP<A_\mathrm{max},
\label{eq:Hpm-derivative-inequality}
\end{equation}
that is, $H^\pm(\SP;x)$ is real and strictly increasing upwards along the imaginary $\SP$-axis provided $\pm(x-x_0)>0$.  In particular, the root of $H^\pm(\SP;x)$ at $\SP=\ii A(x)$ is a simple zero.  To prove this, we first obtain a simple formula for $H^\pm(\SP;x)$ by integrating with respect to $x$ the identity
\begin{equation}
h^\pm_x(\SP;x)=\pm\left(g^\pm_x(\SP;x)-\ii\SP\right)=\mp\ii R(\SP;x),
\label{eq:h-x-derivative}
\end{equation}
which follows from the first line of \eqref{eq:h-def-zeta} and property G8 of Proposition~\ref{prop:g-properties}.
Since we have $H^\pm(\SP;x_\pm(-\ii\SP))=0$ (as an equivalent way of writing $H^\pm(\ii A(x);x)=0$ for $\pm (x-x_0)>0$), and also $R(\SP;x_\pm(-\ii \SP))=0$ for $0<-\ii\SP<A_\mathrm{max}$, it follows that $h^\pm(\SP;x_\pm(-\ii\SP))=0$ for $\pm(x-x_0)>0$.  Therefore, for $x\in [X_-,X_+]$,
\begin{equation}
h^\pm(\SP;x)=\mp \ii\int_{x_\pm(-\ii\SP)}^x R(\SP;y)\,\dd y,\quad \pm (x-x_0)>0,\quad A(x)<-\ii\SP<A_\mathrm{max},
\label{eq:h-integral}
\end{equation}
from which it follows that
\begin{equation}
H^\pm(\SP;x)=\pm\frac{2\pi}{R(\SP;x)}\int_{x_\pm(-\ii\SP)}^x R(\SP;y)\,\dd y,\quad\pm (x-x_0)>0,\quad A(x)<-\ii\SP<A_\mathrm{max}.
\label{eq:H-integral}
\end{equation}
In these formul\ae, the lower limit of integration is a real value between $x_0$ and $x$ under the indicated assumption that $A(x)<-\ii\SP<A_\mathrm{max}$.
However, for semicircular Klaus-Shaw potentials $A$, the turning points (inverse function branches) $x_\pm(s)$ are analytic on the interval $0<s<A_\mathrm{max}$, and hence the formula \eqref{eq:h-integral} for $h^\pm(\SP;x)$ can be analytically continued to a domain of the form $\delta<\imag\{\SP\}<A_\mathrm{max}-\delta$ and $|\re\{\SP\}|<\delta$ omitting the vertical branch cut connecting $\pm \ii A(x)$; it only becomes necessary to replace the real integration with a complex contour connecting $x_\pm(-\ii\SP)$ with the real value $x$.  In the case of the formula \eqref{eq:H-integral}, the two boundary values taken on the cut necessarily agree as it has already been shown that $H^\pm(\SP;x)$ is analytic at $\SP=\ii A(x)$ for $\pm(x-x_0)>0$, making $H^\pm(\SP;x)$ an analytic function of $\SP$ in the domain $\delta<\imag\{\SP\}<A_\mathrm{max}-\delta$ and $|\re\{\SP\}|<\delta$.
Differentiation with respect to $\SP$ using Leibniz' rule yields
\begin{equation}
\ii\frac{\partial H^\pm}{\partial\SP}(\SP;x)=\mp\frac{2\pi \ii\SP}{R(\SP;x)^3}\int_{x_\pm(-\ii\SP)}^x
\frac{A(y)^2-A(x)^2}{R(\SP;y)}\,\dd y,\quad \pm (x-x_0)>0.
\label{eq:Hpm-derivative-expression}
\end{equation}
The strict inequality \eqref{eq:Hpm-derivative-inequality} now follows from \eqref{eq:Hpm-derivative-expression}.  For example, consider the case $x>x_0$.  If also $A(x)<-\ii\SP<A_\mathrm{max}$, then $x_0<x_+(-\ii\SP)<y<x$ for the integral in \eqref{eq:Hpm-derivative-expression}, so $A(y)>A(x)>0$ since $A'(x)<0$ for $x>x_0$ and also, $\SP$, $R(\SP;x)$, and $R(\SP;y)$ are all positive imaginary, confirming \eqref{eq:Hpm-derivative-inequality}.  On the other hand, if $0<-\ii\SP<A(x)$, then instead $x_0<x<y<x_+(-\ii\SP)$ for the integral in \eqref{eq:Hpm-derivative-expression}, so $A(x)>A(y)>0$ and taking the boundary value from the right half-plane (arbitrarily, since $H^\pm(\SP;x)$ has no jump discontinuity on the imaginary axis) we see that $R(\SP;x)$ and $R(\SP;y)$ are positive real while $\SP$ remains positive imaginary, confirming \eqref{eq:Hpm-derivative-inequality} again.  Finally, taking the limit $-\ii\SP\downarrow A(x)$ gives
\begin{equation}
\ii\frac{\partial H^+}{\partial\SP}(\ii A(x);x)=-\frac{2}{3A'(x)}>0,\quad x>x_0
\end{equation}
confirming \eqref{eq:Hpm-derivative-inequality} in the (most important for our purposes) boundary case.  The argument for $x<x_0$ is similar.  Since $H^\pm(\SP;x)$ is analytic in a suitable neighborhood $D(x)$ of $\SP=\ii A(x)$ at which point it has a simple zero, property H1 is confirmed.

Property H2 also follows from the representation \eqref{eq:hH}, the fact that $H^\pm(\ii A(x);x)=0$, and the inequality \eqref{eq:Hpm-derivative-inequality}.  These show immediately that $h^\pm(\SP;x)$ is positive and strictly increasing in the positive imaginary direction along the imaginary axis above $\SP=\ii A(x)$.  To obtain the corresponding inequalities on $\re\{h^\pm(\SP;x)\}$, one notes that the boundary values taken by $h^\pm(\SP;x)$ on the imaginary branch cut below $\SP=\ii A(x)$ are themselves purely imaginary and monotone, from which the desired inequalities are consequences of the Cauchy-Riemann equations.  Uniformity for $x\in J_c^\pm$ holds by continuity of $h^\pm(\SP;x)$ as $J_c^\pm$ is a compact subset of $J^\pm$.

To prove property H3, note first that the formul\ae \eqref{eq:gplus-gminus-x-integral} following from properties G7 and G8 of Proposition~\ref{prop:g-properties} allow us to characterize the difference of boundary values taken by $g^\pm(\SP;x)$ when $0<-\ii \SP<A(x)$, assuming that $\pm (x-x_0)>0$.  Indeed, if we use the subscript $+$ (resp., $-$) to denote the boundary value taken from the right (resp., left) half-plane, we can derive the following formula:
\begin{equation}
\begin{split}
g^\pm_+(\SP;x)-g^\pm_-(\SP;x)&=2\ii\int_x^{x_\pm(-\ii\SP)}R_+(\SP;y)\,\dd y \\ &= 2\ii\int_x^{x_\pm(-\ii\SP)}\sqrt{\SP^2+A(y)^2}\,\dd y,\\ &\qquad\qquad \pm (x-x_0)>0,\quad 0<-\ii\SP<A(x).
\end{split}
\end{equation}
Comparing with \eqref{eq:ZS-phase-integral}, we see that
\begin{multline}
2\phaseint(\SP)\mp\left(-\ii[g^\pm_+(\SP;x)-g^\pm_-(\SP;x)]\right)=\pm 2\int_{x_\mp(-i\SP)}^x\sqrt{\SP^2+A(y)^2}\,\dd y >0,\\ \pm (x-x_0)>0,\quad 0<-\ii\SP<A(x).
\end{multline}
Applying Leibniz' rule to differentiate this formula gives
\begin{multline}
\ii\frac{\partial}{\partial\SP}\left[2\phaseint(\SP)\mp\left(-\ii[g^\pm_+(\SP;x)-g^\pm_-(\SP;x)]\right)
\right]
= \pm 2\ii\SP\int_{x_\mp(-\ii\SP)}^x\frac{\dd y}{\sqrt{\SP^2+A(y)^2}}<0,\\ \pm (x-x_0)>0,\quad 0<-\ii\SP<A(x),
\label{eq:gdiff-derivative-inequality}
\end{multline}
indicating that the positive quantity in square brackets is strictly decreasing as $-\ii\SP$ increases from $0$ to $A(x)$.  The derivative is strictly negative even in the limit $-\ii\SP\uparrow A(x)$.

Now note that \eqref{eq:g-zeta-plus-minus} and the first line of \eqref{eq:h-def-zeta} immediately imply property H6.  
In turn, this implies the boundary values taken by $h^\pm(\SP;x)$ on the cut are purely imaginary, and \eqref{eq:gdiff-derivative-inequality} can be equivalently written in two ways:
\begin{multline}
\ii\frac{\partial}{\partial\SP}\left[2\phaseint(\SP)-(-2\ii h_+^\pm(\SP;x))\right]<0\quad\text{and}\quad
\ii\frac{\partial}{\partial\SP}\left[2\phaseint(\SP)+(-2\ii h_-^\pm(\SP;x))\right]<0,\\ \pm(x-x_0)>0,\quad 0<-\ii\SP<A(x),
\label{eq:upper-constraint}
\end{multline}
with the inequalities being strict even in the limit $-\ii\SP\uparrow A(x)$.  
Property H3 then follows from H2 and a Cauchy-Riemann argument applied to \eqref{eq:upper-constraint}.  Again, uniformity of the estimates for $x\in J_c^\pm$ follows from continuity.  

To prove property H4, firstly note that the analyticity of the boundary values and the fact that they sum to zero both follow immediately from \eqref{eq:hH} because $R(\SP;x)$ changes sign across the branch cut.  Now by the second line of \eqref{eq:h-def-zeta} it is obvious that $H^\pm(\SP;x)$ is an even analytic function of $\SP$ and hence its power series at $\SP=0$ consists of only even powers of $\SP$.  While $R(\SP;x)$ is an odd function of $\SP$, its boundary value $R_+(\SP;x)$ taken from the right half-plane can be written in terms of the principal branch square root as $R_+(\SP;x)=(A(x)^2+\SP^2)^{1/2}$, and hence is an even analytic function of $\SP$ as well.    It remains to calculate the first two terms of the Taylor expansion about $\SP=0$ of $h^\pm_+(\SP;x)$.  Clearly 
\begin{equation}
R_+(\SP;x)=A(x) +\SP^2/(2A(x)) + \bigo{\SP^4},\quad \SP\to 0.
\label{eq:R-plus-expand}
\end{equation}
For $H^\pm(\SP;x)$ we use the second line of \eqref{eq:h-def-zeta} to get 
\begin{equation}
\begin{split}
H^\pm(\SP;x)&=\pm\frac{1}{4}\oint_L\frac{\phi^\pm(s;x)\,\dd s}{R(\ii s;x)(s^2+\SP^2)} \\ &= 
\pm\frac{1}{4}\oint_L\frac{\phi^\pm(s;x)\,\dd s}{s^2R(\ii s;x)} \mp \frac{\SP^2}{4}\oint_L\frac{\phi^\pm(s;x)\,\dd s}{s^4R(\ii s;x)} +\bigo{\SP^4},\quad \SP\to 0.
\end{split}
\end{equation}
Therefore, using \eqref{eq:hH}, 
\begin{equation}
\begin{split}
h_+^\pm(\SP;x)&=\pm\frac{A(x)}{8\pi \ii}\oint_L\frac{\phi^\pm(s;x)\,\dd s}{s^2R(\ii s;x)} \\&\quad\quad{}\pm\frac{\SP^2}{8\pi \ii}\left[\frac{1}{2A(x)}\oint_L\frac{\phi^\pm(s;x)\,\dd s}{s^2R(\ii s;x)} - A(x)\oint_L\frac{\phi^\pm(s;x)\,\dd s}{s^4R(\ii s;x)}\right] + \bigo{\SP^4}\\
&=\pm\frac{A(x)}{8\pi \ii}\oint_L\frac{\phi^\pm(s;x)\,\dd s}{s^2R(\ii s;x)} \\&\quad\quad{}\mp\frac{\SP^2}{16\pi \ii A(x)}\left[
\oint_L\frac{\phi^\pm(s;x)\,\dd s}{s^2R(\ii s;x)} + 2\oint_L\frac{R(\ii s;x)\phi^\pm(s;x)}{s^4}\,\dd s\right] + \bigo{\SP^4},\quad\SP\to 0.
\end{split}
\end{equation}
Now we recall the definition \eqref{eq:phi-pm-define}, into which we may substitute from the right-hand side of \eqref{eq:mu-pm-Lbar}.  Therefore,
\begin{equation}
\begin{split}
\oint_L\frac{\phi^\pm(s;x)\,\dd s}{s^2R(\ii s;x)}&= 4(x-X_\pm)\oint_L\frac{\dd s}{R(\ii s;x)} -4\oint_L \frac{s^2}{R(\ii s;x)} \int_0^1\sqrt{1-z}\,x'(s^2z)\,\dd z\,\dd s\\
&= 8\pi (x-X_\pm) -8\int_0^1\sqrt{1-z}\int_{-A(x)}^{A(x)}\frac{s^2x'(s^2z)}{\sqrt{A(x)^2-s^2}}\,\dd s\,\dd z,
\end{split}
\end{equation}
a purely real expression in which $x'(y)$ denotes the derivative of the branch of the inverse function $y=A(x)^2$ for which $\pm(x-x_0)>0$.  We observe that $\pm x'(y)<0$ holds strictly for $0\le y<A_\mathrm{max}^2$.  Similarly,
\begin{multline}
\oint_L\frac{R(\ii s;x)\phi^\pm(s;x)}{s^4}\,\dd s \\
\begin{aligned}
&= 4(x-X_\pm)\oint\frac{R(\ii s;x)\,\dd s}{s^2} -4\oint_L R(\ii s;x)\int_0^1\sqrt{1-z}\,x'(s^2z)\,\dd z\,\dd s\\
&=-8\pi (x-X_\pm)-8\int_0^1\sqrt{1-z}\int_{-A(x)}^{A(x)}\sqrt{A(x)^2-s^2}\,x'(s^2z)\,\dd s\,\dd z.
\end{aligned}
\end{multline}
Therefore,
\begin{multline}
\oint_L\frac{\phi^\pm(s;x)\,\dd s}{s^2R(\ii s;x)}+2\oint_L\frac{R(\ii s;x)\phi^\pm(s;x)}{s^4}\,\dd s = \\
-8\pi(x-X_\pm)-8\int_0^1\sqrt{1-z}\int_{-A(x)}^{A(x)}\left[\frac{s^2}{\sqrt{A(x)^2-s^2}}+2\sqrt{A(x)^2-s^2}\right]x'(s^2z)\,\dd s\,\dd z,
\end{multline}
an expression in which both terms are real and nonzero and have exactly the same sign, namely that of $x-x_0$.  Therefore property H4 holds pointwise for $x\in J_c^\pm$, and the uniformity of the inequality $\beta^\pm(x)>0$ follows by continuity.  

Finally, all of the statements in property H5 follow from H4, with the exception of the inequality $\beta^\pm(x)-\phaseint_1<0$.  To prove this, first note that the opposite inequality $\beta^\pm(x)-\phaseint_1>0$ would be in contradiction with \eqref{eq:upper-constraint} taken in a neighborhood of $\SP=0$ on the positive imaginary axis.  Therefore $\beta^\pm(x)-\phaseint_1\le 0$ and it remains to rule out the possibility of zeros.  For this purpose, it is sufficient to show that $\beta^\pm(x)-\phaseint_1$ is monotone for $\pm(x-x_0)>0$.  Using \eqref{eq:h-x-derivative}
and \eqref{eq:R-plus-expand} we get
\begin{equation}
\frac{\partial}{\partial x}\left[h^\pm_+(\SP;x)-\ii\phaseint(\SP)\right]=\frac{\partial h^\pm_+}{\partial x}(\SP;x)=\mp \ii R_+(\SP;x)=\mp \ii A(x)\mp\frac{\ii\SP^2}{2A(x)} + \bigo{\SP^4},\quad\SP\to 0.
\end{equation}
Furthermore,
\begin{equation}
\frac{\partial}{\partial x}\left[h^\pm_+(\SP;x)-\ii\phaseint(\SP)\right]=\frac{\partial}{\partial x}\left[\ii(\alpha^\pm(x)-\phaseint_0) + \ii(\beta^\pm(x)-\phaseint_1)\SP^2 + \bigo{\SP^4}\right],
\end{equation}
from which we deduce that $(\beta^\pm(x)-\phaseint_1)_x=\mp 1/(2A(x))\neq 0$, and the proof is complete.  The pointwise strict inequality $\beta^\pm(x)-\phaseint_1<0$ is uniform for $x$ in the compact set $J_c^\pm$ by continuity.
\end{proof}